\documentclass[12pt]{article}
\usepackage{amsthm, amsmath, amssymb} 
\usepackage{graphicx}
\usepackage{enumerate}
\usepackage{natbib}
\usepackage{url} 
\usepackage{float}
\usepackage[colorlinks=true, allcolors=blue]{hyperref}
\usepackage{multirow}
\usepackage{makecell}
\usepackage{rotating}
\usepackage{subcaption} 
\usepackage{enumitem} 
\usepackage{xcolor} 
\usepackage{caption}

\newcommand{\blind}{1}

\newcommand{\TPR}{\textcolor{black}}
\newcommand{\TPM}{\textcolor{black}}
\numberwithin{equation}{section}
\newtheorem{thm}{Theorem}[section]
\newtheorem{prop}{Proposition}[section]

\newtheorem{lem}{Lemma}[section]
\newtheorem{define}{Definition}[section]
\newtheorem{eg}{Example}[section]
\newtheorem{remark}{Remark}[section]

\makeatletter
\renewenvironment{proof}[1][\proofname]{\par
  \pushQED{\qed}%
  \normalfont \topsep6\p@\@plus6\p@\relax
  \trivlist
  \item[\hskip\labelsep
        \itshape
    #1\@addpunct{}]%
}{%
  \popQED\endtrivlist\@endpefalse
}
\makeatother
\renewcommand{\proofname}{}

\usepackage{algorithm}
\usepackage[compatible]{algpseudocode} 
\makeatletter
\newenvironment{breakablealgorithm}
{
	\begin{center}
		\refstepcounter{algorithm}
		\hrule height0.8pt depth0pt \kern0pt
		\renewcommand{\caption}[2][\relax]{
			{\raggedright\textbf{\ALG@name~\thealgorithm} ##2\par}%
			\ifx\relax##1\relax 
			\addcontentsline{loa}{algorithm}{\protect\numberline{\thealgorithm}##2}%
			\else 
			\addcontentsline{loa}{algorithm}{\protect\numberline{\thealgorithm}##1}%
			\fi
			\kern2pt\hrule\kern2pt
		}
	}{
		\kern0pt\hrule\relax
	\end{center}
}
\makeatother

\newcommand{\mf}{\mathbf}

\begin{document}

\def\spacingset#1{\renewcommand{\baselinestretch}%
{#1}\small\normalsize} \spacingset{1}

\newcommand{\setword}[2]{%
  \phantomsection
  #1\def\@currentlabel{\unexpanded{#1}}\label{#2}%
}


\if1\blind
{
  \title{\bf  Simultaneous  inference for  monotone and smoothly time-varying functions under complex temporal dynamics}
  \author{Tianpai Luo\footnote{E-mail addresses: ltp21@mails.tsinghua.edu.cn(T.Luo), wuweichi@mail.tsinghua.edu.cn(W.Wu). Weichi Wu is the corresponding author and acknowledges NSFC 12271287.} \qquad Weichi Wu\footnotemark[1]\vspace{0.2cm}\\ \footnotemark[1] Department of Statistics and Data Science,\\ Tsinghua University, Beijing 100084, China}
    \date{}
  \maketitle
} \fi

\if0\blind
{
  \bigskip
  \bigskip
  \bigskip
  \begin{center}
    {\LARGE\bf  Simultaneous  inference for  monotone and smoothly time-varying functions under complex temporal dynamics} 
\end{center}
  \medskip
} \fi

\bigskip
\begin{abstract}

We propose a new framework for the simultaneous inference of monotone and smoothly time-varying functions under complex temporal dynamics. This will be done utilizing the monotone rearrangement and the nonparametric estimation. We capitalize the Gaussian approximation for the nonparametric monotone estimator and construct the asymptotically correct simultaneous confidence bands (SCBs) using designed bootstrap methods. We investigate two general and practical scenarios. The first is the simultaneous inference of monotone smooth trends from moderately high-dimensional time series. The proposed algorithm has been employed for the joint inference of temperature curves from multiple areas. Specifically, most existing methods are designed for a single monotone smooth trend. In such cases, our proposed SCB empirically exhibits the narrowest width among existing approaches while maintaining confidence levels. It has also been used for testing several hypotheses tailored to global warming. The second scenario involves simultaneous inference of monotone and smoothly time-varying regression coefficients in time-varying coefficient linear models. The proposed algorithm has been utilized for testing the impact of sunshine duration on temperature which is believed to be increasing due to severe greenhouse effect. The validity of the proposed methods has been justified in theory as well as by extensive simulations. 
\end{abstract}  



\noindent%
\vfill

\newpage
\spacingset{1.5} 

\section{Introduction}
\label{Introduction}
Estimating functions under a monotone constraint is fundamental within the sphere of shape constraint problems, offering diverse applications in fields including global warming (e.g., \cite{woodward1993global}), environmental investigation (e.g., \cite{hussian2005monotonic}), biostatistics (e.g., \cite{Mario2002bio}), and financial analysis (e.g., \cite{ait2003nonparametric}). Over the past two decades, research on monotone regression has attracted increasing research interest, developing from traditional isotonic regression (e.g., \cite{brunk1969estimation}), to refined smooth estimators (e.g., \cite{mukerjee1988monotone}; \cite{mammen1991estimating}; \cite{dette2006simple}); and from univariate function as in the above mentioned works to multivariate functions (e.g., \cite{chatterjee2018matrix}; \cite{deng2021confidence}). Despite these advancements, the majority of existing approaches rely heavily on the assumptions of independence or stationarity,  
which often falls short in long-span time series datasets as discussed in \cite{dette2022prediction}.  Meanwhile, the smooth and increasing trends have been widely identified in climate change.   
A prominent example is \cite{nature2018}, which shows global temperature has a smooth and increasing trend using climate models. \cite{nature2018} further warned of accelerated global warming according to this trend.  In our real data analysis, we also identify the smooth and increasing trends as well as nonstationary patterns in temperature data, see Figure \ref{fig:monotone+nonstationary} for example. \TPM{Moreover, we utilize monotonicity tests proposed by \cite{chetverikov2018econometrics}, \cite{Bowman1998testmono}, and \cite{birke2007testing}  for the Heathrow’s temperature trend in Figure \ref{fig:monotone+nonstationary}. The tests from \cite{chetverikov2018econometrics} and \cite{Bowman1998testmono} are under the nondecreasing null hypothesis while \cite{birke2007testing} considers the strictly increasing hypothesis.
The tests yield the monotonicity hypothesis with $p$-values $0.35$, $0.20$ and $0.21$ respectively, 
which justifies the need for the estimation of smooth and monotone regression functions for real-world data in practice.}
\begin{figure}[htbp]
  \centering
  \begin{subfigure}[b]{0.4\textwidth}
    \centering
    \includegraphics[width=0.75\textwidth]{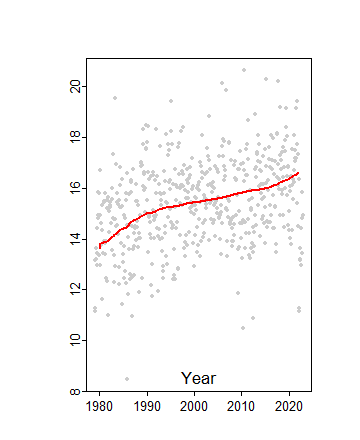} 
    \captionsetup{font={small, stretch=1}}
    \caption{Monotone trend }
    \label{fig:monotone trend}
  \end{subfigure}
  \hfill
  \begin{subfigure}[b]{0.4\textwidth}
    \centering
    \includegraphics[width=\textwidth]{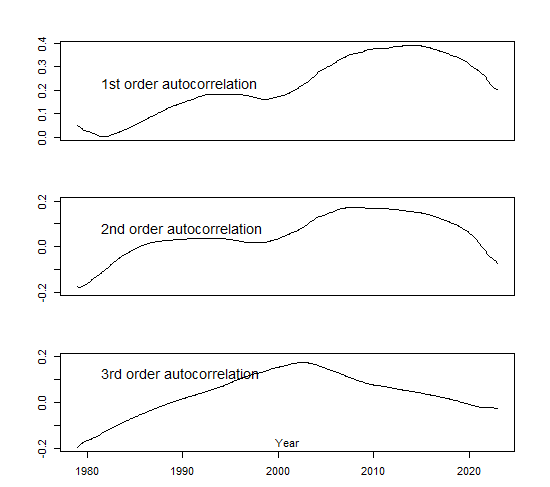} 
    \captionsetup{font={small, stretch=1}}
    \caption{Nonstationarity}
    \label{fig:nonstationary}
  \end{subfigure}
  \captionsetup{font={small, stretch=1}}
  \caption{\small (a) The maximum temperature (°C) data at Heathrow with a fitted smooth time-varying trend, see Section \ref{Empirical study}; (b) Estimated autocorrelation functions by the method in \cite{dette2022prediction} based on data in (a). The time-varying pattern implies the nonstationarity.}
  \label{fig:monotone+nonstationary}
\end{figure}

In this study, our goal is to estimate and make inferences on the monotone and smoothly time-varying functions under complex temporal dynamics. In particular, we consider two very important scenarios where our methodology demonstrates its effectiveness. One is the estimation and jointly simultaneous inference of monotone trends from moderately high-dimensional nonstationary processes. This capability proves especially valuable when dealing with multiple time series exhibiting monotone behavior. These include temperature data collected from various weather stations within a geographic region. The other is the estimation and inference of the monotone time-varying coefficients, or more generally, monotone linear combination of time-varying coefficients in time-varying coefficient linear models. This enables us to assess the monotonically changing relationship between the response variable and predictor variables as time evolves. The time-varying coefficient linear models with monotone condition is useful in climate science, see for example \cite{SubhraWu2023}. 

Our proposed simultaneous confidence band (SCB), which centers around the monotone estimator, is asymptotically correct. 
While some recent studies have considered dependent observation in the context of monotone regression (e.g., \cite{anevski2006general}, \cite{zhao2012estimating}), they primarily focus on pointwise limit distributions.  Recently, \cite{chernozhukov2009improving} and \cite{bagchi2016inference} proposed methods for constructing SCBs for monotone signals. However, both their resulting SCBs are conservative. \cite{bagchi2016inference} focuses on the inference under minimal smooth assumptions, hence their confidence band along with their estimates produces ``jump’’ spots, which appears suboptimal for the analysis of data sets with smoothly changing mechanisms. Moreover,
except \cite{chernozhukov2009improving}, the above-mentioned literature and most current inference methods for monotone regression functions applicable to time series require the data to be strictly stationary. In contrast, via improving an original SCB, \cite{chernozhukov2009improving} produces monotone SCBs. However, those SCBs are not necessarily centered around the monotone estimator of the regression function. 

 In fact, inference on the entire monotone regression function is a fundamental and challenging problem which has received considerable attention and effort in the literature. Besides the literature discussed in the last paragraph, the problem has been investigated under various specific settings and scenarios, including but not limited to the construction of SCBs in special monotone models (e.g., \cite{sampson2009simultaneous}; \cite{huang2017restoration}; \cite{gu2021smooth}) or only over a finite grid in the domain with certain stochastic equicontinuity condition (e.g., \cite{Ted2020correcting}).
 In contrast to the existing literature, our results significantly expand the application scope of inference methods for monotone regression functions 
to a broader range of real-world scenarios. This is done via allowing general time series nonstationarity as well as high-dimensional trend or multiple regression. We compare our methods with several mainstream methods for the statistical analysis of the monotone regression functions and summarize the corresponding results in Table \ref{tab:comparison}.

\begin{table}[h]
  \centering
  \setlength{\belowcaptionskip}{0cm}
  \captionsetup{font={small}, skip=3pt}
  \caption{Comparison between our methodology and previous major study} 
    \scalebox{0.75}{
    {\renewcommand{\arraystretch}{0.5}
    \begin{tabular}{|l|l|l|}
    \hline
    \multicolumn{1}{|c|}{Our methodology} & \multicolumn{1}{c|}{Previous major work} & \multicolumn{1}{c|}{Literature} \\
    \hline
    \makecell[l]{Monotone rearrangement-based:\\- smooth function\\- produce strictly monotone estimator}
    & \makecell[l]{Isotonic regression-based:  \\- constrained optimization \\- almost inevitable flat point}& \makecell[l]{\cite{brunk1969estimation}, \\ \cite{mukerjee1988monotone},\\ \cite{mammen1991estimating}. }\\
    \hline
    \makecell[l]{SCB under asymptotically correct\\ significance level} & \makecell[l]{Pointwise confidence interval \\ or conservative SCB} & \makecell[l]{\cite{deng2021confidence}, \\ \cite{bagchi2016inference},\\\cite{chernozhukov2009improving}. } \\
    \hline
    Nonstationary time series & \makecell[l]{Independent/Stationary observation} &  \makecell[l]{\cite{zhao2012estimating},\\ \cite{bagchi2016inference}.}\\
    \hline
    \makecell[l]{- High-dimensional model i.e. $\mathbf{m}(t)\in\mathbb{R}^p$, $p\rightarrow\infty$.\\ - Regression model i.e. $Y(t)=X^{\top}(t)\mathbf{m}(t)+e(t)$\\
    $\mathbf{C}\mathbf{m}(t)\in\mathbb{R}^s$ for known $\mathbf{C}\in\mathbb{R}^{s\times p}$ is monotone\\(monotone in each coordinate)}
     & \makecell[l]{Multivariate model: $\mathbf{t}\in[0,1]^p$ \\ but $m(\mathbf{t})\in \mathbb{R}$. \textit{The extension of}\\ \textit{our method adapting to this setting}\\ \textit{is interesting for future research.}} &
     \makecell[l]{\cite{chatterjee2018matrix},\\ \cite{deng2021confidence},\\ \cite{Ted2020correcting}}. \\
    \hline
    \end{tabular}%
    }}
  \label{tab:comparison}%
\end{table}%

Our simultaneous inference is based on monotone estimates for either high-dimensional monotone vector functions or for monotone regression coefficient functions, which combines the strength of nonparametric estimation and monotone rearrangement. These estimates belong to the class of two-step estimators that combines isotonization or rearrangement with smoothing. This class of estimators has been extensively discussed in the single trend with stationary noise setting; see, for example, \cite{zhao2012estimating}, \cite{bagchi2016inference}. Our chosen rearranged estimates can be obtained by unconstrained optimization algorithm, which enables us to derive the corresponding stochastic expansion. We then apply the state-of-the-art Gaussian approximation technique in \cite{Mies2022seq_high-dim} to approximate the distribution of maximum deviation of the monotone estimates by certain Gaussian random vectors. Upon this, we design 2 bootstrap algorithms to construct the (joint) SCBs for both scenarios of high-dimensional trends and time-varying coefficient linear models under smooth and monotone constraints. It is worth noting that our method is applicable to so-called piecewise locally stationary noise (see \cite{zhou2013hetero} for details) which allows both smooth changes and abrupt changes in the underlying data-generating mechanisms. The validity of our proposed methods is justified theoretically, to which the key is controlling the approximation error via Nazarov’s inequality (\cite{nazarov2003maximal}). 

The remainder of this paper is structured as follows. In Section \ref{Monotone estimator from monotone rearrangement}, we introduce the monotone rearrangement-based monotone estimator. Section \ref{High dimensional SCB} presents our main results for the high-dimensional model, including model assumptions,  the theoretic results of Gaussian approximation, and the bootstrap procedure mimicking the limiting distribution of the maximum deviation of the monotone and smooth estimator. Similarly, the monotone and smooth estimation for time-varying coefficient linear model is discussed thoroughly in Section \ref{Time-varying coefficient regression}. In Section \ref{Bandwidth selection}, we detail the selection scheme for tuning parameters. 
Furthermore, Section \ref{Simulation study} reports our simulation results, while Section \ref{Empirical study} presents the application of our method to the analysis of historical temperature data in the UK. Finally, in Section \ref{Discussion}, we discuss potential directions for future work. Detailed proofs
and additional simulation results can be found in the online supplement. In particular, additional simulations compare the width of SCB obtained by our method with alternative SCBs and find that our SCB enjoys the narrowest width while maintaining confidence levels.

\section{Notation}
\label{Notation}
Before stating our results formally, we list the notations that will be used throughout the paper below. For a vector $\mathbf{v}=\left(v_1, \cdots, v_p\right)^{\top} \in \mathbb{R}^p$, let $|\mathbf{v}|=(\sum_{j=1}^p v_j^2)^{1 / 2}$. For a random vector $\mathbf{V}$, let $\|\mathbf{V}\|_q=\left(\mathbb{E}|\mathbf{V}|^q\right)^{1 / q}$ ($q>0$) be the $\mathcal{L}^q$-norm of the random vector $\mathbf{V}$. The notion $\|\cdot\|$ refers to $\|\cdot\|_2$ if no extra clarification. 
Denote $\mathrm{diag}\{a_k\}_{k=1}^p$ as a diagonal matrix with diagonal elements $a_1,\dots,a_p$. Let $\lfloor x\rfloor$ represent the largest integer smaller or equal to $x$. For any two positive real sequences $a_n$ and $b_n$, write $a_n\asymp b_n$ if $\exists$ $ 0<c<C<\infty$ such that $c\leq \liminf _{n \rightarrow \infty} a_n / b_n\leq \limsup _{n \rightarrow \infty}a_n / b_n\leq C$. We use the notion $\hat{\mf m}^{\prime}(t)$ to  represent the local linear estimator of derivative $\mf m'(t)$ where $\mf m'(t)=(\frac{\partial}{\partial t}m_j(t))_{1\leq j\leq p}$.

\section{Monotone estimator via monotone rearrangement}
\label{Monotone estimator from monotone rearrangement}
For simplicity, we illustrate our methodology for univariate series in this section. \TPR{Suppose our observation $y_{i,n},i=1,\dots,n$ can be written as following nonparametric regression model with monotone constraint}
\begin{equation}
    y_{i,n} = m( i/n ) + e_{i,n}, \quad i=1,\dots,n.
    \label{monotone model_n}
\end{equation}
\TPR{The model \eqref{monotone model_n} is closely related to the ‘‘infill’’ modeling in spatial statistics in the sense that the rescaled time points are increasingly dense in $[0,1]$, and has been widely considered in nonstationary time-varying linear models, see \cite{zhou2010simultaneous} among others.} \TPR{The mean function $m(t)$ is smoothly monotone function of $t\in[0,1]$. As illustrated in \cite{fan2003nonlinear}, modelling the mean function of time series as $m(i/n)$ is a simple technical device for capturing the feature that the trend is much more slowly varying than the noise. It is also widely used in nonparametric time series regression; see \cite{hall1990nonparametric}, \cite{johnstone1997wavelet}.}  \TPR{$\{e_{i,n}\}_{i=1}^n$ is a zero-mean error process and can be nonstationary. The detailed assumption of $e_{i,n}$ is deferred to Section \ref{sec:Assumptions of high-dimensional time series}. In this paper, all our results established are under a strictly increasing context.} In other words, we focus on the mean function $m:[0,1]\rightarrow\mathbb{R}$ that is strictly increasing in $t$. 
For strictly decreasing cases our methods still apply. This is done by considering $\{-y_{i,n}\}_{i=1}^n$.

Numerous estimators have been derived for model \eqref{monotone model_n} to obtain constrained estimators of $m(t)$ that satisfy the continuity and monotonicity constraint. A prevalent method for the inference of monotone shape functions is the fundamental isotonic regression, which yields a discontinuous ``step’’ function with flat segments. To bridge the gap between step fitting and the continuous nature of the data, isotonization mixed with kernel smoothing procedure has been studied by several researchers, see for example \cite{mammen1991estimating}, \cite{van2003smooth}, \cite{durot2014Kiefer-Wolf}. However, such approaches 
almost inevitably produce a flat area in estimated curves even when the true function contains no flat part. 
\TPM{Additionally, the spline-based methods, for example \cite{tantiyaswasdikul1994isotonic}, \cite{ramsay1988monotone}, \cite{meyer2008inference}, can estimate smoothly increasing function. However, the above isotonization-based and spline-based methods rely on constrained optimization techniques.} As a comparison, \cite{dette2006simple} introduces a smooth and strictly monotone estimator via monotone rearrangement, which does not rest on constrained optimization.  Therefore, the statistical properties of estimator suggested by \cite{dette2006simple} are easy to analyze. The monotone rearrangement techniques have been applied widely in solving statistical problems with monotone constraints, e.g., \cite{chernozhukov2009improving}, \cite{Chernouzhukov2010noncrossing}, \cite{dette2008noncrossing}, \cite{SubhraWu2023}.

The key idea of monotone rearrangement is the use of the following fact. For any function $f$ defined on $[0,1]$, and a kernel density function $K_d(\cdot)$, define $g_{h_d}\circ f$ on $\mathbb R$ as 
$$
g_{h_d}\circ f:s\rightarrow \int_{0}^1\left\{\int_{-\infty}^s\frac{1}{h_d}K_d\left(\frac{f(x)-u}{h_d}\right)\mathrm{d}u\right\}\mathrm{d}x,
$$
which is a smooth and monotone approximation of  $f^{-1}$ when  $f$ is strictly increasing; see \cite{ryff1970measure}. Moreover, $g_{h_d}\circ f$ is always smooth and monotone, even if $f$ is non-monotone.
Thus, a natural smooth and monotone estimator of $m^{-1}$ can be defined through Riemann sum approximation to  $g_{h_d}\circ \tilde{m}$, that is,
\begin{equation}
    \hat m_I^{-1}(s)=:\int_{-\infty}^s \frac{1}{Nh_d}\sum_{i=1}^NK_d\left(\frac{\tilde m(i/N)-u}{h_d}\right)\mathrm{d}u,
    \label{eq:m_I}
\end{equation}
where $\tilde m(\cdot)$ is a jackknife bias-corrected local linear estimator using kernel function $K_r(\cdot)$ and bandwidth $h_r$. In specific, the $\tilde m(t)$ is defined as
\begin{equation}
    \label{eq: def for jackknife ll}
    \tilde{m}(t)=:2\hat{m}_{h_{r}/\sqrt{2}}(t)-\hat{m}_{h_{r}}(t),
\end{equation}
\begin{equation}
\left(\hat{m}_{h_{r}}(t), \hat{m}_{h_{r}}^{\prime}(t)\right)=:\underset{\eta_0, \eta_1 \in \mathbb{R}}{\arg \min }\left[\sum_{i=1}^n\left\{y_{i}-\eta_0-\eta_1\left(t_i-t\right)\right\}^2 K_{r}\left(\frac{t_i-t}{h_{r}}\right)\right].\label{eq:ll def}
\end{equation}
\TPR{
To handle the simultaneous inference rather than the pointwise problem in \cite{dette2006simple}, we apply monotone rearrangement on the jackknife bias-corrected local linear estimator $\tilde{m}(t)$, as defined in \eqref{eq: def for jackknife ll}. This is instead of the Naradaya-Watson estimator used in \cite{dette2006simple}. As discussed in \cite{zhou2010simultaneous}, introducing this jackknife bias correction in the construction of SCB reduces the bias of the local linear estimator to an asymptotically negligible rate compared with its stochastic variation. Moreover, $\tilde{m}(t)$ is also used in the cumulative long-run covariance estimation in \eqref{eq:cumulative long-run cov} due to the small bias.} The simultaneous inference target also necessitates the examination of a time span rather than a single time point in \cite{dette2006simple}. Therefore, we introduce the index set $\hat{\mathcal{T}}$ on which all our simultaneous results are built:
\begin{equation}
    \hat{\mathcal{T}}=:\left\{\hat{m}_I^{-1}(s):s\in\hat{\mathcal{T}}^{-1}\right\}\cap[h_d\log h_d^{-1},1-h_d\log h_d^{-1}],
    \label{index_I}
\end{equation}
where $\hat{\mathcal{T}}^{-1}=:\{s\in\mathbb{R}:\min_{t\in[h_r,1-h_r]}\tilde m(t)\leq s\leq \max_{t\in[h_r,1-h_r]}\tilde m(t)\}$. The interval $\hat{\mathcal T}$ enjoys several advantageous properties.  
First, on the set $\hat{\mathcal T}^{-1}$, $\hat m_I^{-1}(s)$ is smooth and strictly monotone, which enables us to naturally derive a smooth monotone estimator $\hat{m}_I(t)$ for $m(t),t\in\hat{\mathcal{T}}$ by directly inverting $\hat{m}_I^{-1}(s)$. Second, this interval circumvents boundary issues that arise in monotone rearrangement using the kernel density function $K_d(\cdot)$. Despite $\hat{\mathcal{T}}$ being a proper subset of $[0,1]$, it is sufficiently large for simultaneous inference since the length of $[0,1]$ not covered by $\hat{\mathcal{T}}$ converges to zero in probability, 
as supported by the Proposition \ref{Boundary} 
in the online supplement.

In the following Section \ref{High dimensional SCB} and Section \ref{Time-varying coefficient regression}, we further consider the high-dimensional version and time-varying regression extension of the monotone estimator $\hat m_I(t)$ and the corresponding simultaneous inference methods. The simultaneous inference of $m(t)$ for \eqref{monotone model_n} in this section could be performed via the methods in Section \ref{High dimensional SCB} and Section \ref{Time-varying coefficient regression}. 


\section{Joint SCBs for high-dimensional monotone trends}
\label{High dimensional SCB}
Historically, the increasing need to study contemporary time series with rapidly growing size and progressively complex structures has necessitated the construction of SCBs that cover many time-varying curves jointly. This is essential for the simultaneous inference in numerous applications involving processes that are $p$-dimensional. A prominent example is the investigation of global warming trends across multiple districts. While the temperatures recorded in each district may exhibit variations, they commonly adhere to the monotonicity condition. Constructing joint SCBs that simultaneously cover all the monotone temperature trends in those districts at a desired significance level can be useful for comprehensively understanding extreme climate change on a global scale. This motivates us to consider the high-dimensional extension of model \eqref{monotone model_n} in the form: 
\begin{equation}
    y_{i,k}=m_k(i/n)+e_{i,k},\quad i=1,\dots,n;k=1,\dots,p
    \label{high-monotone model}
\end{equation}
where $\mathbf{y}_{i,n}=(y_{i,1},\dots,y_{i,p})^{\top}$ is our $p$-dimensional observation and $\mathbf{e}_{i,n}=(e_{i,1},\dots,e_{i,p})^{\top}$ is a high-dimensional nonstationary error process with mean $\mf 0$. Moreover, each coordinate of the mean vector function $\mathbf{m}(t)=(m_{1}(t),\dots,m_p(t))^{\top}$ is assumed to be monotone, 
that is, $m_k(x)\leq m_k(y)$ if $x\leq y$ for all $1\leq k\leq p$. It is worth pointing out that our monotonicity is different from multivariate monotone regression model, for example \cite{deng2021confidence}, which considers a $\mathbb R$-valued function $m(\mf t)$ on $p$-dimensional index set $\mf t\in[0,1]^p$. In the remaining section, we discuss assumptions on the time series in Subsection \ref{sec:Assumptions of high-dimensional time series}, the Gaussian approximation in Subsection \ref{sec:Gaussian approximations} and the algorithm of generating SCBs in Subsection \ref{sec:construction of joint SCBs}.
\subsection{Assumptions of high-dimensional time series}
\label{sec:Assumptions of high-dimensional time series}
To begin with, we impose the following assumptions on the high-dimensional trends:
\begin{description}[itemsep=0pt,parsep=0pt,topsep=0pt,partopsep=0pt]
\setlength{\baselineskip}{0.7\baselineskip}
    \item[\setword{\textbf{(A1)}}{(A1)}] \textit{For the vector function $\mf m(t)=( m_1(t),\dots, m_p(t))^\top$ where $t\in[0,1]$, the second derivative $m_k^{\prime\prime}(t)$ of the function $m_k(t)$ exists and is Lipschitz continuous on $[0,1]$. The Lipschitz constants are bounded for all $k=1,\dots,p$.}
    \item[\setword{\textbf{(A2)}}{(A2)}] \textit{For the vector function $\mf m(t)=( m_1(t),\dots, m_p(t))^\top$, there exists a universal constant $M>0$ s.t. $\inf_{t\in [0,1]} m_k'(t)\geq M$, $\forall k=1,\dots,p$.}
\end{description}
Let $\varepsilon_i$, $i\in \mathbb{Z}$ be i.i.d. random elements. Denote $\Upsilon_i=(\dots, \varepsilon_{i-1}, \varepsilon_i)$ \TPR{and $\mathcal{S}^{\mathbb{Z}}$ as all the possible values of random elements in $\Upsilon_i$. We assume that the high-dimensional nonstationary error process $\mathbf{e}_i$ is generated by following causal representation
\begin{equation}
    \mathbf{e}_{i,n}=:\mathbf{G}(t_i,\Upsilon_i),\quad t_i=:i/n,\quad i=1,\dots,n,\label{eq: def for G}
\end{equation}
where $\mathbf{G}(u,\Upsilon_i)=:(G_1(u,\Upsilon_i),\dots,G_{p}(u,\Upsilon_i))^{\top}$ is a measurable vector function $\mf G:[0,1]\times \mathcal{S}^{\mathbb{Z}}\rightarrow \mathbb{R}^p$.}
We further introduce the physical dependence measure of \cite{wu2005nonlinear} for $p$-dimensional $\mathbf{G}(u,\Upsilon_i)$.
\begin{define}{(physical dependence measure)}\label{def:phd} Let $\{ \varepsilon_i'\}_{i\in \mathbb{Z}}$ be an i.i.d. copy of $\{ \varepsilon_i\}_{i\in \mathbb{Z}}$ and $\Upsilon_{i,k}=(\dots, \varepsilon_{k-1},\varepsilon_{k}',\varepsilon_{k+1},\dots, \varepsilon_i)$ for $k\leq i$. If for all $u\in[0,1]$, $\|\mf G(u,\Upsilon_i)\|_q<\infty$ ($\|\mf G(u,\Upsilon_i)\|_q$ might depend on dimension $p$), then we can define the physical dependence measure for the stochastic system $\mathbf{G}\left(u,\Upsilon_i\right)$ as
\begin{equation}
    \delta_q(\mathbf{G}, k)=:\sup_{u\in[0,1]}\left\{\left\|\mathbf{G}\left(u,\Upsilon_i\right)-\mathbf{G}\left(u,\Upsilon_{i,i-k}\right)\right\|_q\right\}.
    \label{phd}
\end{equation}
\end{define}
\TPR{To help understand the generality of generation mechanism $\mathbf{G}(u,\Upsilon_i)$ and the physical dependence measure \eqref{phd}, we list two typical examples as follows. We also provide some specific time series models such as ARMA and GARCH models in Section \ref{sec:specific examples} 
of the online supplement.
\begin{description}[itemsep=0pt,parsep=0pt,topsep=0pt,partopsep=0pt]
    \setlength{\baselineskip}{0.7\baselineskip}
    \item[(i)] High-dimensional linear process: Given time-varying $p\times p'$ ($p'$ is fixed) matrices $\mf B_j(t)$, $j\in \mathbb Z, t\in [0,1]$, write $\mf e_{i,n}=\mf G(t_i,\Upsilon_i)$ where $\mf G(t,\Upsilon_h)=:\sum_{j=0}^{\infty} \mf B_{j}(t) \boldsymbol{\varepsilon}_{h-j}$ and $\boldsymbol{\varepsilon}_i=(\varepsilon_{i,1},\dots,\varepsilon_{i,p'})^{\top}$, $\left(\varepsilon_{i,s}\right)_{i\in\mathbb{Z},1\leq s\leq p'}$ are i.i.d. random variables. If $\|\boldsymbol{\varepsilon}_0\|_q<\infty$, then straightforward calculations show $\delta_q(\mf G,k)=O\left(\sup_{t\in[0,1]} \sqrt{\sum_{v=1}^p\sum_{s=1}^{p'}b_{k,v,s}^2(t)}\right)$ where $b_{k,v,s}(t)$ is the element of $\mathbf{B}_k(t)$. 
    \item[(ii)] High-dimensional nonlinear process: Consider $\mf e_{i,n}=\mf G(t_i,\Upsilon_i)$ where
    \begin{equation}
        \mf G(t,\Upsilon_i)=\mf R\left(t,\mf G(t,\Upsilon_{i-1}),\varepsilon_i\right).\label{eq: eg nonlinear process}
    \end{equation}
    Similarly to \cite{zhou2013hetero}, when $\mf G(t,\cdot)$ and $\mf R(t,\cdot,\cdot)$ do not depend on $t$, a variety of $p$ dimensional stationary nonlinear time series models can be written in the form \eqref{eq: eg nonlinear process}. These include GARCH models (\cite{BOLLERSLEV1986GARCH}), threshold models (\cite{tong1990non}) and bilinear models.
\end{description}
}

The physical dependence measure we defined in \eqref{phd} shows an input-output-based dependence measure quantifying the influence of the input $\varepsilon_{i-k}$ on the output $\mathbf{G}(u,\Upsilon_i)$, which is different from the classic mixing conditions. 
Alternative definitions of dependence measure for high-dimensional time series can be found, for example in \cite{zhang2018GA}, where the dependence measure is specified to each dimension, that is,  $\delta_q(G_{j},k),j=1,\dots,p$ and requires a universal summable decay, that is, $\sum_{k=l}^{\infty}\sup_{1\leq j\leq p}\delta_q(G_{j},k)<\infty$. \TPR{Our assumptions on dependence measure are related to their framework in the sense that $\sum_{k=l}^{\infty}\delta_q(\mathbf{G}, k)\leq \sqrt{p}\sum_{k=l}^{\infty}\sup_{1\leq j\leq p}\delta_q(G_{j},k)$ by condition $q\geq 2$ and triangle inequality. 
Furthermore, we define the long-run covariance matrix function for the nonstationary process $\mathbf{e}_{i,n}=\mathbf{G}(t_i,\Upsilon_i)$.
\begin{define}
    For process $\mathbf{G}(t,\Upsilon_i)$, define the long-run covariance matrix function
    \begin{equation}
    \mathbf{\Sigma}_{\mathbf{G}}(t)=:\sum_{l\in\mathbb{Z}} \mathrm{cov}(\mathbf{G}(t,\Upsilon_0),\mathbf{G}(t,\Upsilon_l)).
    \label{eq:long-run var def}
    \end{equation}
\end{define}
The long-run covariance matrix is important for quantifying the stochastic variation of the sums of the nonstationary process. 
 We now introduce the assumptions on physical dependence measures, long-run covariance and other properties of the high-dimensional nonstationary process $\mathbf{G}(u,\Upsilon_i)$.} For some constants $\alpha,\chi\in(0,1)$, $q\geq 2$ and $C>0$: 
\begin{description}[itemsep=0pt,parsep=0pt,topsep=0pt,partopsep=0pt]
\setlength{\baselineskip}{0.7\baselineskip}
    \item[\setword{\textbf{(B1)}}{(B1)}] \textit{$\max\{p,\Theta(p)\}=O(n^\alpha)$, $\sup_{u\in[0,1]}\|\mathbf{G}(u,\Upsilon_0)\|_q\leq C\Theta(p)$.}
    \item[\setword{\textbf{(B2)}}{(B2)}] \textit{$\delta_q(\mathbf{G},k)\leq C \Theta(p) \chi^k$.}
    \item[\setword{\textbf{(B3)}}{(B3)}] \textit{$\sum_{i=2}^n\|\mathbf{G}(t_i,\Upsilon_0)-\mathbf{G}(t_{i-1},\Upsilon_0)\|\leq C\Theta(p)$.}
    \item[\setword{\textbf{(B4)}}{(B4)}] \textit{There exists constant $\Lambda\geq\underline{\Lambda}>0$ such that all the eigenvalues of long-run covariance matrix $\mathbf{\Sigma}_{\mathbf{G}}(t)$ are bounded between $\underline{\Lambda}$ and $\Lambda$ for all $t\in[0,1]$. }
    \item[\setword{\textbf{(B5)}}{(B5)}] \textit{
    For any $e_{i,k}$, $i=1,\dots,n$, $k=1,\dots,p$, there exists a universal constant $t>0$ such that $\mathbb{E}(\exp(t|e_{i,k}|))\leq C<\infty$.}
\end{description}
\begin{remark}
   $\Theta(p)$ reflects the dimension factor for both physical dependence measure and the moment conditions of the $p$-dimensional process. If $\delta_q(G_j,k)\leq C\chi^k$ for all $j=1,\dots,p$ with constant $C>0$, then $\Theta(p)$ would be a rate of $\sqrt{p}$.   Further suppose that all $p$ components are independent, then \hyperref[(B5)]{(B5)} implies  that $\Theta(p)$ in \hyperref[(B1)]{(B1)} will also be a rate of $\sqrt{p}$.
\end{remark}
Assumptions \hyperref[(B1)]{(B1)}, \hyperref[(B2)]{(B2)}, and \hyperref[(B4)]{(B4)} are in line with many kernel-based nonparametric analysis of nonstationary time series such as \cite{Zhao2015inference} and \cite{dette2022prediction}. While this paper adopts a geometric decay in \hyperref[(B2)]{(B2)}, it is worth noting that a polynomial decay can yield similar results. However, this requires more complicated conditions and substantially more intricate mathematical arguments. Therefore, for simplicity, we stick to the geometric decay assumption. Assumption \hyperref[(B3)]{(B3)} can be regarded as an extension of the locally stationary assumption. For example, it encompasses scenarios akin to the \textbf{piecewise locally stationary} concept, as discussed in \cite{zhou2013hetero}, where the time series can experience abrupt changes at several breakpoints over time. This renders it a more flexible and realistic representation for various practical applications. Assumption \hyperref[(B5)]{(B5)} allows sub-exponential tails, and is mild for high-dimensional models. It can be relaxed to allow polynomial tails for fixed dimension $p$. \TPR{For more detailed illustration, we verify these assumptions by a general class of high-dimensional error process in specific examples in Section \ref{sec:specific examples} 
in the online supplement.}

\subsection{Monotone estimation and Gaussian approximation}
\label{sec:Gaussian approximations}
To estimate the monotone high-dimensional trend $\mathbf{m}(t)$, we apply monotone rearrangement on each dimension. The monotone estimator is written as $\hat{\mathbf{m}}_I(t)=(\hat m_{I,1}(t),\dots,\hat m_{I,p}(t))^\top$ where $\hat m_{I,k}(t)$ is the inverse of 
\begin{equation}
\label{high-dim monotone estimator}
    \hat m_{I,k}^{-1}(s)=:\int_{-\infty}^s \frac{1}{Nh_{d,k}}\sum_{i=1}^NK_d\left(\frac{\tilde m_k(i/N)-u}{h_{d,k}}\right) \mathrm{d}u,
\end{equation}
where $\tilde{m}_k(t)=:2\hat{m}_{k,h_{r,k}/\sqrt{2}}(t)-\hat{m}_{k,h_{r,k}}(t)$ and
\begin{equation}
\left(\hat{m}_{k,h_{r,k}}(t), \hat{m}_{k,h_{r,k}}^{\prime}(t)\right)=:\underset{\eta_0, \eta_1 \in \mathbb{R}}{\arg \min }\left[\sum_{i=1}^n\left\{y_{i,k}-\eta_0-\eta_1\left(t_i-t\right)\right\}^2 K_{r}\left(\frac{t_i-t}{h_{r,k}}\right)\right].\label{eq:high-dim ll def}
\end{equation}

Before stating our main results, we make several assumptions and simplify our kernel methods. Suppose there exists $h_r,h_d$ s.t. $\min_{1\leq k\leq p}h_{r,k}\asymp h_r\asymp \max_{1\leq k\leq p}h_{r,k}$ and $\min_{1\leq k\leq p}h_{d,k}\asymp h_d\asymp \max_{1\leq k\leq p}h_{d,k}$. 
Then we assume:
\begin{description}[itemsep=0pt,parsep=0pt,topsep=0pt,partopsep=0pt]
\setlength{\baselineskip}{0.7\baselineskip}
    \item[\setword{\textbf{(C1)}}{(C1)}] \textit{The kernel functions $K_d,K_r$ are symmetric kernel density with compact support $[-1,1]$ and bounded second order derivatives.}
    \item[\setword{\textbf{(C2)}}{(C2)}] \textit{$n=O(N),nh_r^3\rightarrow\infty,nh_r^7\rightarrow0,h_d=o(h_r), Nh_d^3\rightarrow \infty$, $R_n/h_d=o(1)$ where $R_n=h_r^2+ \log^{3}n/\sqrt{nh_r}$}.
\end{description}
\begin{remark}
    We use the Epanechnikov kernel for both $K_d$ and $K_r$ throughout the simulation and data analysis. Other kernels satisfying \hyperref[(C1)]{(C1)} work similarly. 
    Our assumption $R_n/h_d=o(1)$ is also weaker than the condition $R_n^2/h_d^3=o(1)$ in \cite{dette2006simple}  
\end{remark}

Define the time span $\hat{\mathcal{T}}=\cap_{k=1}^p\hat{\mathcal{T}}_k$ where $\hat{\mathcal{T}}_k$ is defined in the same way as \eqref{index_I} with $\tilde m_{k}(t)$ so that $\hat{m}_{I,k}(t)$ is well defined for each $k=1,\dots,p$. To construct joint SCBs for $\hat{\mathbf{m}}_I(t)-\mathbf{m}(t)$ on the time span $\hat{\mathcal{T}}$, the key result is to learn the maximum deviation, that is, $\max_{1\leq k\leq p}\sup_{t\in\hat{\mathcal{T}}}|\hat{m}_{I,k}(t)-m_k(t)|$. The following proposition provides the Gaussian approximation of this maximum deviation and serves as the basis for our simultaneous inference and further bootstrap procedure.


\begin{prop}
    \label{prop:GA in high-dim}Suppose  \hyperref[(A1)]{(A1)},\hyperref[(A2)]{(A2)},\hyperref[(B1)]{(B1)}-\hyperref[(B5)]{(B5)} and \hyperref[(C1)]{(C1)}-\hyperref[(C2)]{(C2)} are satisfied. Then there exists independent $\mathbf{V}_j=(V_{j,1},\dots,V_{j,p})^{\top}\sim N_p(0,\mathbf{\Sigma}_{\mathbf{G}}(t_j)),j=1,\dots,n$,
    on a richer probability space s.t.
    \begin{equation}
        \max_{1\leq k\leq p}\sup_{t\in\hat{\mathcal{T}}}\sqrt{nh_r}\left|m_k(t)-\hat{m}_{I,k}(t)-\mathcal{V}_k(t)\right|=O_p\left\{\frac{(R_n^2+h_d^3)\sqrt{nh_r}}{h_d}+\frac{\Theta(p)\sqrt{n\log n}\left(\frac{p}{n}\right)^{\frac{q-2}{6q-4}}}{\sqrt{nh_r} } \right\},
        \label{eq:GA in prop:GA in high-dim}
    \end{equation}
    where the $p$-dimensional  Gaussian vector process $\mathcal{V}(t)=(\mathcal{V}_1(t),\dots,\mathcal{V}_p(t))^{\top}$ with 
    $$
    \mathcal{V}_k(t)=\frac{m_k'(t)}{Nh_{d,k}nh_{r,k}}\sum_{j=1}^n\sum_{i=1}^N K_{d}\left(\frac{m_k(i/N)-m_k(t)}{h_{d,k}}\right)\tilde K_r^*\left(\frac{j/n-i/N}{h_{r,k}},\frac{i}{N}\right)V_{j,k},
    $$
\begin{equation}
    \tilde K_r^*\left(\frac{j/n-t}{h_{r,k}},t\right)=\frac{\nu_{2,k}(t) \tilde K_r\left(\frac{j / n-t}{h_{r,k}}\right)-\nu_{1,k}(t)\tilde K_r\left(\frac{j / n-t}{h_{r,k}}\right)\left(\frac{j / n-t}{h_{r,k}}\right)}{\nu_{0,k}(t) \nu_{2,k}(t)-\nu_{1,k}^2(t)},\label{eq:K^*}
\end{equation}
and $\nu_{l,k}(t)=\int_{-t/h_{r,k}}^{(1-t)/h_{r,k}}x^lK_r(x)\mathrm{d}x$, $\tilde K_r(x)=:2\sqrt{2}K_r\left(\sqrt{2}x\right)-K_r\left(x\right)$.
\end{prop}
\begin{remark}
    \label{remark:GA in high-dim}
    $\sqrt{nh_r}\mathcal{V}(t)$ is a valid nondegenerate Gaussian process in the sense that there exists a constant $\underline{\sigma}>0$ such that $\min_{1\leq k\leq p}\operatorname{Var}(\sqrt{nh_r}\mathcal{V}_k(t))\geq \underline{\sigma}^2$,  which has been shown in the proof of Theorem \ref{thm:high-dim bootstrap theorem}. The rate in the right side of \eqref{eq:GA in prop:GA in high-dim} can be $o_p(1)$ if divergent $p = O(nh_r^3)^{1/4-\alpha}$ for some $\alpha\in(0,1/4)$ with typical scaling $\Theta(p)\asymp\sqrt{p}$, sufficiently large $q$.
\end{remark}


\subsection{Bootstrap algorithm of joint SCBs}   
\label{sec:construction of joint SCBs}
Following Proposition \ref{prop:GA in high-dim}, we can investigate  $\max_{1\leq k\leq p}\sup_{t\in \hat {\mathcal T}}\sqrt{nh_r}|m_k(t)-\hat m_{I,k}(t)|$ through studying $\max_{1\leq k\leq p}\sup_{t\in \hat{\mathcal{T}}}|\sqrt{nh_r}\mathcal V_k(t)|$. However, it is sophisticated to study the limiting distribution of $\max_{1\leq k\leq p}\sup_{t\in \hat{\mathcal{T}}}|\sqrt{nh_r}\mathcal V_k(t)|$ due to the high dimensionality and the complicated time-varying covariance structure of $\mathcal V_k(t)$.  One direct approach is to generate copies of the estimated $\mathcal V_k(t)$ and obtain empirical quantiles of their maxima, where the estimated $\mathcal V_k(t)$ is a Gaussian random variable having the form analogous to $\mathcal V_k(t)$, with the unknown quantities $m_k(t)$, $m_k'(t)$, and $\mf \Sigma_{\mf G}(t_j)$ replaced by appropriate estimators. However, it is widely recognized that accurately estimating $m_k'(t)$ can be challenging. 
Moreover, the estimation of $\mf\Sigma_{\mf G}(t_j)$ can also be difficult if $\mf G(t_i,\Upsilon_i)$ is piecewise locally stationary since the breakpoints are difficult to identify. This yields inconsistency around the abrupt changes of the long run covariance for usual nonparametric estimators as seen in works such as \cite{zhou2013hetero}, \cite{zhang2015time}, \cite{bai2024difference}.
\begin{breakablealgorithm}
	\caption{Bootstrap for joint SCBs of $\mathbf{m}(t)$}
	\label{alg:high-dim bootstrap}
    \begin{algorithmic}[0]
\setlength{\baselineskip}{0.7\baselineskip}
		\STATE{\textbf{Data:}} $\mathbf{y}_i\in\mathbb{R}^p,i=1,\dots,n$
        \STATE{\textbf{Initialization}}: Choose bandwidths $h_{r,k}$, $h_{d,k}$ and the window size $L$ by Section \ref{Bandwidth selection}.
        \
        \STATE{\textbf{Step 1:}} Obtain the jackknife estimator $\tilde m_k(t)=2\hat{m}_{k,h_{r,k}/\sqrt{2}}(t)-\hat{m}_{k,h_{r,k}}(t)$ with $\hat{m}_{k,h_{r,k}}(t)$ defined in \eqref{eq:high-dim ll def}.
        \STATE{\textbf{Step 2:}} Obtain the monotone estimator $\hat{\mathbf{m}}_I(t)$ defined above \eqref{high-dim monotone estimator}.
        \STATE{\textbf{Step 3:}} Obtain  the residual $\hat{\boldsymbol{\varepsilon}}_i=(\hat \varepsilon_{i,1},\dots,\hat \varepsilon_{i,p})^{\top}$ where $\hat \varepsilon_{i,k}=y_{i,k}-\tilde m_k(t_i)$ and derive the estimated cumulative long-run covariance $\hat{\mf Q}(j)$ defined in \eqref{eq:cumulative long-run cov}.
        \STATE{\textbf{Step 4:}} Given data, generate $\mathbf{V}_j^*|\Upsilon_n\sim N_p(0,\hat{\mathbf{Q}}(j)-\hat{\mathbf{Q}}(j-1))$, $j=1,\dots,n$, where $V_j^*s$ are independent of each other conditional on the data, and calculate
    \begin{equation}
        \mathcal{V}^*(t)=:\sum_{j=1}^n\sum_{i=1}^N \mathcal{W}(i/N,t)\mathcal{K}_r(j/n,i/N) \mathbf{V}_j^*,
        \label{GA-V-star}
    \end{equation}
where $\mathcal{W}(i/N,t)=\mathrm{diag}\left\{ \hat W_{i,k,I}^*(t) \right\}_{k=1}^p$, $\hat W_{i,k,I}^*(t)=\hat W_{i,k,I}(t)/\sum_{i=1}^N \hat W_{i,k,I}(t)$ and
        $$
            \hat W_{i,k,I}(t)=K_d\left(\frac{\tilde m_k(i/N)-\hat m_{I,k}(t)}{h_{d,k}}\right),\quad \mathcal{K}_r(j/n,i/N)=\mathrm{diag}\left\{ \frac{1}{nh_{r,k}}\tilde{K}^*_r\left(\frac{j/n-i/N}{h_{r,k}},\frac{i}{N}\right) \right\}_{k=1}^p. 
        $$       
      \STATE{\textbf{Step 5:}} Repeat step 4 for $B$ times and obtain the sample $\Big\{\max_{1\leq k\leq p}\sup_{t\in\hat{\mathcal{T}}} \Big| \mathcal{V}^{*(r)}_k(t) \Big|\Big\}_{r=1,...B}$. Let $\hat q_{1-\alpha}$ be the $(1-\alpha)_{th}$ sample quantile.
        
        \STATE{\textbf{Output:}} Level $(1-\alpha)$ joint SCBs $\hat{\mf m}_{I}(t)\pm \hat q_{1-\alpha}\mathbf{1}_p$, where $\mathbf{1}_p=(1,\dots,1)^\top$ is a $p$-dimensional vector.  
	\end{algorithmic}
\end{breakablealgorithm}

To overcome the above two difficulties we consider the bootstrapping $\mathcal V^*(t)$ in \eqref{GA-V-star} of Algorithm \ref{alg:high-dim bootstrap},  whose formula \eqref{GA-V-star}  does not involve $m_k'(t)$. Moreover, recent progress in Gaussian approximation (\cite{Mies2022seq_high-dim}) shows that the consistency of the {\it cumulative} long-run covariance estimator, instead of the uniform consistency of the time-varying long-run covariance estimator, is sufficient for the jointly simultaneous inference. In this paper, we estimate $\mf Q(k)=:\sum_{i=1}^k\boldsymbol \Sigma_{\mf G}(t_i)$ by
\begin{equation}
    \hat{\mathbf{Q}}(k)=\sum_{i=L}^k\frac{1}{L}\left(\sum_{j=i-L+1}^i \hat{\boldsymbol{\varepsilon}}_j\right)\left(\sum_{j=i-L+1}^i \hat{\boldsymbol{\varepsilon}}_j\right)^\top\label{eq:cumulative long-run cov},
\end{equation} 
 where $\hat{\boldsymbol{\varepsilon}}_i=(\hat \varepsilon_{i,1},\dots,\hat \varepsilon_{i,p})^{\top}$ are the nonparametric residuals, that is, $\hat \varepsilon_{i,k}=y_{i,k}-\tilde m_k(t_i)$. 
A similar estimator for the cumulative long-run covariance has been studied by \cite{Mies2022seq_high-dim}, where the original series, instead of the residuals, is used. This is because \cite{Mies2022seq_high-dim} assumes the data has zero mean when estimating $\mathbf{Q}(k)$. In the  online supplement, we show that the estimation error for cumulative long-run covariance  ${\mf Q}(k)$ using $\hat {\mf Q}(k)$ is relatively small w.r.t ${\mathbf Q}(k)$ by Lemma \ref{lem:high long-run var}.
With the use of $\hat {\mf Q}(k)$ in Algorithm \ref{alg:high-dim bootstrap}, we introduce the following theorem which shows that the distribution of our bootstrap samples given data uniformly approximates the distribution of $\max_{1\leq k\leq p}\sup_{t\in\hat{\mathcal{T}}} \left|\mathcal{V}_k(t) \right|$. 

\begin{thm}
    \label{thm:high-dim bootstrap theorem}
    Recall the Gaussian process $\mathcal V(t)$ defined in Proposition \ref{prop:GA in high-dim}. Suppose
    \begin{equation}
    \label{eq: def union rate}
    \rho_n=:  \frac{\Theta(p)\left(\frac{p}{n}\right)^{\frac{q-2}{6q-4}}\sqrt{n\log n} }{nh_r} \bigvee \frac{\sqrt{\log (n)(\sqrt{n \varphi_n p}+\varphi_n+p)} }{nh_r},
\end{equation} 
    where $\varphi_n=:\Theta^2(p)(\sqrt{npL}+nL^{-1})+p^{3/2}L\sqrt{n}\left(\Theta(p)R_n+\sqrt{n}h_r^5+\sqrt{n}\log^6(n)/(nh_r)\right)$.
    Assuming the conditions in Proposition \ref{prop:GA in high-dim} hold and 
        \begin{equation}
        \rho_n\sqrt{nh_r\log(Np)}=o(1),
        \quad\sqrt{nh_r\log(Np)\log(np)}=o(Nh_d),\label{eq:rate for high-dim boot + continu}
    \end{equation}
    then \eqref{eq:GA in prop:GA in high-dim} holds and for $\mathcal{V}^*(t)=(\mathcal{V}^*_1(t),\dots,\mathcal{V}^*_1(t))^\top$ defined in \eqref{GA-V-star}, we have
    \begin{equation}
        \sup_{x\in\mathbb{R}} \left|\mathbb{P}\left\{\max_{1\leq k\leq p}\sup_{t\in\hat{\mathcal{T}}} \left|\sqrt{nh_r}\mathcal{V}_k(t) \right|\leq x \right\} - \mathbb{P}\left\{\max_{1\leq k\leq p}\sup_{t\in\hat{\mathcal{T}}} \left| \sqrt{nh_r}\mathcal{V}^*_k(t) \right|\leq x \,\middle\vert\, \Upsilon_n \right\} \right|\rightarrow_p 0,
        \label{eq:boot in thm:high-dim bootstrap theorem}
    \end{equation}
\end{thm}\begin{remark}
\label{remark:high-dim boot}Under the typical scaling $\Theta(p)\asymp\sqrt{p}$, \eqref{eq:rate for high-dim boot + continu}  can be satisfied with $nh_r^6\rightarrow\infty$, $\lfloor L\rfloor\asymp (n/p)^{1/3}$, $p\asymp (nh_r^{6})^{1/7-\alpha}$ for a sufficiently small $\alpha>0$, 
 if $q$ and $N$ are sufficiently large. Simulation results in Section \ref{Simulation study} reveal that our method works well with $p$ diverging at a rate faster than $(nh_r^6)^{1/7-\alpha}$. We leave the refinement on the theory to allow a faster diverging rate of $p$ using sharper inequalities in future studies. 
 \TPM{Constant $M$ in Assumption \hyperref[(A2)]{(A2)} is allowed to approach 0 as the sample size increases, which has been further discussed in the online supplement. Moreover, to deal with the scenario of $M=0$, we propose a penalized SCB for simultaneous inference under monotonicity. Specifically, we consider the pseudo data 
$\mathbf{y}_i^*=\mathbf{y}_i+\lambda_n(t_i)\mathbf{1}_p=\mathbf{f}(t_i)+\mathbf{e}_i$ where $\mathbf{f}(t)=:(f_1(t),\dots,f_p(t))^\top$, and $f_k(t)=:m_k(t)+\lambda_n(t)$ with a pre-determined penalization $\lambda_n(t)=C_1(n)t+C_2(n)t^2+C_3(n)t^3$, where  $C_i(n)s$ are appropriate decaying positive coefficients. 
The mean of the pseudo data is strictly monotone, thus we could perform simultaneous inference for $\mathbf{f}(t)$ using Algorithm \ref{alg:high-dim bootstrap}. We then obtained the penalized estimator and the corresponding SCBs by subtracting  $\lambda(t)-C_3(n)g_{n,1}t$ from the estimator of $\mf f(t)$ and the associate lower bound and upper bound of the joint SCBs, where a suitable choice of $g_{n,1}\to 0$ introduces negligible bias while guarantees the monotonicity. A thorough discussion of this method, including the rigorous proof of its validity, and the simulation evidence is presented in Section \ref{sec:penalization}, \ref{sec:penalization simulation} 
of the online supplement.}
\end{remark}
\TPR{Together with Gaussian approximations in Proposition \ref{prop:GA in high-dim}, we show in the online supplement  that \eqref{eq:boot in thm:high-dim bootstrap theorem} enjoys a rate of $O_p\left(\sqrt{\rho_n\sqrt{nh_r\log(Np)}}+\sqrt{\sqrt{nh_r\log(Np)\log(np)}/Nh_d}\right)$. This ensures the validity of our bootstrap Algorithm \ref{alg:high-dim bootstrap} for the construction of joint SCBs by Remark \ref{remark:high-dim boot} and the rate condition \eqref{eq:rate for high-dim boot + continu}.} \TPR{Based on the non-asymptotic Gaussian approximation in Proposition \ref{prop:GA in high-dim}, we proposed a bootstrap algorithm to mimic the finite sample behavior of the maximal deviation of the proposed monotone estimator with diminishing approximation error as $n\rightarrow \infty.$ We show that the yielding SCBs are asymptotically correct with shrinking width in simulation.  We thus reconcile the theoretical asymptotic where $n\rightarrow \infty$ yet the time interval remains bounded. }

\section{Time-varying coefficient linear regression}
\label{Time-varying coefficient regression}
Besides monotone trends, monotonicity also arises in time-varying relationships between variables. For example, greenhouse gases such as carbon dioxide, methane, and water vapor, can help to regulate Earth's temperature by trapping heat from the sun that would otherwise be radiated back into space, creating so-called natural greenhouse effect; see \cite{anderson2016CO2} for details. 
Based on the growth of greenhouse gases, we can identify that there exists a smoothly increasing relationship between the response variable (temperature) and the predictor (sunshine duration). 
This motivates us to 
consider the following time-varying coefficient linear model
\begin{equation}
    y_{i}=\mathbf{x}_{i}^{\top}\mathbf{m}(t_i)+e_{i},\label{Time-varying_n}
\end{equation}
where $y_{i},\mathbf{x}_{i},e_{i}$ represent the response (e.g., temperature), $p$-dimensional covariate process (e.g., sunshine duration, rain falls, etc.), and the zero-mean error process, respectively, while $\mathbf{m}(t)=(m_1(t),\dots,m_p(t))^{\top}$ denotes the $p$-dimensional time-varying coefficients. \TPM{The covariate process $\mf x_i$ can also include intercept and lagged terms such as $\mf x_i=(1,x_i,x_{i-1})^\top$.} In this section, we consider model \eqref{Time-varying_n} with fixed dimension $p$; the scenario of diverging $p$ is much more complicated, and potential future work might explore different methods for high-dimensional covariates. To reflect the increasing relationship, in model \eqref{Time-varying_n} we further assume the following monotonicity constraint: there exists a known $\mathbf{C}\in \mathbb{R}^{s\times p}$ s.t. each coordinate function of the $\mathbf{m}_{\mathbf{C}}(t)=:\mathbf{C}\mathbf{m}(t)$ is smoothly increasing. An example is \eqref{eq:tmax-SD} in the data analysis where the choice of $\mathbf{C} = (0, 1)$ reflects a smoothly increasing relationship between sunshine duration and temperature. Our inference for model \eqref{Time-varying_n} is based on the jackknife local linear estimator $\tilde{\mf m}(t)=2\hat{\mf m}_{h_r/\sqrt{2}}(t)-\hat{\mf m}_{h_r}(t)$ where
\begin{equation}
    \left(\hat{\mathbf{m}}_{h_r}(t), \hat{\mathbf{m}}^{\prime}_{h_r}(t)\right)=:\underset{\eta_0, \eta_1 \in \mathbb{R}^p}{\arg \min }\left[\sum_{i=1}^n\left\{y_i-\mathbf{x}_i^{\top} \eta_0-\mathbf{x}_i^{\top} \eta_1\left(t_i-t\right)\right\}^2 K_{r}\left(\frac{t_i-t}{h_r}\right)\right].
    \label{local linear}
\end{equation} 
Similar to \eqref{eq:m_I}, we define our monotone estimator through monotone rearrangement
\begin{equation}
    \hat m_{\mathbf{C},I,k}^{-1}(v) = \int_{-\infty}^v \frac{1}{Nh_d}\sum_{i=1}^NK_d\left(\frac{\tilde m_{\mathbf{C},k}(i/N)-u}{h_d}\right)\mathrm{d}u\label{eq:m_I.inv multi-dim}
\end{equation}
where $\tilde{m}_{\mathbf{C},k}(t)$ is the $k_{th}$ coordinate of $\tilde{\mathbf{m}}_\mathbf{C}(t)=:\mathbf{C}\tilde{\mathbf{m}}(t)$. Define $\hat{\mathbf{m}}_{\mathbf{C},I}(t)=:(\hat m_{\mathbf{C},I,1}(t),\dots,\hat m_{\mathbf{C},I,s}(t))^{\top}$ where $\hat m_{\mathbf{C},I,k}(t)$ is the inverse of $\hat m_{\mathbf{C},I,k}^{-1}(v)$ defined in \eqref{eq:m_I.inv multi-dim}. Consequently $\hat{\mathbf{m}}_{\mathbf{C},I}(t)$ serves as our monotone estimator of ${\mathbf{m}}_{\mathbf{C}}(t)$, and the time span $\hat{\mathcal{T}}=\cap_{k=1}^p \hat{\mathcal{T}}_k$ where $\hat{\mathcal{T}}_k$ is defined in the same way as \eqref{index_I} with $\hat{\mathcal{T}}_k^{-1}=\{u\in\mathbb{R}:\min_{1\leq i\leq N}\tilde{m}_{\mf C,k}(i/N)\leq u\leq \max_{1\leq i\leq N}\tilde{m}_{\mf C,k}(i/N)\}$ so that $\hat{\mf m}_{\mf C,I}(t)$ is well defined on $\hat{\mathcal{T}}$. We consider $(\mathbf{x}_i)$ and $(e_i)$  generated from \begin{equation*}
    \mathbf{x}_i=\mathbf{H}(t_i,\Upsilon_i)\quad \text{and}\quad  e_i=G(t_i,\Upsilon_i),\quad i=1,\dots,n,
\end{equation*}
where $\mathbf{H}=:(H_{1},\dots,H_{p})^\top$ and $G$ are measurable functions. The quantities $H_{i,1},i=1,\dots,n,$ are fixed to be $1$, representing the intercept of the regression coefficients. Furthermore, we make the following assumptions on both $(\mathbf{x}_i)$ and $(e_i)$:
\begin{description}[itemsep=0pt,parsep=0pt,topsep=0pt,partopsep=0pt]
\setlength{\baselineskip}{0.7\baselineskip}
    \item[\setword{\textbf{(B1')}}{(B1')}] \textit{There exists a Lipschitz continuous matrix function $\mf M(t)$ s.t. $\mathbf{M}(t)=:\mathbb{E}\left\{\mathbf{H}\left(t,\Upsilon_0\right) \mathbf{H}\left(t,\Upsilon_0\right)^{\top}\right\}$ and the smallest eigenvalue of $\mf M(t)$ is bounded away from zero for all $t\in[0,1]$. }
    \item[\setword{\textbf{(B2')}}{(B2')}] \textit{$\sup _{u\in[0,1]}\left\{\left\|\mathbf{H}\left(u,\Upsilon_0\right)\right\|_4\right\}<\infty$ and $\|\mathbf{H}\left(u,\Upsilon_0\right)-\mathbf{H}\left(v,\Upsilon_0\right)\|\leq L|u-v|,\forall u,v\in[0,1]$ for some constant $L>0$. }
    \item[\setword{\textbf{(B3')}}{(B3')}] For \textit{$\mathbf{U}\left(t,\Upsilon_i\right)=:\mathbf{H}\left(t,\Upsilon_i\right) G\left(t,\Upsilon_i\right)$, $\sum_{i=2}^n\|\mathbf{U}(t_i,\Upsilon_0)-\mathbf{U}(t_{i-1},\Upsilon_0)\|<\infty$.}
    \item[\setword{\textbf{(B4')}}{(B4')}] \textit{$\delta_4(\mathbf{U}, k)=O(\chi^k)$ for some $\chi\in(0,1)$ and $\sup_{t\in[0,1]}\left\{\left\|\mathbf{U}\left(t,\Upsilon_0\right)\right\|_4\right\}<\infty$}
    \item[\setword{\textbf{(B5')}}{(B5')}] \textit{The smallest eigenvalue of the long-run covariance matrix function $\mf \Lambda(t)=:\mathbf{\Sigma}_{\mathbf{U}}(t)$ is bounded away from 0 for all $t\in[0,1]$. }
\end{description}

Assumptions \hyperref[(B1')]{(B1')} and \hyperref[(B2')]{(B2')} are standard in the literature for analyzing time-varying coefficient linear models and have also been used in \cite{zhou2010simultaneous}. Assumptions \hyperref[(B3')]{(B3')}-\hyperref[(B5')]{(B5')} coincide with \hyperref[(B1)]{(B1)}-\hyperref[(B4)]{(B4)} with fixed dimension,  allowing the error process $(e_i)$ to be piecewise locally stationary. It is noteworthy that \hyperref[(B2')]{(B2')} requires $(\mf x_i)$ to be locally stationary, since if  $(\mf x_i)$ is piecewise locally stationary the local linear estimator \eqref{local linear} is no longer consistent. 
With the above assumptions, the following Theorem \ref{thm:multi-SCB} presents Gaussian approximation of the maximum deviation $\max_{1\leq k\leq s}\sup_{t\in\hat{\mathcal{T}}}|\hat{m}_{\mf C,I,k}(t)-m_{\mf C,k}(t)|$ and provides a valid bootstrap method for mimicking the behavior of the Gaussian process which approximates the maximum deviation. 
\begin{thm}
    \label{thm:multi-SCB}
    Let conditions \hyperref[(B1')]{(B1')}-\hyperref[(B5')]{(B5')}, \hyperref[(C1)]{(C1)}-\hyperref[(C2)]{(C2)} hold 
    and the $\mathbb{R}^s$-valued function $\mathbf{m}_{\mf C}(t)$ satisfy conditions \hyperref[(A1)]{(A1)}-\hyperref[(A2)]{(A2)}. Recall $\tilde K_r^*$ defined in \eqref{eq:K^*}, we have following results:\\
    (i). There exists independent Gaussian random vector $\mathbf{V}_j=:(V_{j,1},\dots,V_{j,s})^{\top}\sim N_s(0,\mathbf{\Sigma}_{\mathbf{C}}(j))$, $j=1,\dots,n$ such that
    \begin{equation}
        \max_{1\leq k\leq s}\sup_{t\in\hat{\mathcal{T}}}\left|\sqrt{nh_r}(\hat{m}_{\mathbf{C},I,k}(t)-m_{\mathbf{C},k}(t))-\sqrt{nh_r}\mathcal{V}_{\mathbf{C},k}(t)\right|=O_p(\rho_n\sqrt{nh_r}),\label{eq:GA in prop:GA in multi-dim}
    \end{equation}
    where $\rho_n=:h_d^2+R_n^2/h_d+n^{1/4}\log^2 (n)/(nh_r)$, $\mathbf{\Sigma}_{\mathbf{C}}(j)=:\mathbf{C}\mathbf{M}^{-1}(t_j)\mf\Lambda(t_j)\mathbf{M}^{-1}(t_j)\mathbf{C}^{\top}$ and 
    \begin{equation*}
        \mathcal{V}_{\mathbf{C},k}(t)=:\frac{m_{\mathbf{C},k}'(t)}{Nh_d}\sum_{i=1}^N K_d\left(\frac{m_{\mathbf{C},k}(i/N)-m_{\mathbf{C},k}(t)}{h_{d}}\right)\frac{1}{nh_r}\sum_{j=1}^n\tilde K_r^*\left(\frac{j/n-i/N}{h_r},\frac{i}{N}\right)V_{j,k}.
    \end{equation*}
    (ii). Suppose $\hat{\mf\Sigma}_{\mf C}(t_j)=:\mathbf{C}\hat{\mathbf{M}}^{-1}(t_j)\hat{\mf\Lambda}(t_j)\hat{\mathbf{M}}^{-1}(t_j)\mathbf{C}^{\top}$ where 
    $$
    \hat{\mathbf{M}}(t_j)=:\frac{1}{nh_r}\sum_{i=1}^n\mathbf{x}_i\mathbf{x}_i^{\top}K_{r}\left(\frac{t_i-t_j^*}{h_r}\right), t_j^*=\max\{h_r,\min(t_j,1-h_r)\},
    $$ 
    \begin{equation}
        \hat{\mf \Lambda}(t_j)=:\frac{1}{L}\left(\sum_{i=j-L+1}^j\mf{x}_i\hat{\varepsilon}_i\right)\left(\sum_{i=j-L+1}^j\mf{x}_i\hat{\varepsilon}_i\right)^\top,\quad \hat{\varepsilon}_i=:y_i-\mf x_i^\top\tilde{\mf m}(t_i).\label{eq:long-run estimation multi-dim}
    \end{equation}    
    Let $\bar{\varphi}_n=:\sqrt{n L}+n L^{-1}+\sqrt{n}L\left(h_r^2+\log(n)/\sqrt{nh_r}\right)+nL\left(h_r^5+\log^2n/(nh_r)\right)$. If 
    \begin{equation*}
        \frac{\log(N)\log^2(n)(\sqrt{n\bar{\varphi}_n}+\bar{\varphi}_n)}{nh_r}=o(1),\quad\sqrt{nh_r\log(N)\log(n)}=o(Nh_d),
    \end{equation*}
    and $\rho_n\sqrt{nh_r}=o(1)$, then there exists random vectors $\mf V_j^*=:(V_{j,1}^*,\dots,V_{j,s}^*)^\top,j=1,\dots,n,$ which are conditionally independent given data and follow $\mf{V}_j^*|\Upsilon_n\sim N_s(0,\hat{\mf \Sigma}_{\mf C}(t_j))$ such that             
    \begin{equation}
        \sup_{x\in\mathbb{R}} \left|\mathbb{P}\left\{\max_{1\leq k\leq s}\sup_{t\in\hat{\mathcal{T}}} \left|\sqrt{nh_r}\mathcal{V}_{\mf C,k}(t) \right|\leq x \right\} - \mathbb{P}\left\{\max_{1\leq k\leq s}\sup_{t\in\hat{\mathcal{T}}} \left| \sqrt{nh_r}\mathcal{V}^*_{\mf C,k}(t) \right|\leq x \,\middle\vert\, \Upsilon_n \right\} \right|\rightarrow_p 0,
        \label{eq:boot in thm:multi-dim bootstrap theorem}
    \end{equation}
    where
    \begin{equation}
        \mathcal{V}_{\mf C,k}^{*}(t)=:\frac{1}{nh_{r}}\sum_{j=1}^n\sum_{i=1}^N \hat W_{i,k,I}^*(t)\tilde K_r^*\left(\frac{j/n-i/N}{h_{r}},\frac{i}{N}\right) V_{j,k}^*\label{eq:V star in multi-dim} , 
    \end{equation}
    $$
    \hat W_{i,k,I}(t)=K_d\left(\frac{\tilde m_{C,k}(i/N)-\hat m_{C,I,k}(t)}{h_{d}}\right),\quad \hat W_{i,k,I}^*(t)=\hat W_{i,k,I}(t)/\sum_{i=1}^N \hat W_{i,k,I}(t).
    $$
\end{thm} 

Result (i) provides the Gaussian approximation towards the maximum deviation but leaves challenges in estimating $\mf{m}_{\mf C}^\prime(t)$ and the long-run covariance related quantity $\mathbf{\Sigma}_{\mathbf{C}}(t_j)$. Result (ii) avoids estimating $\mf m_{\mf C}'(t)$ and provides an appropriate estimator $\hat{\mf \Sigma}_{\mf C}(t_j)$ for $\mathbf{\Sigma}_{\mathbf{C}}(t_j)$ with negligible cumulative errors. Combining (i) and (ii) in Theorem \ref{thm:multi-SCB}, we propose Algorithm \ref{alg:multi-dim} for constructing the joint SCBs of $\mf{m}_{\mf C}(t)$ in practice, of which the validity is guaranteed by Theorem \ref{thm:multi-SCB}.

\begin{breakablealgorithm}
	\caption{Bootstrap for joint SCBs of $\mathbf{m}_{\mathbf{C}}(t)$}
	\label{alg:multi-dim}
    \begin{algorithmic}[0]
\setlength{\baselineskip}{0.7\baselineskip}
		\STATE{\textbf{Data:}} $(y_i,\mf x_i)$, $i=1,\dots,n$ and known matrix $\mathbf{C}\in\mathbb{R}^{s\times p}$.
        \STATE{\textbf{Initialization}}: Choose bandwidth $h_{r}$, $h_{d}$ and window size $L$ by Section \ref{Bandwidth selection}.
        \
        \STATE{\textbf{Step 1:}} Obtain the estimator $\tilde{\mf m}_{\mf C}(t)=2\hat{\mf m}_{h_r/\sqrt{2}}(t)-\hat{\mf m}_{h_r}(t)$ with $\hat{\mf m}_{h_r}(t)$ defined in \eqref{local linear}.
        \STATE{\textbf{Step 2:}} Obtain the monotone estimator $\hat{\mathbf{m}}_{\mf C, I}(t)$ defined below \eqref{eq:m_I.inv multi-dim}.
        \STATE{\textbf{Step 3:}} 
        Calculate
        $\hat{\mf \Sigma}_{\mf C}(j)=\mf C\hat{\mf M}^{-1}(t_j)\hat{\mf \Lambda}(t_j)\hat{\mf M}^{-1}(t_j)\mf C^\top$ by plugging in 
       $\hat{\mathbf{M}}(t_j)$ and $\hat{\mf \Lambda}(t_j)$ in \eqref{eq:long-run estimation multi-dim} with bandwidth $h_r$ and window size $L$, respectively.
        \STATE{\textbf{Step 4:}} Given data, generate the Gaussian vector process $\mathcal{V}^*_{\mf C}(t)=:(\mathcal{V}^*_{\mf C,1}(t),\dots,\mathcal{V}^*_{\mf C,s}(t))^\top$ where $\mathcal{V}_{\mf C,k}^*(t)$ is defined in \eqref{eq:V star in multi-dim}.
         \STATE{\textbf{Step 5:}} Repeat step 4 for $B$ times and obtain the samples $\Big\{\max_{1\leq k\leq s}\sup_{t\in\hat{\mathcal{T}}} \Big| \mathcal{V}^{*(r)}_{\mf C,k}(t) \Big|\Big\}_{1\leq r\leq B}$. Let $\hat q_{1-\alpha}$ be the   $(1-\alpha)_{th}$ sample quantile. 
        \STATE{\textbf{Output:}} Level $(1-\alpha)$ joint SCBs $\hat{\mf m}_{\mf C,I}(t)\pm \hat q_{1-\alpha}\mathbf{1}_s$ where $\mathbf{1}_s=(1,\dots,1)^\top\in \mathbb{R}^s$. 
	\end{algorithmic}
\end{breakablealgorithm}

\section{Tuning parameter selection}
\label{Bandwidth selection}
To implement our method it is necessary to select bandwidths $h_r$, $h_d$, and window size $L$ in \eqref{eq:ll def}, \eqref{eq:m_I} and \eqref{eq:long-run var def} respectively. The first bandwidth $h_r$ is introduced to apply local linear estimation, 
 which we select via the General Cross Validation (GCV) selector proposed by \cite{craven1978smoothing}. This method has been applied in various fields of nonstationary time series regressions (see for example \cite{wu2017nonparametric,bai2023detecting}) and also works well in our simulation and data analysis. To be precise, for each dimension $k$ in model \eqref{high-monotone model}, we can write $\hat{\mathbf{Y}}_k=\mf Q(h_k) \mathbf{Y}_k$ for some square matrix $\mf Q$, where $\mathbf{Y}_k=(y_{1,k},\dots,y_{n,k})^\top$ and $\hat{\mathbf{Y}}_k=(\hat{y}_{1,k},\dots,\hat{y}_{n,k})^\top$ denote the vector of observed values and estimated values respectively, and $h_k$ is the candidate bandwidth. Then we choose bandwidth $\hat h_{r,k}$ where
\begin{equation}
    \hat{h}_{r,k}=\underset{h_k}{\arg \min }\{\operatorname{GCV}(h_k)\}, \quad \operatorname{GCV}(h_k)=\frac{n^{-1}|\mathbf{Y}_k-\hat{\mathbf{Y}}_k|^2}{[1-\operatorname{tr}\{\mf Q(h_k)\} / n]^2} .\notag
\end{equation}
We can apply the above procedure to each dimension and then select an average bandwidth $\hat{h}_r=\sum_{k=1}^p\hat{h}_{r,k}/p$ for high-dimensional model. \TPR{For the range of candidate $h_k$s, we recommend a rule of thumb strategy in \cite{bai2024difference} which is wrapped into the ``rule\_of\_thumb’’ function in R package ``mlrv’’.}
The GCV for model \eqref{Time-varying_n} follows the same formula with $\mf Y=(y_1,...,y_n)^\top$ and $\hat{\mf Y}=(\hat y_1,...,\hat y_n)^\top$. 

The second bandwidth $h_d$ is introduced for the monotone rearrangement based on the local linear estimates, which in practice is usually chosen to be small as long as the assumption \hyperref[(C2)]{(C2)} on $h_r$ and $h_d$ is satisfied. 
Following \cite{dette2006simple}, choosing $h_d$ that is much smaller than $h_r$ can mitigate bias caused by monotone rearrangement and boundary effects (see also \cite{dette2019detecting}) while the analysis results are also insensitive to $h_d$ as long as it is sufficiently small. In our simulation and data analysis practice, we recommend choosing $h_d=R_n^{2/3}$ while $h_d$ is much smaller than $h_r$.

As for window size $L$ in \eqref{eq:cumulative long-run cov}, we recommend the minimum volatility (MV) method advocated by \cite{politis1999subsampling}. Let $\hat{\mf Q}(k)^{(L)}$ denote the quantity $\hat{\mf Q}(k)$ in \eqref{eq:cumulative long-run cov} using window size $L$, then the MV method suggests that the estimator $\hat{\mf\Sigma}_k^{(L)}=:\hat{\mf Q}(k)^{(L)}-\hat{\mf Q}(k-1)^{(L)}$ becomes stable when $L$ is in an appropriate range. Specifically, one could firstly set a series of candidate window sizes $L_1<L_2,\dots< L_r$. Compute $\bar{\hat{\mf\Sigma}}_k=\sum_{j=1}^r \hat{\mf\Sigma}_k^{(L_j)}/M$ and
\begin{equation}
    \mf{MV}(j)=:\max_{L_r\leq k\leq n}\text{se}\left( \{\hat{\mf\Sigma}_k^{(L_{i})}\}_{i=1\vee(j-3)}^{r\wedge(j+3)} \right),\quad \text{se}\left( \{\hat{\mf\Sigma}_k^{(L_{j})}\}_{j=1}^r \right)=:\operatorname{tr}\left[\frac{1}{r-1}\sum_{j=1}^r[\hat{\mf\Sigma}_k^{(L_{j})}-\bar{\hat{\mf\Sigma}}_k]^2 \right]^{1/2}.\notag
\end{equation}
Then we finally choose $L\in\{L_j:j=1,\dots,r\}$ which minimizes $\mf{MV}(j)$.  
\section{Simulation study} 
\label{Simulation study}
In this section, we investigate the empirical coverage probabilities of our joint SCBs for high-dimensional monotone trend and smoothly monotone time-varying coefficient in Section \ref{High dimensional SCB} and Section \ref{Time-varying coefficient regression}, respectively, via extensive simulation studies. \TPR{Noting that the width of an SCB is important for verifying SCB's validity, we also document the widths of our joint SCBs and compare them with the SCB method without monotone constraints for nonstationary data (e.g.,\cite{zhou2010simultaneous}) in the Section \ref{sec: widths of SCB in simulation} 
of the online supplement.
All simulation results are based on 400 simulation runs with sample size $n=300,500,1000$ and bootstrap sample size $B=2000$.} 
We also set $N=4000$ which maximizes the effectiveness of the rearrangement step under our computational ability constraints.
We consider three smooth monotone functions for both model \eqref{high-monotone model} and \eqref{Time-varying_n}, which are $m_1(t)=0.5t^2+t$, $m_2(t)=\exp(t)$, $m_3(t)=2\ln(t+1)$.

In the high-dimensional model \eqref{high-monotone model} with $p\geq 3$, we define our $p$-dimensional regression function as $\mathbf{m}(t)=\left(\mathbf{A}_1 ^\top m_1(t),\mathbf{A}_2^\top m_2(t),\mathbf{A}_3^\top m_3(t)\right)^\top$ where $\mathbf{A}=(a_1,\dots,a_p)^{\top}$ is a $p\times 1$ matrix with $a_i=1+0.2(i/p)^{0.5}$ and $\mathbf{A}_1, \mathbf{A}_2, \mathbf{A}_3$ are the sub-matrices of $\mathbf{A}$ consisting of the first $\lfloor p/3 \rfloor_{th}$ rows, the $(\lfloor p/3 \rfloor+1)_{th}$ to $\lfloor 2p/3 \rfloor_{th}$ and the $(\lfloor 2p/3 \rfloor+1)_{th}$ to $p_{th}$ rows of $\mathbf{A}$, respectively. The error terms are generated by the following two models:
\begin{description}
[itemsep=0pt,parsep=0pt,topsep=0pt,partopsep=0pt]
\setlength{\baselineskip}{0.7\baselineskip}
    \item[(a)] Locally stationary: $\mathbf{e}_i=\mathbf{G}(t_i,\Upsilon_i) = b(i/n)\mathbf{G}_{i}(\Upsilon_{i-1})+\boldsymbol{\xi}_i$, $b(t)=0.15(0.9+0.1\sin(2\pi t))$,
    \item[(b)] Piecewise stationary: $\mathbf{e}_i=\mathbf{G}^{(0)}(t_i,\Upsilon_i)$ for $i/n\leq 1/3$ and $\mathbf{e}_i=\mathbf{G}^{(1)}(t_i,\Upsilon_i)$ for $i/n>1/3$ where $\mathbf{G}^{(0)}(t_i,\Upsilon_i)=0.5\mathbf{G}^{(0)}(t_i,\Upsilon_{i-1})+\boldsymbol{\xi}_i$, $\mathbf{G}^{(1)}(t_i,\Upsilon_i)=-0.5\mathbf{G}^{(1)}(t_i,\Upsilon_{i-1})+\boldsymbol{\xi}_i,$
\end{description}
where $\boldsymbol{\xi}_i$ for both models follows i.i.d. $p$-dimensional Gaussian distribution $N_p(0,\mathbf{\Sigma}_e)$ with a block-diagonal covariance matrix $\boldsymbol{\Sigma}_e=$ $\left(\Sigma_{j, l}\right)_{j, l=1}^p$, where $\Sigma_{j, j}=1$ for $j=1, \ldots, p$, and $\Sigma_{j, l}=(-0.95)^{|j-l|}$ for $j$ and $l$ lower than $\lfloor p/2\rfloor$, and zero elsewhere. 


Table \ref{tab:SCB coverage high main} presents the simulated coverage probabilities for model \eqref{high-monotone model} with both (a) and (b). For simplicity we use the same bandwidth for each dimension, that is, $h_{r,k}=h_r$. 
We apply the simultaneous inference procedure presented in Algorithm \ref{alg:high-dim bootstrap} to obtain 90\% and 95\% joint SCBs on the support set $\hat{\mathcal{T}}$. \TPR{We examine various bandwidths over a wide range and find that the performance of our methods improves and becomes more robust with respect to the choices of bandwidth when the sample size increases. Importantly, the results yielded by GCV methods are all reasonably good at all sample sizes. The empirical coverage probabilities tend to be inaccurate with the smallest bandwidth $h_r=0.05$. This may be due to the small sample size, as predicted by our theorem and supported by the fact that $h_r=0.05$ is excluded by the range of our GCV selection result; please refer to Section \ref{sec:GCV selection} 
in the online supplement for detailed discussion.}

\begin{table}[htb]
  \centering
  \setlength{\belowcaptionskip}{-0.5cm}
  \captionsetup{font={small}, skip=5pt}
  \caption{Simulated coverage probabilities of joint SCBs for high-dimensional $\mathbf{m}$}
  \scalebox{0.8}{
  {\renewcommand{\arraystretch}{0.5}
     \begin{tabular}{ccccccc c cccccc}
    \hline
    $(n=300)$ & \multicolumn{6}{c}{\textit{Model (a)}}& & \multicolumn{6}{c}{\textit{Model (b)}} \\
    \cline{2-7}\cline{9-14}
          & \multicolumn{2}{c}{$p=9$} & \multicolumn{2}{c}{$p=18$} & \multicolumn{2}{c}{$p=27$} & & \multicolumn{2}{c}{$p=9$} & \multicolumn{2}{c}{$p=18$} & \multicolumn{2}{c}{$p=27$} \\
    $h_r$ & 90\%  & 95\%  & 90\%  & 95\%  & 90\%  & 95\% & & 90\%  & 95\%  & 90\%  & 95\%  & 90\%  & 95\% \\
    \hline
     GCV   & 0.890 & 0.945 & 0.855 & 0.935 & 0.813 & 0.923 &       & 0.915 & 0.968 & 0.860 & 0.943 & 0.870 & 0.948 \\
    0.05  & 0.998 & 1.000 & 1.000 & 1.000 & 1.000 & 1.000 &       & 0.998 & 1.000 & 1.000 & 1.000 & 1.000 & 1.000 \\
    0.1   & 0.945 & 0.985 & 0.958 & 0.993 & 0.980 & 0.995 &       & 0.958 & 0.988 & 0.985 & 1.000 & 0.988 & 0.998 \\
    0.15  & 0.935 & 0.985 & 0.960 & 0.983 & 0.965 & 0.990 &       & 0.938 & 0.980 & 0.965 & 0.993 & 0.985 & 1.000 \\
    0.2   & 0.910 & 0.973 & 0.938 & 0.978 & 0.935 & 0.980 &       & 0.920 & 0.965 & 0.953 & 0.980 & 0.963 & 0.993 \\
    0.25  & 0.893 & 0.968 & 0.923 & 0.980 & 0.903 & 0.970 &       & 0.890 & 0.965 & 0.943 & 0.980 & 0.940 & 0.980 \\
    0.3   & 0.913 & 0.958 & 0.923 & 0.973 & 0.898 & 0.960 &       & 0.880 & 0.968 & 0.940 & 0.978 & 0.915 & 0.963 \\
    0.35  & 0.895 & 0.960 & 0.920 & 0.973 & 0.888 & 0.955 &       & 0.878 & 0.943 & 0.923 & 0.975 & 0.893 & 0.950 \\
    \hline
     $(n=500)$ & \multicolumn{6}{c}{\textit{Model (a)}}& & \multicolumn{6}{c}{\textit{Model (b)}} \\
    \cline{2-7}\cline{9-14}
          & \multicolumn{2}{c}{$p=9$} & \multicolumn{2}{c}{$p=18$} & \multicolumn{2}{c}{$p=27$} & & \multicolumn{2}{c}{$p=9$} & \multicolumn{2}{c}{$p=18$} & \multicolumn{2}{c}{$p=27$} \\
    $h_r$ & 90\%  & 95\%  & 90\%  & 95\%  & 90\%  & 95\% & & 90\%  & 95\%  & 90\%  & 95\%  & 90\%  & 95\% \\
    \hline
    GCV   & 0.843 & 0.930 & 0.878 & 0.960 & 0.883 & 0.958 &       & 0.870 & 0.948 & 0.915 & 0.970 & 0.913 & 0.960 \\
    0.05  & 0.260 & 0.428 & 0.265 & 0.463 & 0.218 & 0.455 &       & 0.315 & 0.523 & 0.278 & 0.518 & 0.280 & 0.483 \\
    0.1   & 0.795 & 0.890 & 0.823 & 0.910 & 0.850 & 0.958 &       & 0.828 & 0.910 & 0.858 & 0.930 & 0.873 & 0.950 \\
    0.15  & 0.888 & 0.948 & 0.920 & 0.980 & 0.948 & 0.983 &       & 0.890 & 0.935 & 0.935 & 0.980 & 0.938 & 0.978 \\
    0.2   & 0.875 & 0.955 & 0.895 & 0.973 & 0.905 & 0.973 &       & 0.898 & 0.955 & 0.925 & 0.978 & 0.930 & 0.980 \\
    0.25  & 0.858 & 0.945 & 0.893 & 0.978 & 0.890 & 0.960 &       & 0.888 & 0.955 & 0.918 & 0.975 & 0.923 & 0.975 \\
    0.3   & 0.848 & 0.925 & 0.890 & 0.955 & 0.875 & 0.948 &       & 0.883 & 0.945 & 0.895 & 0.960 & 0.903 & 0.965 \\
    0.35  & 0.845 & 0.925 & 0.880 & 0.953 & 0.868 & 0.930 &       & 0.880 & 0.950 & 0.893 & 0.955 & 0.893 & 0.955 \\
    \hline
    $(n=1000)$ & \multicolumn{6}{c}{\textit{Model (a)}}& & \multicolumn{6}{c}{\textit{Model (b)}} \\
    \cline{2-7}\cline{9-14}
          & \multicolumn{2}{c}{$p=9$} & \multicolumn{2}{c}{$p=18$} & \multicolumn{2}{c}{$p=27$} & & \multicolumn{2}{c}{$p=9$} & \multicolumn{2}{c}{$p=18$} & \multicolumn{2}{c}{$p=27$} \\
    $h_r$ & 90\%  & 95\%  & 90\%  & 95\%  & 90\%  & 95\% & & 90\%  & 95\%  & 90\%  & 95\%  & 90\%  & 95\% \\
    \hline
    GCV   & 0.880 & 0.950 & 0.878 & 0.950 & 0.895 & 0.953 &       & 0.873 & 0.943 & 0.885 & 0.953 & 0.930 & 0.970 \\
    0.05  & 0.488 & 0.683 & 0.243 & 0.440 & 0.153 & 0.340 &       & 0.238 & 0.443 & 0.200 & 0.390 & 0.225 & 0.440 \\
    0.1   & 0.935 & 0.975 & 0.900 & 0.975 & 0.863 & 0.950 &       & 0.830 & 0.925 & 0.865 & 0.955 & 0.903 & 0.965 \\
    0.15  & 0.903 & 0.958 & 0.910 & 0.973 & 0.918 & 0.965 &       & 0.890 & 0.953 & 0.913 & 0.985 & 0.940 & 0.980 \\
    0.2   & 0.870 & 0.945 & 0.875 & 0.963 & 0.913 & 0.970 &       & 0.880 & 0.945 & 0.905 & 0.978 & 0.925 & 0.970 \\
    0.25  & 0.855 & 0.930 & 0.850 & 0.950 & 0.900 & 0.970 &       & 0.868 & 0.938 & 0.890 & 0.955 & 0.908 & 0.960 \\
    0.3   & 0.878 & 0.933 & 0.848 & 0.928 & 0.890 & 0.958 &       & 0.865 & 0.938 & 0.880 & 0.955 & 0.898 & 0.955 \\
    0.35  & 0.860 & 0.925 & 0.845 & 0.930 & 0.895 & 0.955 &       & 0.865 & 0.938 & 0.898 & 0.958 & 0.895 & 0.950 \\
        \hline
    \end{tabular}
  }}\label{tab:SCB coverage high main}%
\end{table}%

We also conduct a simulation scenario for the time-varying coefficient linear model in Section \ref{Time-varying coefficient regression} and the results yielded by GCV methods are all reasonably good at all sample sizes. The results also show that the performance of joint SCBs generated by Algorithm \ref{alg:multi-dim} is reasonably accurate in most $h_r$ and becomes more robust with respect to $h_r$ as the sample size grows. Due to the page limit, we have placed the detailed simulation set-ups and results in the Section \ref{sec: simulation for time-varying linear regression} 
of the online supplement.

\section{Climate data analysis}
\label{Empirical study}
In this section, we study the historical monthly highest temperature collected from various weather stations in the United Kingdom. This dataset can be obtained from the website (https://www.metoffice.gov.uk/research/climate/maps-and-data/historic-station-data). The maximum temperature serves as an index for discerning temperature anomalies and extremes within the context of global warming, and we study the joint SCBs for this data set. \TPR{Although preliminary studies show that the maximum temperature is slightly skewed, our theorem framework allows asymmetric innovation or error process. We also add simulation with skewed data to justify this as displayed in the Section \ref{sec: additional simulation} 
of the online supplement.}
We investigate the monthly highest temperature series over 27 stations from 1979-2023; in other words, we consider model \eqref{high-monotone model} with $p=27$ and $n=540$. The names of the stations are listed in the online supplement. It should be noted that there are only a few missing points, which we interpolate using ``approx’’ in R via observations from the same months in nearby years. \TPM{In the online supplement, we also display analysis results based on other interpolation methods including spline interpolation and Kalman filtering. The results remain robust across different interpolations, demonstrating that our statistical findings are not sensitive to the choice of methods for preprocessing missing values.}
Our goal is to conduct inference for all the deseasonalized temperature trends simultaneously, where we implement  ``stl’’ in R for the deseasonalization. \TPR{We checked the error induced by the ``stl’’ procedure is of rate $O_p(\Theta(p)n^{-1/2})$, which is negligible in contrast with the nonparametric estimation under conditions in Proposition \ref{prop:GA in high-dim}. Therefore, our simultaneous theoretical framework can be applied to the deseasonalized time series.}  
Taking into account global warming, the inference is performed under the constraint that all the temperature trends are monotone. 

We apply Algorithm \ref{alg:high-dim bootstrap} to generate the joint SCBs for the 27 monotone trends via $B=2000$ bootstraps. The tuning parameters $h_r,h_d$, and $L$ are selected according to the methods in Section \ref{Bandwidth selection}. 
In Figure \ref{fig:joint SCBs} we present our joint SCBs 
for testing quadratic trend. \TPR{Specifically, we test the quadratic null hypothesis: $m_k(t)\equiv a_kt^2+b_kt+c_k$ for all $k=1,\dots,p$, with parameters $a_k,b_k,c_k$ estimated under constraints $a_k,b_k>0$ by the constrained least square method. We impose the constraints to fit an accelerating increasing quadratic function, testing whether an accelerated global warming as warned in \cite{nature2018} exists. The least square fitting procedure is commonly used in testing parametric assumptions of trends in time series such as \cite{brown1975techniques}, \cite{nyblom1989testing}, and \cite{wu2011testparam}. Following \cite{Zhou2015constrain}, it is not difficult to show the $O_p(\Theta(p)n^{-1/2})$ consistent rate of our constrained least square estimate with nonnegative constraint under the null hypothesis. This means the parametric fitting error is asymptotically negligible compared with the approximation error of the nonparametric estimation under conditions in Proposition \ref{prop:GA in high-dim}. Therefore the estimated $a_k, b_k, c_k$ can be utilized to test the existence of the accelerating trend, combining with our joint SCBs.}

    The test result rejects the null hypothesis under significance level $5\%$/$10\%$ because the fitted quadratic trends offend the joint SCBs in multiple stations. \TPR{Both our monotone estimator and local linear estimator find the temperature curves increasing fastest at the beginning years. Simultaneously, our accelerating quadratic trends under the null hypothesis increase most slowly at the beginning time points. Therefore, we can see the beginning time points are the most likely offending area for the SCB testing.} This also reveals that the quadratic trend on regional scale may not be universally present even though previous studies supported the quadratic trend on a global scale (e.g., \cite{wu2007inference}).

\begin{figure}[htb]
    \centering
    \includegraphics[width=\linewidth]{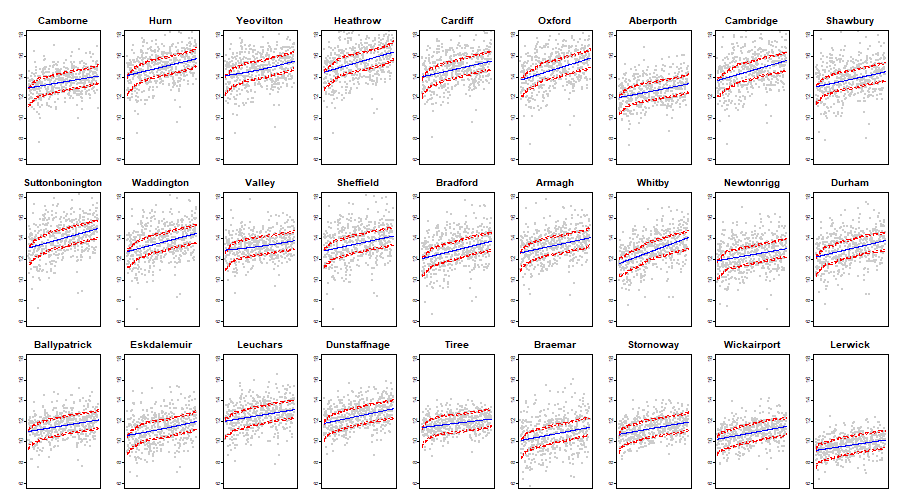}
    \captionsetup{font={small, stretch=1},skip=0cm}
    \caption{Joint SCBs for monthly highest temperature (°C, y-axis) in 27 stations from 1979 to 2023 (x-axis). The points represent the time series adjusted by removing seasonal effects. The solid lines represent quadratic trend \TPR{fitted by constrained LSE with function ``nls()’’ in R}. The dotted and dashed lines represent the 90\% and 95\% joint SCBs respectively. Notice that the solid line is not covered by the 95\% SCB at stations like Cardiff, Durham, Hurn and Valley. }
    \label{fig:joint SCBs}
\end{figure}

\TPR{
Our joint SCBs can be further used to test whether an increase of value $C$ occurred over the entire time span $[t_0,t_1]$. In specific, the null hypothesis should be
\begin{equation}\label{eq:fundement test}
    H_0:\sup_{t,t'\in[t_0,t_1]} |m_k(t)-m_k(t')|\leq C,\quad \forall k=1,\dots,p, ~\text{$m_k(t)s$ are monotone increasing.}
\end{equation}
Due to the monotone increasing, 
$m_k(t_1)-m_k(t_0)=\sup_{t,t'\in[t_0,t_1]} |m_k(t)-m_k(t')|$. 
However, this 
does not enable us to construct a test straightforwardly using the local linear estimator $\hat m_k(t_0)$ and $\hat m_k(t_1)$  or the jackknife estimator $\tilde m_k(t_0)$ and $\tilde m_k(t_1)$, since those estimators cannot capture the monotone shape of the curve. 
Based on our level (1-$\alpha$) joint SCBs $\hat{\mf m}_I\pm \hat{q}_{1-\alpha}\mathcal{B}_p$ obtained from Algorithm \ref{alg:high-dim bootstrap}, we shall reject the null hypothesis in \eqref{eq:fundement test} at significance $\alpha$ if $\mf m(t)$ cannot be covered by $\hat{\mf m}_I\pm \hat{q}_{1-\alpha}\mathcal{B}_p$, that is
\begin{equation}
    \label{eq:rejection criteria}
    \max_{1\leq k\leq p}(\hat{m}_{I,k}(t_1)-\hat{q}_{1-\alpha})-(\hat{m}_{I,k}(t_0)+\hat{q}_{1-\alpha})>C.
\end{equation}
The value of $C$ can be determined based on field knowledge. In the context of global warming, a report from \cite{ipcc2022} has identified approximately $0.5$°C difference in temperature extremes over the time spans. To investigate whether the finding in \cite{ipcc2022} appears significant in the UK data, we set $C=0.5$, and make the time $t_0$ and $t_1$ correspond to the beginning (1979) and end (2023) of our data set. We find the highest confidence level at which \eqref{eq:rejection criteria} holds is $1-\alpha=99.5\%$, thus we reject the null hypothesis in \eqref{eq:fundement test} with a small $p$-value of $0.005$, which supports the conclusion of the IPCC report.
}

\TPR{More generally, we can also test whether there exists a time span $[t_0,t_1]$ whose length is less than $\Delta$ ($\Delta\leq 1$) such that the maximum temperature has significantly increased over a given amount, through testing the null hypothesis 
\begin{equation}
    H_0:\quad \max_{1\leq k\leq p}\sup_{0<t_1-t_0\leq \Delta}m_k(t_1)-m_k(t_0)\leq C,  \text{$m_k(t)s$ are monotone increasing.}\label{eq:null delta}
\end{equation}
For example, we set $C = 0.5$, and a length of $\Delta=3/4$ (about 3-decade long). Based on similar arguments of \eqref{eq:rejection criteria}, we shall reject the null hypothesis in \eqref{eq:null delta} at significance $\alpha$ if 
\begin{equation}
    \label{eq:rejection criteria delta}
    \max_{1\leq k\leq p}\sup_{0<t_1-t_0\leq \Delta}(\hat{m}_{I,k}(t_1)-\hat{q}_{1-\alpha})-(\hat{m}_{I,k}(t_0)+\hat{q}_{1-\alpha})>C,
\end{equation}
The highest confidence level at which \eqref{eq:rejection criteria delta} holds is 1-$\alpha$ = 96\%. Thus we reject \eqref{eq:null delta} with a small $p$-value of 0.04. We also test \eqref{eq:null delta} with ``$\min_{1\leq k\leq p}$",``$>$" instead of ``$\max_{1\leq k\leq p}$",``$\leq$" respectively, and the $p$-value is very close to $1$. The results support the IPCC's report.} 

Simultaneous inference for the relationship between atmospheric temperature and other environmental factors can also be studied by our joint SCBs with time-varying coefficient linear model as introduced in Section \ref{Time-varying coefficient regression}. Prior research \cite{van2015relationship} has indicated that the maximum temperature exhibits the highest degree of correlation to sunshine duration when compared with minimum and mean temperatures. Therefore, we can investigate the following time-varying coefficient linear model.
\begin{equation}
    T_{max,i} = m_0(t_i)+m_1(t_i)SD_i+e_i,\quad i=1,\dots,n,\label{eq:tmax-SD}
\end{equation}
where $(T_{max,i})_{i=1}^n$ is the series of monthly maximum temperature and $SD_i$ represents the series of monthly sunshine duration. 
\TPM{Notice that a lagged sunshine duration such as $SD_{i-1}$ can also be added into model \eqref{eq:tmax-SD} and Theorem \ref{thm:multi-SCB} is still valid. The associated analysis of regression with lagged terms and the comparison with \eqref{eq:tmax-SD} are included in Section \ref{regressionsection} 
of the online supplement. The comparison shows both models draw the same conclusions. 
Furthermore, based on the VIC variable selection criterion proposed by \cite{zhang2012inference}, we recommend using \eqref{eq:tmax-SD} without lagged terms.  }

In model \eqref{eq:tmax-SD}, we assume $m_1(t)$ is strictly increasing 
which has been supported by evidence from climatology research. \cite{matuszko2015relationship} finds more significant and higher correlation coefficients between sunshine duration and air temperature in 1954-2012 than 1884-1953. 
Using the methodology proposed in Section \ref{Time-varying coefficient regression}, we can obtain the monotone estimates as well as the corresponding SCB for the coefficient $m_1(t)$. These can be used to test whether the sunshine duration always has a positive effect on temperature including monthly maximum temperature. For example, the data at Heathrow suggests a significant positive relationship between sunshine duration and temperature (at $p$-value 0.046) while the data at Lerwick doesn't. \TPM{We additionally apply unconstrained SCBs \citep{zhou2010simultaneous} and improved monotone SCBs in \cite{chernozhukov2009improving} with isotonic regression \citep{brunk1969estimation} and monotone spline \citep{tantiyaswasdikul1994isotonic}. The comparison shows our monotone SCBs are narrowest with the lowest $p$-values.}
Due to the page limitation, we present detailed inference analysis for \eqref{eq:tmax-SD} using the historical data from Heathrow and Lerwick as illustrative examples in Section \ref{regressionsection} 
of the online supplement.

\section{Discussion and future work}
\label{Discussion}
In this study, we develop a simultaneous inference framework for monotone and smoothly time-varying functions based on the nonparametric and monotone rearrangement estimator, allowing complex temporal dynamics.  We study the inference problem in detail for two important and practical scenarios, namely, the high-dimensional trends model and the time-varying coefficient linear model. In these contexts, we have developed Gaussian approximations and the associated bootstrap algorithms for constructing joint SCBs. Importantly, the proposed joint SCBs are asymptotically correct. 
In the numerical study, our proposed joint SCBs also show a narrower width compared with the existing method. 

Skewness and optimal bandwidth choice are new issues that naturally arise in our theoretical application. Additional simulations in Section \ref{sec:skew simulation} 
of the online supplement confirm our framework's ability to handle asymmetric innovation. Future improvements to our Gaussian approximation could include skewness corrections, as suggested by \citet{hall2013bootstrap}. \TPR{Unlike kernel-based nonparametric estimation, the optimal bandwidth for our SCBs isn't determined by the bias-variance trade-off. We use the jackknife method, making the bias of the local linear estimator asymptotically negligible compared with its stochastic variation. This results in undersmoothing, a common practice in nonparametric inference (see \cite{bjerve1985uniform}, \cite{hall1992bootstrap}, \cite{wasserman2006all}). An alternative approach to bandwidth selection in this context is to minimize type 2 error while keeping type 1 error within $(\alpha-c_{min}, \alpha+c_{min})$ for sufficiently small $c_{min}$, see \cite{gao2008bandwidth}. We propose exploring this approach for optimal SCB widths as a promising future research direction.}
\bigskip
\begin{center}
{\large\bf SUPPLEMENTARY MATERIAL}
\end{center}
\begin{description}
[itemsep=0pt,parsep=0pt,topsep=0pt,partopsep=0pt]
\setlength{\baselineskip}{1\baselineskip}
\item[A Proofs.] Providing proofs of main results in this paper with essential lemmas.
\item[B Appendix.] Giving specific examples of nonstationary process, discussion on theoretical results, additional simulations, and extra analysis of the real data in Section \ref{Empirical study}.
\end{description}

\noindent \textbf{Acknowledgments}: The authors gratefully acknowledge the editor, associate editor, and the referees
for their helpful comments and suggestions.


\clearpage

  \bigskip
  \bigskip
  \bigskip
  \begin{center}
    {\LARGE\bf  Supplement to ``Simultaneous inference for monotone and smoothly time-varying functions under complex temporal dynamics"} 
\end{center}
  \medskip







\renewcommand{\thesection}{\Alph{section}}
\renewcommand{\thetable}{\thesection.\arabic{table}}
\renewcommand{\thefigure}{\thesection.\arabic{figure}}
\setcounter{section}{0} 
This supplementary material will provide detailed proofs in Section \ref{sec:Proofs} and a further appendix in Section \ref{sec:Appendix}. Section \ref{sec:Proofs} contains proofs of the theorems and propositions in main paper with essential lemmas and auxiliary results. Section \ref{sec:Appendix} displays specific examples, further discussion, additional simulation and extra analysis of the data used in main paper. 

\TPR{In particular, Section \ref{sec:specific examples} introduces a series of time series models including high-dimensional linear process, AR, MA, and GARCH in the formulation of $\mf G(u,\Upsilon_i)$ and verifies their dependence structure.}

Section \ref{sec:discussion supp} presents discussion on 4 aspects. Subsection \ref{sec:reconcile discussion} discusses how we reconcile the asymptotic results with the bounded time interval setting. Subsection \ref{sec:context discussion} discusses the difference between our theorem framework and works in \cite{dette2006simple}. Subsection \ref{sec:width dicussion} discusses the difficulty in minimal volume of SCB under the monotone condition. \TPM{Subsection \ref{sec:interpolation} discusses how the missing values affect our results if different interpolation methods are used.}

\TPR{Section \ref{sec: additional simulation} includes a series of additional simulations. Subsection \ref{sec: widths of SCB in simulation} contains the width comparisons between our monotone SCB and other SCBs without monotone constraint. Subsection \ref{sec:skew simulation} shows the simulation scenarios for skewed high-dimensional time series. Our GCV selections for the simulation set-ups in the main paper can be found in Subsection \ref{sec:GCV selection}. \TPM{Subsection \ref{sec:penalization simulation} validates our penalized SCB proposed in Section \ref{sec:penalization} through simulations of weakly monotone trends}
Subsection \ref{sec: simulation for time-varying linear regression} presents the simulation for time-varying linear models.}

Section \ref{regressionsection} applies the time-varying linear model to investigate the relationship between sunshine duration and atmospheric temperature. Location information for stations investigated in the main paper is also provided.

\section{Proofs}
\label{sec:Proofs}This supplementary material will provide detailed proofs in Section \ref{sec:Proofs} and a further appendix in Section \ref{sec:Appendix}. Section \ref{sec:Proofs} contains proofs of the theorems and propositions in main paper with essential lemmas and auxiliary results. Section \ref{sec:Appendix} displays specific examples, further discussion, additional simulation and extra analysis of the data used in main paper. 

\TPR{In particular, Section \ref{sec:specific examples} introduces a series of time series models including high-dimensional linear process, AR, MA, and GARCH in the formulation of $\mf G(u,\Upsilon_i)$ and verifies their dependence structure.}

Section \ref{sec:discussion supp} presents discussion on 4 aspects. Subsection \ref{sec:reconcile discussion} discusses how we reconcile the asymptotic results with the bounded time interval setting. Subsection \ref{sec:context discussion} discusses the difference between our theorem framework and works in \cite{dette2006simple}. Subsection \ref{sec:width dicussion} discusses the difficulty in minimal volume of SCB under the monotone condition. \TPM{Subsection \ref{sec:interpolation} discusses how the missing values affect our results if different interpolation methods are used.}

\TPR{Section \ref{sec: additional simulation} includes a series of additional simulations. Subsection \ref{sec: widths of SCB in simulation} contains the width comparisons between our monotone SCB and other SCBs without monotone constraint. Subsection \ref{sec:skew simulation} shows the simulation scenarios for skewed high-dimensional time series. Our GCV selections for the simulation set-ups in the main paper can be found in Subsection \ref{sec:GCV selection}. \TPM{Subsection \ref{sec:penalization simulation} validates our penalized SCB proposed in Section \ref{sec:penalization} through simulations of weakly monotone trends}
Subsection \ref{sec: simulation for time-varying linear regression} presents the simulation for time-varying linear models.}

Section \ref{regressionsection} applies the time-varying linear model to investigate the relationship between sunshine duration and atmospheric temperature. Location information for stations investigated in the main paper is also provided.
For $k \in \mathbb{Z}$, define the projection operator $\mathcal{P}_k \cdot=\mathbb{E}\left(\cdot \mid \Upsilon_k\right)-\mathbb{E}\left(\cdot \mid \Upsilon_{k-1}\right)$. We write $a_n\lesssim b_n$($a_n\gtrsim b_n$) to mean that there exists a universal constant $C>0$ such that $a_n\leq Cb_n$($Ca_n\geq b_n$) for all $n$. For set $A$, denote $\# A$ as the number of elements in $A$. For matrix $A, \operatorname{tr}(A)$ denotes the trace of $A$, and $\|A\|_{\operatorname{tr},p}=:\operatorname{tr}\left(\left(A^{\top} A\right)^{p/2}\right)^{1/p}$ is the Schatten norm of order $p$. $\|\cdot \|_{\operatorname{tr}}$ refers to $\|\cdot \|_{\operatorname{tr},1}$ if no extra clarification, which is known as the trace norm. 

Introduce $m_N(t)$ as auxiliary function towards $m(t)$, which is defined by the inverse of function below:
\begin{equation}\label{eq:m_N}
    m_N^{-1}(s)=: \int_{-\infty}^s \frac{1}{N h_d} \sum_{i=1}^N K_d\left(\frac{m(i / N)-u}{h_d}\right) \mathrm{d} u.
\end{equation}
The error between $m_N^{-1}$ and $m^{-1}$ can be viewed as the error of monotone rearrangement step which is nonrandom and totally depends on $N$ and $h_d$. \cite{dette2006simple} has proposed Lemmas to show $m_N(t)$ is pretty close to true $m(t)$ pointwise and we give a generalized global version Lemma \ref{lem:m_N} in auxiliary.

If $m$ is decreasing, we shall firstly reverse the observation data and follow the same procedure in increasing situations to obtain $\hat m_I$. It should be mentioned that $\hat m_I$ is defined from its inverse
$$
\hat m_I^{-1}(s)=:\int_{-\infty}^s\frac{1}{Nh_d}\sum_{i=1}^N K_d\left(\frac{\widehat{m_-}(i/N)-u}{h_d}\right),
$$
where $\widehat{m_-}$ is the local linear estimator from the origin sample $-y_i,i=1,\dots,n$. Therefore, $|\hat m_I- (-m)|$ can follow all the theorems in increasing situations. In this way, we take $-\hat m_I$ as the decreasing estimator towards true function $m$, then we can find that simultaneous results for $|-\hat m_I-m|$ are actually equivalent to the increasing situation $|\hat m_I-(-m)|$.

For $p$-dimensional vector function $\mathbf{m}(t)=(m_1(t),\dots,m_p(t))^{\top}$, we define $\mathbf{m}_{N}(t)=:(m_{N,1}(t),\dots,m_{N,p}(t))^{\top}$ where $m_{N,k}$ is the inverse of $m_{N,k}^{-1}$ defined by replacing $m(i/N),h_d$ in \eqref{eq:m_N} as $m_k(i/N),h_{d,k}$ respectively.

\subsection{Proofs of main results}

\begin{lem}
\label{m_I}
Suppose Assumptions \hyperref[(A1)]{(A1)}-\hyperref[(A2)]{(A2)}, \hyperref[(C1)]{(C1)}-\hyperref[(C2)]{(C2)} and $\max_{1\leq k\leq p}\sup_{t\in[0,1]}|\tilde{m}_k(t)-m_k(t)|=O_p(R_n)$ are satisfied, recalling the $M>0$ defined in Assumption \hyperref[(A2)]{(A2)}, we have
\begin{equation}
    \max_{1\leq k\leq p}\sup_{t\in\hat{\mathcal{T}}} |m_{N,k}(t)-\hat m_{I,k}(t)-\frac{\hat{m}_{I,k}^{-1}-m_{N,k}^{-1}}{\left(m_k^{-1}\right)^{\prime}} \circ m_k(t)|=O_p\left(R_n^2/h_d  \right).
    \label{e7}
\end{equation}
\end{lem}

\begin{proof}[Proof of Lemma \ref{m_I}.]
We firstly introduce an operator which maps an increasing function $m$ to its ``quantile'' $m^{-1}(t)$. Let $\mathcal{M}$ be the function class such that all $H\in \mathcal{M}$ are strictly increasing and exist second order derivative. For a fixed $t$ which lies in the interior of $H$'s image, consider the functional
    $$
    \Phi_t:\left\{\begin{array}{l}
    \mathcal{M} \mapsto\mathbb{R} \\
    H \mapsto H^{-1}(t)
    \end{array}\right.
    $$
and define for $H_1, H_2 \in \mathcal{M}$ the function
\begin{equation*}
    Q_t:\left\{\begin{array}{l}
{[0,1] \mapsto \mathbb{R}} \\
\lambda \mapsto \Phi_t\left(H_1+\lambda\left(H_2-H_1\right)\right).
\end{array}\right.
\end{equation*}

Let $F(x,y)=(H_1+x(H_2-H_1))(y)-t$, then $Q_t(\lambda)$ satisfies $F(\lambda,Q_t(\lambda))=0$. Then
\begin{equation}
   Q_t^{\prime}(\lambda)=-\frac{H_2-H_1}{h_1+\lambda\left(h_2-h_1\right)} \circ\left(H_1+\lambda\left(H_2-H_1\right)\right)^{-1}(t),
    \label{e15}
\end{equation}
where $h_1,h_2$ are the derivatives of $H_1,H_2$. Moreover, we have
\begin{equation}
Q_t^{\prime \prime}(\lambda)=Q_t^{\prime}(\lambda)\cdot\left\{\frac{-2\left(h_2-h_1\right)}{h_1+\lambda\left(h_2-h_1\right)}+\frac{\left(H_2-H_1\right)\left(h_1^{\prime}+\lambda\left(h_2^{\prime}-h_1^{\prime}\right)\right)}{\left\{h_1+\lambda\left(h_2-h_1\right)\right\}^2}\right\} \circ Q_t(\lambda),
\label{e15.1}
\end{equation}
Let $H_2=\hat m_{I,k}^{-1}$ and $H_1=m_{N,k}^{-1}$ in \eqref{e15} and \eqref{e15.1} and use the fact
$Q_t(1) = \Phi_t(H_2),Q_t(0) = \Phi_t(H_1)$, by Lagrange’s
\begin{equation}
    \hat m_{I,k}(t)-m_{N,k}(t)=Q_t(1)-Q_t(0)=Q_t^\prime(0)+\frac{1}{2}Q_t^{\prime\prime}(\lambda_t^*)=A_{N,k}(t)+\frac{1}{2}B_{N,k}(t),
    \label{e16}
\end{equation}
where $\lambda_t^*\in[0,1]$ and
$$
\begin{aligned}
A_{N,k}(t)= & -\frac{\hat{m}_{I,k}^{-1}-m_{N,k}^{-1}}{\left(m_{N,k}^{-1}\right)^{\prime}} \circ m_{N,k}(t), \\
B_{N,k}(t)= & \frac{2\left(\hat{m}_{I,k}^{-1}-m_{N,k}^{-1}\right)\left(\hat{m}_{I,k}^{-1}-m_{N,k}^{-1}\right)^{\prime}}{\left\{\left(m_{N,k}^{-1}+\lambda_t^*\left(\hat{m}_{I,k}^{-1}-m_{N,k}^{-1}\right)\right)^{\prime}\right\}^2} \circ\left(m_{N,k}^{-1}+\lambda_t^*\left(\hat{m}_{I,k}^{-1}-m_{N,k}^{-1}\right)\right)^{-1}(t) \\
& -\frac{\left(\hat{m}_{I,k}^{-1}-m_{N,k}^{-1}\right)^2\left(m_{N,k}^{-1}+\lambda_t^*\left(\hat{m}_{I,k}^{-1}-m_{N,k}^{-1}\right)\right)^{\prime \prime}}{\left\{\left(\hat{m}_I^{-1}+\lambda_t^*\left(\hat{m}_{I,k}^{-1}-m_{N,k}^{-1}\right)\right)^{\prime}\right\}^3} \circ\left(m_{N,k}^{-1}+\lambda_t^*\left(\hat{m}_{I,k}^{-1}-m_{N,k}^{-1}\right)\right)^{-1}(t) .
\end{aligned}
$$
Then we prove following two assertions
\begin{equation}
    \max_{1\leq k\leq p}\sup_{t\in\hat{\mathcal{T}}} |A_{N,k}(t)+\frac{\hat{m}_{I,k}^{-1}-m_{N,k}^{-1}}{\left(m_k^{-1}\right)^{\prime}} \circ m_k(t)|=O_p\left(\TPR{R_nh_d}\right),
    \label{e17}
\end{equation}
\begin{equation}
    \max_{1\leq k\leq p}\sup_{t\in\hat{\mathcal{T}}}|B_{N,k}(t)|=O_p\left(\TPR{\frac{R_n^2}{h_d}}\right).
    \label{e18}
\end{equation}
Firstly, we prove \eqref{e17}. By Lagrange's mean value
\begin{equation}
    (\hat m_{I,k}^{-1}-m_{N,k}^{-1})\circ m_{N,k}(t)=(\hat m_{I,k}^{-1}-m_{N,k}^{-1})\circ m_k(t)+\left[(\hat m_{I,k}^{-1}-m_{N,k}^{-1})^\prime(\xi_{t,N,k})\right]\cdot\left(m_{N,k}(t)-m_k(t)\right),
    \label{e19}
\end{equation}
where $\xi_{t,N,k}$ is between $m_k(t)$ and $m_{N,k}(t)$. 
By \eqref{eq:rate of (m_I.inv-m_N.inv)'} in Lemma \ref{lem:(m_I.inv-m_N.inv)'}
\begin{equation}
    \max_{1\leq k\leq p}\sup_{t\in\hat{\mathcal{T}}}|(\hat m_{I,k}^{-1}-m_{N,k}^{-1})^\prime(\xi_{t,N,k})|\leq\max_{1\leq k\leq p}\sup_{t\in\mathbb{R}}|(\hat m_{I,k}^{-1}-m_{N,k}^{-1})^\prime(t)| =O_p\left( \TPR{\frac{R_n}{h_d}}\right).\label{e22}
\end{equation}
Combine \eqref{e22} and \eqref{eq:m_N-m} in Lemma \ref{lem:m_N}, the \eqref{e19} can be written as
\begin{equation}
    \max_{1\leq k\leq p}\sup_{t\in\hat{\mathcal{T}}}|(\hat m_{I,k}^{-1}-m_{N,k}^{-1})\circ m_{N,k}(t)-(\hat m_{I,k}^{-1}-m_{N,k}^{-1})\circ m_k(t)|=O_p\left( \TPR{R_nh_d} \right).\label{e23}
\end{equation}

For $B_{N,k}$, we decompose $B_{N,k}=2B_{N1,k}-B_{N2,k}$
$$
\begin{aligned}
B_{N 1,k}(t) & =\frac{\left(\hat{m}_{I,k}^{-1}-m_{N,k}^{-1}\right)\left(\hat{m}_{I,k}^{-1}-m_{N,k}^{-1}\right)^{\prime}}{\left\{\left[m_{N,k}^{-1}+\lambda_t^*\left(\hat{m}_{I,k}^{-1}-m_{N,k}^{-1}\right)\right]^{\prime}\right\}^2}\circ t_n, \\
B_{N 2,k}(t) & =\frac{\left(\hat{m}_{I,k}^{-1}-m_{N,k}^{-1}\right)^2\left(m_N^{-1}+\lambda_t^*\left(\hat{m}_{I,k}^{-1}-m_{N,k}^{-1}\right)\right)^{\prime \prime}}{\left\{\left[\hat{m}_{I,k}^{-1}+\lambda_t^*\left(\hat{m}_{I,k}^{-1}-m_{N,k}^{-1}\right)\right]^{\prime}\right\}^3}\circ t_n,
\end{aligned}
$$
where $t_n=(m_{N,k}^{-1}+\lambda_t^*(\hat m_{I,k}^{-1}-m_{N,k}^{-1}))^{-1}(t)$. Note that $t_n\rightarrow_pm_{N,k}(t)$ and 
$$
\max_{1\leq k\leq p}\sup_{t\in\mathbb{R}}|(m_{N,k}^{-1})''(t)|=\max_{1\leq k\leq p}\frac{1}{h_{d,k}}\sup_{t\in\mathbb{R}}|\frac{1}{Nh_{d,k}}\sum_{i=1}^NK_d'\left(\frac{m_k(i/N)-t}{h_{d,k}}\right)|=O(\TPR{\frac{1}{h_d}}          ),
$$
together with \eqref{eq:rate of m_I.inv-m_N.inv}, \eqref{eq:rate of (m_I.inv-m_N.inv)'}, \eqref{eq:rate of (m_I.inv-m_N.inv)''} in Lemma \ref{lem:(m_I.inv-m_N.inv)'} and Lemma \ref{lem:m_I.inv' inf>0},
\begin{equation}
    \max_{1\leq k\leq p}\sup_{t\in\hat{\mathcal{T}}} |B_{N1,k}(t)|=O_p\left(\TPR{\frac{R_n^2}{h_d}}\right),\quad\max_{1\leq k\leq p}\sup_{t\in\hat{\mathcal{T}}} |B_{N2,k}(t)|=O_p\left(\TPR{\frac{R_n^2}{h_d}}\right).
    \label{e25}
\end{equation}
Combining \eqref{e23} and \eqref{e25}, we can obtain the desired rate in \eqref{e17} and \eqref{e18}. Finally, combining \eqref{e16}, \eqref{e17} and \eqref{e18}, it immediately yields
\begin{equation}
    \max_{1\leq k\leq p}\sup_{t\in\hat{\mathcal{T}}} |m_{N,k}(t)-\hat m_{I,k}(t)-\frac{\hat{m}_{I,k}^{-1}-m_{N,k}^{-1}}{\left(m_{k}^{-1}\right)^{\prime}} \circ m_k(t)|=O_p\left(\TPR{\frac{R_n^2}{h_d}} \right).
    \label{eq:inverse}
\end{equation}
\end{proof}

\subsubsection{Proof of Proposition \ref{prop:GA in high-dim}}

\begin{proof}
For $k=1,2,\dots,p$, denote 
    \begin{equation*}
        \mathcal{S}_k^{(0)}(t)=:\frac{\hat{m}_{I,k}^{-1}-m_{N,k}^{-1}}{(m_{k}^{-1})'}\circ m_{k}(t).
    \end{equation*}
By Lemma \ref{lem:High-dim ll}, we have $\max_{1\leq k\leq p}\sup_{t\in[0,1]}|\tilde{m}_k(t)-m_k(t)|=O_p(R_n)$, then using Lemma \ref{m_I}, we shall have
\begin{equation}
    \label{eq:recovery in eq:GA in thm:high-dim bootstrap theorem}
        \max_{1\leq k\leq p}\sup_{t\in\hat{\mathcal{T}}}|m_{N,k}(t)-\hat{{m}}_{I,k}(t)-\mathcal{S}_k^{(0)}(t)|=O_p\left(\TPR{\frac{R_n^2}{h_d}} \right).
\end{equation}
By Taylor's expansion
\begin{align}
    \hat m_{I,k}^{-1}(t)-m_{N,k}^{-1}(t) &=\frac{1}{N h_{d,k}^2} \int_{-\infty}^t\sum_{i=1}^{N}(\tilde{m}_k(i / N)-m_k(i / N)) K_{d}'\left(\frac{m_k(i/N)-u}{h_{d,k}}\right)du\notag\\
   &+\frac{1}{2Nh_{d,k}^3}\int_{-\infty}^t \sum_{i=1}^{N}(\tilde{m}_k(i / N)-m_k(i / N))^2 K_{d}''\left(\frac{\eta_{i,N,k}-u}{h_{d,k}}\right)du, \label{e24}
\end{align}
for some $\eta_{i,N,k}$ between $\hat m_k(i/N )$ and $m_k(i/N )$. 
 By Lemma \ref{lem:High-dim ll}, uniformly $t\in(m(0),m(1))$, the second line in \eqref{e24} is of rate $O_p(\TPR{R_n^2/Mh_d})$. Then based on \eqref{e24}, we have
\begin{equation}
    \max_{1\leq k\leq p}\sup_{t\in\mathbb{R}}|\hat m_{I,k}^{-1}(t)-m_{N,k}^{-1}(t)+\frac{1}{N h_{d,k}} \sum_{i=1}^{N}(\tilde{m}_k(i / N)-m_k(i / N)) K_{d}\left(\frac{m_k(i/N)-t}{h_{d,k}}\right)|=O_p(\TPR{R_n^2/h_d}).
    \label{e26}
\end{equation}
 Combining definition of $\mathcal{S}_k^{(0)}(t)$ and \eqref{e26}, we have
\begin{align}
        &\max_{1\leq k\leq p}\sup_{t\in\hat{\mathcal{T}}}|\mathcal{S}_k^{(0)}(t)-\frac{m_k^{\prime}(t)}{N h_{d,k}} \sum_{i=1}^{N}(\tilde{m}_k(i / N)-m_k(i / N)) K_{d}\left(\frac{m_k(i/N)-m_k(t)}{h_{d,k}}\right)|\notag\\
        &=O_p\left(\TPR{\frac{R_n^2}{h_d}} \right).\label{eq:stochastic in high GA}
\end{align}

By Lemma \ref{high-GA}, on a richer space, there exists independent $\mathbf{V}_j\sim N_p(0,\mf\Sigma_{\mf G}(t_j))$, $j=1,\dots,n$, such that for $\mathbf{e}_l$ defined in \eqref{high-monotone model}
\begin{equation}
    \left( \mathbb{E} \max_{k\leq n}\left|\sum_{l=1}^k(\mathbf{e}_l-\mathbf{V}_l)\right|^2 \right)^{\frac{1}{2}}=O\left(\Theta(p)\sqrt{ n\log n}\left(\frac{p}{n}\right)^{\frac{q-2}{6q-4}}\right)
    \label{high-dim GA application}.
\end{equation}
Denote vectors $\mathcal{E}^{(0)}(t)=(\mathcal{E}_1^{(0)}(t),\dots,\mathcal{E}_p^{(0)}(t))^\top$ and $\mathcal{V}^{(0)}(t)=(\mathcal{V}_1^{(0)}(t),\dots,\mathcal{V}_p^{(0)}(t))^\top$,
$$
\mathcal{E}^{(0)}(t)=:\mathcal{M}(t)\sum_{j=1}^n\sum_{i=1}^N  \mathcal{K}_d(t,i/N)\mathcal{K}_r(j/n,i/N) \mathbf{e}_j,
$$
\begin{equation}
    \mathcal{V}^{(0)}(t)=:\mathcal{M}(t)\sum_{j=1}^n\sum_{i=1}^N  \mathcal{K}_d(t,i/N)\mathcal{K}_r(j/n,i/N) \mathbf{V}_j,
    \label{GA-V-0}
\end{equation}
where $\mathcal{M}(t)=\operatorname{diag}\left\{m_{k}^{\prime}(t)\right\}_{k=1}^p$ and
$$
\mathcal{K}_r(j/n,i/N)=\operatorname{diag}\left\{\frac{1}{nh_{r,k}}\tilde{K}^*_r\left(\frac{j/n-i/N}{h_{r, k}},\frac{i}{N}\right)\right\}_{k=1}^p,
$$
$$
\mathcal{K}_d(t,i/N)=\operatorname{diag}\left\{\frac{1}{Nh_{d,k}}K_d\left(\frac{m_k(i/N)-m_{k}(t)}{h_{d, k}}\right)\right\}_{k=1}^p.
$$
Combining \eqref{eq:stochastic in high GA} and Lemma \ref{local linear stochastic expansion} with jackknife bias correction, it can yield that
\begin{equation}
    \label{eq:stochastic expansion in eq:GA in thm:high-dim bootstrap theorem}
    \max_{1\leq k\leq p}\sup_{t\in\hat{\mathcal{T}}} |\mathcal{S}_k^{(0)}(t)-\mathcal{E}_k^{(0)}(t)|=O_p\left(\TPR{\frac{R_n^2}{h_d}} \right).
\end{equation}
Recall $\nu_{l,k}(t)=\int_{-t/h_{r,k}}^{(1-t)/h_{r,k}}x^lK_r(x)\mathrm{d}x$, for $\kappa_l=:\int |x|^lK_r(x)\mathrm{d}x$, using the fact $\kappa_2-\kappa_1^2>0$, we have
\begin{align}
    \inf_{t\in[0,1]}|\nu_{0,k}(t)\nu_{2,k}(t)-\nu^2_{1,k}(t)|&\geq \frac{1}{2}\cdot\frac{1}{2}\int_{-1}^1 x^2K_r(x)\mathrm{d}x-\left(\frac{1}{2}\int_{-1}^1|x|K_r(x)\mathrm{d}x\right)^2,\notag\\
    &\geq \frac{1}{4}(\kappa_2-\kappa_1^2).\label{eq:lower c(t)}
\end{align}
For $c_k(t)=:\nu_{0,k}(t)\nu_{2,k}(t)-\nu^2_{1,k}(t)$, we have $\inf_{t\in[0,1]}c_k(t)\geq (\kappa_2-\kappa_1^2)/4$, thus $\tilde{K}^*$ defined in \eqref{eq:K^*} is bounded. 
Using summation by parts, we have
\begin{align}
    &\mathbb{E}\left(\max_{1\leq k\leq p}\sup_{t\in\hat{\mathcal{T}}}|\mathcal{V}_k^{(0)}(t)-\mathcal{E}_k^{(0)}(t)|\right)\notag\\
    &\leq \max_{1\leq k\leq p}\sup_{t\in\hat{\mathcal{T}}}\left|\frac{m_{k}'(t)}{Nh_{d,k}}\sum_{i=1}^NK_d\left(\frac{m_k(i/N)-m_{k}(t)}{h_{d, k}}\right)\frac{1}{nh_{r,k}}\sum_{j=1}^{n-1}\Delta_{j}(\tilde K_r^*) \right|\cdot\mathbb{E}\left(\max_{j\leq n}\left|\sum_{l=1}^j(\mathbf{e}_l-\mathbf{V}_l)\right|\right)\notag\\
    &+\max_{1\leq k\leq p}\sup_{t\in\hat{\mathcal{T}}}\left|\frac{m_{k}'(t)}{Nh_{d,k}}\sum_{i=1}^NK_d\left(\frac{m_k(i/N)-m_{k}(t)}{h_{d, k}}\right)\frac{1}{nh_{r,k}}\tilde{K}^*\left(\frac{1-i/N}{h_{r,k}},\frac{i}{N}\right)\right|\cdot\mathbb{E}\left(\max_{j\leq n}\left|\sum_{l=1}^j(\mathbf{e}_l-\mathbf{V}_l)\right|\right)\notag\\
    &=O_p\left\{ \frac{\Theta(p)\sqrt{n\log n}\left(\frac{p}{n}\right)^{\frac{q-2}{6q-4}}}{nh_r}  \right\},\label{eq:GA in eq:GA in thm:high-dim bootstrap theorem}
\end{align}
where 
$$
\Delta_j(\tilde{K}_r^*)=:\left[\tilde K_r^*\left(\frac{j/n-i/N}{h_{r,k}},\frac{i}{N}\right)-\tilde K_r^*\left(\frac{(j+1)/n-i/N}{h_{r,k}},\frac{i}{N}\right)\right]
$$
and the last line in \eqref{eq:GA in eq:GA in thm:high-dim bootstrap theorem} is obtained in view of \eqref{high-dim GA application}.

Combining \eqref{eq:recovery in eq:GA in thm:high-dim bootstrap theorem}, \eqref{eq:stochastic expansion in eq:GA in thm:high-dim bootstrap theorem} and \eqref{eq:GA in eq:GA in thm:high-dim bootstrap theorem}, we have
\begin{equation}
    \label{eq:GA in thm:high-dim bootstrap theorem}
    \max_{1\leq k\leq p}\sup_{t\in\hat{\mathcal{T}}}|m_{N,k}(t)-\hat{m}_{I,k}(t)-\mathcal{V}_k^{(0)}(t)|=O_p\left\{\TPR{\frac{R_n^2}{h_d}}+\frac{\Theta(p)\sqrt{n\log n}\left(\frac{p}{n}\right)^{\frac{q-2}{6q-4}}}{nh_r}\right\}.
\end{equation}
Then by Lemma \ref{lem:m_N}, $\mathcal{V}^{(0)}(t)$ satisfies \eqref{eq:GA in prop:GA in high-dim}.
\end{proof}

\subsubsection{Proof of Theorem \ref{prop:GA in high-dim}}
\begin{proof}
Recall the union rate we define in \eqref{eq: def union rate}
\begin{equation*}
    \rho_n=:  \frac{\Theta(p)\left(\frac{p}{n}\right)^{\frac{q-2}{6q-4}}\sqrt{n\log n} }{nh_r} \bigvee \frac{\sqrt{\log (n)(\sqrt{n \varphi_n p}+\varphi_n+p)} }{nh_r},
\end{equation*}
where $\varphi_n=\Theta^2(p)(\sqrt{npL}+nL^{-1})+p^{3/2}L\sqrt{n}\left(\Theta(p)R_n+\sqrt{n}\left(h_r^5+\frac{\log^6(n)}{nh_r}\right)\right)$.

Denote
    \begin{equation*}
        \mathcal{V}^{(1)}(t)=:\sum_{j=1}^n\sum_{i=1}^N \mathcal{W}(i/N,t)\mathcal{K}_r(j/n,i/N) \mathbf{V}_j,
    \end{equation*}
where $\mathbf{V}_j,j=1,\dots,n$ is the same as \eqref{GA-V-0}. Note that $\hat{Q}(k)-\hat{Q}(k-1)$ is positive semidefinite, given data, Lemma \ref{lem:gaussian cov structure} indicates that there exists independent $\mathbf{V}_j^*\sim N_p(0,\hat{\mathbf{Q}}(j)-\hat{\mathbf{Q}}(j-1)),j=1,\dots,n$ on a richer space s.t.
\begin{equation}
    \label{eq:diff V-V*}
    \mathbb{E} \left(\max _{k=1, \ldots, n}\left|\sum_{j=1}^k \mathbf{V}_j-\sum_{j=1}^k \mathbf{V}^*_j\right|^2 \,\middle\vert\, \Upsilon_n\right) \leq C \log (n)(\sqrt{n\varphi^\Delta\Phi}+\varphi^\Delta+\Phi)
\end{equation}
where
$$
\varphi^\Delta=\max_{k=1,\dots,n}\left\| \sum_{j=1}^k\mf\Sigma_{\mf G}(t_j)-\hat{\mathbf{Q}}(k)   \right\|_{\operatorname{tr}},\quad \Phi=\max_{t=1,\dots,n}\|\mf\Sigma_{\mf G}(t_j)\|_{\operatorname{tr}}.
$$
Thus with summation by parts formula, given the data, 
\begin{align}
        \sup_{t\in\hat{\mathcal{T}}}\left| \mathcal{V}^{(1)}(t)-\mathcal{V}^{*}(t)\right|&\lesssim \sup_{t\in\hat{\mathcal{T}}}\left|\sum_{i=1}^N\mathcal{W}(i/N,t)\sum_{j=1}^{n-1}[\mathcal{K}_r(j/n,i/N)-\mathcal{K}_r((j+1)/n,i/N)]\right| \notag\\
        &\cdot \max _{j=1, \ldots, n}\left|\sum_{t=1}^j \mathbf{V}_t-\sum_{t=1}^j \mathbf{V}^*_t\right|.\notag\\
\end{align}
Then by \eqref{eq:diff V-V*} and the fact $\sum_{i=1}^N\mathcal{W}(i/N,t)=I_p$, 
\begin{equation}
    \mathbb{E} \left( \sup_{t\in\hat{\mathcal{T}}}\left| \mathcal{V}^{(1)}(t)-\mathcal{V}^{*}(t)\right|^2 \,\middle\vert\, \Upsilon_n\right) \leq \frac{C{\log (n)(\sqrt{n\varphi^\Delta\Phi}+\varphi^\Delta+\Phi)}}{n^2h_r^2}.\label{eq:Mies high-dim given data}
\end{equation}

By Assumption \hyperref[(B4)]{(B4)}, $p\underline{\Lambda}\leq \|\Sigma(t)\|_{\operatorname{tr}}\leq p\Lambda$ for any $t\in[0,1]$, thus $\Phi\asymp p$. Note that $\varphi^\Delta=O_p(\varphi_n)$ by Lemma \ref{lem:high long-run var},  then we can have 
\begin{equation}
    \label{eq:Mies high-dim no given data}
     \mathbb{E} \left( \sup_{t\in\hat{\mathcal{T}}}\left| \mathcal{V}^{(1)}(t)-\mathcal{V}^{*}(t)\right|\,\middle\vert\, \Upsilon_n\right)=O_p\left( \frac{\sqrt{\log (n)(\sqrt{n \varphi_n p}+\varphi_n+p)}}{nh_r}\right).
\end{equation}

Recall the proof of Proposition \ref{prop:GA in high-dim}, it has shown that $\mathcal{V}^{(0)}(t)$ defined in \eqref{GA-V-0} satisfies \eqref{eq:GA in prop:GA in high-dim}. Therefore, Theorem \ref{prop:GA in high-dim} can hold if
\begin{align*}
    &\sup_{x\in\mathbb{R}}\left|\mathbb{P}\left\{\max_{1\leq k\leq p}\sup_{t\in\hat{\mathcal{T}}} \sqrt{nh_r}\left| \mathcal{V}^{(0)}_k(t) \right|\leq x \right\} - \mathbb{P}\left\{\max_{1\leq k\leq p}\sup_{t\in\hat{\mathcal{T}}} \sqrt{nh_r}\left| \mathcal{V}^*_k(t) \right|\leq x \,\middle\vert\, \Upsilon_n \right\} \right|\rightarrow_p 0.
\end{align*}
By Lemma \ref{lem:prob inequality} and \TPR{Chebyshev's inequality, for any $\rho_n^*>0$,}
\begin{align}
    &\sup_{x\in\mathbb{R}}\left|\mathbb{P}\left\{\max_{1\leq k\leq p}\sup_{t\in\hat{\mathcal{T}}} \left|\sqrt{nh_r} \mathcal{V}^{(0)}_k(t) \right|\leq x \right\} - \mathbb{P}\left\{\max_{1\leq k\leq p}\sup_{t\in\hat{\mathcal{T}}} \left| \sqrt{nh_r}\mathcal{V}^*_k(t) \right|\leq x \,\middle\vert\, \Upsilon_n \right\} \right|,\notag\\
    &\leq\mathbb{P}\left\{ \max_{1\leq k\leq p}\sup_{t\in\hat{\mathcal{T}}}\sqrt{nh_r}\left| \mathcal{V}_k^{(0)}(t)-\mathcal{V}_k^{*}(t)\right|>\rho_n^*\sqrt{nh_r}\,\middle\vert\,\Upsilon_n \right\} \notag\\
    &\quad +\sup_{x\in\mathbb{R}}\mathbb{P}\left\{\left| \max_{1\leq k\leq p}\sup_{t\in\hat{\mathcal{T}}}|\sqrt{nh_r}\mathcal{V}^{(0)}_k(t)|-x\right|\leq \rho_n^*\sqrt{nh_r} \right\},\notag\\
    &\leq \TPR{\frac{1}{\rho_n^*}\mathbb{E}\left(  \max_{1\leq k\leq p}\sup_{t\in\hat{\mathcal{T}}}\left| \mathcal{V}_k^{(0)}(t)-\mathcal{V}_k^{*}(t)\right|\,\middle\vert\,\Upsilon_n \right)+\sup_{x\in\mathbb{R}}\mathbb{P}\left\{\left| \max_{1\leq k\leq p}\sup_{t\in\hat{\mathcal{T}}}|\sqrt{nh_r}\mathcal{V}^{(0)}_k(t)|-x\right|\leq \rho_n^*\sqrt{nh_r} \right\},\notag}\\
    &\leq \mathcal{I}_1+\mathcal{I}_2+\mathcal{I}_3,\notag
\end{align}
where
\TPR{\begin{align}
    \mathcal{I}_1&=\frac{1}{\rho_n^*}\mathbb{E}\left(  \max_{1\leq k\leq p}\sup_{t\in\hat{\mathcal{T}}}\left| \mathcal{V}_k^{(0)}(t)-\mathcal{V}_k^{(1)}(t)\right|\,\middle\vert\,\Upsilon_n \right),\notag \\
    \mathcal{I}_2&=\frac{1}{\rho_n^*}\mathbb{E}\left(  \max_{1\leq k\leq p}\sup_{t\in\hat{\mathcal{T}}}\left| \mathcal{V}_k^{(1)}(t)-\mathcal{V}_k^{*}(t)\right|\,\middle\vert\,\Upsilon_n \right)\label{eq:Mies high-dim},\\
    \mathcal{I}_3&=\sup_{x\in\mathbb{R}}\mathbb{P}\left\{\left| \max_{1\leq k\leq p}\sup_{t\in\hat{\mathcal{T}}}|\sqrt{nh_r}\mathcal{V}_k^{(0)}(t)|-x\right|\leq \rho_n^*\sqrt{nh_r} \right\}\label{eq:target boot in thm:high-dim bootstrap theorem}.
\end{align}
Recall the definition of $\rho_n$ in \eqref{eq: def union rate}, by \eqref{eq:Mies high-dim no given data},  
\begin{equation}
    \mathcal{I}_2=O_p\left( \frac{\sqrt{\log (n)(\sqrt{n \varphi_n p}+\varphi_n+p)}}{\rho_n^* nh_r}\right)=O_p(\rho_n/\rho_n^*).
\end{equation}
Therefore, we only need to bound $\mathcal{I}_1$ and $\mathcal{I}_3$ in separately two steps.
\begin{itemize}
    \item Step 1  
    \begin{equation}
    \mathcal{I}_1=O_p\left( \frac{\Theta(p)\left(\frac{p}{n}\right)^{\frac{q-2}{6q-4}}\sqrt{n\log n} }{ \rho_n^*nh_r}\right)=O_p(\rho_n/\rho_n^*).
    \label{eq:eq1:target boot in thm:high-dim bootstrap theorem}
    \end{equation}
    \item Step 2
    \begin{align}
        &\sup_{x\in\mathbb{R}}\mathbb{P}\left\{x-\rho_n^*\sqrt{nh_r} \leq \max_{1\leq k\leq n}\sup_{t\in\hat{\mathcal{T}}}|\sqrt{nh_r}\mathcal{V}^{(0)}_k(t)|\leq x+\rho_n^*\sqrt{nh_r} \right\}\notag\\
        &=O\left( \rho_n^*\sqrt{nh_r\log(Np)}+\sqrt{\frac{\sqrt{(nh_r)\log(Np)\log(np)}}{Nh_d}} \right).
    \label{eq:eq2:target boot in thm:high-dim bootstrap theorem}
    \end{align}
\end{itemize}
}
\paragraph{Step 1.} 
Firstly, we consider another form of stochastic expansion on $m_{I,k}-m_{N,k}$. As discussed in Remark \ref{remark:GA in high-dim}, our proof can also allow the $M$ in Assumption \hyperref[(A2)]{(A2)} appropriately decaying to 0. For simplicity we only consider fixed $M$ in our following proof, which has been specified in Assumption \hyperref[(A2)]{(A2)}. Use the same technique in \eqref{e16} but swap $\hat m_{I,k}^{-1}$ and $m_{N,k}^{-1}$ i.e. $H_2=m_{N,k}^{-1}$ and $H_1=\hat m_{I,k}^{-1}$, then by \eqref{e15} and \eqref{e15.1}, for some $\lambda_t^{**}\in[0,1]$,
\begin{equation}
    m_{N,k}(t) - \hat m_{I,k}(t) = \tilde{A}_{N,k}(t)+\frac{1}{2}\tilde B_{N,k}(t)
    \label{e29}
\end{equation}
$$
\begin{aligned}
\tilde{A}_{N,k}(t)&=-\frac{m_{N,k}^{-1}-\hat m_{I,k}^{-1}}{\left(\hat m_{I,k}^{-1}\right)^{\prime}} \circ \hat m_{I,k}(t),\\
\tilde B_{N,k}(t)&= \frac{2\left(m_{N,k}^{-1}-\hat{m}_{I,k}^{-1}\right)\left(m_{N,k}^{-1}-\hat{m}_{I,k}^{-1}\right)^{\prime}}{\left\{\left(m_{I,k}^{-1}+\lambda_t^{**}\left(\hat{m}_{N,k}^{-1}-m_{I,k}^{-1}\right)\right)^{\prime}\right\}^{2}} \circ\left(m_{I,k}^{-1}+\lambda_t^{**}\left(\hat{m}_{N,k}^{-1}-m_{I,k}^{-1}\right)\right)^{-1}(t) \\
&-\frac{\left(\hat{m}_{N,k}^{-1}-m_{I,k}^{-1}\right)^{2}\left(m_{I,k}^{-1}+\lambda_t^{**}\left(\hat{m}_{N,k}^{-1}-m_{I,k}^{-1}\right)\right)^{\prime \prime}}{\left\{\left(\hat{m}_{N,k}^{-1}+\lambda_t^{**}\left(\hat{m}_{N,k}^{-1}-m_{I,k}^{-1}\right)\right)^{\prime}\right\}^{3}} \circ\left(m_{I,k}^{-1}+\lambda_t^{**}\left(\hat{m}_{N,k}^{-1}-m_{I,k}^{-1}\right)\right)^{-1}(t) .
\end{aligned}
$$
By similar arguments in the proof of Lemma \ref{m_I} and Proposition \ref{prop:GA in high-dim}, we can show that 
\begin{equation}
    \max_{1\leq k\leq p}\sup_{t\in\hat{\mathcal{T}}}|\tilde{B}_{N,k}(t)|=O_p\left(\frac{R_n^2}{h_d}\right).\label{eq:tilde Bn}
\end{equation}
Note that
$$
(\hat m_{I,k}^{-1})'(t) = \frac{d}{dt}\left\{\frac{1}{Nh_{d,k}}\sum_{i=1}^N\int_{-\infty}^t K_d\left(\frac{\tilde m_k(i/N)-u}{h_d}\right)du\right\}=\frac{1}{Nh_{d,k}}\sum_{i=1}^NK_d\left(\frac{\tilde m_k(i/N)-t}{h_{d,k}}\right),
$$
thus
\begin{equation}
    (\hat m_{I,k}^{-1})'\circ \hat m_{I,k}(t) = \frac{1}{Nh_{d,k}}\sum_{i=1}^NK_d\left(\frac{\tilde m_k(i/N)-\hat m_{I,k}(t)}{h_{d,k}}\right).
    \label{e30}
\end{equation}
By \eqref{e24}, for some $\eta_{i,N,k}$ between $\tilde{m}_k(i/N)$ and $m_k(i/N)$,  
\begin{align*}
    (m_{N,k}^{-1}-\hat m_{I,k}^{-1}) \circ \hat m_{I,k}(t)
    &=\frac{1}{N h_{d,k}^2} \sum_{i=1}^{N}(m_k(i / N)-\tilde m_k(i / N))\int_{-\infty}^{\hat m_{I,k}(t)} K_{d}'\left(\frac{\tilde m_k(i/N)-u}{h_{d,k}}\right)du,\\
    &+\frac{1}{2Nh_{d,k}^3}\sum_{i=1}^{N}({m}_k(i / N)-\tilde m_k(i / N))^2\int_{-\infty}^{\hat m_{I,k}(t)}K_{d}''\left(\frac{\eta_{i,N,k}-u}{h_{d,k}}\right)du,
\end{align*}
together with \eqref{e30}, Lemma \ref{lem:m_I.inv' inf>0} and Lemma \ref{lem:High-dim ll}, we have
\begin{equation}
\label{eq:An stochastic expansion}
\max_{1\leq k\leq p}\sup_{t\in\hat{\mathcal{T}}}\left|\tilde A_{N,k}(t) - \frac{\sum_{i=1}^N K_d\left(\frac{\tilde m_k(i/N)-\hat m_{I,k}(t)}{h_{d,k}}\right)[ \tilde m_k(i/N)- m_k(i/N)]}{\sum_{i=1}^N K_d\left(\frac{\tilde m_k(i/N)-\hat m_{I,k}(t)}{h_{d,k}}\right)}\right|=O_p\left(\frac{R_n^2}{h_d}\right).    
\end{equation}
By definition of $\hat{W}^*_{i,k,I}(t)$ in the algorithm, \eqref{eq:An stochastic expansion} can be written as
\begin{equation}
    \max_{1\leq k\leq p}\sup_{t\in\hat{\mathcal{T}}}\left|\tilde A_{N,k}(t) - \sum_{i=1}^N\hat{W}_{i,k,I}^*(t)[\tilde{m}_k(i/N)- m_k(i/N)]\right|=O_p\left(\frac{R_n^2}{h_d}\right).\label{eq:An W stochastic expansion}
\end{equation}
Combining \eqref{e29},\eqref{eq:tilde Bn} and \eqref{eq:An W stochastic expansion}, we have
\begin{equation}
        \max_{1\leq k\leq p}\sup_{t\in\hat{\mathcal{T}}}|m_{N,k}(t)-\hat{{m}}_{I,k}(t)-\sum_{i=1}^N\hat{W}_{i,k,I}^*(t)[\tilde{m}_k(i/N)- m_k(i/N)]|=O_p\left(\frac{R_n^2}{h_d}\right).\label{eq:rearrange rate in step1 firstly}
\end{equation}

Secondly, denote vectors $\mathcal{S}^{(1)}(t)=(\mathcal{S}_1^{(1)}(t),\dots,\mathcal{S}_p^{(1)}(t))^{\top}$ and $\mathcal{E}^{(1)}(t)=(\mathcal{E}_1^{(1)}(t),\dots,\mathcal{E}_p^{(1)}(t))^{\top}$
    \begin{equation*}
        \mathcal{S}^{(1)}(t)=:\sum_{i=1}^N \mathcal{W}(i/N,t)[\tilde{\mathbf{m}}(i/N)-\mathbf{m}(i/N)],
    \end{equation*}
$$
\mathcal{E}^{(1)}(t)=:\sum_{j=1}^n\sum_{i=1}^N \mathcal{W}(i/N,t)\mathcal{K}_r(j/n,i/N) \mathbf{e}_j,
$$
where $\tilde{\mf m}(t)=(\tilde{m}_1(t),\dots,\tilde{m}_p(t))^\top$ and $\mathcal{W}(i/N,t),\mathcal{K}_r(j/n,i//N)$ are defined in \eqref{GA-V-star}. 
Then \eqref{eq:rearrange rate in step1 firstly} can be written as
\begin{equation}
        \max_{1\leq k\leq p}\sup_{t\in\hat{\mathcal{T}}}|m_{N,k}(t)-\hat{{m}}_{I,k}(t)-\mathcal{S}_k^{(1)}(t)|=O_p\left(\frac{R_n^2}{h_d}\right).\label{eq:rearrange rate in step1}
\end{equation}
Similar to \eqref{eq:stochastic expansion in eq:GA in thm:high-dim bootstrap theorem}, combining \eqref{eq:rearrange rate in step1} and Lemma \ref{local linear stochastic expansion} with jackknife bias correction, it can yield that
\begin{equation}
    \max_{1\leq k\leq p}\sup_{t\in\hat{\mathcal{T}} } |\mathcal{S}_k^{(1)}(t)-\mathcal{E}_k^{(1)}(t)|=O_p\left(\frac{R_n^2}{h_d}\right).\label{eq:stochastic expansion rate in step1}
\end{equation}
Note that $\mathcal{W}(i/N,t)>0$ and $\sum_{i=1}^N\mathcal{W}(i/N,t)=I_p$, using summation by parts and \eqref{high-dim GA application}, on given data, we have 
\begin{align}
    &\mathbb{E}\left(\max_{1\leq k\leq p}\sup_{t\in\hat{\mathcal{T}}}|\mathcal{V}_k^{(1)}(t)-\mathcal{E}_k^{(1)}(t)| \,\middle\vert\,\Upsilon_n \right)\notag\\
    &\lesssim \sup_{t\in\hat{\mathcal{T}}}\left|\sum_{i=1}^N\mathcal{W}(i/N,t)\sum_{j=1}^{n-1}\left[\mathcal{K}_r(j/n,i/N)-\mathcal{K}_r((j+1)/n,i/N)\right] \right|\cdot \mathbb{E}\left(\max_{j\leq n}\left|\sum_{l=1}^k(\mathbf{e}_l-V_l)\right|\,\middle\vert\,\Upsilon_n \right)\notag\\
    &=  O_p\left\{ \frac{\Theta(p)\sqrt{n\log n}\left(\frac{p}{n}\right)^{\frac{q-2}{6q-4}}}{nh_r}  \right\}.\label{eq:GA rate in step1}
\end{align}
Combining \eqref{eq:GA rate in step1}, \eqref{eq:stochastic expansion rate in step1}, \eqref{eq:rearrange rate in step1}, \eqref{eq:GA in eq:GA in thm:high-dim bootstrap theorem}, \eqref{eq:stochastic expansion in eq:GA in thm:high-dim bootstrap theorem} and  \eqref{eq:recovery in eq:GA in thm:high-dim bootstrap theorem}, we have
\begin{align*}
     &\mathbb{E}\left(\max_{1\leq k\leq p}\sup_{t\in\hat{\mathcal{T}}}\left| \mathcal{V}_k^{(0)}(t)-\mathcal{V}_k^{(1)}(t)\right| \,\middle\vert\,\Upsilon_n \right)\\
     &\leq \mathbb{E}\left(\max_{1\leq k\leq p}\sup_{t\in\hat{\mathcal{T}}}\left| m_{N,k}(t)-\hat{m}_{I,k}(t)-\mathcal{V}_k^{(0)}(t)\right| \,\middle\vert\,\Upsilon_n \right)+\mathbb{E}\left(\max_{1\leq k\leq p}\sup_{t\in\hat{\mathcal{T}}}\left| m_{N,k}(t)-\hat{m}_{I,k}(t)-\mathcal{V}_k^{(1)}(t)\right| \,\middle\vert\,\Upsilon_n \right).\notag\\
     &\leq \max_{1\leq k\leq p}\sup_{t\in\hat{\mathcal{T}}}\left| m_{N,k}(t)-\hat{m}_{I,k}(t)-\mathcal{E}_k^{(0)}(t)\right|+\max_{1\leq k\leq p}\sup_{t\in\hat{\mathcal{T}}}\left| m_{N,k}(t)-\hat{m}_{I,k}(t)-\mathcal{E}_k^{(1)}(t)\right|\notag\\
     &\quad+\mathbb{E}\left(\max_{1\leq k\leq p}\sup_{t\in\hat{\mathcal{T}}}|\mathcal{V}_k^{(0)}(t)-\mathcal{E}_k^{(0)}(t)| \,\middle\vert\,\Upsilon_n \right)+\mathbb{E}\left(\max_{1\leq k\leq p}\sup_{t\in\hat{\mathcal{T}}}|\mathcal{V}_k^{(1)}(t)-\mathcal{E}_k^{(1)}(t)| \,\middle\vert\,\Upsilon_n \right)\notag\\
     &=O_p\left\{ \frac{\Theta(p)\sqrt{n\log n}\left(\frac{p}{n}\right)^{\frac{q-2}{6q-4}}}{nh_r}\right\},
\end{align*}
which yields \eqref{eq:eq1:target boot in thm:high-dim bootstrap theorem}.
\paragraph{Step 2.}
Recall the definition of $\mathcal{V}^{(0)}(t)=(\mathcal{V}^{0}_1(t),\dots,\mathcal{V}^{(0)}_p(t))^{\top}$ in \eqref{GA-V-0},
$$
\mathcal{V}_k^{(0)}(t)=\frac{m_k'(t)}{Nh_{d,k}}\sum_{i=1}^NK_d\left(\frac{m_k(i/N)-m_{k}(t)}{h_{d,k}}\right)\frac{1}{nh_r}\sum_{j=1}^n\tilde K_r^*\left(\frac{j/n-i/N}{h_{r,k}},\frac{i}{N}\right)V_{j,k}
$$
where $V_{j,k}$ is the $k$-th coordinate of independent Gaussian random vector $\mathbf{V}_j\sim N_p(0,\mf\Sigma_{\mf G}(t_j)),j=1,\dots,n$. By Assumption \hyperref[(B4)]{(B4)}, $\underline{\Lambda}>0$ is smaller than any eigenvalues of $\mf\Sigma_{\mf G}(t_j)$ for any $j=1,2,\dots,n$, thus we have $\operatorname{Var}(V_{j,k})\geq \underline{\Lambda}$ for any $j,k$.

Firstly, we divide the Gaussian process $\sqrt{nh_r}\mathcal{V}^{(0)}(t)$ with discrete grids $s_l\in \Delta_l$ where $g_n=\lfloor{1/\Delta}\rfloor$ and $\Delta_l=[(l-1)\Delta,l\Delta]\cap\hat{\mathcal{T}}$. And we consider following centered Gaussian random vector in $\mathbb{R}^{(p\times g_n)}$
$$ 
(\sqrt{nh_{r}}\mathcal{V}^{(0)}_1(s_1),\dots,\sqrt{nh_{r}}\mathcal{V}^{(0)}_1(s_{g_n}),\dots,\sqrt{nh_{r}}\mathcal{V}^{(0)}_p(s_1),\dots,\sqrt{nh_{r}}\mathcal{V}^{(0)}_p(s_{g_n}))^{\top}.
$$
For any $k,l$, note that $V_{1,k},V_{2,k},\dots,V_{n,k}$ are independent,
\begin{align}
    &\operatorname{Var}[\sqrt{nh_{r}}\mathcal{V}_k^{(0)}(s_l)]\notag\\
    &\geq nh_r\underline{\Lambda}(m_k'(s_l))^2\sum_{j=1}^n\left[\sum_{i=1}^N\frac{1}{nh_{r,k}}\frac{1}{Nh_{d,k}}K_d(\frac{ m_k(i/N)-m_{k}(s_l)}{h_{d,k}})\tilde K_r^*\left(\frac{j/n-i/N}{h_{r,k}},\frac{i}{N}\right)\right]^2.\notag
\end{align}
Using $\min_{1\leq k\leq p}\inf_{t\in[0,1]}m_{k}'(t)\geq M>0$ in Assumption \hyperref[(A2)]{(A2)} 
then we can have for some constant $C>0$,
\begin{align}
    &\operatorname{Var}[\sqrt{nh_{r}}\mathcal{V}_k^{(0)}(s_l)]\notag\\
    &\geq C\underline{\Lambda}M^2\frac{h_r}{h_{r,k}}\iint_{[0,1]^2}K_d(\frac{m_k(u)-m_{k}(s_l)}{h_{d,k}})K_d(\frac{ m_k(v)-m_{k}(s_l)}{h_{d,k}})I(u,v)\frac{\mathrm{d}u}{h_{d,k}}\frac{\mathrm{d}v}{h_{d,k}}+o(1),\label{eq:min var}
\end{align}
where
$$
    I(u,v)=:\int_{0}^1\frac{1}{h_{r,k}}\tilde K_r^*\left(\frac{x-u}{h_{r,k}},u\right)\tilde K_r^*\left(\frac{x-v}{h_{r,k}},v\right)\mathrm{d}x.
$$
Note that for $\forall u,v\in[0,1]$ s.t. 
\begin{equation}
    |m_k(u)-m_{N,k}(s_l)|\leq h_{d,k} \text{ and } |m_k(v)-m_{N,k}(s_l)|\leq h_{d,k},\label{eq:Kd>0}
\end{equation}
we have $M|u-v|\leq |m_k(u)-m_k(v)|\leq |m_k(u)-m_{N,k}(s_l)|+|m_k(v)-m_{N,k}(s_l)|\leq 2h_{d,k}$, thus $|\frac{u-v}{h_{r,k}}|\leq \frac{2h_{d,k}}{Mh_{r,k}}$. For all $u,v\in[0,1]$ satifying \eqref{eq:Kd>0}, by the definition of $\tilde{K}_r^*$ in \eqref{eq:K^*} and conclusion in \eqref{eq:lower c(t)} $\inf_{t\in[0,1]}c_k(t)\geq \frac{1}{4}(\kappa_2-\kappa_1^2)$ where $c_k(t)=:\nu_{0,k}(t)\nu_{2,k}(t)-\nu^2_{1,k}(t)$,
\begin{align}
    \left|\tilde K_r^*\left(\frac{x-u}{h_{r,k}},u\right)-\tilde K_r^*\left(\frac{x-v}{h_{r,k}},v\right)\right|&\leq  \frac{1}{c_k(u)}\left|\nu_{2,k}(u)\tilde K_r(\frac{x-u}{h_{r,k}})-\nu_{2,k}(v)\tilde K_r(\frac{x-v}{h_{r,k}})\right|\notag\\
    &+\frac{1}{c_k(u)}\left|\nu_{1,k}(u)\tilde K_r(\frac{x-u}{h_{r,k}})(\frac{x-u}{h_{r,k}})-\nu_{1,k}(v)\tilde K_r(\frac{x-v}{h_{r,k}})(\frac{x-v}{h_{r,k}})\right|  \notag\\
    &+\frac{|c_k(u)-c_k(v)|}{c_k(u)c_k(v)}\left|\nu_{2,k}(v)\tilde K_r(\frac{x-v}{h_{r,k}})+\nu_{1,k}(v)\tilde K_r(\frac{x-v}{h_{r,k}})(\frac{x-v}{h_{r,k}})\right|\notag\\
    &\leq C\left|\frac{u-v}{h_{r,k}}\right|\leq \frac{2Ch_{d,k}}{Mh_{r,k}},\notag
\end{align}
where $C>0$ is a universal constant and does not depend on $x,u,v$. 

In the following we show $I(u,v)$ has strictly positive lower bound for all $u,v\in[0,1]$ satisfying \eqref{eq:Kd>0}. Note that for any $k=1,\dots,p$, $c_k(t)$ is bounded and $\min_{1\leq k\leq p} h_{r,k}\asymp h_r\asymp \max_{1\leq k\leq p} h_{r,k},\min_{1\leq k\leq p} h_{d,k}\asymp h_d\asymp \max_{1\leq k\leq p} h_{d,k}$ with $h_d/h_r\rightarrow0$, we have
\begin{align}
    I(u,v)&=I(u,u)+O(\frac{h_{d,k}}{h_{r,k}})\notag\\
    &=\frac{1}{c_k(u)^2}\int_{0}^1 \frac{1}{h_{r,k}}\left[\nu_{2,k}(u)\tilde K_r(\frac{x-u}{h_{r,k}})-\nu_{1,k}(u)\tilde K_r(\frac{x-u}{h_{r,k}})(\frac{x-u}{h_{r,k}})\right]^2\mathrm{d}x+O(\frac{h_{d,k}}{h_{r,k}})\notag\\
    &\geq \frac{1}{\max_{t\in[0,1]}c_k(t)^2}\int_{\frac{0-u}{h_{r,k}}}^{\frac{1-u}{h_{r,k}}} [\nu_{2,k}(u)-\nu_{1,k}(u)z]^2\tilde K_r^2(z)\mathrm{d}z+o(1),\label{eq:I(u,u)}
\end{align}
with uniformly $u,v\in[0,1]$ satisfying \eqref{eq:Kd>0}.
\begin{description}
    \item[Case 1.] For any $u\in[h_{r,k},1-h_{r,k}]$ and $\kappa_l=\int_{-1}^1 |z|^lK_r(z)\mathrm{d}z$, $\phi_l=\int_{-1}^1|z|^l\tilde K_r^2(z)\mathrm{d}z$, 
\begin{equation}
    \int_{\frac{0-u}{h_{r,k}}}^{\frac{1-u}{h_{r,k}}} [\nu_{2,k}(u)-\nu_{1,k}(u)z]^2\tilde K_r^2(z)\mathrm{d}z=\kappa_2^2\phi_0+o(1).\label{eq:case1 u}
\end{equation}
    \item[Case 2.] For any $u\in[0,h_{r,k})$, 
    \begin{equation}
    \int_{\frac{0-u}{h_{r,k}}}^{\frac{1-u}{h_{r,k}}} [\nu_{2,k}(u)-\nu_{1,k}(u)z]^2\tilde K_r^2(z)\mathrm{d}z\geq \int_0^1 [\nu_{2,k}(u)-\nu_{1,k}(u)z]^2\tilde K_r^2(z)\mathrm{d}z.\label{eq:case2 u}
    \end{equation}
    Denote the right side of \eqref{eq:case2 u} as $G(u)$, then
    \begin{equation*}
        G(u)\geq \min\{G(0),G(h_{r,k}),G(u^*)\}
    \end{equation*}
    where $0=G'(u^*).$ By calculation, we have
    \begin{equation}
        G(u)\geq \min\left\{ \frac{1}{4}\int_0^1(\kappa_2-\kappa_1z)^2\tilde K_r^2(z)\mathrm{d}z, \frac{1}{2}\kappa_2^2\phi_0 ,\frac{\kappa_2^2(\phi_2\phi_0-\phi_1^2)}{8(\phi_1+\phi_0)} \right\}.\label{eq:case2 G(u)}
    \end{equation}
    \item[Case 3.] For any $u\in(1-h_{r,k},1]$, by similar arguments in Case 2, we have
    \begin{align}
        &\int_{\frac{0-u}{h_{r,k}}}^{\frac{1-u}{h_{r,k}}} [\nu_{2,k}(u)-\nu_{1,k}(u)z]^2\tilde K_r^2(z)\mathrm{d}z\geq \int_{-1}^0 [\nu_{2,k}(u)-\nu_{1,k}(u)z]^2\tilde K_r^2(z)\mathrm{d}z\notag\\
        &\geq\min\left\{ \frac{1}{4}\int_{-1}^0(\kappa_2-\kappa_1z)^2\tilde K_r^2(z)\mathrm{d}z, \frac{1}{2}\kappa_2^2\phi_0 ,\frac{\kappa_2^2(\phi_2\phi_0-\phi_1^2)}{8(\phi_1+\phi_0)} \right\}.\label{eq:case3 u}
    \end{align}
\end{description}
Combining \eqref{eq:I(u,u)}, \eqref{eq:case1 u}, \eqref{eq:case2 u}, \eqref{eq:case2 G(u)}, \eqref{eq:case3 u} and the fact $\phi_2\phi_0-\phi_1^2>0$, 
\begin{equation}
    I(u,v)\geq  \frac{1}{\max_{t\in[0,1]}c^2(t)}\min\left\{ \frac{1}{4}\int_{-1}^0(\kappa_2-\kappa_1z)^2\tilde K_r^2(z)\mathrm{d}z, \frac{1}{2}\kappa_2^2\phi_0 ,\frac{\kappa_2^2(\phi_2\phi_0-\phi_1^2)}{8(\phi_1+\phi_0)} \right\}>0,\label{eq:min I(u,v)}
\end{equation}
uniformly with $u,v\in[0,1]$ satisfying \eqref{eq:Kd>0}. Combining \eqref{eq:min I(u,v)} and \eqref{eq:min var}, there exists constant $\underline{\sigma}>0$ s.t.
\begin{equation}
    \min_{1\leq k\le p}\min_{1\leq l\leq g_n} \operatorname{Var}[\sqrt{nh_{r}}\mathcal{V}_k^{(0)}(s_l)]\geq \underline{\sigma}^2.\label{eq:min Var sigma}
\end{equation}
Using the fact \eqref{eq:min Var sigma}, by Nazarov's inequality in Lemma \ref{nazarov}, for every $\rho>0$,
\begin{equation}
    \sup_{x\in\mathbb{R}}\mathbb{P}\left\{x-\rho\leq \max_{1\leq k\leq p}\max_{1\leq l\leq g_n} \sqrt{nh_{r}}\left| \mathcal{V}^{(0)}_k(s_l) \right| \leq  x+\rho\right\}\leq \frac{2\rho}{\underline{\sigma}}(\sqrt{2\log g_n+2\log 2p}+2).\label{eq:anti con for disc V}
\end{equation}

Secondly, we consider the difference between $\mathcal{V}^{(0)}(t)$ and $\mathcal{V}^{(0)}(s_l)$. Denote
$$
\mathcal{D}_{n,k,l,i}(t)=:m_{k}'(t)K_d(\frac{m_k(i/N)-m_{k}(t)}{h_{d,k}})-m_{k}'(s_l)K_d(\frac{m_k(i/N)-m_{k}(s_{l})}{h_{d,k}}),
$$
then we have 
\begin{align}
    0&\leq \max_{1\leq k\leq p}\sup_{t\in\hat{\mathcal{T}}}|\mathcal{V}_k^{(0)}(t)|-\max_{1\leq k\leq p}\max_{1\leq l\leq g_n}|\mathcal{V}_k^{(0)}(s_l)| \notag\\
    &\leq \max_{1\leq k\leq p}\max_{1\leq l\leq g_n}\left\{\sup_{t\in\Delta_l}|\mathcal{V}_k^{(0)}(t)|-|\mathcal{V}_k^{(0)}(s_l)|\right\}\notag\leq \max_{1\leq k\leq p}\max_{1\leq l\leq g_n}\sup_{t\in\Delta_l}\left|\mathcal{V}_k^{(0)}(t)-\mathcal{V}_k^{(0)}(s_l)\right|\notag\\
    &\leq \max_{1\leq k\leq p}\max_{1\leq l\leq g_n}\sup_{t\in\Delta_l}\frac{1}{Nh_{d,k}}\sum_{i=1}^N|\mathcal{D}_{n,k,l,i}(t)|\sum_{j=1}^n\frac{1}{nh_{r,k}}\tilde K_r^*\left(\frac{j/n-i/N}{h_{r,k}},\frac{i}{N}\right)|V_{j,k}|\notag\\
    &\leq \left(\max_{1\leq k\leq p}\max_{1\leq j\leq n}|V_{j,k}|\right)\cdot \left( \max_{1\leq k\leq p}\max_{1\leq l\leq g_n}\sup_{t\in\Delta_l}\frac{1}{Nh_{d,k}}\sum_{i=1}^N|\mathcal{D}_{n,k,l,i}(t)| \right)\notag\\
    & \cdot \left( \max_{1\leq k\leq p}\max_{1\leq i\leq N}\sum_{j=1}^n\frac{1}{nh_{r,k}}\tilde K_r^*\left(\frac{j/n-i/N}{h_{r,k}},\frac{i}{N}\right) \right) \notag\\
    &\lesssim   \left(\max_{1\leq k\leq p}\max_{1\leq j\leq n}|V_{j,k}|\right)\cdot \left( \max_{1\leq k\leq p}\max_{1\leq l\leq g_n}\sup_{t\in\Delta_l}\frac{1}{Nh_{d,k}}\sum_{i=1}^N|\mathcal{D}_{n,k,l,i}(t)| \right).\label{eq:max V and max D}
\end{align}
For any $k,l$ and $t\in\Delta_l$, by Lagrange's mean value, for some $\eta_{k,t}$ between $m_{k}(s_l)$ and $m_{k}(t)$,
\begin{align}
    \frac{1}{Nh_{d,k}}\sum_{i=1}^N|\mathcal{D}_{n,k,l,i}(t)|&\leq \frac{1}{Nh_{d,k}}\sum_{i=1}^NK_d(\frac{m_{k}(i/N)-m_{k}(t)}{h_{d,k}})|m_{k}'(t)-m_{k}'(s_l)|,\notag\\
    &+ m_{k}'(t)\frac{1}{Nh_{d,k}}\sum_{i=1}^NK_d'(\frac{m_{k}(i/N)-\eta_{k,t}}{h_{d,k}})|\frac{m_{k}(t)-m_{k}(s_l)}{h_{d,k}}|,\notag\\
    &= O(\frac{1}{g_nh_{d,k}}),\notag
\end{align}
combining universal Lipschitz condition in Assumption \hyperref[(A1)]{(A1)}, we have
\begin{equation}
    \max_{1\leq k\leq p}\max_{1\leq l\leq g_n}\sup_{t\in\Delta_l}\frac{1}{Nh_{d,k}}\sum_{i=1}^N|\mathcal{D}_{m,k,l,i}(t)| =O(\frac{1}{g_nh_d})\label{eq:max D}.
\end{equation}
By Assumption \hyperref[(B4)]{(B4)}, $\Lambda$ is larger than any eigenvalues of $\mathbf{\Sigma}_i,i=1,\dots,n$, thus 
$$
\max_{1\leq k\leq p}\max_{1\leq j\leq n}\operatorname{Var}(V_{j,k})\leq \Lambda.$$
Using Lemma 2.3.4 in \cite{gine2021mathematical}, it yields
\begin{equation}
    \mathbb{E}(\max_{1\leq k\leq p}\max_{1\leq j\leq n}|V_{j,k}|)\leq \Lambda^{1/2}\sqrt{2\log(np)}.\label{eq:max V}
\end{equation}
Combining \eqref{eq:max V}, \eqref{eq:max D} and \eqref{eq:max V and max D}, for every $\delta^*>0$
\begin{align}
    &\mathbb{P}\left\{ \max_{1\leq k\leq p}\sup_{t\in\hat{\mathcal{T}}}|\mathcal{V}_k^{(0)}(t)|-\max_{1\leq k\leq p}\max_{1\leq l\leq g_n}|\mathcal{V}_k^{(0)}(s_l)| >\delta^*\right\}\notag\\
    &\leq \frac{1}{\delta^*}\mathbb{E}\left\{\max_{1\leq k\leq p}\sup_{t\in\hat{\mathcal{T}}}|\mathcal{V}_k^{(0)}(t)|-\max_{1\leq k\leq p}\max_{1\leq l\leq g_n}|\mathcal{V}_k^{(0)}(s_l)|\right\}\notag\\
    &\lesssim \frac{1}{\delta^*g_nh_d}\mathbb{E}\left(\max_{1\leq k\leq p}\max_{1\leq j\leq n}|V_{j,k}|\right)\lesssim\frac{\sqrt{\log(np)}}{\delta^*g_nh_d}.\label{eq:max V V continu}
\end{align}

Finally, combining \eqref{eq:max V V continu} and \eqref{eq:anti con for disc V}, we can have
\begin{align}
    &\sup_{x\in\mathbb{R}}\mathbb{P}\left\{x-\rho_n^*\sqrt{nh_r}\leq \max_{1\leq k\leq p}\sup_{t\in\hat{\mathcal{T}}}\sqrt{nh_r}|\mathcal{V}_k^{(0)}(t)| \leq x+\rho_n^*\sqrt{nh_r} \right\}\notag\\
    &\leq \sup_{x\in\mathbb{R}}\mathbb{P}\left\{x-(\rho_n^*+\delta^*)\sqrt{nh_r}\leq \max_{1\leq k\leq p}\max_{1\leq l\leq g_n}\sqrt{nh_r}|\mathcal{V}_k^{(0)}(s_l)| \leq x+(\rho_n^*+\delta^*)\sqrt{nh_r} \right\}\notag\\
    &+\mathbb{P}\left\{ \max_{1\leq k\leq p}\sup_{t\in\hat{\mathcal{T}}}|\mathcal{V}_k^{(0)}(t)|-\max_{1\leq k\leq p}\max_{1\leq l\leq g_n}|\mathcal{V}_k^{(0)}(s_l)| >\delta^*\right\}\notag\\
    &\lesssim (\rho_n^*+\delta^*)\sqrt{nh_r}\sqrt{\log(g_np)}+\frac{\sqrt{\log (np)}}{\delta^*g_nh_d}.\label{eq:rate for eq2:target boot in thm:high-dim bootstrap theorem}
\end{align}
Let $g_n=N$, then we can minimize \eqref{eq:rate for eq2:target boot in thm:high-dim bootstrap theorem} by setting 
$$
\delta^*=\sqrt{\frac{\log^{1/2}(np)}{Nh_d\sqrt{nh_r\log(Np)}}},
$$
which yields the desired rate in \eqref{eq:eq2:target boot in thm:high-dim bootstrap theorem}. 

\TPR{Finally, by assertions in Step 1 and Step 2, setting $\rho_n^*=\sqrt{\rho_n/\sqrt{nh_r\log(Np)}}$, we have
\begin{align*}
    &\sup_{x\in\mathbb{R}}\left|\mathbb{P}\left\{\max_{1\leq k\leq p}\sup_{t\in\hat{\mathcal{T}}} \sqrt{nh_r}\left| \mathcal{V}^{(0)}_k(t) \right|\leq x \right\} - \mathbb{P}\left\{\max_{1\leq k\leq p}\sup_{t\in\hat{\mathcal{T}}} \sqrt{nh_r}\left| \mathcal{V}^*_k(t) \right|\leq x \,\middle\vert\, \Upsilon_n \right\} \right|\notag\\
    &=O_p\left(\sqrt{\rho_n\sqrt{nh_r\log(Np)}}+\sqrt{\sqrt{nh_r\log(Np)\log(np)}/Nh_d}\right),
\end{align*}
which proves Theorem \ref{thm:high-dim bootstrap theorem} by condition \eqref{eq:rate for high-dim boot + continu}.}
\end{proof}

\subsubsection{Proof of Theorem \ref{thm:multi-SCB}}
\begin{proof}
Proof of (i). By Lemma \ref{lem:multi-dim ll}, the new jackknife corrected local linear estimator $\tilde{\mf m}(t)=2\hat{\mf m}_{h_r/\sqrt{2}}(t)-\hat{\mf m}_{h_r}(t)$ defined from \eqref{local linear} also follows 
$$
\max_{1\leq k\leq p}\sup_{t\in[0,1]}|\tilde{m}_k(t)-m_k(t)|=O_p(R_n),
$$
thus all the conditions in Lemma \ref{m_I} also hold in Theorem \ref{thm:multi-SCB}.
Denote
$$
m_{\mathbf{C},N,k}^{-1}(v)=:\frac{1}{Nh_d}\sum_{i=1}^N\int_{-\infty}^vK_d\left(\frac{m_{\mathbf{C},k}(i/N)-u}{h_d}\right)\mathrm{d}u,
$$
then by similar arguments in \eqref{eq:recovery in eq:GA in thm:high-dim bootstrap theorem} and \eqref{e26},
\begin{equation}
    \max_{1\leq k\leq s}\sup_{t\in\hat{\mathcal{T}}} |m_{\mf C,N,k}(t)-\hat{m}_{\mf C,I,k}(t)-\frac{\hat{m}_{\mf C,I,k}^{-1}-m_{\mf C,N,k}^{-1}}{(m_{\mf C,N,k}^{-1})^{\prime}}\circ m_{\mf C,N,k}(t)|=O_p(\TPR{\frac{R_n^2}{h_d}}),\label{eq:key0 to thm:multi-SCB}
\end{equation}
\begin{equation}
    \max_{1\leq k \leq s}\sup_{v\in\hat{\mathcal{T}}^{-1}}\left|\hat{m }_{\mathbf{C},I,k}^{-1}(v)-m_{\mathbf{C},N,k}^{-1}(v)-\tilde{A}_{\mathbf{C},k}(v)\right|=O_p(\TPR{\frac{R_n^2}{h_d}}),\label{eq:key1 to thm:multi-SCB}
\end{equation}
where 
$$
\tilde{A}_{\mathbf{C},k}(v)=:\frac{1}{Nh_d}\sum_{i=1}^NK_d(\frac{m_{C,k}(i/N)-v}{h_{d}})(\tilde{m}_{C,k}(i/N)-m_{C,k}(i/N) ),
$$
and $\tilde{\mf m}_{\mf C}=\mf C\tilde{\mf m}(t)$ is obtained by jackknife corrected local linear estimator $\tilde{\mathbf{m}}(t)$. By \eqref{eq:multi-dim stochastic expansion of ll} in Lemma \ref{lem:multi-dim ll} with jackknife correction,
\begin{equation*}
\max_{1\leq k\leq s}\sup_{v\in\hat{\mathcal{T}}^{-1}}\left|\tilde{A}_{\mathbf{C},k}(v)-
\frac{1}{Nh_d}\sum_{i=1}^NK_d\left(\frac{m_{C,k}(i/N)-v}{h_{d}}\right)\Omega_{\mf C,k}(i/N)
\right|=O_p(h_r^3+\frac{1}{nh_r}),
\end{equation*}
where $\mf{\Omega}_{\mf C}(t)=(\Omega_{\mf C,1}(t),\dots,\Omega_{\mf C,s}(t))^\top$ and
$$
    \mf \Omega_{\mf C}(t)=\frac{1}{nh_r}\sum_{j=1}^n\tilde K_r^*\left(\frac{t_j-t}{h_r},t\right)\mf C\mf M^{-1}(t_j)\mathbf{x}_{j}e_j.
$$
By Corollary 1 in \cite{Wu2011gaussian}, there exists independent $\mathbf{V}_j\sim N(0,\mathbf{\Sigma}_{\mf C}(t_j)),j=1,\dots,n$ on a richer space s.t.
$$
\max_{i\leq n}\left| \sum_{j=1}^i\mf C\mf M^{-1}(t_j)\mathbf{x}_je_j-\sum_{j=1}^i\mathbf{V}_j  \right|=o_p(n^{1/4}\log^2n).
$$
Summation by parts yields that
\begin{align}
    &\max_{1\leq k\leq s}\sup_{v\in\hat{\mathcal{T}}^{-1}}\left|\frac{1}{Nh_d}\sum_{i=1}^NK_d(\frac{m_{C,k}(i/N)-v}{h_{d}})\left(\Omega_{\mf C,k}(i/N)-\frac{1}{nh_r}\sum_{j=1}^n\tilde K_r^*\left(\frac{j/n-i/N}{h_r},\frac{i}{N}\right)\mathbf{V}_{j,k}\right) \right|\notag\\
    &\lesssim \max_{1\leq k\leq s}\sup_{v\in\hat{\mathcal{T}}^{-1}}\left|\frac{1}{Nh_d}\sum_{i=1}^NK_d(\frac{m_{C,k}(i/N)-v}{h_{d}})\frac{1}{nh_r}\sum_{k=1}^{n-1}\left[\tilde K_r^*\left(\frac{\frac{k}{n}-\frac{i}{N}}{h_r},\frac{i}{N}\right)-\tilde K_r^*\left(\frac{\frac{k+1}{n}-\frac{i}{N}}{h_r},\frac{i}{N}\right)\right]\right|\notag\\
    &\cdot \max_{i\leq n}\left| \sum_{j=1}^i\mf C\mf M^{-1}(t_j)\mathbf{x}_je_j-\sum_{j=1}^i\mathbf{V}_j  \right|=o_p(\frac{n^{1/4}\log^2(n)}{nh_r}).
    \label{eq:key3 to thm:multi-SCB}
\end{align}
Therefore, combining \eqref{eq:key0 to thm:multi-SCB}, \eqref{eq:key1 to thm:multi-SCB} and \eqref{eq:key3 to thm:multi-SCB}, there exists independent $\mathbf{V}_j=(V_{j,1},\dots,V_{j,s})^{\top}\sim N_s(0,\mathbf{\Sigma}_{\mf C}(t_j)),j=1,\dots,n$ on a richer space s.t.
\begin{equation}
    \max_{1\leq k\leq s}\sup_{t\in \hat{\mathcal{T}} }\left| \hat{m }_{\mathbf{C},I,k}(t)-m_{\mathbf{C},N,k}(t)-\mathcal{V}_{\mathbf{C},k}^{(0)}(t) \right|=O_p(\rho_n)
    \label{eq:key to thm:multi-SCB}
\end{equation}
where $\rho_n=R_n^2/h_d+\frac{n^{1/4}\log^2(n)}{nh_r}$ and $\mathcal{V}_{\mathbf{C}}^{(0)}(t)=(\mathcal{V}^{(0)}_{\mathbf{C},1}(t),\dots,\mathcal{V}^{(0)}_{\mathbf{C},s}(t))^{\top}$
$$
\mathcal{V}^{(0)}_{\mathbf{C},k}(t)=:\frac{m_{\mf C,k}^{\prime}(t)}{Nh_d}\sum_{i=1}^NK_d\left(\frac{m_k(i/N)-t}{h_d}\right)\frac{1}{nh_r}\sum_{j=1}^n\tilde K_r^*\left(\frac{j/n-i/N}{h_r},\frac{i}{N}\right)V_{j,k}.
$$
Combining \eqref{eq:key to thm:multi-SCB} and Lemma \ref{lem:m_N}, \eqref{eq:GA in prop:GA in multi-dim} holds.

Proof of (ii). Denote $\mathcal{V}^{(1)}_{\mf C}(t)=:(\mathcal{V}^{(1)}_{\mf C,1},\dots,\mathcal{V}^{(1)}_{\mf C,s}(t))^\top$ where
\begin{equation*}
    \mathcal{V}_{\mf C,k}^{(1)}(t)=:\frac{1}{nh_{r}}\sum_{j=1}^n\sum_{i=1}^N \hat W_{i,k,I}^*(t)\tilde K_r^*\left(\frac{j/n-i/N}{h_{r}},\frac{i}{N}\right) V_{j,k}.
\end{equation*}
By similar arguments \eqref{eq:Mies high-dim given data} and \eqref{eq:Mies high-dim no given data} with Lemma \ref{lem:multi-dim long-run var}, we can have
\begin{equation}
    \max_{1\leq k\leq s}\sup_{t\in\mathcal{T}}|\mathcal{V}_{\mf C,k}^{(1)}(t)-\mathcal{V}_{\mf C,k}^{*}(t)|=O_p\left(\frac{\sqrt{\log(n)(\sqrt{n\bar{\varphi}_n}+\bar{\varphi}_n)}}{nh_r}\right)\label{eq:I3 multi-dim},
\end{equation}
where $\bar{\varphi}_n=\sqrt{n L}+n L^{-1}+\sqrt{n}L\left(h_r^2+\frac{\log n}{\sqrt{nh_r}}\right)+nL\left(h_r^5+\frac{\log^2n}{nh_r}\right)$. Denote new union rate
$$
\rho_n^*=\frac{n^{1/4}\log^2 n }{nh_r} \bigvee \frac{\sqrt{\sqrt{n \bar{\varphi}_n}+\bar{\varphi}_n}\log (n)}{nh_r}.
$$
By similar arguments in Step 1 and Step 2 of Theorem \ref{thm:high-dim bootstrap theorem}, we can have
\begin{equation}
    \max_{1\leq k\leq s}\sup_{t\in\hat{\mathcal{T}}} \left|\mathcal{V}_{\mf C,k}^{(0)}(t)-\mathcal{V}_{\mf C,k}^{(1)}(t)\right|=o_p(\rho_n^*)\label{eq:I1 multi-dim}
\end{equation}
\begin{equation}
        \sup_{x\in\mathbb{R}}\mathbb{P}\left\{x-\rho_n^* \leq \max_{1\leq k\leq n}\sup_{t\in\hat{\mathcal{T}}}|\mathcal{V}^{(0)}_{\mf C,k}(t)|\leq x+\rho_n^* \right\}=o(1).
    \label{eq:I2 multi-dim}
\end{equation}
By Lemma \ref{lem:prob inequality},
\begin{align}
    &\sup_{x\in\mathbb{R}} \left|\mathbb{P}\left\{\max_{1\leq k\leq s}\sup_{t\in\hat{\mathcal{T}}} \left|\mathcal{V}_{\mf C,k}(t) \right|\leq x \right\} - \mathbb{P}\left\{\max_{1\leq k\leq s}\sup_{t\in\hat{\mathcal{T}}} \left| \mathcal{V}^*_{\mf C,k}(t) \right|\leq x \,\middle\vert\, \Upsilon_n \right\} \right|\notag\\
    &\lesssim \mathbb{P}\left\{\max_{1\leq k\leq s}\sup_{t\in\mathcal{T}}|\mathcal{V}_{\mf C,k}^{(1)}(t)-\mathcal{V}_{\mf C,k}^{*}(t)|>\rho_n^*\,\middle\vert\, \Upsilon_n\right\}+\mathbb{P}\left\{\max_{1\leq k\leq s}\sup_{t\in\hat{\mathcal{T}}} \left|\mathcal{V}_{\mf C,k}^{(0)}(t)-\mathcal{V}_{\mf C,k}^{(1)}(t)\right|>\rho_n^*\right\}\notag\\
    &+\sup_{x\in\mathbb{R}}\mathbb{P}\left\{x-\rho_n^* \leq \max_{1\leq k\leq n}\sup_{t\in\hat{\mathcal{T}}}|\mathcal{V}^{(0)}_{\mf C,k}(t)|\leq x+\rho_n^* \right\}\label{eq:I1+I2+I3 multi-dim}.
\end{align}
Combining \eqref{eq:I3 multi-dim},\eqref{eq:I1 multi-dim},\eqref{eq:I2 multi-dim} and \eqref{eq:I1+I2+I3 multi-dim}, \eqref{eq:boot in thm:multi-dim bootstrap theorem} holds.
\end{proof}

\TPM{
\subsection{penalized SCB for weakly monotone assumption}
\label{sec:penalization}
}

\TPM{To include the case when the trend stays flat for a period, we propose a penalized SCB allowing $M=0$. Specifically, we estimate $\mathbf{f}(t)=:(f_1(t),\dots,f_p(t))^\top$ from pseudo data
$$
\mathbf{y}_i^*=\mathbf{y}_i+\lambda_n(t_i)\mathbf{1}_p=\mathbf{f}(t_i)+\mathbf{e}_i,
$$
where $f_k(t)=:m_k(t)+\lambda_n(t)$ with a pre-determined penalization 
$$
\lambda_n(t)=C_1(n)t+C_2(n)t^2+C_3(n)t^3.
$$ 
$C_1(n)$, $C_2(n)$, $C_3(n)$ are strictly positive coefficients, satisfying $\max\{C_1(n),C_2(n),C_3(n)\}\to 0$, $C_1(n)/C_2(n)\to \infty$, and $C_2(n)\asymp C_3(n)$ as $n\rightarrow\infty$. 
We denote $\tilde{\mathbf{f}}(t)=(\tilde{f}_1(t),\dots,\tilde{f}_p(t))^\top$ as the jackknife bias-corrected local linear estimator obtained by replacing $\mathbf{y}_i$ with $\mathbf{y}_i^*$ in \eqref{eq:high-dim ll def}.
Observing that $f_k(t)$ has a strictly positive derivative, therefore we could apply the increasing rearrangement in Section \ref{sec:Gaussian approximations} to $\tilde{\mathbf{f}}(t)$. Then the monotone estimator of $\mf f(t)$ can be expressed as  $\hat{\mathbf{f}}_I(t)=(\hat{f}_{I,1}(t),\dots,\hat{f}_{I,p}(t))^\top$ where $\hat{f}_{I,k}(t)$ is defined by the inverse of
$$
\hat{f}_{I,k}^{-1}(t)=\frac{1}{Nh_d}\sum_{i=1}^N \int_{-\infty}^tK_d\left(\frac{\tilde f_k(i/N)-u}{h_d}\right)\mathrm{d} u,
$$
and our final estimator is $\hat {\mf m}_{correct}(t)=(\hat {m}_{correct,1}(t),\dots,\hat { m}_{correct,p}(t))^\top$ where
\begin{equation}
 \hat {\mf m}_{correct}(t)=:  \hat{\mf f}_I(t)-(\lambda_n(t)-C_3(n)g_{n,1}t)\mathbf{1}_p,\quad g_{n,1}=\frac{\log^4 n}{h_d\sqrt{nh_r}}+\frac{h_d^2\log n}{C_1^{5/2}(n)}.\label{eq:corrected penal estimator}
\end{equation}
The purpose of substracting $\lambda_n(t)-C_3(n)g_{n,1}t$ from $\hat{\mathbf f}_I(t)$ instead of $\lambda_n(t)$ is to maintain the monotonicty of $\hat{\mf m}_{correct}(t)$. The confidence bands are derived by subtracting $\lambda_n(t)-C_3(n)g_{n,1}t$ from the lower bound and upper bound of SCBs for $\mathbf{f}(t)$ derived from Algorithm \ref{alg:high-dim bootstrap} with new inputs $\mathbf{y}_i^*,i=1,\dots,n$.}

\TPM{The penalized SCB extends the application scope to allow for flat curves while retaining several key advantages: the monotonicity of the estimates and confidence bands, asymptotic correctness, and the property that the estimates lie in the middle of the bands. The cost of the penalization lies in that if $m'(t)>0$, the penalization introduces additional bias in the estimation. Additionally, in simulations, we observe that larger values of $C_1(n)$ and $C_2(n)$ tend to result in wider SCBs.}

\TPM{We show in Proposition \ref{prop:penalization} that subtracting $\lambda_n(t)-C_3(n)g_{n,1}t$  from the estimate and the confidence bands obtained by Algorithm \ref{alg:high-dim bootstrap} will yield asymptotically correct SCBs for $\mf m(t)$, while $\hat{\mathbf{m}}_{correct}(t)$ and the confidence bands both maintain monotonicity with high probability even when $m_k(t)$ exists flat part for some $1\leq k\leq p$. In practice, we recommend to start applying penalization with an initial $C_1(n)=C_{1,0}$ which is not too small, and $C_2(n)=C_3(n)=C_1^{19/8}(n)$. If the mononicity is violated, we could slightly lower $C_1(n)$; say, calculate $C_1(n)=aC_{1,0}$ for $a=0.9, 0.8,..$ and use the biggest $a$ such that  $\hat{\mathbf{m}}_{correct}(t)$ is monotone.
Notice that by construction, the monotonicity of $\hat{\mathbf{m}}_{correct}(t)$ implies the monotonicity of both the lower and upper bound of the simultaneous confidence bands. }

\TPM{
\begin{description}[itemsep=0pt,parsep=0pt,topsep=0pt,partopsep=0pt]
    \item[\setword{\textbf{(A1')}}{(A1')}] \textit{For the vector function $\mathbf{m}(t) = (m_1(t), . . . ,m_p(t))^\top$ where $t \in[0, 1]$, the third derivative $m_k^{'''}(t)$ exists and is Lipschitz continuous on $[0,1]$ with Lipschitz constants  bounded for all $k = 1,\dots,p$.}
    \item[\setword{\textbf{(A2')}}{(A2')}] \textit{$\min_{1\leq k\leq p}\min_{t\in[0,1]}m_k^\prime(t)\geq 0$. If equality holds, then all roots $s$ of $m_k^\prime(\cdot)$ satisfy either of the following 2 conditions for some universal constant $\eta>0$,\\
(1) $m_k^{(\alpha)}(s)=0,\text{ for all integer $\alpha\geq 1$}$ where $m_k^{(\alpha)}(\cdot)$ is the $\alpha_{th}$ derivative;\\
(2) $s$ is an interior point and $|m_k^{'''}(s)|\geq \eta>0$.}
\end{description}
}

\TPM{The condition (A2') (2) is a technical condition and can be possibly generalized to any integer $\alpha\geq 3$ with a substantially more involved mathematical argument.}

\TPM{
\begin{prop}
    \label{prop:penalization}
    Under Assumption \hyperref[(B1)]{(B1)}-\hyperref[(B5)]{(B5)}, and \hyperref[(C1)]{(C1)}-\hyperref[(C2)]{(C2)}, suppose \eqref{eq:rate for high-dim boot + continu} holds and
     $\tilde{q}_{1-\alpha}$ is the output quantile of  Algorithm \ref{alg:high-dim bootstrap} with input $\mathbf{y}_i+\lambda_n(i/n)\mathbf{1}_p,i=1,\dots,n$ and $B$ is the bootstrap sample size. If $\mathbf{m}(t)$ satisfies Assumption \hyperref[(A1')]{(A1')} and \hyperref[(A2')]{(A2')}, and $C_1(n)$, $C_2(n)$, $C_3(n)$ are strictly positive coefficients, satisfying $\max\{C_1(n),C_2(n),C_3(n)\}\to 0$, $C_1(n)/C_2(n)\to \infty$, and $C_2(n)\asymp C_3(n)>0$ as $n\rightarrow\infty$, we then have following assertions:\\
    (i). If $C_3(n)g_{n,1}=o(1/\sqrt{nh_r})$, then
    \begin{equation}
        \lim_{n\rightarrow\infty}\lim_{B\rightarrow\infty}\mathbb P\left( \max_{1\leq k\leq p}\sup_{t\in\hat{\mathcal{T}} }|\hat{m}_{correct,k}(t)-m_k(t)|\leq \tilde{q}_{1-\alpha}\right)= 1-\alpha,\label{eq:penalized correct SCB}
    \end{equation}
    where $\hat{\mathcal{T}}=\cap_{k=1}^p \hat{\mathcal{T}}_k$ and $\hat{\mathcal{T}}_k$ is defined in the same way as \eqref{index_I} with $\tilde{m}_k(t)$ replaced by $\tilde{f}_k(t)$.\\
    (ii). Further suppose $ C_1^3(n)/C_3(n)=o(h_d^3\sqrt{nh_r}/\log^3 n)$ and $C_3(n)\asymp C_1^{s}(n)$ for $s\in(1,5/2)$, then we have
    \begin{equation}
    \lim_{n\rightarrow\infty}\mathbb{P}\left\{\min_{1\leq k\leq p}\inf_{t\in\hat{\mathcal{T} }} \hat{m}_{correct,k}^\prime(t)\geq 0\right\}=1.\label{eq:monotone shape}    
    \end{equation}
\end{prop}}

\TPM{\begin{remark}
\label{remark:penalization}
    For brevity, we simply write $C_1,C_2,C_3$, $\lambda(t)$ instead of $C_1(n),C_2(n),C_3(n)$, $\lambda_n(t)$ if no confusion is made.We now discuss the feasible $C_1,C_2, C_3$ satifying conditions in both (i) and (ii). Suppose $h_d=h_r^{1+k}$ with $k\in(0,1)$ satisfying \hyperref[(C2)]{(C2)} and $C_3=C_1^s$, $s\in(1,5/2]$ then $C_3g_{n,1}=o(\log (n)/\sqrt{nh_r})$ and $ C_1^3/C_3=o(h_d^3\sqrt{nh_r}/\log^3 n)$ can hold if $nh_r^{4s(k+1)-6k-5}=o(1)$. For example, one can choose $s=19/8$ and $C_1\asymp(nh_r^{(16+13k)/3})^{12/5}$  with $k\in(5/7,1)$.
\end{remark}
}

\begin{proof}
\TPM{\textit{Proof of (i).}  
By Remark \ref{remark:GA in high-dim}, $\sqrt{nh_r}\mathcal{V}(t)$ is a valid nondegenerate Gaussian process. Thus conditions $C_3(n)g_{n,1}=o_p(1/\sqrt{nh_r})$ ensure that $C_3g_{n,1}t$ are negligible in contrast with our SCB quantile $\tilde{q}_{1-\alpha}$. By applying Proposition \ref{prop:GA in high-dim} and Theorem \ref{thm:high-dim bootstrap theorem}, \eqref{eq:penalized correct SCB} holds.}

\TPM{\textit{Proof of (ii).} 
For any $t\in\hat{\mathcal{T}}$, there exists random vectors  $\hat{s}_k(t)\in[h_r,1-h_r],k=1,2,\dots,p$, s.t. $\tilde{f}_k(\hat{s}_k(t))=\hat{f}_{I,k}(t)$. For brevity, we write $\hat{s}_k$ instead of $\hat{s}_k(t)$ if no confusion is caused. By the definition of $\hat{f}_{I,k}^{-1}$, we have
\begin{align}
        \frac{\mathrm d}{\mathrm d t}\hat{f}_{I,k}(t)&=\frac{1}{(\hat{f}_{I,k}^{-1})^\prime( \hat f_{I,k}(t) )}=\frac{1}{\frac{1}{Nh_d}\sum_{i=1}^N K_d\left(\frac{\tilde f_k(i/N)-\tilde f_k(\hat{s}_k)}{h_d}\right)}.\notag\\
\end{align}
Observing the following decomposition for uniformly $s\in[h_r,1-h_r]$,
\begin{align}
    &\frac{1}{Nh_d}\sum_{i=1}^N K_d\left( \frac{\tilde{f}_k(i/N)-\tilde{f}_k(s) }{h_d}\right)\notag\\
    =&\frac{1}{Nh_d}\sum_{i=1}^N K_d\left( \frac{f_k(i/N)-f_k(s) }{h_d}\right)+V_k(s),  \notag\\
    =&\int_0^1 K_d\left(\frac{f_k(u)-f_k(s)}{h_d}\right)\frac{\mathrm{d}u}{h_d}+V_k(s)+I_1,\notag\\
    =&(f_k^{-1})^\prime(f_k(s))+B_k(s)+V_k(s)+I_2+I_1,\notag
\end{align}
where
\begin{align}
    B_k(s)&=h_d^2(f_k^{-1})^{'''}(f_k(s) )\int K_d(x)x^2\mathrm{d}x,\notag\\
    V_k(s)&=\frac{1}{Nh_d}\sum_{i=1}^N \left[K_d\left( \frac{\tilde{f}_k(i/N)-\tilde{f}_k(s) }{h_d}\right)-K_d\left( \frac{f_k(i/N)-f_k(s) }{h_d}\right)\right],\notag\\
    I_1&=O(\frac{1}{N}),\quad I_2=o\left(B_k(s)\right).\label{eq:decompose of 1/f_I'(t)}
\end{align}
Denote $r_k(s)=B_k(s)+V_k(s)+I_2+I_1$, we can write
\begin{align}
        \frac{\mathrm d}{\mathrm d t}\hat{f}_{I,k}(t)&=\frac{1}{(\hat{f}_{I,k}^{-1})^\prime( \hat f_{I,k}(t) )}=\frac{1}{\frac{1}{Nh_d}\sum_{i=1}^N K_d\left(\frac{\tilde f_k(i/N)-\tilde f_k(\hat{s}_k)}{h_d}\right)},\notag\\
        &=\frac{1}{(f_k^{-1})^\prime (f_k(\hat s_k))+r_k(\hat s_k) },\notag\\
        &=\frac{f_k^\prime(\hat{s}_k)}{1+f_k^\prime(\hat{s}_k)r_k(\hat s_k) }.\label{eq:derivative compute}
\end{align}
}

\TPM{
We divide the proof into two cases based on the behavior of $m_k^\prime(\hat{s}_k)$. 
\begin{itemize}
    \item \textbf{Case 1:} $\hat{s}_k(t)\in \mathcal{I}_{k,n}$ where $\mathcal{I}_{k,n}=:\{s: m_k^\prime(s)\geq C_1 \}\cap[h_r,1-h_r]$. 
    \item \textbf{Case 2:} $\hat{s}_k(t)\in \mathcal{I}_{k,n}^c$ where $\mathcal{I}_{k,n}^c=:\{s:  m_k^\prime(s)<  C_1 \}\cap[h_r,1-h_r]$. 
    \begin{itemize}
        \item \textbf{Case 2.1:} $\hat{s}_k(t)\in \mathcal{I}_{k,n}^c(1)$ where $\mathcal{I}_{k,n}^c(1)=:\mathcal{I}_{k,n}^c\cap\{s:m_k^\prime(s)=0\}$;
        \item \textbf{Case 2.2:} $\hat{s}_k(t)\in \mathcal{I}_{k,n}^c(2)$ where $\mathcal{I}_{k,n}^c(2)=:\mathcal{I}_{k,n}^c\cap\{s: m_k^\prime(s)\geq  C_1/\delta_n\}$ where positive $\delta_n=1\sqrt{C_1r_n}$ diverges and $r_n=h_d^2C_3/C_1^4+\log^{3}(n)/(C_1h_d\sqrt{nh_r})$.
        \item \textbf{Case 2.3:} $\hat{s}_k(t)\in \mathcal{I}_{k,n}^c(3)$ and $\mathcal{I}_{k,n}(3)=:\mathcal{I}_{k,n}^c\cap\{s:0<m_k^\prime(s)< C_1/\delta_n\}$.
    \end{itemize}
\end{itemize}
Notice that \textbf{Case 2}, or \textbf{Case 2.1, 2.2, 2.3} might not always happen in the sense that the corresponding interval is empty; this is not an issue since at least one of \textbf{Case 1}  and subcases of \textbf{Case 2} will happen. The behavior of \eqref{eq:derivative compute} can be decomposed as
\begin{align}
    &\mathbb{P}\left(\min_{1\leq k\leq p}\inf_{t\in\hat{\mathcal{T}}} \left[\frac{f_k^\prime(\hat{s}_k(t))}{1+f_k^\prime(\hat{s}_k(t))r_k(\hat{s}_k(t))}-\lambda^\prime(\hat{s}_k(t))\right]\geq 0\right)\notag\\
 &\geq \mathbb{P}\left(\min_{1\leq k\leq p}\inf_{s_k\in\mathcal{I}_{k,n}} \left[\frac{f_k^\prime(s_k)}{1+f_k^\prime(s_k)r_k(s_k)}-\lambda^\prime(s_k)\right]\geq 0\right)\notag\\
 &+\mathbb{P}\left(\min_{1\leq k\leq p}\inf_{s_k\in\mathcal{I}_{k,n}^c(1)} \left[\frac{f_k^\prime(s_k)}{1+f_k^\prime(s_k)r_k(s_k)}-\lambda^\prime(s_k)\right]\geq 0\right)\notag\\
 &+\mathbb{P}\left(\min_{1\leq k\leq p}\inf_{s_k\in\mathcal{I}_{k,n}^c(2)} \left[\frac{f_k^\prime(s_k)}{1+f_k^\prime(s_k)r_k(s_k)}-\lambda^\prime(s_k)\right]\geq 0\right)\notag\\
 &+\mathbb{P}\left(\min_{1\leq k\leq p}\inf_{s_k\in\mathcal{I}_{k,n}^c(3)} \left[\frac{f_k^\prime(s_k)}{1+f_k^\prime(s_k)r_k(s_k)}-\lambda^\prime(s_k)\right]\geq 0\right)-3,\label{eq:Case derivative decomp}
\end{align}
where we adopt the convention that infimum $\inf\{A\}=0$ if $A$ is an empty set $A$.
}

\TPM{\paragraph{Case 1:}
For any $s_k\in\mathcal{I}_{k,n}$, by \eqref{eq:decompose of 1/f_I'(t)} and Lemma \ref{lem:derivative variance},
\begin{align}
    \frac{f_k^\prime(s_k)}{1+f_k^\prime(s_k)r_k(s_k) }-\lambda^\prime(s_k)&=\frac{m_k^\prime(s_k)+O_p(f_k^\prime(s_k)r_k( s_k)\lambda^\prime(s_k))
    }{1+f_k^\prime(s_k)r_k( s_k) },\notag\\
    &\geq\frac{m_k^\prime(s_k)+o_p(C_1)
    }{2},
\end{align}
where $o_p(\cdot)$ is uniformly with $k=1,\dots,p$. Note that $m_k^\prime(s_k)\geq C_1$ under \textbf{Case 1}, thus we have 
\begin{equation}
    \mathbb{P}\left(\min_{1\leq k\leq p}\inf_{s_k\in\mathcal{I}_{k,n}} \left[\frac{f_k^\prime(s_k)}{1+f_k^\prime(s_k)r_k(s_k)}-\lambda^\prime(s_k)\right]\geq \frac{C_1}{4}\right)=1-o(1).\label{eq:lower derivative for case 1}
\end{equation}
}

\TPM{
\paragraph{Case 2.1:}
For any $s_k^*\in\mathcal{I}_{k,n}^c(1)$, we have $m_k^\prime(s_k^*)=0$. By Assumption \hyperref[(A2')]{(A2')} and Lemma \ref{lem:0 derivative}, $m_k^\prime(s_k^*)=0=m_k^{''}(s_k^*)$ and $m_k^{'''}(s_k^*)\geq 0$. Note that 
\begin{align}
    (f_k^{-1})^{'''}(f_k(s_k^*) )&=-\frac{f_k^{'''}(s_k^*)f_k^{'}(s_k^*)-3(f_k^{''}(s_k^*))^2}{(f_k^\prime(s_k^*))^5},\notag\\
    &\leq -\frac{\lambda^{'''}(s_k^*)\lambda^{'}(s_k^*)-3(\lambda^{''}(s_k^*))^2}{(\lambda^\prime(s_k^*))^5}.\label{eq:negative bias}
\end{align}
Recall $\lambda(t)=C_1t+C_2t^2+C_3t^3$, $\min\{C_1,C_2,C_3\}>0$, $\max\{C_1,C_2,C_3\}\rightarrow0$, and $C_1/ C_2\to \infty$, $C_2\asymp C_3$, we then have
\begin{align}
    \lambda^{'''}(s_k^*)\lambda^{'}(s_k^*)-3(\lambda^{''}(s_k^*))^2&= 6C_3(C_1+2C_2s_k^*+3C_3(s_k^*)^2)-3(2C_2+6C_3s_k^*)^2 \notag\\
    &\gtrsim C_1C_3>0.\label{eq:negative bias lambda}
\end{align}
Therefore, we have $-B_k(s_k^*)\gtrsim h_d^2C_3/C_1^4>0$ for any $s_k^*,1\leq k\leq p$, s.t. $m_k^\prime(s_k^*)=0$.  By the condition $ C_1^3/C_3=o(h_d^3\sqrt{nh_r}/\log^3 n)$ and Lemma \ref{lem:derivative variance}, we have 
\begin{equation}
      \max_{1\leq k\leq p}\sup_{s_k^*\in\mathcal{I}_{k,n}^c(2)}V_k(s_k^*)=O_p\left(\frac{\log^3 n}{C_1h_d\sqrt{nh_r}}\right)=o_p(h_d^2C_3/C_1^4)\label{eq:penal condition}.
\end{equation}
Combining \eqref{eq:penal condition} and \eqref{eq:decompose of 1/f_I'(t)}, we have $r_k(s_k^*)<0$ with probability tending to $1$ for any $s_k^*\in\mathcal{I}_{k,n}^c(1)$ for all $1\leq k\leq p$, which yields
\begin{equation}
    \mathbb{P}\left(\min_{1\leq k\leq p}\inf_{s_k^*\in\mathcal{I}_{k,n}^c(1)} \left[\frac{f_k^\prime(s_k^*)}{1+f_k^\prime(s_k^*)r_k(s_k^*)}-\lambda^\prime(s_k^*)\right]\geq 0\right)=1-o(1).\label{eq:lower derivative for case 2.1}
\end{equation}
}

\TPM{\paragraph{Case 2.2:} For any $s_k\in\mathcal{I}_{k,n}^c(2)$, we have $m_k^\prime(s_k)\geq C_1/\delta_n$ and $m_k^\prime(s_k)\leq  C_1$. By Lemma \ref{lem:derivative variance} and \eqref{eq:decompose of 1/f_I'(t)}, we have $\max_{1\leq k\leq p}\sup_{s_k\in\mathcal{I}_{k,n}^c(2)}|f_k^\prime(s_k)r_k(s_k)|=O_p(C_1r_n)=o_p(1)$.
Recalling \eqref{eq:derivative compute}, we have for sufficiently large $n$,
\begin{equation}
    \frac{f_k^\prime(s_k)}{1+f_k^\prime(s_k)r_k(s_k)}-\lambda^\prime(s_k)=\frac{m_k^\prime(s_k)-\lambda^\prime(s_k)f_k^\prime(s_k)r_k(s_k)}{1+o_p(1)}\geq \frac{C_1(1/\delta_n-f_k^\prime(s_k)r_k(s_k)) )}{2}.\label{eq:lower derivative for case 2.2 delta}
\end{equation}
 Note that $\delta_n=1/\sqrt{C_1r_n}$, then \eqref{eq:lower derivative for case 2.2 delta} yields
\begin{equation}
\mathbb{P}\left(\min_{1\leq k\leq p}\inf_{s_k\in\mathcal{I}_{k,n}^c(2)} \left[\frac{f_k^\prime(s_k)}{1+f_k^\prime(s_k)r_k(s_k)}-\lambda^\prime(s_k)\right]\geq 0\right)=1-o(1).\label{eq:lower derivative for case 2.2}
\end{equation}
}  


\TPM{\paragraph{Case 2.3} For all $s_k\in\mathcal{I}_{k,n}^c(3)$, we have $0<m_k^\prime(s_k)=o(  C_1)$. By Taylor's expansion and condition (2) of \hyperref[(A2')]{(A2')}, there exists $s_k^*$ staisfying s.t. $|s_k-s_k^*|=o(1)$ and the $3$-th derivative  $m_k^{'''}(s_{0,k})\geq \eta$. By Lemma \ref{lem:0 derivative} and recall $\kappa_2(K_d)=\int x^2 K_d(x)\mathrm{d}x$, we can decompose $B_k(s_k)-B_k(s_k^*)$ as
\begin{align}
    B_k(s_k)-B_k(s_k^*)&=h_d^2\kappa_2(K_d) \left[ (f_k^{-1})^{'''}(s_k)-(f_k^{-1})^{'''}(s_k^*) \right],\notag\\
    &=h_d^2\kappa_2(K_d)\left[ \frac{f_k^{'''}(s_k^*)f_k^{'}(s_k^*)-3(f_k^{''}(s_k^*))^2}{(f_k^\prime(s_k^*))^5} -\frac{f_k^{'''}(s_k)f_k^{'}(s_k)-3(f_k^{''}(s_k))^2}{(f_k^\prime(s_k))^5} \right],\notag\\
    &=h_d^2\kappa_2(K_d)\left( II_1+II_2 \right),\label{eq:decomp B(s)}
\end{align}
where
\begin{align}
    II_1&=\frac{f^{'''}(s_k^*)f^{'}(s_k^*)-3(f^{''}(s_k^*))^2}{(f^\prime(s_k^*))^5}-\frac{f^{'''}(s_k^*)f^{'}(s_k^*)-3(f^{''}(s_k^*))^2}{(f^\prime(s_k))^5},\notag\\
    II_2&= \frac{f^{'''}(s_k^*)f^{'}(s_k^*)-3(f^{''}(s_k^*))^2}{(f^\prime(s_k))^5}-\frac{f^{'''}( s_k)f^{'}(s_k)-3(f^{''}(s_k))^2}{(f^\prime(s_k))^5}.
\end{align}
Using the fact $m_k^\prime(s_k^*)=0$,  $m_k^\prime(s_k)=o(C_1)$ in Case 2.3, we have
\begin{equation}
    \frac{f^\prime(s_k^*)}{f^\prime(s_k)}-1=\frac{\lambda^\prime(s_k^*)}{m_k^\prime(s_k)+ \lambda^\prime(s_k)}-1=o_p(1),\label{eq:rate of II_1}
\end{equation}
which yields $h_d^2\kappa_2(K_d)II_1=B_k(s_k^*) \left[1-(f^\prime(s_k^*)/f^\prime(s_k))^5 \right]=o\left(B_k(s_k^*)\right).$
}

\TPM{For $II_2$, by the condition (2) of \hyperref[(A2')]{(A2')}, i.e., $m_k^{'''}(s_k^*)\geq \eta>0$ then 
by the fact $m_k^{^\prime}(s_k)-m_k^\prime(s_k^*)=o(C_1)$ and the Taylor's expansion, we have $(s_k-s_k^*)^2=o(C_1)$. Therefore,
\begin{align}
    |f^{'''}(s_k^*)f^\prime(s_k^*)-f^{'''}(s_k)f^\prime(s_k)|&\leq|f_k^{'''}(s_k^*)[f_k^\prime(s_k^*)-f_k^\prime(s_k)]|+|f_k^{'''}(s_k^*)-f_k^{'''}(s_k)|f_k^\prime( s_k)  \notag\\
    &=o(C_1 );\label{eq:rate of II_2 1'}\\
    |(f^{''}(s_k^*))^2-(f^{''}(s_k))^2|&\leq|f^{''}(s_k^*)-f^{''}(s_k)|\cdot|f^{''}(s_k^*)+f^{''}(s_k)|\notag\\
    &=|f^{''}(s_k^*)-f^{''}(s_k)|^2+2|f^{''}(s_k^*)-f^{''}(s_k)||f_k^{''}(s_k^*)|\notag\\
    &=o(C_1 ).\label{eq:rate of II_2 2'}
\end{align}
Note that
\begin{align}II_2h_d^2\kappa_2(K_d)=B_k(s_k^*)\left(\frac{f_k^\prime(s_k^*)}{f_k^\prime(s_k)}\right)^5\left\{\frac{[f^{'''}(s_k^*)f^{'}(s_k^*)-3(f^{''}(s_k^*))^2]-[f^{'''}(\hat s_k)f^{'}(s_k)-3(f^{''}(s_k))^2]}{f_k^{'''}(s_k^*)f_k^\prime(s_k^*)-3(f^{''}(s_k^*))^2}\right\},\label{eq:II_2hdk}
\end{align}
then combining \eqref{eq:rate of II_1}, \eqref{eq:rate of II_2 1'} and \eqref{eq:rate of II_2 2'}, and using the fact $f_k^{'''}(s_k^*)f_k^\prime(s_k^*)- (f^{''}(s_k^*))^2\asymp C_1 m_k^{'''}(s_k^*)$ if $m_k^{'''}(s_k^*)\geq \eta$, straight forward calculations show that $h_d^2\kappa_2(K_d)II_2=B_k(s_k^*)\frac{o(C_1)}{C_1m_k^{'''}(s_k^*)}=o(B_k(s_k^*))$. Therefore $B_k(s_k)=B_k(s_k^*)+o(B_k(s_k^*))$ and using the conclusion $-B_k(s_k^*)\gtrsim h_d^2 C_3/C_1^4>0$ under Case 2.1, $-B_k(s_k)\gtrsim h_d^2 C_3/C_1^4>0$.
}

\TPM{By similar arguments in \eqref{eq:penal condition}, with condition $ C_1^3/C_3=o(h_d^3\sqrt{nh_r}/\log^3 n)$ and Lemma \ref{lem:derivative variance}, we can also have 
\begin{equation}
      \max_{1\leq k\leq p}\sup_{s_k\in\mathcal{I}{k,n}^c(3)}V_k(s_k)=O_p\left(\frac{\log^3 n}{C_1h_d\sqrt{nh_r}}\right)=o_p(h_d^2C_3/C_1^4)\label{eq:penal condition case 2.3}.
\end{equation}
Combining \eqref{eq:penal condition case 2.3} and \eqref{eq:decompose of 1/f_I'(t)}, we have $r_k(s_k)<0$ with probability tending to $1$ for any $s_k\in\mathcal{I}_{k,n}^c(3)$ for all $1\leq k\leq p$, which yields
\begin{equation}
    \mathbb{P}\left(\min_{1\leq k\leq p}\inf_{s_k\in\mathcal{I}_{k,n}^c(3)} \left[\frac{f_k^\prime(s_k)}{1+f_k^\prime(s_k)r_k(s_k)}-\lambda^\prime(s_k)\right]\geq 0\right)=1-o(1).\label{eq:lower derivative for case 2.3}
\end{equation} 
}

\TPM{    
To sum up the discussions under \textbf{Case 1} and \textbf{Case 2}, combining \eqref{eq:Case derivative decomp}, \eqref{eq:lower derivative for case 1}, \eqref{eq:lower derivative for case 2.1},\eqref{eq:lower derivative for case 2.2}, and \eqref{eq:lower derivative for case 2.3}, we have
\begin{align}
 &\mathbb{P}\left(\min_{1\leq k\leq p}\inf_{t\in\hat{\mathcal{T}}} \left[\frac{f_k^\prime(\hat{s}_k(t))}{1+f_k^\prime(\hat{s}_k(t))r_k(\hat{s}_k(t))}-\lambda^\prime(\hat{s}_k(t))\right]\geq 0\right)=1-o(1).\notag
\end{align}
Then by Lemma \ref{lem:s-t}, we finally have for some constant $M_0\in(0,1/2)$,
\begin{align}
    &\mathbb{P}\left(\min_{1\leq k\leq p}\inf_{t\in\hat{\mathcal{T}}} \left[\frac{\mathrm{d}}{dt}\hat{f}_{I,k}(t)-\lambda^\prime(t)+g_{n,1}C_2\right]\geq 0  \right) ,\notag\\ 
    \geq &\mathbb{P}\left(\min_{1\leq k\leq p}\inf_{t\in\hat{\mathcal{T}}} \left[\frac{f_k^\prime(\hat{s}_k(t))}{1+f_k^\prime(\hat{s}_k(t))r_k(\hat{s}_k(t))}-\lambda^\prime(\hat{s}_k(t))\right]\geq \max_{1\leq k\leq p}\sup_{t\in\hat{\mathcal{T}}}|\lambda^\prime(t)-\lambda^\prime(\hat{s}_k(t))|-g_{n,1}C_2  \right),\notag\\
    \geq&  \mathbb{P}\left(\min_{1\leq k\leq p}\inf_{t\in\hat{\mathcal{T}}} \left[\frac{f_k^\prime(\hat{s}_k(t))}{1+f_k^\prime(\hat{s}_k(t))r_k(\hat{s}_k(t))}-\lambda^\prime(\hat{s}_k(t))\right]\geq 0,\quad \max_{1\leq k\leq p}\sup_{t\in\hat{\mathcal{T}}}|t-\hat{s}_k(t)|\leq M_0g_{n,1} \right),\notag\\
    \geq&  \mathbb{P}\left(\min_{1\leq k\leq p}\inf_{t\in\hat{\mathcal{T}}} \left[\frac{f_k^\prime(\hat{s}_k(t))}{1+f_k^\prime(\hat{s}_k(t))r_k(\hat{s}_k(t))}-\lambda^\prime(\hat{s}_k(t))\right]\geq 0\right)+\mathbb{P}\left( \max_{1\leq k\leq p}\sup_{t\in\hat{\mathcal{T}}}|t-\hat{s}_k(t)|\leq M_0 g_{n,1} \right)-1\notag\\
    =&1-o(1).
\end{align}
Notice that $\hat{m}_{correct,k}^\prime(t)=\frac{\mathrm{d}}{dt}\hat{f}_{I,k}(t)-(\lambda^\prime(t)-g_{n,1}C_2)$, \eqref{eq:monotone shape} holds.
}


\end{proof}
\TPM{
\begin{lem}   
\label{lem:s-t}
    Under same conditions in (i) and (ii) of Proposition \ref{prop:penalization}, for any $t\in\hat{\mathcal{T}}$, there exists $\hat{s}_k(t)\in[h_r,1-h_r],1\leq k\leq p$ s.t. $\tilde{f}_k(\hat{s}_k(t))=\hat{f}_{I,k}(t)$ and
    \begin{equation}
        \max_{1\leq k\leq p}\sup_{t\in\hat{\mathcal{T}}}|\hat s_k(t)-t|=O_p\left(\frac{\log^3 n}{h_d\sqrt{nh_r}}+\frac{h_d^2}{C_1^{5/2}}\right).\label{eq:rate of s-t}
    \end{equation}
\end{lem}  
}

\begin{proof}[proof.]
 \TPM{By the definition of $\hat{f}^{-1}_{I,k}(t)$ and $\hat{\mathcal{T}}$ in \eqref{index_I}, $\{\hat{f}_{I,k}(t):t\in\hat{\mathcal{T}}\}$ is a subset of the range $\{\tilde{f}_k(t),t\in[h_r,1-h_r]\}$. Since local linear estimate $\tilde f_k(t)$ is smooth, there exists a random $\hat s_k(t)\in[h_r,1-h_r]$ s.t. $\tilde{f}_k(\hat s_k(t))=\hat{f}_{I,k}(t)$ and
\begin{align}
    t&=\frac{1}{Nh_d}\sum_{i=1}^N \int_{-\infty}^{\hat f_{I,k}(t)}K_d\left(\frac{\tilde{f}_k(i/N)-u}{h_d}\right)\mathrm{d}u= \frac{1}{Nh_d}\sum_{i=1}^N \int_{-\infty}^{\tilde f_k(\hat s_k(t))}K_d\left(\frac{\tilde{f}_k(i/N)-u}{h_d}\right)\mathrm{d}u,\notag\\
    \hat s_k(t)&=f_k^{-1}(f_k(\hat s_k(t)))=\frac{1}{Nh_d}\sum_{i=1}^N \int_{-\infty}^{f_k(\hat s_k(t))}K_d\left(\frac{f_k(i/N)-u}{h_d}\right)\mathrm{d}u+f_k^{-1}(f_k(\hat s_k(t)))-f_{N,k}^{-1}(f_k(\hat s_k(t))),
\end{align}
where $f_{N,k}^{-1}(t)$ is defined as
$$
f_{N,k}^{-1}(t)=:\frac{1}{Nh_d}\sum_{i=1}^N\int_{-\infty}^t K_d\left(\frac{f_k(i/N)-u}{h_d}\right)\mathrm{d}u.
$$
For brevity, we write $\hat{s}_k$ instead of $\hat{s}_k(t)$ if no confusion is made. For any $k=1,\dots,p$, when $f_k(\hat{s}_k)\geq \tilde{f}_k(\hat s_k)$,
\begin{align}
    |t-\hat{s}_k|&\leq\left|\int_{\tilde f_k(\hat{s}_k)}^{ f_k(\hat s_k)} \frac{1}{N}\sum_{i=1}^N K_d\left(\frac{f_k(i/N)-u}{h_d}\right)\frac{\mathrm{d}u}{h_d}\right|+\left|f_k^{-1}(f_k(\hat s_k))-f_{N,k}^{-1}(f_k(\hat s_k))\right|,
     \notag\\
    &+\left|\hat f_{I,k}^{-1}(\tilde{f}_k(\hat{s}_k))-f_{N,k}^{-1}(\tilde{f}_k(\hat{s}_k))\right|\notag\\
    &\leq |f_k(\hat{s}_k)-f_k(\hat s_k)|\sup_{\tilde{f}_k(\hat{s}_k)\leq u\leq f_k(\hat{s}_k)}\left|\int_0^1K_d\left(\frac{f_k(v)-u}{h_d}\right)\frac{\mathrm{d}v}{h_d}\right|+\left|f_k^{-1}(f_k(\hat s_k))-f_{N,k}^{-1}(f_k(\hat s_k))\right|\notag\\
    &+\left|\hat f_{I,k}^{-1}(\tilde{f}_k(\hat{s}_k))-f_{N,k}^{-1}(\tilde{f}_k(\hat{s}_k))\right|.\label{eq:s-t decomp}
\end{align}
By Lemma \ref{local linear stochastic expansion}, and with setting $\lambda=\log n$ in \eqref{eq:nonstationary rate in lem:High-dim ll} of Lemma \ref{lem:High-dim ll}, and $nh_r^7=o(1)$ in Assumption \hyperref[(C2)]{(C2)}, we have
\begin{align}
    &|\tilde f_k(\hat{s}_k)-f_k(\hat s_k)|\sup_{\tilde{f}_k(\hat{s}_k)\leq u\leq f_k(\hat{s}_k)}\left|\int_0^1K_d\left(\frac{f_k(v)-u}{h_d}\right)\frac{\mathrm{d}v}{h_d}\right|=O_p\left( \frac{\log^3 n}{h_d\sqrt{nh_r}} \right).\label{eq:s-t decomp 1}
\end{align}
where $O_p(\cdot)$ are uniform for all $1\leq k\leq p$ and $\hat{s}_k\in[h_r,1-h_r]$. By \eqref{eq:Rieman approx}, \eqref{eq:Ahd def} and \eqref{eq:m_N.inv Ahd(t)} in Lemma \ref{lem:m_N}, for some $|\theta|<1$,
\begin{align}
    &\left|f_k^{-1}(f_k(\hat s_k))-f_{N,k}^{-1}(f_k(\hat s_k))\right|=O(\frac{1}{Nh_d})+|\frac{1}{2}\kappa_2(K_d)h_d^2 (f_k^{-1})^{''}(f(\hat s_k+\theta h_d))|.\label{eq:s-t decomp 2 part}
\end{align}
For $|(f_k^{-1})^{''}(f_k(t))|,\forall t\in(0,1)$, note that
\begin{equation}
    |(f_k^{-1})^{''}(f_k(t))|=\left|\frac{f_k^{''}(t)}{(f_k^\prime(t))^3}\right|=\left|\frac{m_k^{''}(t)+\lambda^{''}(t)}{(m_k^\prime(t)+\lambda^\prime(t))^3}\right|\asymp \left| \frac{m_k^{''}(t)+C_3}{(m_k^\prime(t)+C_1)^3} \right|.\label{eq:f.inv '' bound}
\end{equation}
Similar with our discussions in the proof of Proposition \ref{prop:penalization}, we bound $\max_{1\leq k\leq p}\sup_{s_k\in(0,1)}|(f_k^{-1})^{''}(s_k)|$ by dividing the proof into 2 cases.
\begin{itemize}
    \item \textbf{Case 1}: $s_k\in\mathcal{I}_k$ where $\mathcal{I}_k=:\{s:m_k^{'}(s)\geq \log^{-1} n\}$;
    \item \textbf{Case 2}: $s_k\in\mathcal{I}_k^c$ where $\mathcal{I}_k^c=:\{s:m_k^{'}(s)< \log^{-1} n\}$;
    \begin{itemize}
        \item \textbf{Case 2.1}: $s_k\in\mathcal{I}_k^c(1)$ where $\mathcal{I}_k^c(1)=:\mathcal{I}_k^c\cap\{s:m_k^{''}(s)\leq C_3\}$;
        \item \textbf{Case 2.2}: $s_k\in\mathcal{I}_k^c(2)$ where $\mathcal{I}_k^c(2)=:\mathcal{I}_k^c\cap\{s:m_k^{''}(s)> C_3\}$;
    \end{itemize}
\end{itemize}
}

\TPM{
\textbf{Case 1}: When $m_k^\prime(t)\geq \log^{-1} n$, it immediately yields $|(f^{-1})^{''}(f_k(t))|\lesssim 1/|m_k^\prime(t)|^3 \lesssim\log^3 n=O(C_1^{-5/2})$.
}

\TPM{
\textbf{Case 2.1}: When $m_k^{''}(t)\leq C_3$, by \eqref{eq:f.inv '' bound},
\begin{equation}
    |(f_k^{-1})^{''}(f_k(t))|=O\left(\frac{C_3}{C_1^3}\right).\label{eq:f.inv '' bound 2.1}
\end{equation}
}

\TPM{
\textbf{Case 2.2}: When $m_k^{''}(s_k)> C_3$ and $m_k^{'}(s_k)=o(1)$, by Taylor's expansion, there exists $s_k^*$ staisfying \hyperref[(A2')]{(A2')} s.t. $|s_k-s_k^*|=o(1)$ and the $3$-th derivative  $m_k^{'''}(s_{0,k})\geq \eta$. By Lemma \ref{lem:0 derivative} and Taylor's expansion on $s_k^*$,
\begin{align}
   m_k^{''}(s_k)\asymp |s_k-s_k^*|,\quad m_k^\prime(s_k)\asymp (s_k-s_k^*)^{2}.\notag
\end{align}
By $m_k^{''}(s_k)>C_3$,  $m_k^{''}(s_k)\gtrsim C_3^{h_k}$ where $h_k\in(0,1]$ and
we have $|s_k-s_k^*|\gtrsim C_3^{h_k}$, thus $m_k^\prime(s_k)\gtrsim C_3^{2h_k}$. By \eqref{eq:f.inv '' bound}, write $C_3=C_1^s$, then 
\begin{align}
    |(f_k^{-1})^{''}(f_k(t))|&\lesssim \frac{C_3^{h_k}}{ \left(\max\left\{ C_1, C_3^{2h_k}\right\}\right)^3}\notag\\
    &\lesssim C_1^{sh_k-3\min\left\{ 2sh_k,1\right\}}.
\end{align}
If $1<2sh_k$, then $sh_k-1>-\frac{1}{2}$,
\begin{equation}
|(f_k^{-1})^{''}(f_k(t))|=O(C_1^{-2+sh_k-1})=O(C_1^{-5/2}).
\end{equation}
If $1\geq 2sh_k$, then 
\begin{equation}
|(f_k^{-1})^{''}(f_k(t))|=O(C_1^{-5sh_k})=O(C_1^{-5/2}).
\end{equation}
Summarizing all the discussions under \textbf{Case 1} and \textbf{Case 2}, combining \eqref{eq:s-t decomp 2 part}, we have
\begin{align}
    &\left|f_k^{-1}(f_k(\hat s_k))-f_{N,k}^{-1}(f_k(\hat s_k))\right|=O\left(\frac{h_d^2}{C_1^{5/2}}\right)\label{eq:s-t decomp 2}
\end{align}
}

\TPM{By Taylor's expansion, for some $\xi_{i,k}$ between $\tilde{f}_k(i/N)$ and $f_k(i/N)$, 
\begin{align}
    \hat f_{I,k}^{-1}(t)-f_{N,k}^{-1}(t)&=\frac{1}{N h_{d}} \sum_{i=1}^{N}(\tilde{f}_k(i / N)-f_k(i / N))\int_{-\infty}^t K_{d}^\prime\left(\frac{f_k(i/N)-u}{h_{d}}\right)\frac{\mathrm{d}u}{h_d}\notag\\
    &+ \frac{1}{2N h_{d}^2} \sum_{i=1}^{N}(\tilde{f}_k(i / N)-f_k(i / N))^2\int_{-\infty}^t K_{d}^{''}\left(\frac{\xi_{i,k}-u}{h_{d}}\right)\frac{\mathrm{d}u}{h_d},\notag\\
    &=\frac{1}{N h_{d}} \sum_{i=1}^{N}(\tilde{f}_k(i / N)-f_k(i / N)) K_{d}\left(\frac{f_k(i/N)-t}{h_{d}}\right)\notag\\
    &+ \frac{1}{2N h_{d}^2} \sum_{i=1}^{N}(\tilde{f}_k(i / N)-f_k(i / N))^2 K_{d}^{'}\left(\frac{\xi_{i,k}-t}{h_{d}}\right).\label{eq:taylor f_I}
\end{align}
Then using the fact that $K_d(x)$ and $K_d^\prime(x)$ are bounded, by condition $nh_r^7=o(1)$ and $R_n/h_d=o(1)$ in Assumption \hyperref[(C2)]{(C2)}, Lemma \ref{local linear stochastic expansion}, and with setting $\lambda=\log n$ in \eqref{eq:nonstationary rate in lem:High-dim ll} of Lemma \ref{lem:High-dim ll},
\begin{align}
    &\left|\hat f_{I,k}^{-1}(t)-f_{N,k}^{-1}(t)\right|\notag\\
    \lesssim&\frac{1}{h_d}\max_{1\leq k\leq p}\left(\sup_{u\in [h_r,1-h_r]}|\tilde{f}_k(u)-f_k(u)|+h_r \sup_{u\in [0,h_r)\cup(1-h_r,1]}|\tilde{f}_k(u)-f_k(u)|\right),\notag\\
    =&O_p\left( \frac{\log^3 n}{h_d\sqrt{nh_r}} \right).\label{eq:s-t decomp 3}
\end{align}
Combining \eqref{eq:s-t decomp}, \eqref{eq:s-t decomp 1}, \eqref{eq:s-t decomp 2}, \eqref{eq:s-t decomp 3}, we finally have
$$
\max_{1\leq k\leq p}\sup_{t\in\hat{\mathcal{T}}}|t-\hat{s}_k(t)|=O_p\left(\frac{\log^3 n}{h_d\sqrt{nh_r}}+\frac{h_d^2}{C_1^{5/2}}\right).
$$
Similar arguments can apply to the case $f_k(\hat{s}_k)<\tilde f_k(\hat s_k)$, therefore, \eqref{eq:rate of s-t} holds.
}

\end{proof}
\TPM{
\begin{lem}
\label{lem:0 derivative}
    For monotone increasing trend function $m_k(t),k=1,\dots,p$, satisfying Assumption \hyperref[(A1')]{(A1')}, if $m_k^\prime(s_0)=0,s_0\in(0,1)$, then $m_k^{''}(s_0)=0$ and $m_k^{'''}(s_0)\geq 0$. 
\end{lem}
}

\begin{proof}[Proof.]
\TPM{If $m_k^{''}(s_0)>0$, $\exists \delta>0$ s.t. $m_k^{''}(s)>0$, $\forall s\in[s_0-\delta,s_0+\delta]$. Then 
$$
0>m_k^\prime(s_0)-\int_{s_0-\delta}^{s_0}m_k^{''}(s)\mathrm d s=m_k^\prime(s_0-\delta),
$$
which contradicts to monotone condition $m_k^\prime(t)\geq 0$. }

\TPM{
Similarly, if $m_k^{''}(s_0)<0$, $\exists \delta>0$ s.t. $m_k^{''}(s)<0$, $\forall s\in[s_0-\delta,s_0+\delta]$. Then 
$$
0>\int^{s_0+\delta}_{s_0}m_k^{''}(s)\mathrm d s+m_k^\prime(s_0)=m_k^\prime(s_0+\delta),
$$
which contradicts to monotone condition $m_k^\prime(t)\geq 0$. Therefore, $m_k^{''}(s_0)=0$.
}

\TPM{
If $m_k^{'''}(s_0)<0$, $\exists \delta>0$ s.t. $m_k^{'''}(s)<0$, $\forall s\in[s_0-\delta,s_0+\delta]$. Then $0=m_k^{''}(s_0)<m_k^{''}(s),\forall s\in[s_0-\delta,s_0)$ and
$$
0=m_k^\prime(s_0)=\int_{s_0-\delta}^{s_0} m_k^{''}(s)\mathrm{d}s+m_k^\prime(s_0-\delta)>0+m_k^\prime(s_0-\delta),
$$
which contradicts to monotone condition $m_k^\prime(t)\geq 0$. Therefore, $m_k^{'''}(s_0)\geq 0$.
}
\end{proof}

\TPM{
\begin{lem}
\label{lem:derivative variance}
Under same conditions in (i) and (ii) of Proposition \ref{prop:penalization}, we have
\begin{equation}
\label{eq:derivative variance}
    \max_{1\leq k\leq p}\sup_{s\in[h_r,1-h_r]}\left|\frac{1}{Nh_d}\sum_{i=1}^N \left[K_d\left( \frac{\tilde{f}_k(i/N)-\tilde{f}_k(s) }{h_d}\right)-K_d\left( \frac{f_k(i/N)-f_k(s) }{h_d}\right)\right]\right|=O_p\left( \frac{\log^3 n}{C_1h_d\sqrt{nh_r}} \right).
\end{equation}
\end{lem}
}

\begin{proof}[Proof.]  
\TPM{By Lagrange's mean value, for some $\eta_{i,k}$ between $\tilde{f}_k(i/N)$ and $f_k(i/N)$, we have
\begin{align}
    &\frac{1}{Nh_d}\sum_{i=1}^N \left[K_d\left( \frac{\tilde{f}_k(i/N)-f_k(s) }{h_d}\right)-K_d\left( \frac{f_k(i/N)-f_k(s) }{h_d}\right)\right]\notag\\
    =&\frac{1}{Nh_d^2}\sum_{i=1}^N K_d^\prime\left( \frac{\eta_{i,N}-f_k(s) }{h_d}\right)\left[\tilde{f}_k(i/N)-f_k(i/N)\right]. 
\end{align}
By similar arguments in \eqref{eq:number of Kd'} and \eqref{eq:number bound}, the number of set $\#\{i:|\eta_{i,N}-f_k(s)|\leq h_d\}=O_p(Nh_d/C_1),\forall s\in[h_r,1-h_r]$. Here $\#\{i:|\eta_{i,N}-f_k(s)|\leq h_d\}$ is bounded by $Nh_d/C_1$ instead of $Nh_d/M$ in \eqref{eq:number bound} because $f_k^\prime(t)\geq \lambda^\prime(t)\geq C_1$ instead of $M$. By condition $nh_r^7=o(1)$ and $R_n/h_d=o(1)$ in Assumption \hyperref[(C2)]{(C2)}, Lemma \ref{local linear stochastic expansion}, and with setting $\lambda=\log n$ in \eqref{eq:nonstationary rate in lem:High-dim ll} of Lemma \ref{lem:High-dim ll}
, we then have
\begin{equation}
    \max_{1\leq k\leq p}\sup_{s\in[h_r,1-h_r]}\left|\frac{1}{Nh_d}\sum_{i=1}^N \left[K_d\left( \frac{\tilde{f}_k(i/N)-f_k(s) }{h_d}\right)-K_d\left( \frac{f_k(i/N)-f_k(s) }{h_d}\right)\right]\right|=O_p\left( \frac{\log^3 n}{C_1h_d\sqrt{nh_r}} \right).\label{eq:derivative variance 1}
\end{equation}
Similarly, by Lagrange's, for some $\eta_{k}$ between $\tilde{f}_k(s)$ and $f_k(s)$, then we have
\begin{align}
    &\frac{1}{Nh_d}\sum_{i=1}^N \left[K_d\left( \frac{\tilde{f}_k(i/N)-\tilde f_k(s) }{h_d}\right)-K_d\left( \frac{\tilde{f}_k(i/N)-f_k(s) }{h_d}\right)\right]\notag\\
    =&\frac{1}{Nh_d^2}\sum_{i=1}^N K_d^\prime\left( \frac{\tilde{f}_k(i/N)-\eta_{k} }{h_d}\right)\left[\tilde{f}_k(s)-f_k(s)\right]. 
\end{align}
By Lemma \ref{lem:High-dim ll}, $\max_{1\leq k\leq p}\sup_{1\leq i\leq N}|\tilde{f}_k(i/N)-f_k(i/N)|=O_p(R_n)$ and $R_n/h_d=o(1)$ in Assumption \hyperref[(C2)]{(C2)}, we can also have for any $s\in[h_r,1-h_r]$,
$$
\max_{1\leq k\leq p}\#\{i:|\tilde{f}_k(i/N)-\eta_{k}|\leq h_d\}\leq \max_{1\leq k\leq p}\#\{i:|f_k(i/N)-f_k(s)|\leq 2R_n+h_d\}=O(Nh_d/C_1).
$$
Then by condition $nh_r^7=o(1)$ and $R_n/h_d=o(1)$ in Assumption \hyperref[(C2)]{(C2)}, Lemma \ref{local linear stochastic expansion}, and with setting $\lambda=\log n$ in \eqref{eq:nonstationary rate in lem:High-dim ll} of Lemma \ref{lem:High-dim ll}, we then have
\begin{equation}
    \max_{1\leq k\leq p}\sup_{s\in[h_r,1-h_r]}\left|\frac{1}{Nh_d}\sum_{i=1}^N \left[K_d\left( \frac{\tilde{f}_k(i/N)-\tilde{f}_k(s) }{h_d}\right)-K_d\left( \frac{\tilde{f}_k(i/N)-f_k(s) }{h_d}\right)\right]\right|=O_p\left( \frac{\log^3 n}{C_1h_d\sqrt{nh_r}} \right).\label{eq:derivative variance 2}
\end{equation}
Combining \eqref{eq:derivative variance 1} and \eqref{eq:derivative variance 2}, \eqref{eq:derivative variance} holds.
}
\end{proof}

\subsection{Covariance estimation}
\begin{lem}
    \label{lem:thm5 in mies}
    Let $p$-dimensional nonstationary process $\mf e_i=\mathbf{G}\left(t_i,\Upsilon_i\right),i=1,\dots,n$, with $\mathbb{E}\left(\mf e_i\right)=0$ where $(\mathbf{e}_i),p,n$ satisfy \hyperref[(B1)]{(B1)} and \hyperref[(B2)]{(B2)} with $q\geq 2$. If \hyperref[(B3)]{(B3)} is satisfied, then for some constant $C>0$ only depends on $q$,
    $$
    \mathbb{E}\max_{1\leq k\leq n}\left\| \bar{\mf Q}(k)-\sum_{j=1}^k\mf{\Sigma}_{\mf G}(t_j) \right\|_{\operatorname{tr}}\leq C\Theta^2(p)(\sqrt{npL}+nL^{-1}),
    $$
    where $\mf{\Sigma}_{\mf G}(t)$ is the long-run covariance matrix function of process $\mf G(t,\Upsilon_i)$ and 
    $$
    \bar{\mf Q}(k)=\sum_{i=L}^k \frac{1}{L}\left(\sum_{j=i-L+1}^i \mf e_j\right)\left(\sum_{j=i-L+1}^i \mf e_j\right)^\top.
    $$
\end{lem}
\begin{proof}[Proof of Lemma \ref{lem:thm5 in mies}.]
    The proof of Lemma \ref{lem:thm5 in mies} is a straightforward application of that of Theorem 5 in \cite{Mies2022seq_high-dim}. 
\end{proof}

\begin{lem}
\label{lem:high long-run var}
Under same conditions in Theorem \ref{prop:GA in high-dim}, for $\hat{\mathbf{Q}}(k)$ defined in \eqref{eq:cumulative long-run cov},
\begin{equation}
    \mathbb{E} \max _{k=1, \ldots, n}\left\|\hat{\mathbf{Q}}(k)-\sum_{j=1}^k \mathbf{\Sigma}_{\mathbf{G}}(t_j)\right\|_{t r} \leq  C \varphi_n\label{eq:final bound residual cumulative cov}
\end{equation}
where constant $C>0$ only depends on $q$ in Assumption \hyperref[(B1)]{(B1)} and 
$$
\varphi_n=:\Theta^2(p)(\sqrt{npL}+nL^{-1})+\Theta(p)p^{3/2}L\sqrt{n}\left(h_r^2+\frac{\log^{3}(n)}{\sqrt{nh_r}}\right)+p^{3/2}Ln\left(h_r^5+\frac{\log^6(n)}{nh_r}\right).
$$
\end{lem}
\begin{proof}[Proof of Lemma \ref{lem:high long-run var}.]
Denote 
    $$
    \Delta_{i}=:\left(\sum_{j=i-L+1}^i \hat{\boldsymbol{\varepsilon}}(t_j)\right),\quad\bar{\Delta}_i=:\left(\sum_{j=i-L+1}^i \mathbf{e}_j\right),\quad \bar{\mathbf{Q}}(k)=:\sum_{i=L}^k\frac{1}{L}\bar{\Delta}_i\bar{\Delta}_i^{\top},
    $$ 
    where residual $\hat{\boldsymbol{\varepsilon}}(t_i)=(\hat \varepsilon_1(t_i),\dots,\hat \varepsilon_p(t_i))^{\top}$ and $\hat \varepsilon_k(t_i)=y_{i,k}-\tilde m_k(t_i)$.
    Using the fact $\|uu^{\top}-vv^{\top}\|_{\operatorname{tr}}=\|(u-v)v^\top+v(u-v)^\top+(u-v)(u-v)^{\top}\|_{\operatorname{tr}}$ for any $u,v\in\mathrm{R}^p$,
    \begin{align}
        \left\|\hat{\mathbf{Q}}(k)-\bar{\mathbf{Q}}(k)\right\|_{\operatorname{tr}} 
        \leq \frac{2}{L}\left\|\sum_{i=L}^k \left(\Delta_i-\bar{\Delta}_i\right)\bar{\Delta}_i^{\top}\right\|_{\operatorname{tr}}+\frac{1}{L}\left\|\sum_{i=L}^k\left(\Delta_i-\bar{\Delta}_i\right)\left(\Delta_i-\bar{\Delta}_i\right)^{\top}\right\|_{\operatorname{tr}}\label{eq:bound 1 residual cumulative cov}.
    \end{align}
    Note that
    \begin{align}
        &\left\|\sum_{i=L}^k \left(\Delta_i-\bar{\Delta}_i\right)\bar{\Delta}_i^{\top}\right\|_{\operatorname{tr}}^2\leq p\left\|\sum_{i=L}^k \left(\Delta_i-\bar{\Delta}_i\right)\bar{\Delta}_i^{\top}\right\|_{\operatorname{tr},2}^2,\notag\\
        &=p\operatorname{tr}\left[\left(\sum_{i=L}^k\bar{\Delta}_i\left(\Delta_i-\bar{\Delta}_i\right)^{\top}\right)\cdot\left(\sum_{i=L}^k\left(\Delta_i-\bar{\Delta}_i\right) \bar{\Delta}_i^{\top}\right) \right],\notag\\
        &\leq p^3\max_{1\leq l\leq p}\max_{1\leq s\leq p}\left(  \sum_{i=L}^k \left(\sum_{j=i-L+1}^i (\tilde{m}_s(t_j)-m_s(t_j) )\right)\sum_{j=i-L+1}^ie_{j,l} \right)^2.\label{eq:bound 2 residual cumulative cov}
    \end{align}
    By similar arguments in \eqref{eq:bound 2 residual cumulative cov}, we can also have
    \begin{align}
        &\left\|\sum_{i=L}^k \left(\Delta_i-\bar{\Delta}_i\right)\left(\Delta_i-\bar{\Delta}_i\right)^{\top}\right\|_{\operatorname{tr}}^2\leq p\left\|\sum_{i=L}^k \left(\Delta_i-\bar{\Delta}_i\right)\left(\Delta_i-\bar{\Delta}_i\right)^{\top}\right\|_{\operatorname{tr},2}^2,\notag\\
        &\leq p^3\max_{1\leq l\leq p}\max_{1\leq s\leq p}\left[  \sum_{i=L}^k \left(\sum_{j=i-L+1}^i (\tilde{m}_s(t_j)-m_s(t_j) )\right)\left(\sum_{j=i-L+1}^i (\tilde{m}_l(t_j)-m_l(t_j) )\right) \right]^2,\notag\\
        &\leq p^3\max_{1\leq s\leq p} \left[  \sum_{i=L}^k\left(\sum_{j=i-L+1}^i (\tilde{m}_s(t_j)-m_s(t_j) )\right)^2\right] \max_{1\leq l\leq p} \left[  \sum_{i=L}^k \left(\sum_{j=i-L+1}^i (\tilde{m}_l(t_j)-m_l(t_j) )\right)^2 \right]\notag\\
        &\leq p^3\max_{1\leq s\leq p} \left[  \sum_{i=L}^kL\sum_{j=i-L+1}^i |\tilde{m}_s(t_j)-m_s(t_j) |^2\right]^2.\label{eq:bound 2.1 residual cumulative cov}
    \end{align}
    Denote 
    \begin{align}
            \mathcal{I}_1(k)&=: \max_{1\leq l\leq p}\max_{1\leq s\leq p}\left|  \sum_{i=L}^k \left(\sum_{j=i-L+1}^i (\tilde{m}_s(t_j)-m_s(t_j)) \right)\sum_{j=i-L+1}^ie_{j,l} \right|,   \notag\\
            \mathcal{I}_2(k)&=: L \max_{1\leq s\leq p}\sum_{i=L}^k\sum_{j=i-L+1}^i |\tilde{m}_s(t_j)-m_s(t_j) |^2,  \notag
    \end{align}
    then combining \eqref{eq:bound 1 residual cumulative cov}, \eqref{eq:bound 2 residual cumulative cov} and \eqref{eq:bound 2.1 residual cumulative cov}, we have
    \begin{equation}
        \left\|\hat{\mathbf{Q}}(k)-\bar{\mathbf{Q}}(k)\right\|_{\operatorname{tr}} \lesssim \frac{p^{3/2}}{L}(\mathcal{I}_1(k)+\mathcal{I}_2(k)).\label{eq:I1+I2}
    \end{equation}
    In following we shall bound $\mathcal{I}_1(k),\mathcal{I}_2(k)$ separately.
    
    For $\mathcal{I}_1(k)$, using summation by parts,
    \begin{align}
        &\mathcal{I}_1(k)\leq  \max_{1\leq l\leq p}\max_{1\leq s\leq p}\left|\sum_{j=k-L+1}^k(\tilde m_s(t_j)-m_s(t_j)) \right|\cdot \max_{m\leq k}\left|\sum_{i=L}^m\sum_{j=i-L+1}^ie_{j,l}\right|  \notag\\
        &+\max_{1\leq l\leq p}\max_{1\leq s\leq p}\left|\sum_{i=L}^{k-1} \left[\sum_{j=i-L+1}^i (\tilde{m}_s(t_j)-m_s(t_j))\right]-\left[\sum_{j=i-L+2}^{i+1} (\tilde{m}_s(t_j)-m_s(t_j))\right]  \right|\cdot \max_{m\leq k}\left|\sum_{i=L}^m\sum_{j=i-L+1}^ie_{j,l}\right|\notag\\
        &\leq L\left(\sum_{j=k-L+1}^k\max_{1\leq s\leq p}|\tilde m_s(t_j)-m_s(t_j)| + \sum_{i=1}^{L}\max_{1\leq s\leq p}|\tilde{m}_s(t_{i})-m_s(t_{i})|+\sum_{i=k-L+2}^{k}\max_{1\leq s\leq p}|\tilde m_s(t_{i})-m_s(t_{i})|  \right)\notag\\
        &\cdot \max_{1\leq l\leq p}\max_{m\leq k}\left|\sum_{j=1}^me_{j,l}\right|\notag\\
        &\lesssim L^2\max_{1\leq s\leq p}\sup_{t\in[0,1]}|\tilde m_s(t)-m_s(t)|\max_{1\leq l\leq p}\max_{m\leq n}\left|\sum_{j=1}^me_{j,l}\right|.\label{eq:I1 partial sum bound}
    \end{align}
    By Theorem 3.2 in \cite{Mies2022seq_high-dim}, for $q\geq 2$ in Assumption \hyperref[(B1)]{(B1)}, we have
    \begin{equation}
        \label{eq:e partial sum bound}
        \left\|\max_{1\leq l\leq p}\max_{m\leq n}\left|\sum_{j=1}^me_{j,l}\right| \right\|_q\leq \left\|  \max_{m\leq n}\left|\sum_{j=1}^m\mf e_j\right|\right\|_q\leq C\sqrt{n}\sum_{k=0}^\infty \delta_q(\mf G,k)\leq C\Theta(p)\sqrt{n},
    \end{equation}
    where constant $C>0$ only depends on $q$.
    Combining \eqref{eq:I1 partial sum bound},\eqref{eq:e partial sum bound} and \eqref{eq:final deduction in High-dim ll} in Lemma \ref{lem:High-dim ll}, by Hölder's inequality, we can have for sufficiently large $\lambda \geq q/(q-1)$,
    \begin{align}
        \mathbb{E}\max_{1\leq k\leq n}\mathcal{I}_1(k)&\lesssim  L^2 \left\|\max_{1\leq s\leq p} \sup_{t\in[0,1]}|\tilde m_s(t)-m_s(t)| \right\|_\lambda\cdot\left\|\max_{1\leq l\leq p}\max_{m\leq n}\left|\sum_{j=1}^me_{j,l}\right| \right\|_q\notag\\
        &\lesssim L^2p^{1/\lambda}(h_r^2+\frac{h_r^{-1/\lambda}\lambda^{3}\Theta^{1/\lambda}(p)}{\sqrt{nh_r}})\Theta(p)\sqrt{n}\label{eq:rate for I1_r}.
    \end{align}

    For $\mathcal{I}_2(k)$, by \eqref{e11} and \eqref{e10.j} in Lemma \ref{local linear stochastic expansion}, we have
    \begin{align}
        \max_{1\leq k\leq n}\mathcal{I}_2(k)&\leq L^2\max_{1\leq s\leq p}\sum_{j=1}^n|\tilde{m}_s(t_j)-m_s(t_j)|^2,\notag\\
        & \leq L^2\max_{1\leq s\leq p} (n-2\lfloor nh_{r,s}\rfloor)\sup_{t\in\mathfrak{T}_{n,s}}\left|\tilde{m}_{s}(t)-m_s(t)-\frac{1}{n h_{r,s}} \sum_{i=1}^n e_{i,s} \tilde K_r^*\left(\frac{t_i-t}{h_{r,s}},t\right)\right|^2\notag\\
        &+L^2\max_{1\leq s\leq p} (2\lfloor nh_{r,s}\rfloor)\sup_{t\in\mathfrak{T}_{n,s}^\prime}\left|\tilde{m}_{s}(t)-m_s(t)-\frac{1}{n h_{r,s}} \sum_{i=1}^n e_{i,s} \tilde K_r^*\left(\frac{t_i-t}{h_{r,s}},t\right)\right|^2\notag\\
        &+L^2\max_{1\leq s\leq p}\sum_{j=1}^n \left|\frac{1}{n h_{r,s}} \sum_{i=1}^n e_{i,s} \tilde K_r^*\left(\frac{t_i-t_j}{h_{r,k}},t_j\right)\right|^2,\notag\\
        &\leq O(L^2 nh_r^5+ \frac{L^2}{nh_r^2})+L^2n\max_{1\leq s\leq p}\sup_{t\in[0,1]}\left|\frac{1}{n h_{r,s}} \sum_{i=1}^n e_{i,s} \tilde K_r^*\left(\frac{t_i-t}{h_{r,k}},t\right)\right|^2,\label{eq:bound for I2}
    \end{align}
    where $\mathfrak{T}_{n,s}=[h_{r,s},1-h_{r,s}]$ and $\mathfrak{T}_{n,k}^{\prime}=[0,h_{r,s})\cup(1-h_{r,s},1]$.
    Combining \eqref{eq:bound for I2} and \eqref{eq:nonstationary rate in lem:High-dim ll} in the proof of  Lemma \ref{lem:High-dim ll}, for any $\lambda\geq 2$
    \begin{equation}
        \left\|\max_{1\leq k\leq n}\mathcal{I}_2(k)\right\|_\lambda = O\left( L^2nh_r^5+L^2n\frac{h_r^{-1/\lambda}\lambda^{6}\Theta^{1/\lambda}(p)}{nh_r} \right).\label{eq:rate for I2_r}
    \end{equation}
    Combining \eqref{eq:I1+I2}, \eqref{eq:rate for I1_r} and \eqref{eq:rate for I2_r} with setting $\lambda\asymp\log(n)$, by Assumption \hyperref[(B1)]{(B1)} and \hyperref[(C2)]{(C2)}, we have
    \begin{equation}
        \mathbb{E}\max_{1\leq k\leq n}\left\|\hat{\mathbf{Q}}(k)-\bar{\mathbf{Q}}(k)\right\|_{\operatorname{tr}} \leq Cp^{3/2}\left(\Theta(p)L\sqrt{n}\left(h_r^2+\frac{\log^{3}(n)}{\sqrt{nh_r}}\right)+Ln\left(h_r^5+\frac{\log^6(n)}{nh_r}\right)\right). \label{eq:bound 4 residual cumulative cov}
    \end{equation}
    
    Finally, by Lemma \ref{lem:thm5 in mies}, for some constant $C>0$ only depends on $q$, 
    \begin{equation*}
        \mathbb{E} \max _{1\leq k\leq  n}\left\|\bar{\mathbf{Q}}(k)-\sum_{j=1}^k \mathbf{\Sigma}_{\mathbf{G}}(t_j)\right\|_{\operatorname{tr} } \leq C \Theta(p)^2\left(\sqrt{n p L}+n L^{-1}\right),
    \end{equation*}
    which yields desired bound in \eqref{eq:final bound residual cumulative cov} in view of \eqref{eq:bound 4 residual cumulative cov}.
\end{proof}

\begin{lem}
\label{lem:multi-dim long-run var}
Under same conditions in Theorem \ref{thm:multi-SCB}, let $\hat{\mf Q}(k)=:\sum_{j=L}^k \hat{\mf \Lambda}(t_j)$ and $\hat{\mathbf{\Lambda}}(t)$ is defined in \eqref{eq:long-run estimation multi-dim}, then we have
\begin{equation}
    \mathbb{E} \max _{k=1, \ldots, n}\left\|\hat{\mf Q}(k) -\sum_{j=1}^k\mathbf{\Lambda}(t_j)\right\|_{\operatorname{tr} } \leq  C \bar{\varphi}_n.\label{eq:final bound residual cumulative cov multi-dim}
\end{equation}
where constant $C>0$ and 
$$
\bar{\varphi}_n=:\sqrt{n L}+n L^{-1}+\sqrt{n}L\left(h_r^2+\frac{\log n}{\sqrt{nh_r}}\right)+nL\left(h_r^5+\frac{\log^2n}{nh_r}\right).
$$
\end{lem}
\begin{proof}[Proof of Lemma \ref{lem:multi-dim long-run var}.]
Denote 
    $$
    \Delta_{i}=:\left(\sum_{j=i-L+1}^i \mf x_j\hat{\varepsilon}(t_j)\right),\quad\bar{\Delta}_i=:\left(\sum_{j=i-L+1}^i \mathbf{x}_je_j\right),\quad  \bar{\mathbf{Q}}(k)=:\sum_{i=L}^k\frac{1}{L}\bar{\Delta}_i\bar{\Delta}_i^{\top},\quad{\mathbf{Q}}(k)=:\sum_{j=1}^k {\mf \Lambda}(t_j)
    $$
    where residual $\hat{{\varepsilon}}(t_j)=y_j-\mf x_j^\top\tilde{\mf m}(t_j)$. Then $\hat{\mathbf{Q}}(k)=\sum_{i=L}^k\frac{1}{L}{\Delta}_i{\Delta}_i^{\top}$ by definition \eqref{eq:long-run estimation multi-dim}. Note that process $\mf{x}_ie_i$ follows conditions \hyperref[(B1)]{(B1)}-\hyperref[(B3)]{(B3)} with fixed dimension $p$, by Lemma \ref{lem:thm5 in mies}, we have
    \begin{equation}
        \mathbb{E} \max _{k=1, \ldots, n}\left\|\bar{\mathbf{Q}}(k)-\mf Q(k)\right\|_{t r} \leq C\left(\sqrt{n L}+n L^{-1}\right).\label{eq:apply mies multi-dim}
    \end{equation}
    By similar arguments in \eqref{eq:I1+I2}
    \eqref{eq:I1 partial sum bound} and \eqref{eq:bound 4 residual cumulative cov}, we can have
    \begin{equation}
        \left\|\bar{\mathbf{Q}}(k)-\hat{\mf Q}(k)\right\|_{\operatorname{tr}}\lesssim \frac{1}{L}(\mathcal{I}_1(k)+\mathcal{I}_2(k))\label{eq:I1+I2 residual multi-dim}
    \end{equation}
    where
    \begin{align}
            \mathcal{I}_1(k)&=: \max_{1\leq l\leq p}\max_{1\leq s\leq p}\left|  \sum_{i=L}^k \left(\sum_{j=i-L+1}^i (\tilde{m}_s(t_j)-m_s(t_j)) \right)\sum_{j=i-L+1}^ix_{j,l}e_j \right|,   \notag\\
            \mathcal{I}_2(k)&=: L \max_{1\leq s\leq p}\sum_{i=L}^k\sum_{j=i-L+1}^i |\tilde{m}_s(t_j)-m_s(t_j) |^2.  \notag
    \end{align}
    Note that by Theorem 2 in \cite{wu2005nonlinear} and Assumption \hyperref[(B4')]{(B4')}, we have 
    \begin{equation*}
        \left\|\max_{m\leq n}\left|\sum_{j=1}^m \mf{x}_je_j\right| \right\|_4\lesssim \sqrt{n}\sum_{k=0}^\infty\chi^k\lesssim \sqrt{n}. \label{eq:e partial sum multi-dim}
    \end{equation*}
    Then by similar arguments in \eqref{eq:I1 partial sum bound},\eqref{eq:rate for I1_r}, together with \eqref{eq:multi-dim stochastic expansion of ll}, \eqref{eq:multi-dim ll GA} and \eqref{eq:multi-dim rate of Tn} in Lemma \ref{lem:multi-dim ll}, we have
    \begin{align}
                \mathbb{E}\max _{k=1, \ldots, n}\mathcal{I}_1(k)&\lesssim 
        L^2\left\| \sup_{t\in[0,1]}|\tilde{\mf m}(t)-\mf m(t)|\right\|_4\cdot \left\|\max_{m\leq n}\left|\sum_{j=1}^m \mf{x}_je_j\right| \right\|_4.\notag\\
        &\lesssim \sqrt{n}L^2(h_r^2+\frac{\log n}{\sqrt{nh_r}}).\label{eq:I1 residual multi-dim}
    \end{align}
    Besides, by similar arguments in \eqref{eq:bound for I2} and \eqref{eq:rate for I2_r}, together with \eqref{eq:multi-dim stochastic expansion of ll}, \eqref{eq:multi-dim ll GA} and \eqref{eq:multi-dim rate of Tn} in Lemma \ref{lem:multi-dim ll}, we can also have
    \begin{equation}
        \mathbb{E}\max_{1\leq k\leq n}\mathcal{I}_2(k)\lesssim L^2nh_r^5+L^2n\frac{\log^2n}{nh_r}.\label{eq:I2 residual multi-dim}
    \end{equation}
    Combining \eqref{eq:I1+I2 residual multi-dim}, \eqref{eq:I1 residual multi-dim}, \eqref{eq:I2 residual multi-dim}, we have
    \begin{equation}
    \mathbb{E}\max _{k=1, \ldots, n}\left\|\bar{\mathbf{Q}}(k)-\hat{\mf Q}(k)\right\|_{\operatorname{tr} }\leq C\left(\sqrt{n}L(h_r^2+\frac{\log n}{\sqrt{nh_r}})+nLh_r^5+nL\frac{\log^2n}{nh_r}\right).   \label{eq:residual bound multi-dim}
    \end{equation}
    Therefore, \eqref{eq:final bound residual cumulative cov multi-dim} holds in view of \eqref{eq:apply mies multi-dim} and \eqref{eq:residual bound multi-dim}.
\end{proof}

\subsection{Boundary issues}
\begin{prop}
\label{Boundary}
    Suppose \hyperref[(C1)]{(C1)}-\hyperref[(C2)]{(C2)} hold, then we have following results:
    \begin{description}
        \item[(i).] In model \eqref{high-monotone model}, if Assumptions \hyperref[(B1)]{(B1)}-\hyperref[(B5)]{(B5)} hold and $\mf m(t)$ follows Assumptions \hyperref[(A1)]{(A1)}-\hyperref[(A2)]{(A2)}, then the random Lebesgue's measure $\lambda\{ \hat{\mathcal{T}}\}=:\lambda\{\cap_{k=1}^p \hat{\mathcal{T}}_k \}\rightarrow_p 1$ where
        $$
            \hat{\mathcal{T}}_k=:\{\hat{m}_{I,k}^{-1}(s):s\in\hat{\mathcal{T}}_k^{-1}\}\cap [h_{d}\log h_{d}^{-1},1-h_{d}\log h_{d}^{-1}],
        $$
        and $\hat{\mathcal{T}}^{-1}_k=:\{s\in\mathbb{R}:\min_{1\leq i\leq N}\hat{m}_k(i/N)\leq s\leq \max_{1\leq i\leq N}\hat{m}_k(i/N) \}$.
        \item[(ii).] In model \eqref{Time-varying_n}, if Assumptions \hyperref[(B1')]{(B1')}-\hyperref[(B5')]{(B5')} hold and $\mf m_{\mf C}(t)$ follows Assumptions \hyperref[(A1)]{(A1)}-\hyperref[(A2)]{(A2)}, then the random Lebesgue's measure $\lambda\{\hat{\mathcal{T}}\}=:\lambda\{\cap_{k=1}^s \hat{\mathcal{T}}_k \}\rightarrow_p 1$ where
        $$
            \hat{\mathcal{T}}_k=:\{\hat{m}_{\mf C,I,k}^{-1}(v):v\in\hat{\mathcal{T}}_k^{-1}\}\cap [h_{d}\log h_{d}^{-1},1-h_{d}\log h_{d}^{-1}],
        $$
        and $\hat{\mathcal{T}}^{-1}_k=:\{v\in\mathbb{R}:\min_{1\leq i\leq N}\hat{m}_{\mf C,k}(i/N)\leq s\leq \max_{1\leq i\leq N}\hat{m}_{\mf C,k}(i/N) \}$.
    \end{description}
\end{prop}
\begin{proof}[Proof of Proposition \ref{Boundary}.]
Proof of (i). Note that $h_d\rightarrow0$, we only need to prove 
\begin{align}
    &\max_{1\leq k\leq p}\hat m_{I,k}^{-1}\left(\min_{1\leq i\leq N}\hat m_k(i/N)\right)-0=o_p(1) \label{eq:left boundary}\\
    &1-\min_{1\leq k\leq p}\hat m_{I,k}^{-1}\left(\max_{1\leq i\leq N}\hat m_k(i/N)\right)=o_p(1).\label{eq:right boundary}
\end{align}
By \eqref{eq:rate of m_I.inv-m_N.inv} in Lemma \ref{lem:(m_I.inv-m_N.inv)'}, we have
\begin{equation}
    \max_{1\leq k\leq p}\left|\hat m_{I,k}^{-1}\left(\min_{1\leq i\leq N}\hat m(i/N)\right)- m_{N,k}^{-1}\left(\min_{1\leq i\leq N}\hat m(i/N)\right)\right|=O_p(R_n).\label{eq:m_N.inv(min)}
\end{equation}
Denote $l_k(v)=:1\bigwedge \frac{\hat m_k(v)-\hat m_k(0)}{h_{d,k} }$, then
\begin{align}
    m_{N,k}^{-1}\left(\min_{1\leq i\leq N}\hat m(i/N)\right)&=\int_{-\infty}^{\min_{1\leq i\leq N}\hat m(i/N)}\frac{1}{Nh_{d,k}}\sum_{i=1}^N K_d\left(\frac{m_k(i/N)-u}{h_{d,k}}\right)\mathrm{d}u,\notag\\
    &=\int_0^1 \int_{-\infty}^{\min_{1\leq i\leq N}\hat m(i/N)} K_d\left(\frac{m_k(v)-u}{h_{d,k}}\right)\mathrm{d}u\mathrm{d}v+O(\frac{1}{Nh_d}),\notag\\
    &\lesssim \int_{0}^1\int^1_{l_k(v)}K_d(x)\mathrm{d}x\mathrm{d}v+O(\frac{1}{Nh_d}),\label{eq:bound for m_N.inv(min)}
\end{align}
where big $O$ for $k=1,\dots,p$ can be uniformly bounded by Assumption \hyperref[(A1)]{(A1)}-\hyperref[(A2)]{(A2)}. To bound \eqref{eq:bound for m_N.inv(min)}, we only need to bound the random Lebesgue's measure for all $k=1,\dots,p$ i.e.
\begin{equation}
    \lambda\{v:\hat m_k(v)-\hat m_k(0)\leq h_{d,k},v\in[0,1]\}.
    \label{length_T_I}
\end{equation}
Note that $\exists M>0 \text { s.t. } \inf _{t \in[0,1]} m_k^{\prime}(t) \geq M$ for all $k=1,\dots,p$ in Assumption \hyperref[(A2)]{(A2)}, then we have
\begin{equation}
 \{v:m_k(v)-m_k(0)\leq h_{d,k},v\in[0,1]\}\subset \{v:v\leq h_{d,k},v\in[0,1] \}.  \label{eq:measure m(v)-m(0)} 
\end{equation}
Introducing $U$ which is independent of $\hat m$ and follows uniform distribution on $[0,1]$, then for every $x\in[0,m_k(1))$
\begin{align*}
    |\mathbb{P}(m_k(U)-m_k(0)\leq x)-\mathbb{P}(\hat m_k(U)-\hat m_k(0)\leq x)|&\leq\mathbb{P}(|m_k(U)-\hat m_k(U)-(m_k(0)-\hat m_k(0))|>\delta)\\
    &+\mathbb{P}(x-\delta\leq m_k(U)-m_k(0)\leq x+\delta)\\
    &\leq\mathbb{P}(|m_k(U)-\hat m_k(U)-(m_k(0)-\hat m_k(0))|>\delta)\\
    &+m_k^{-1}(x+\delta+m_k(0))-m_k^{-1}(x-\delta+m_k(0)).
\end{align*}
Setting $\delta=\sqrt{R_n}$ , by Lemma \ref{lem:High-dim ll} and Assumption \hyperref[(A1)]{(A1)}-\hyperref[(A2)]{(A2)}, we can have
\begin{align}
    &\max_{1\leq k\leq p}\sup_{x\in[0,m_k(1)]}|\mathbb{P}(m_k(U)-m_k(0)\leq x)-\mathbb{P}(\hat m_k(U)-\hat m_k(0)\leq x)|\notag\\
    &\leq \max_{1\leq k\leq p}\mathbb{P}(|m_k(U)-\hat m_k(U)-(m_k(0)-\hat m_k(0))|>\sqrt{R_n})+O(\sqrt{R_n})\notag\\
    &\leq \mathbb{P}(\max_{1\leq k\leq p}\sup_{t\in[0,1]}|\hat{m}_k(t)-m_k(t)|>\sqrt{R_n}/2)+O(\sqrt{R_n})=o(1)\label{eq:diff hat measure}.
\end{align}
Combining \eqref{eq:m_N.inv(min)}, \eqref{eq:bound for m_N.inv(min)}, \eqref{length_T_I}, \eqref{eq:measure m(v)-m(0)} and \eqref{eq:diff hat measure}, we can finally have
$$
\max_{1\leq k\leq p}\hat m_{I,k}^{-1}\left(\min_{1\leq i\leq N}\hat m(i/N)\right)\lesssim O_p(R_n)+O(\frac{1}{Nh_d})+o(1)+O(h_d)=o_p(1),
$$
which yields \eqref{eq:left boundary}. Similarly, we can also prove \eqref{eq:right boundary}.

Note that local linear estimator $\hat{\mf m}(t)$ in model \eqref{Time-varying_n} also follows $\sup_{t\in[0,1]}|\hat{\mf m}(t)-\mf m(t)|=O_p(R_n)$ by Lemma \ref{lem:multi-dim ll}. Thus all the Lemmas used in the proof of (i) can also be applied under model \eqref{Time-varying coefficient regression} and Assumptions \hyperref[(B1')]{(B1')}-\hyperref[(B5')]{(B5')}. Therefore, the assertion (ii) holds by similar arguments in the proof of (i).
\end{proof}

\subsection{Lemmas of Gaussian approximations}
\begin{lem}
Let $p$-dimensional nonstationary process $\mf X_i=\mathbf{G}\left(t_i,\Upsilon_i\right),i=1,\dots,n$, with $\mathbb{E}\left(X_i\right)=0$ where $\mathbf{G},p,n$ satisfy \hyperref[(B1)]{(B1)} and \hyperref[(B2)]{(B2)}. If \hyperref[(B3)]{(B3)} is satisfied, then there exists random vectors $(\tilde{X}_i)_{i=1}^n \stackrel{d}{=}$ $(X_i)_{i=1}^n$ and independent, mean zero, Gaussian random vectors $\tilde{Y}_i \sim \mathcal{N}\left(0, \mf\Sigma_{\mf G}(t_i)\right)$ such that
$$
\left(\mathbb{E} \max _{k \leq n}\left|\frac{1}{\sqrt{n}} \sum_{i=1}^k\left(\tilde{X}_i-\tilde{Y}_i\right)\right|^2\right)^{\frac{1}{2}} \leq C \Theta(p) \sqrt{\log (n)}\left(\frac{p}{n}\right)^{\chi(q)} 
$$
where $\chi(q)= (q-2)/(6 q-4)$, $\mf \Sigma_{\mf G}(t)=\sum_{h=-\infty}^{\infty} \operatorname{Cov}\left(\mathbf{G}\left(t,\Upsilon_0\right), \mathbf{G}\left(t,\Upsilon_h\right)\right),t\in[0,1]$ and the universal constant $C>0$ depends on $q$.
\label{high-GA}
\end{lem}
\begin{proof}[Proof of Lemma \ref{high-GA}.]
    See proof of Theorem 3.1 in \cite{Mies2022seq_high-dim}.
\end{proof}

\begin{lem}
    Let $\mathbf{\Sigma}_1, \mathbf{\Sigma}_2 \in \mathbb{R}^{p \times p}$ be symmetric, positive semidefinite matrices, and let $v_i \in \mathbb{R}^p, \lambda_i \in \mathbb{R}, i=1, \ldots, p$, be the eigenvectors and eigenvalues of $\Delta=\mathbf{\Sigma}_2-\mathbf{\Sigma}_1$. Define the matrix $|\Delta|=\sum_{i=1}^d v_i v_i^T\left|\lambda_i\right|$. Consider a random vector $Y \sim \mathcal{N}\left(0, \mathbf{\Sigma}_1\right)$ defined on a sufficiently rich probability space. Then there exists a random vector $\eta \sim \mathcal{N}(0,|\Delta|)$ such that $Y+\eta \sim \mathcal{N}\left(0, \mathbf{\Sigma}_2\right)$.\label{lem:2-Wasserstein bound}
\end{lem}
\begin{proof}[Proof of Lemma \ref{lem:2-Wasserstein bound}.]
    See proof of Lemma 6.2 in \cite{Mies2022seq_high-dim}. 
\end{proof}

\begin{lem}
\label{lem:gaussian cov structure}
Let $\mathbf{\Sigma}_t, \tilde{\mathbf{\Sigma}}_t \in \mathbb{R}^{p \times p}$ be symmetric, positive semidefinite matrices, for $t=$ $1, \ldots, n$, and consider independent random vectors $Y_t \sim \mathcal{N}\left(0, \mathbf{\Sigma}_t\right)$. Denote
$$
\varphi=:\max _{k=1, \ldots, n}\left\|\sum_{t=1}^k \mathbf{\Sigma}_t-\sum_{t=1}^k \tilde{\mathbf{\Sigma}}_t\right\|_{t r},\quad \Phi=:\max _{t=1, \ldots, n}\left\|\mathbf{\Sigma}_t\right\|_{t r}.
$$
On a potentially larger probability space, there exist independent random vectors $\tilde{Y}_t \sim \mathcal{N}\left(0, \tilde{\mathbf{\Sigma}}_t\right)$ , such that
$$
\mathbb{E} \max _{k=1, \ldots, n}\left|\sum_{t=1}^k Y_t-\sum_{t=1}^k \tilde{Y}_t\right|^2 \leq C \log (n)(\sqrt{n\varphi\Phi}+\varphi+\Phi)
$$
for universal constant $C>0$.
\end{lem}

\begin{proof}[Proof of Lemma \ref{lem:gaussian cov structure}.]
The proof of Lemma \ref{lem:gaussian cov structure} is similar to the proof of Proposition 5.2 in \cite{Mies2022seq_high-dim} which only considers positive definite matrices. We find their proof can also be applied to positive semidefinite matrices in this paper. Therefore, we give a similar proof from \cite{Mies2022seq_high-dim} with slight modification for positive semidefinite matrices as follows.

Let $B \in\{1, \ldots, n\}$, and $M=\lceil n / B\rceil$, to be specified later. Denote $\xi_l=\sum_{t=(l-1) B+1}^{(l B) \wedge n} Y_t$, and $S_l=\sum_{t=(l-1) B+1}^{(l B) \wedge n} \mathbf{\Sigma}_t$, and $\tilde S_l=\sum_{t=(l-1) B+1}^{l B \wedge n} \tilde{\mathbf{\Sigma}}_t$, for $l=1, \ldots, M$. Then $\xi_l$ are independent Gaussian random vectors, $\xi_l \sim \mathcal{N}\left(0, S_l\right)$. Denoting $\Delta_l=S_l-\tilde{S_l}$, and $\left|\Delta_l\right|$ as in Lemma \ref{lem:2-Wasserstein bound}, we find Gaussian random vectors $\zeta_l \sim \mathcal{N}\left(0,\left|\Delta_l\right|\right)$ such that $\tilde \xi_l=\xi_l+\zeta_l \sim \mathcal{N}\left(0, \tilde S_l\right)$. We may also split $\tilde \zeta_l$ into independent terms, i.e. we find independent Gaussian random vectors $\tilde Y_t \sim \mathcal{N}\left(0, \tilde{\mathbf{\Sigma}}_t\right)$ such that $\tilde\xi_l=\sum_{t=(l-1) B+1}^{(l B) \wedge n} \tilde Y_t$. This construction yields that the $(\tilde Y_t)_{t=1}^n$ and $\left(Y_t\right)_{t=1}^n$ are sequences of independent random vectors, while $\tilde Y_t$ and $Y_{t+1}$ are not necessarily independent. We also introduce the notation $\zeta_s=\sum_{t=(l-1) B+1}^s Y_t$ for $s=(l-1) B+1, \ldots,(l B) \wedge n$, and $\tilde \zeta_s$ analogously. Then
\begin{align*}
& \mathbb{E} \max _{k=1, \ldots, n}\left|\sum_{t=1}^k\left(Y_t-\tilde Y_t\right)\right|^2 \\
& \leq \mathbb{E} \max _{k=1, \ldots, M}\left|\sum_{l=1}^k\left(\xi_l-\tilde \xi_l\right)\right|^2+\mathbb{E} \max _{l=1, \ldots, M} \max _{s=(l-1) B+1, \ldots,(l B) \wedge n}\left|\sum_{t=(l-1) B+1}^s\left(Y_t-\tilde Y_t\right)\right|^2 \\
& \leq \sum_{l=1}^M\left\|\Delta_l\right\|_{\mathrm{tr}}+2 \mathbb{E} \max _{s=1, \ldots, n}\left|\zeta_s\right|^2+2 \mathbb{E} \max _{s=1, \ldots, n}\left|\tilde \zeta_s\right|^2 \\
& \leq \sum_{l=1}^M\left[\left\|\sum_{t=1}^{(l B) \wedge n} \mathbf{\Sigma}_t-\tilde{\mathbf{\Sigma}}_t\right\|_{\operatorname{tr}}+\left\|\sum_{t=1}^{(l-1) B} \mathbf{\Sigma}_t-\tilde{\mathbf{\Sigma}}_t\right\|_{\operatorname{tr}}\right]+2 \mathbb{E} \max _{s=1, \ldots, n}\left|\zeta_s\right|^2+2 \mathbb{E} \max _{s=1, \ldots, n}\left|\tilde \zeta_s\right|^2 \\
& \leq 2 M \max _{k=1, \ldots, n}\left\|\sum_{t=1}^k\left(\mathbf{\Sigma}_t-\tilde{\mathbf{\Sigma}}_t\right)\right\|_{\operatorname{tr}}+2 \mathbb{E} \max _{s=1, \ldots, n}\left|\zeta_s\right|^2+2 \mathbb{E} \max _{s=1, \ldots, n}\left|\tilde \zeta_s\right|^2 .
\end{align*}
Since the random vectors $\zeta_s$ are Gaussian, the random variable $\left\|\zeta_s\right\|^2$ is sub-exponential with sub-exponential norm bounded by $C \operatorname{tr}\left(\operatorname{Cov}\left(\zeta_s\right)\right) \leq C \operatorname{tr}\left(S_l\right)$, for some universal factor $C$, and for $s=(l-1) B+1, \ldots,(l B) \wedge n$. To see this, denote the sub-exponential norm by $\|\cdot\|_{\psi_1}$. Then
$$
\|| \zeta_s\left|^2\right\|_{\psi_1}=\left\|\sum_{j=1}^d \sigma_j^2 \delta_j^2\right\|_{\psi_1} \leq \sum_{j=1}^d \sigma_j^2\left\|\delta_j^2\right\|_{\psi_1},
$$
where $\sigma_j^2$ are the eigenvalues of $\operatorname{Cov}\left(\zeta_s\right)$, and $\delta_j \sim \mathcal{N}(0,1)$. A consequence of this subexponential bound is that, for a potentially larger $C$,
$$
\mathbb{E} \max _{s=1 \ldots, n}\left|\zeta_s\right|^2 \leq C \log (n) \max _l\left\|S_l\right\|_{\mathrm{tr}} \lesssim \log (n) B \Phi .
$$
Analogously,
$$
\begin{aligned}
\mathbb{E} \max _{s=1, \ldots, n}\left|\tilde\zeta_s\right|^2 \leq C \log (n) \max _l\left\|\tilde S_l\right\|_{\operatorname{tr}} & \leq C \log (n) \max _l\left[\left\|S_l\right\|_{\mathrm{tr}}+\varphi\right] \\
& \lesssim \log (n)(B \Phi+\varphi),
\end{aligned}
$$
such that
$$
\begin{aligned}
\mathbb{E} \max _{k=1, \ldots, n}\left|\sum_{t=1}^k\left(Y_t-\tilde Y_t\right)\right|^2 & \lesssim 2 M \varphi+ \log (n)[B \Phi+\varphi] \\
& \lesssim \log (n)\left(\frac{n}{B} \varphi+B \Phi\right) .
\end{aligned}
$$
If $\varphi\leq n\Phi$, we have $\sqrt{n \varphi / \Phi}\leq n$ so we can choose $B=1 \vee\lceil\sqrt{n \varphi / \Phi}\rceil$, which yields
$$
\mathbb{E} \max _{k=1, \ldots, n}\left|\sum_{t=1}^k\left(Y_t-\tilde Y_t\right)\right|^2\lesssim\log(n)(\sqrt{n\varphi\Phi}+\Phi).
$$
If $\varphi> n\Phi$, $n\varphi/B+B \Phi$ is minimized by setting $B=n$, which yields
$$
\mathbb{E} \max _{k=1, \ldots, n}\left|\sum_{t=1}^k\left(Y_t-\tilde Y_t\right)\right|^2\lesssim\log(n)\varphi.
$$
To sum up, for some universal constant $C>0$,
\begin{equation*}
    \mathbb{E} \max _{k=1, \ldots, n}\left|\sum_{t=1}^k\left(Y_t-\tilde Y_t\right)\right|^2\leq C\log(n)(\sqrt{n\varphi\Phi}+\varphi+\Phi).
\end{equation*}
\end{proof}

\subsection{Lemmas of monotone rearrangement}
\begin{lem}
\label{lem:(m_I.inv-m_N.inv)'} Assume assumptions in Lemma \ref{m_I} hold, then
    \begin{equation}
        \label{eq:rate of m_I.inv-m_N.inv}
        \max_{1\leq k\leq p}\sup_{t\in\mathbb{R}}|(\hat{m}_{I,k}^{-1}-m_{N,k}^{-1})(t)|=O_p\left(\TPR{R_n/M}\right)
    \end{equation}
    \begin{equation}
    \label{eq:rate of (m_I.inv-m_N.inv)'}
        \max_{1\leq k\leq p}\sup_{t\in\mathbb{R}}|(\hat{m}_{I,k}^{-1}-m_{N,k}^{-1})^{\prime}(t)|=O_p\left(\TPR{\frac{R_n}{Mh_d}}\right)
    \end{equation}
    \begin{equation}
    \label{eq:rate of (m_I.inv-m_N.inv)''}
        \max_{1\leq k\leq p}\sup_{t\in\mathbb{R}}|(\hat{m}_{I,k}^{-1}-m_{N,k}^{-1})^{\prime\prime}(t)|=O_p\left(\TPR{\frac{R_n}{Mh_d^2}}\right)
    \end{equation}
\end{lem}

\begin{proof}[Proof of Lemma \ref{lem:(m_I.inv-m_N.inv)'}]
By Lemma \ref{lem:High-dim ll}, the jackknife corrected local linear estimator follows 
$$
\max_{1\leq k\leq p}\sup_{t\in[0,1]}|\tilde{m}_k(t)-m_k(t)|=O_p(R_n),
$$
where $R_n=h_r^2+\log^{3}n/\sqrt{nh_r}$. 

Recall the definitions of $\hat m_I^{-1}$, $m_N^{-1}$ in \eqref{eq:m_I} and \eqref{eq:m_N} respectively, for any $t\in\mathbb{R}$, by Taylor's expansion
\begin{align}
    \hat m_{I,k}^{-1}(t)-m_{N,k}^{-1}(t) &=\frac{1}{N h_{d,k}^2} \sum_{i=1}^{N}(\tilde{m}_k(i / N)-m_k(i / N)) \int_{-\infty}^t K_{d}'\left(\frac{m_k(i/N)-u}{h_{d,k}}\right)\mathrm{d}u\notag\\
   &\quad+\frac{1}{2Nh_{d,k}^3} \sum_{i=1}^{N}(\tilde{m}_k(i / N)-m_k(i / N))^2 \int_{-\infty}^tK_{d}''\left(\frac{\eta_{i,N,k}-u}{h_{d,k}}\right)\mathrm{d}u, \notag\\
   &=\frac{1}{N h_{d,k}} \sum_{i=1}^{N}(\tilde{m}_k(i / N)-m_k(i / N)) K_{d}\left(\frac{m_k(i/N)-t}{h_{d,k}}\right)\mathrm{d}u\notag\\
   &\quad+\frac{1}{2Nh_{d,k}^2} \sum_{i=1}^{N}(\tilde{m}_k(i / N)-m_k(i / N))^2 K_{d}'\left(\frac{\eta_{i,N,k}-t}{h_{d,k}}\right)\mathrm{d}u, \notag
\end{align}
for some $\eta_{i,N,k}$ between $\tilde m_k(i/N)$ and $m_k(i/N)$.
For any sequence 
$$
\{\eta_{i,N,k}: \eta_{i,N,k}\text{ between $\tilde m_k(i/N)$ and $m_k(i/N)$},i=1,\dots,N\},
$$
denote index set
\begin{equation}
    \mathfrak{K}(\{\eta_{i,N,k}\}_{i=1}^N,t)=:\{i:|\eta_{i,N,k}-t|\leq h_{d,k}\}.
    \label{eq:number of Kd'}
\end{equation}
Note that $|\eta_{i,N,k}-m_k(i/N)|\leq |\hat m_k(i/N)-m_k(i/N)|$, we have
$$
|\eta_{i,N,k}-t|\geq |m_k(i/N)-t|-|\eta_{i,N,k}-m_k(i/N)|\geq |m_k(i/N)-t|-\sup_{t\in[0,1]}|\tilde{m}_k(t)-m_k(t)|,
$$
which means
$$
\#\mathfrak{K}(\{\eta_{i,N,k}\}_{i=1}^N,t)\leq \#\left\{i:|m_k(i/N)-t|\leq h_{d,k}+\sup_{t\in[0,1]}|\tilde{m}_k(t)-m_k(t)|\right\}.
$$
\begin{description}
    \item[Case 1.] If $t\in[m_k(0),m_k(1)]$, then there exists unique $t^*\in[0,1]$ s.t. $m_k(t^*)=t$ and
\begin{align}
    \sup_{t\in[m_k(0),m_k(1)]}\#\mathfrak{K}(\{\eta_{i,N,k}\}_{i=1}^N,t)&\leq \sup_{t^*[0,1]}\#\left\{i:|m_k(i/N)-m_k(t^*)|\leq h_{d,k}+\sup_{t\in[0,1]}|\tilde{m}_k(t)-m_k(t)|\right\},\notag\\
    &\leq \sup_{t^*\in[0,1]}\#\left\{i:|i/N-t^*|\leq \frac{h_{d,k}+\sup_{t\in[0,1]}|\tilde{m}_k(t)-m_k(t)|}{\inf_{t\in[0,1]}m_k'(t) } \right\},\notag\\
    &=O_p(\TPR{Nh_d/M})\label{eq:number bound},
\end{align}
where the least line of \eqref{eq:number bound} is obtained by \eqref{eq:m_N-m}, $\min_{1\leq k\leq p}h_{d,k}\asymp h_d\asymp \min_{1\leq k\leq p}h_{d,k}$, $\inf_{1\leq k\leq p}\inf_{t\in[0,1]} m_k'(t)\geq M$ from Assumption \hyperref[(A2)]{(A2)} and $R_n/h_d=o(1)$ from Assumption \hyperref[(C2)]{(C2)}.
    \item[Case 2.] If $t\notin[m_k(0)-h_d,m_k(1)+h_{d,k}]$, $\#\mathfrak{K}(\{\eta_{i,N,k}\}_{i=1}^N,t)=0$ since $|m_k(i/N)-t|>h_{d,k}$ for any $i=1,\dots,N$.
    \item[Case 3.] If $t\in[m_k(0)-h_{d,k},m_k(0))$, then $|t-m_k(0)|\leq h_{d,k}$. Thus
    $$
    \begin{aligned}
        |\eta_{i,N,k}-t|&\geq |m_k(i/N)-t|-\sup_{t\in[0,1]}|\tilde{m}_k(t)-m_k(t)|\\
        &\geq |m_k(i/N)-m_k(0)|-|t-m_k(0)|-\sup_{t\in[0,1]}|\tilde{m}_k(t)-m_k(t)|\\
        &\geq |m_k(i/N)-m_k(0)|-h_d-\sup_{t\in[0,1]}|\tilde{m}_k(t)-m_k(t)|
    \end{aligned}
    $$
    then
    $$
    \sup_{t\in[m_k(0)-h_{d,k},m_k(0))}\#\mathfrak{K}(\{\eta_{i,N,k}\}_{i=1}^N,t)\leq \#\left\{i:|m_k(i/N)-m_k(0)|\leq 2h_{d,k}+\sup_{t\in[0,1]}|\tilde{m}_k(t)-m_k(t)|\right\},
    $$
    which yields $\sup_{t\in[m_k(0)-h_{d,k},m_k(0))}\#\mathfrak{K}(\{\eta_{i,N,k}\}_{i=1}^N,t)=O_p(Nh_d/M)$ by similar arguments in \eqref{eq:number bound}.
    \item[Case 4.] If $t\in(m_k(1),m_k(1)+h_{d,k}]$, then $|m_k(1)-t|\leq h_{d,k}$. By similar arguments in $t\in[m_k(0)-h_{d,k},m_k(0))$, it can also yield $\sup_{t\in(m_k(1),m_k(1)+h_{d,k}]}\#\mathfrak{K}(\{\eta_{i,N,k}\}_{i=1}^N,t)=O_p(\TPR{Nh_d/M})$.
\end{description}

Summarizing Case 1-4, note that the big $O$ for $k$ can be uniformly bounded under the uniformly bounded Lipschitz condition \hyperref[(A1)]{(A1)} and Assumption \hyperref[(A2)]{(A2)}, we have
\begin{equation}
    \label{eq:number rate of Kd'}
    \max_{1\leq k\leq p}\sup_{t\in\mathbb{R}}\#\mathfrak{K}(\{\eta_{i,N,k}\}_{i=1}^N,t)=O_p(\TPR{Nh_d/M}).
\end{equation}
Note that $K_d$ and $K_d'$ are bounded, by Lemma \ref{lem:High-dim ll} and \eqref{eq:number rate of Kd'}
\begin{align}
    &\max_{1\leq k\leq p}\sup_{t\in\mathbb{R}}|(\hat{m}_{I,k}^{-1}-m_{N,k}^{-1})(t)|\notag\\
    &=\max_{1\leq k\leq p}\sup_{t\in[0,1]}|\tilde{m}_k(t)-m_k(t)|\cdot\max_{1\leq k\leq p}\sup_{t\in\mathbb{R}}\left|\frac{1}{Nh_{d,k}}\sum_{i=1}^N K_d\left(\frac{m_k(i/N)-t }{h_{d,k}}\right)\right|\notag\\
    &+ \max_{1\leq k\leq p}\sup_{t\in[0,1]}|\tilde{m}_k(t)-m_k(t)|^2\cdot\max_{1\leq k\leq p}\sup_{t\in\mathbb{R}}\left|\frac{1}{2Nh_{d,k}^2}\sum_{i\in\mathfrak{K}(\{\eta_{i,N,k}\}_{i=1}^N,t)} K_d'\left(\frac{\eta_{i,N,k}-t }{h_{d,k}}\right)\right|,\notag\\
    &=O_p\left(\TPR{R_n/M}\right), \label{eq:deduction of m_I.inv-m_N.inv}
\end{align}
then \eqref{eq:rate of m_I.inv-m_N.inv} holds.
Moreover, by Lagrange's mean value again, we have for some $\eta_{i,N,k}$ between $\tilde{m}_k(i/N)$ and $m_k(i/N)$,
\begin{align*}
    (\hat m_{I,k}^{-1}-m_{N,k}^{-1})^\prime(t)&=\frac{1}{Nh_{d,k}}\sum_{i=1}^N\left\{K_d\left(\frac{\hat m_k(i/N)-t}{h_{d,k}}\right)-K_d\left(\frac{ m_k(i/N)-t}{h_{d,k}}\right)\right\},\\
    &=\frac{1}{Nh_{d,k}^2}\sum_{i=1}^N K_d'\left(\frac{\eta_{i,N,k}-t }{h_{d,k}}\right)\left\{\hat m_k(i/N)-m_k(i/N)\right\}.
\end{align*}
By similar arguments in \eqref{eq:deduction of m_I.inv-m_N.inv},
\begin{align}
    &\max_{1\leq k\leq p}\sup_{t\in\mathbb{R}}|(\hat{m}_{I,k}^{-1}-m_{N,k}^{-1})^{\prime}(t)|\notag\\
    &\leq\max_{1\leq k\leq p}\sup_{t\in[0,1]}|\tilde{m}_k(t)-m_k(t)|\cdot\max_{1\leq k\leq p}\sup_{t\in\mathbb{R}}\left|\frac{1}{Nh_{d,k}^2}\sum_{i\in\mathfrak{K}(\{\eta_{i,N,k}\}_{i=1}^N,t)} K_d'\left(\frac{\eta_{i,N,k}-t }{h_{d,k}}\right)\right|,\notag\\
    &=O_p\left(\TPR{\frac{R_n}{Mh_d}}\right),\notag
\end{align}
thus \eqref{eq:rate of (m_I.inv-m_N.inv)'} holds.
Similarly, for some $\eta_{i,N,k}$ between $\tilde{m}_k(i/N)$ and $m_k(i/N)$ 
\begin{align*}
    (\hat m_{I,k}^{-1}-m_{N,k}^{-1})^{\prime\prime}(t)&=\frac{1}{Nh_{d,k}^2}\sum_{i=1}^N\left\{K_d'\left(\frac{\hat m_k(i/N)-t}{h_{d,k}}\right)-K_d'\left(\frac{ m_k(i/N)-t}{h_{d,k}}\right)\right\},\\
    &=\frac{1}{Nh_{d,k}^3}\sum_{i=1}^N K_d''\left(\frac{\eta_{i,N,k}-t }{h_{d,k}}\right)\left\{\hat m_k(i/N)-m_k(i/N)\right\},\\
    &=\frac{1}{Nh_{d,k}^3}\sum_{i\in\mathfrak{K}(\{\eta_{i,N,k}\}_{i=1}^N,t)} K_d''\left(\frac{\eta_{i,N,k}-t }{h_d}\right)\left\{\hat m_k(i/N)-m_k(i/N)\right\}.
\end{align*} 
Then
\begin{align}
    &\max_{1\leq k\leq p}\sup_{t\in\mathbb{R}}|(\hat{m}_{I,k}^{-1}-m_{N,k}^{-1})^{\prime\prime}(t)|\notag\\
    &\leq\max_{1\leq k\leq p}\sup_{t\in[0,1]}|\tilde{m}_k(t)-m_k(t)|\cdot\max_{1\leq k\leq p}\sup_{t\in\mathbb{R}}\left|\frac{1}{Nh_{d,k}^3}\sum_{i\in\mathfrak{K}(\{\eta_{i,N,k}\}_{i=1}^N,t)} K_d''\left(\frac{\eta_{i,N,k}-t }{h_{d,k}}\right)\right|,\notag\\
    &=O_p\left(\TPR{\frac{R_n}{Mh_d^2}}\right),\notag
\end{align}
thus \eqref{eq:rate of (m_I.inv-m_N.inv)''} holds.
\end{proof}

\begin{lem}
\label{lem:m_I.inv' inf>0} Assume assumptions in Lemma \ref{m_I} hold, then $\inf_{t\in\hat{\mathcal{T}}^{-1} }(\hat{m}_{I,k}^{-1})'(t)>0,\forall k=1,\dots,p$ for sufficiently large $N$ and there exists universal constant $C>0$ s.t.
\begin{equation}
    \label{eq:m_I.inv' inf>0}
    \mathbb{P}\left\{\min_{1\leq k\leq p}\inf_{t\in\hat{\mathcal{T}}_k^{-1} }(\hat{m}_{I,k}^{-1})'(t)\geq C\right\}\rightarrow 1,
\end{equation}
where $\hat{\mathcal{T}}_k^{-1}=[\min_{1\leq i\leq N}\tilde{m}_k(i/N),\max_{1\leq i\leq N}\tilde{m}_k(i/N)]$.
\end{lem}
\begin{proof}[Proof of Lemma \ref{lem:m_I.inv' inf>0}.]
     For any $t\in\hat{\mathcal{T}}_k^{-1}=[\min_{1\leq i\leq N}\tilde{m}_k(i/N),\max_{1\leq i\leq N}\tilde{m}_k(i/N)]$, by the continuity of $\hat{m}_k$, there exists $s_k\in[0,1]$ s.t. $\tilde{m}_k(s_k)=t$. Additionally, due to continuous $K_d$ on [-1,1], there exists $i/N$ for sufficiently large $N$ s.t. 
     $$
     K_d(|\tilde{m}_k(i/N)-\tilde{m}_k(s_k)|/h_{d,k})\geq \max_{z\in[-1,1]}K_d(z)/2>0,
     $$ 
     then
    $$
        (\hat{m}^{-1}_{I,k})'(t)\geq\frac{1}{Nh_{d,k}}K_d\left(\frac{\tilde{m}_k(i/N)-\tilde{m}_k(s_k)}{h_{d,k}}\right)>0, \forall k=1,\dots,p.
    $$
    Moreover, by Lagrange's mean value, for some $\eta_{s,k}$ between $\tilde{m}_k(s_k)$ and $m_k(s_k)$,
    \begin{align}
        (m_{N,k}^{-1})'(t)&=(m_{N,k}^{-1})'\circ\tilde{m}_k(s_k)=\frac{1}{Nh_{d,k}}\sum_{i=1}^NK_d\left(\frac{m_k(i/N)-\tilde{m}_k(s_k)}{h_{d,k}}\right),\notag\\
        &=\frac{1}{Nh_{d,k}}\sum_{i=1}^N\left[K_d\left(\frac{m_k(i/N)-m_k(s)}{h_{d,k}}\right)+K_d'\left(\frac{m_k(i/N)-\eta_{s,k}}{h_{d,k}}\right)\left(\frac{m_k(s_k)-\tilde{m}_k(s_k)}{h_{d,k}}\right)\right],\notag\\
        &=(m_{N,k}^{-1})'\circ m_k(s_k)+\frac{1}{Nh_{d,k}}\sum_{i=1}^NK_d'\left(\frac{m_k(i/N)-\eta_{s,k}}{h_{d,k}}\right)\left(\frac{m_k(s_k)-\tilde{m}_k(s_k)}{h_{d,k}}\right).\notag
    \end{align}
    By similar arguments in the proof of Lemma \ref{lem:(m_I.inv-m_N.inv)'},
    $$
        \max_{1\leq k\leq p}\sup_{s_k\in[0,1]}\#\left\{i:|m_k(i/N)-\eta_{s,k}|\leq h_{d,k} \right\}=O_p(\TPR{Nh_d/M}),
    $$
    thus
    \begin{equation}
    \label{eq:m_N.inv'hatm-m_N.inv'm}
        \max_{1\leq k\leq p}\sup_{s_k\in[0,1]}|(m_{N,k}^{-1})'\circ\tilde{m}_k(s_k)-(m_{N,k}^{-1})'\circ m_k(s_k)|=O_p(\TPR{R_n/Mh_d}).
    \end{equation}
    Note that for uniformly $s\in[0,1]$,
    \begin{align}
        (m_{N,k}^{-1})'\circ m_k(s)&=\int_0^1\frac{1}{h_{d,k}}K_d\left(\frac{m_k(u)-m_k(s)}{h_{d,k}}\right)\mathrm{d}u+o(1)\notag\\
        &= \int_{\frac{m_k(0)-m_k(s)}{h_{d,k}}}^{\frac{m_k(1)-m_k(s)}{h_{d,k}}}(m_k^{-1})'\circ(m_k(s)+zh_{d,k}) K_d\left(z\right)\mathrm{d}z+o(1)\notag,
    \end{align}
    using the fact $m_k(s)+zh_d\in[m_k(0),m_k(1)]$ and $(m_k^{-1})'\circ m_k(t)=1/m_k'(t)$,
    \begin{align}
        (m_{N,k}^{-1})'\circ m_k(s)&\geq \frac{1}{\sup_{t\in[0,1]}m_k'(t)} \left(\int_{\frac{m_k(0)-m_k(s)}{h_d}}^{\frac{m_k(s)-m_k(s)}{h_d}}K_d(z)\mathrm{d}z+\int^{\frac{m_k(1)-m_k(s)}{h_{d,k}}}_{\frac{m_k(s)-m_k(s)}{h_d}}K_d(z)\mathrm{d}z \right)+o(1),\notag\\
        &\geq \frac{1}{\sup_{t\in[0,1]} m_k'(t)}\cdot \frac{1}{2}+o(1).\label{eq:inf m_N.inv'}
    \end{align}
Combining \eqref{eq:m_N.inv'hatm-m_N.inv'm}, \eqref{eq:inf m_N.inv'}, \eqref{eq:rate of (m_I.inv-m_N.inv)'} in Lemma \ref{lem:(m_I.inv-m_N.inv)'} and $R_n/h_d=o(1)$ in Assumption \hyperref[(C2)]{(C2)}, for any $t\in\hat{\mathcal{T}}_k^{-1}$,
\begin{align}
    (\hat{m}_{I,k}^{-1})'(t)&=(m_{N,k}^{-1})'\circ m_k(s_k)+ (m_{N,k}^{-1})'\circ \tilde{m}_k(s_k)-(m_{N,k}^{-1})'\circ m_k(s_k)+(\hat{m}_{I,k}^{-1})'(t)-(m_{N,k}^{-1})'(t),\notag\\
    &\geq \frac{1}{\sup_{t\in[0,1]} m_k'(t)}\cdot \frac{1}{2}+o(1).\notag
\end{align}
By the uniformly bounded Lipschitz condition \hyperref[(A1)]{(A1)}, $\max_{1\leq k\leq p}\sup_{t\in[0,1]}m_k'(t)$ is bounded and the small $o(\cdot)$ for $k$ and $t$ can be uniformly bounded, thus \eqref{eq:m_I.inv' inf>0} holds.
\end{proof}

\begin{lem}{ (Gloabal version for Lemma 2.1 and 2.2 in \cite{dette2006simple}) }
\label{lem:m_N}If \hyperref[(A1)]{(A1)}-\hyperref[(A2)]{(A2)} and \hyperref[(C1)]{(C1)} are satisfied, then
\begin{align}
    &(i).\quad\max_{1\leq k\leq p}\sup_{t\in[m(0)+h_{d,k},m(1)-h_{d,k}]}|m_{N,k}^{-1}(t)-m_k^{-1}(t)|=O(\TPR{h_d^2})
    \label{eq:m_N.inv-m.inv}\\
    &(ii).\quad\max_{1\leq k\leq p}\sup_{t\in \hat{\mathcal{T}} }|m_{N,k}(t)-m_k(t)|=O(\TPR{h_d^2}).    \label{eq:m_N-m}
\end{align}
\end{lem}

\begin{proof}[Proof of Lemma \ref{lem:m_N}.]
By simple calculation, for any $1\leq k\leq p$ and $t\in(m_k(0),m_k(1))$
$$
    (m_k^{-1})^{\prime}(t)=\frac{1}{m_k^{\prime}(m_k^{-1}(t))},\quad
    (m_k^{-1})^{\prime\prime}(t)=-\frac{m_k^{\prime\prime}(m_k^{-1}(t))}{(m_k^{\prime}(m_k^{-1}(t)))^3 }
$$
then we have
\TPR{$$
|(m_k^{-1})^{\prime\prime}(m(t))|= \left|\frac{m_k^{''}(t)}{(m_k'(t))^3}\right|,|(m_k^{-1})^{\prime}(m(t))|= \left|\frac{1}{m_k^\prime(t)}\right|,
$$}
which yields $(m_k^{-1})^{'}(t)$ and $(m_k^{-1})^{''}(t)$ are universally bounded for any $1\leq k\leq p,t\in[0,1]$ by Assumption \hyperref[(A1)]{(A1)} and \hyperref[(A2)]{(A2)}. 

For the sake of brevity, since $k$ will be fixed in the subsequent analysis, we omit it in the subscripts for short. That is we omit the dependence on $k$ when no confusion arises. 

Proof of $(i)$. Note that
\begin{equation}
    \sup_{t\in \mathbb{R} }|m_N^{-1}(t)-A_{h_d}(t)|= O\left(\frac{1}{N h_d}\right)\label{eq:Rieman approx}
\end{equation}
where 
\begin{equation}
A_{h_d}(t)=:\int_0^1 \int_{-\infty}^t K_d\left(\frac{m(x)-u}{h_d}\right) \frac{1}{h_d} \mathrm{~d} u \mathrm{~d} x.   \label{eq:Ahd def} 
\end{equation}

For $t\in[m(0)+h_d,m(1)-h_d]$, note that kernel $K_d$ is zero outside the interval $[-1,1]$, then
$$
\begin{aligned}
A_{h_d}\left(t\right) &=\int_0^{m^{-1}\left(t+h_d\right)} \int_{m(x)-h_d}^t K_d\left(\frac{m(x)-u}{h_d}\right) \frac{\mathrm{d} u}{h_d} \mathrm{~d} x \\
&=\int^{m^{-1}(t-h_d)}_0 \int_{\frac{m(x)-t}{h_d}}^1 K_d(v)\mathrm{d}v\mathrm{~d} x+\int_{m^{-1}(t-h_d)}^{m^{-1}(t+h_d)}\int_{m(x)-h_d}^t K_d\left(\frac{m(x)-u}{h_d}\right) \frac{\mathrm{d} u}{h_d} \mathrm{~d} x\\
&= m^{-1}\left(t-h_d\right)+h_d \int_{(m(0)-t) / h_d\bigvee -1}^{(m(1)-t) / h_d\bigwedge 1} \left(m^{-1}\right)^{\prime}\left(t+z h_d\right) \int_z^1 K_d(v) \mathrm{d} v \mathrm{~d} z,
\end{aligned}
$$ 
where the the second identity is obtained by setting $v=(m(x)-u)/h_d$ and then $z=(m(x)-t)/h_d$. Note that $\int_{-1}^1\int_z^1K_d(v) \mathrm{d} v \mathrm{~d} z =1$, we can obtain from Taylor expansion for $\xi_{t,z}\in (t-|z|h_d,t+|z|h_d)$, 
\begin{align}
A_{h_d}\left(t\right) & =m^{-1}\left(t-h_d\right)+h_d \int_{-1}^1\left(m^{-1}\right)^{\prime}\left(t+z h_d\right) \int_z^1 K_d(v) \mathrm{d} v \mathrm{~d} z,\notag\\
& =m^{-1}\left(t\right)+h_d^2\left(m^{-1}\right)^{\prime \prime}(\xi_{t,z})\left\{\frac{1}{2}+\int_{-1}^1 z \int_z^1 K_d(v) \mathrm{d} v \mathrm{~d} z\right\},\notag \\
& =m^{-1}(t)+\frac{1}{2}\kappa_2\left(K_d\right) h_d^2\left(m^{-1}\right)^{\prime \prime}(\xi_{t,z}),\label{eq:m_N.inv Ahd(t)}
\end{align}
where $\kappa_2(K_d)=\int v^2K_d(v)\mathrm{d}v $ and the last identity is obtained from $\int_{-1}^1 z \int_z^1 K_d(v) \mathrm{d} v \mathrm{~d} z=\frac{1}{2} \kappa_2(K_d)-\frac{1}{2}$.
Then $\sup_{t\in[m(0)+h_d,m(1)-h_d]}|A_{h_d}(t)-m^{-1}(t)|=O(\TPR{h_d^2})$.

Proof of $(ii)$. As for $m_N-m$, let $H_1=m^{-1}$, $H_2=m_N^{-1}$ in \eqref{e15}. By similar arguments in \eqref{e16} with Taylor's expansion
\begin{equation}
m_N(t)-m(t)=\Phi\left(m_N^{-1}\right)-\Phi\left(m^{-1}\right)=Q_t(1)-Q_t(0)=Q^{\prime}\left(\lambda_t^*\right) 
\label{eq:Q(1)-Q(0)}
\end{equation}
for some $\lambda_t^* \in[0,1]$. Recall \eqref{e15}
\begin{equation}
    Q_t^{\prime}\left(\lambda_t^*\right)=-\frac{m_N^{-1}-m^{-1}}{\left(m^{-1}+\lambda_t^*\left(m_N^{-1}-m^{-1}\right)\right)^{\prime}} \circ\left(m^{-1}+\lambda_t^*\left(m_n^{-1}-m^{-1}\right)\right)^{-1}(t). \label{eq:Q(lambda)}
\end{equation}
Denote $t_n=\left(m^{-1}+\lambda_t^*\left(m_N^{-1}-m^{-1}\right)\right)^{-1}(t)$, we have
$$
Q_t^{\prime}\left(\lambda_t^*\right)=\left[-\frac{m_N^{-1}-m^{-1}}{\left(m^{-1}\right)^{\prime}} \circ t_n\right]\cdot\left[\frac{\left(m^{-1}\right)^{\prime}}{\left(m^{-1}+\lambda_t^*\left(m_N^{-1}-m^{-1}\right)\right)^{\prime}}\circ t_n\right].
$$
Note that $1/(m^{-1})^\prime$ is bounded and $\sup_{t\in(m(0),m(1))}|m_N^{-1}(t)-m^{-1}(t)|\rightarrow0$, by \eqref{eq:m_N.inv-m.inv},
$$
\sup_{t\in [m^{-1}(m(0)+h_d),m^{-1}(m(1)-h_d)] }|Q_t^{\prime}(\lambda_t^*)| \leq C \sup_{t\in[m(0)+h_d,m(1)-h_d]}|m_N^{-1}(t)-m^{-1}(t)|/\TPR{M}=O(\TPR{h_d^2}).
$$
By Assumption \hyperref[(A2)]{(A2)}, for small enough $h_d$
$$
m^{-1}(m(0)+h_d)\leq 0+h_d\leq h_d\log h_d^{-1},
$$
$$
m^{-1}(m(1)-h_d)\geq 1-h_d\geq 1- h_d\log h_d^{-1},
$$
thus $\sup_{t\in \hat{\mathcal{T}} }|Q_t^{\prime}(\lambda_t^*)|\leq \sup_{t\in [m^{-1}(m(0)+h_d),m^{-1}(m(1)-h_d)] }|Q_t^{\prime}(\lambda_t^*)| =O(\TPR{h_d^2} ),$ which together with \eqref{eq:Q(1)-Q(0)} and \eqref{eq:m_N.inv-m.inv} shows \eqref{eq:m_N-m}.
\end{proof}

\subsection{Uniform bounds for local linear estimates}

\begin{lem}
\label{local linear stochastic expansion}
If \hyperref[(A1)]{(A1)}, \hyperref[(C1)]{(C1)} hold and $h_r\rightarrow 0,nh_r\rightarrow \infty$, then we have following results:\\
(i). For $\mathfrak{T}_{n,k}=\left[h_{r,k}, 1-h_{r,k}\right]$ and $\kappa_l=\int |u|^lK_r(u)du$,
    \begin{equation}
    \max_{1\leq k\leq p}\sup_{t\in\mathfrak{T}_{n,k}}\left|\hat{m}_{k}(t)-m_k(t)-\frac{\kappa_2 m^{\prime\prime}_k(t)}{2} h_{r,k}^2-\frac{1}{n h_{r,k}} \sum_{i=1}^n e_{i,k} K_r\left(\frac{t_i-t}{h_{r,k}}\right)\right|=O\left(h_r^3+\frac{1}{n h_r}\right),
    \label{e10}
    \end{equation}
(ii). For $\mathfrak{T}_{n,k}'=[0,h_{r,k})\cup(1-h_{r,k},1]$,
\begin{align}
        \max_{1\leq k\leq p}\sup_{t \in \mathfrak{T}_{n,k}^{\prime}} \Bigg|\hat{m}_{k}(t)-m_k(t)&-\frac{h_{r,k}^2}{2c_k(t)} m_k^{\prime\prime}(t)\left(\nu_{2, h_{r,k}}^2(t)-\nu_{1, k}(t) \nu_{3, k}(t)\right)- \notag\\
    &\frac{1}{n h_{r,k}} \sum_{i=1}^ne_{i,k} K_r^*\left(\frac{t_i-t}{h_{r,k}},t\right)\Bigg|=O\left(h_r^3+\frac{1}{n h_r}\right), \label{e11}
\end{align}
where
$$
K_r^*\left(\frac{t_i-t}{h_{r,k}},t\right)=\frac{\nu_{2,k}(t)K_r\left(\frac{t_i-t}{h_{r,k}}\right)-\nu_{1,k}(t)K_r\left(\frac{t_i-t}{h_{r,k}}\right)\left(\frac{t_i-t}{h_{r,k}}\right)}{c_k(t)}
$$
$\nu_{j, k}(t)=\int_{-t / h_{r,k}}^{(1-t)/ h_{r,k}} x^j K(x) d x$ and $c_k(t)=\nu_{0, k}(t) \nu_{2, k}(t)-\nu_{1, k}^2(t)$.\\
\TPR{(iii). Specify the local linear estimator using bandwidth $h_{r,k}$ as $\hat{m}_{k,h_{r,k}}(t)$. Then for the jackknife corrected local linear estimator $\tilde{m}_k(t)=:2\hat{m}_{k,h_{r,k}/\sqrt{2}}(t)-\hat{m}_{k,h_{r,k}}(t)$, we have
\begin{equation}
    \max_{1\leq k\leq p}\sup_{t\in\mathfrak{T}_{n,k}}\left|\tilde{m}_{k}(t)-m_k(t)-\frac{1}{n h_{r,k}} \sum_{i=1}^n e_{i,k} \tilde{K}_r\left(\frac{t_i-t}{h_{r,k}}\right)\right|=O\left(h_r^3+\frac{1}{n h_r}\right),
    \label{e10.j}
\end{equation}
where kernel function $\tilde{K}(x)=2\sqrt{2}K_r(\sqrt{2}x)-K_r(x)$.}
\end{lem}

\begin{proof}[Proof of Lemma \ref{local linear stochastic expansion}.]
Proof of (i). Recall the definition of $\hat{m}_k$ in \eqref{eq:high-dim ll def} and define the notations $\hat{m}_{h_{r,k}}(t)$ refers to local linear estimator $\hat{m}_k(t)$ using bandwidth $h_{r,k}$, 
$$
\begin{aligned}
& S_{n, l,k}(t)=\frac{1}{n h_{r,k}} \sum_{i=1}^n\left(\frac{t_i-t}{h_{r,k}}\right)^l K_{r}\left(\frac{t_i-t}{h_{r,k}}\right), \\
& R_{n, l,k}(t)=\frac{1}{n h_{r,k}} \sum_{i=1}^n y_{i,k}\left(\frac{t_i-t}{h_{r,k}}\right)^l K_{r}\left(\frac{t_i-t}{h_{r,k}}\right),
\end{aligned}
$$
$(l=0,1, \ldots)$, then we obtain the representation
\begin{equation}
\left[\begin{array}{c}
\hat{m}_{h_{r,k}}(t) \\
h_{r,k} \hat{m^{\prime}}_{h_{r,k}}(t)
\end{array}\right]=\left[\begin{array}{cc}
S_{n, 0,k}(t) & S_{n, 1,k}(t) \\
S_{n, 1,k}(t) & S_{n, 2,k}(t)
\end{array}\right]^{-1}\left[\begin{array}{c}
R_{n, 0,k}(t) \\
R_{n, 1,k}(t)
\end{array}\right]=: S_{n,k}^{-1}(t) R_{n,k}(t) \label{eq:solution to ll}
\end{equation}
for the local linear estimate $\hat{m}_{h_{r,k}}$, where the last identity defines the $2 \times 2$ matrix $S_{n,k}(t)$ and the vector $R_{n,k}(t)$ in an obvious manner. Note that
$$
\max_{1\leq k\leq p}\sup_{t\in\mathfrak{T}_{n,k}}|S_{n, 0,k}(t)-1|=O\left(\frac{1}{n h_{r}}\right),
$$
$$
\max_{1\leq k\leq p}\sup_{t\in\mathfrak{T}_{n,k}}|S_{n, 1,k}(t)|=O\left(\frac{1}{n h_{r}}\right),
$$
and $m_k\in C^3[0, 1]$ for each $k = 1, 2, . . . , p$ with universal Lipchitz constant $L > 0$ by Assumption \hyperref[(A1)]{(A1)}, elementary calculation and a Taylor expansion can yield that
\begin{equation}
S_{n,k}(t)\left[\begin{array}{c}
\hat{m}_{h_{r,k}}(t)-m_k(t) \\
h_{r,k}\left(\hat{ m_k^{\prime}}_{h_{r,k}}(t)-m_k^{\prime}(t)\right)
\end{array}\right]=\left[\begin{array}{c}
\frac{1}{n h_{r,k}} \sum_{i=1}^n e_{i,k} K_{r}\left(\frac{t_i-t}{h_{r,k}}\right)\\
\frac{1}{n h_{r,k}} \sum_{i=1}^n e_{i,k} K_{r}\left(\frac{t_i-t}{h_{r,k}}\right)\left(\frac{t_i-t}{h_{r,k}}\right)
\end{array}\right]+O(h_r^3+\frac{1}{nh_r})\notag
\end{equation}
uniformly with respect to $t \in \mathfrak{T}_{n,k}$ and $k=1,\dots,p$. Thus 
$$
\max_{1\leq k\leq p}\sup_{t\in\mathfrak{T}_{n,k}}\left|\hat{m}_{k}(t)-m_k(t)-\frac{\kappa_2 m^{\prime\prime}_k(t)}{2} h_{r,k}^2-\frac{1}{n h_{r,k}} \sum_{i=1}^n e_{i,k} K_r\left(\frac{t_i-t}{h_{r,k}}\right)\right|=O\left(h_r^3+\frac{1}{n h_r}\right).
$$

Proof of (ii). For any $t\in[0,h_{r,k})\cup(1-h_{r,k},0]$, using \eqref{eq:solution to ll} and a Taylor expansion yields
\begin{align}
    &S_{n,k}(t)\left[\begin{array}{c}
\hat{m}_k(t)-m_k(t) \\
h_{r,k}\left(\hat{m^{\prime}}_{k}(t)-m_k^\prime(t)\right)
\end{array}\right]\notag\\
&=\left[\begin{array}{c}
\frac{1}{n h_{r,k}} \sum_{i=1}^n e_{i,k} K_{r}\left(\frac{t_i-t}{h_{r,k}}\right)+\frac{h_{r,k}^2}{2} \nu_{2, k}(t) m_k^{\prime\prime}(t) \\
\frac{1}{n h_{r,k}} \sum_{i=1}^n e_{i,k} K_{r}\left(\frac{t_i-t}{h_{r,k}}\right)\left(\frac{t_i-t}{h_{r,k}}\right)+\frac{h_{r,k}^2}{2} \nu_{3, k}(t) m_k^{\prime\prime}(t)
\end{array}\right]+O\left(h_{r}^3+\frac{1}{n h_{r}}\right)\label{eq:expansion ll}
\end{align}
uniformly with respect to $t \in\left[0, h_{r,k}\right) \cup\left(1-h_{r,k}, 1\right]$ and $k=1,\dots,p$. On the other hand, uniformly with respect to $t \in\left[0, h_{r,k}\right) \cup\left(1-h_{r,k}, 1\right]$ and $k=1,\dots,p$, we have that
$$
S_{n,k}(t)=\left[\begin{array}{ll}
\nu_{0, k}(t) & \nu_{1, k}(t) \\
\nu_{1, k}(t) & \nu_{2, k}(t)
\end{array}\right]+O\left(\frac{1}{n h_{r}}\right) .\label{eq:boundary S}
$$
Therefore, combining \eqref{eq:expansion ll} and \eqref{eq:boundary S}, it follows that
$$
\begin{aligned}
c_k(t)\left(\hat{m}_{k}(t)-m_k(t)\right)= & \frac{1}{n h_{r,k}} \sum_{i=1}^n\left[\nu_{2, k}(t)-\nu_{1, k}(t)\left(\frac{t_i-t}{h_{r,k}}\right)\right] e_{i,k} K_{r}\left(\frac{t_i-t}{h_{r,k}}\right)+ \\
& \frac{h_{r,k}^2}{2} m_k^{\prime\prime}(t)\left(\nu_{2, h_{r,k}}^2(t)-\nu_{1, k}(t) \nu_{3, k}(t)\right)+O\left(h_{r}^3+\frac{1}{n h_{r}}\right)
\end{aligned} 
$$
uniformly with respect to $t \in\left[0, h_{r,k}\right) \cup\left(1-h_{r,k}, 1\right]$ and $k=1,\dots,p$. Note that $\nu_{1,k},\nu_{2,k},\nu_{3,k},m^{\prime\prime}(t)$ are bounded and $\inf_{t\in[0,1]}c_k(t)\geq(\kappa_2-\kappa_1^2)/4>0$ by \eqref{eq:lower c(t)}, thus \eqref{e11} holds.

\TPR{Proof of (iii). Note that when $t\in\mathfrak{T}_{n,k}$, 
$$
\tilde{K}_r\left(\frac{t_i-t}{h_{r,k}}\right)=2\sqrt{2}K_r\left(\frac{t_i-t}{h_{r,k}/\sqrt{2}}\right)-K_r\left(\frac{t_i-t}{h_{r,k}}\right).
$$
Then apply results of (i),
\begin{align}
    &\max_{1\leq k\leq p}\sup_{t\in\mathfrak{T}_{n,k}}\left|\tilde{m}_{k}(t)-m_k(t)-\frac{1}{n h_{r,k}}\sum_{i=1}^n e_{i,k} \tilde{K}_r\left(\frac{t_i-t}{h_{r,k}}\right)\right|\notag \\
    &=\max_{1\leq k\leq p}\sup_{t\in\mathfrak{T}_{n,k}}\left|\tilde{m}_{k}(t)-(2m_k(t)-m_k(t))-\frac{1}{n h_{r,k}} \sum_{i=1}^n e_{i,k} \left[2\sqrt{2}K_r\left(\frac{t_i-t}{h_{r,k}/\sqrt{2}}\right)-K_r\left(\frac{t_i-t}{h_{r,k}}\right)\right]\right|\notag\\
    &\leq \max_{1\leq k\leq p}\sup_{t\in\mathfrak{T}_{n,k}}2\left|\hat{m}_{k,h_{r,k}/\sqrt{2}}(t)-\frac{\kappa_2 m^{\prime\prime}_k(t)}{2} (h_{r,k}/\sqrt{2})^2-m_k(t)-\frac{1}{n h_{r,k}/\sqrt{2}} \sum_{i=1}^n e_{i,k} K_r\left(\frac{t_i-t}{h_{r,k}/\sqrt{2}}\right)\right|\notag\\
    &+ \max_{1\leq k\leq p}\sup_{t\in\mathfrak{T}_{n,k}}\left|\hat{m}_{k,h_{r,k}}(t)-m_k(t)-\frac{\kappa_2 m^{\prime\prime}_k(t)}{2} h_{r,k}^2-\frac{1}{n h_{r,k}} \sum_{i=1}^n e_{i,k} K_r\left(\frac{t_i-t}{h_{r,k}}\right)\right|\notag\\
    &=O\left(h_r^3+\frac{1}{n h_r}\right),
\end{align}
which yields \eqref{e10.j}.}
\end{proof}

\begin{lem}
    \label{lem:High-dim ll}
    Under Assumptions \hyperref[(A1)]{(A1)}, \hyperref[(B1)]{(B1)}-\hyperref[(B3)]{(B3)} and \hyperref[(C1)]{(C1)}-\hyperref[(C2)]{(C2)}, then for local linear estimator $\hat{m}_k$ defined in \eqref{eq:high-dim ll def}, we have
    \begin{equation}
    \label{eq:rate1 in lem:High-dim ll}
        \max_{1\leq k\leq p}\sup_{t\in[0,1]}\left|\hat{m}_k(t)-m_k(t)\right|=O_p\left(h_r^2+\frac{\log^{3}n}{\sqrt{nh_r}}\right).
    \end{equation}
\end{lem}
\begin{proof}[Proof of Lemma \ref{lem:High-dim ll}.]
Combining \eqref{e10}, \eqref{e11} in Lemma \ref{local linear stochastic expansion} and $K_r^*\left(\frac{i / n-t}{h_{r,k}},t\right)$ defined in \eqref{eq:K^*}, we have 
    \begin{equation}
        \label{eq:stochastic in lem:High-dim ll}
        \max_{1\leq k\leq p}\sup_{t \in\left[0,1\right]}\left|\hat{m}_k-m_k(t)-\sum_{i=1}^n \frac{1}{n h_{r,k}} K_r^*\left(\frac{i / n-t}{h_{r,k}},t\right) e_{i,k}\right|=O\left(h_{r}^2\right).
    \end{equation}
    For any $\lambda\geq 2$,
    \begin{align}
        \left\|\sum_{i=1}^n \frac{1}{n h_{r,k}} K_r^*\left(\frac{i / n-t}{h_{r,k}},t\right) e_{i,k}\right\|_\lambda&=\left\|\frac{1}{n h_{r,k}} \sum_{i=1}^n \sum_{m=0}^{\infty} \mathcal{P}_{i-m} K_r^*\left(\frac{i / n-t}{h_{r,k}},t\right) e_{i,k}\right\|_\lambda\notag\\
        &\leq \frac{1}{nh_{r,k}}\sum_{m=0}^{\infty}\left\|\sum_{i=1}^n\mathcal{P}_{i-m}K_r^*\left(\frac{i / n-t}{h_{r,k}},t\right)e_{i,k}\right\|_\lambda.\label{eq: nonstationary term in lem:High-dim ll}
    \end{align}
    Recalling Definition \ref{def:phd}, denote $e_{i,k}^{(i-m)}=\mathbf{G}_{k}(t_i,\Upsilon_{i,i-m})$ where $\Upsilon_{i,i-m}=(\dots,\hat \varepsilon_{i-m}',\hat \varepsilon_{i-m+1},\dots,\hat \varepsilon_i)$, then 
    \begin{align}
        \left\|\mathcal{P}_{i-m}e_{i,k}\right\|_\lambda^\lambda&=\left\|\mathbb{E}\left(e_{i,k}-e_{i,k}^{(i-m)}|\Upsilon_{i-m}\right)\right\|_\lambda^\lambda\notag\\
        &\leq\left\|e_{i,k}-e_{i,k}^{(i-m)}\right\|_\lambda^\lambda=\mathbb{E}\left( \left|e_{i,k}-e_{i,k}^{(i-m)}\right|^{\lambda-1/2}\left|e_{i,k}-e_{i,k}^{(i-m)}\right|^{1/2}  \right)\notag\\
        &\leq \mathbb{E}\left( \left|e_{i,k}-e_{i,k}^{(i-m)}\right|^{2\lambda-1}\right)^{1/2}\mathbb{E}\left(\left|e_{i,k}-e_{i,k}^{(i-m)}\right|  \right)^{1/2}\notag\\
        &\leq (2\|e_{i,k}\|_{2\lambda-1})^{\frac{2\lambda-1}{2}} \delta_1(\mathbf{G}_k,m)\leq 2(M(2\lambda-1) )^{\frac{2\lambda-1}{2}}\Theta(p) \chi^m ,\label{eq:delta_r(m)}
    \end{align}
    where the last line is obtained from the moment property of sub-exponential random variables (see Proposition E.4 of \cite{wuzhou2023multiscale}), i.e.,  $\|e_{i,k}\|_\lambda\leq K\lambda,\forall \lambda\geq 1$  and Assumption \hyperref[(B1)]{(B1)} $\delta_1(\mathbf{G},m)\leq\Theta(p) \chi^m$ for some $\chi\in(0,1)$. 
    Note that $\nu_{1,k},\nu_{2,k},\nu_{3,k}$ are bounded and $\inf_{t\in[0,1]}c_k(t)\geq(\kappa_2-\kappa_1^2)/4>0$ by \eqref{eq:lower c(t)}, there exists universal $C>0$ s.t. 
    $$
    \max_{1\leq k\leq p}\sup_{t\in[0,1]}\sum_{i=1}^n \left[K_r^*\left(\frac{i / n-t}{h_{r,k}},t\right)\right]^2\leq Cnh_r.
    $$
    Then by Burkholder's inequality and \eqref{eq:delta_r(m)}, for some constant $C>0$, it yields
    \begin{align}
    \max_{1\leq k\leq p}\sup_{t\in[0,1]}\left\|\sum_{i=1}^n\mathcal{P}_{i-m}K_r^*\left(\frac{i / n-t}{h_{r,k}},t\right)e_{i,k}\right\|_\lambda^2&\leq (\lambda-1)^2\sum_{i=1}^n \left[K_r^*\left(\frac{i / n-t}{h_{r,k}},t\right)\right]^2\|\mathcal{P}_{i-m}e_{i,k}\|_\lambda^2\notag\\ 
    &\leq C \lambda^4 nh_r \Theta(p)^{2/r}\chi^{2m/r}.\label{eq:delta_r(m)^2}
    \end{align}
    Combining \eqref{eq:delta_r(m)^2} and \eqref{eq: nonstationary term in lem:High-dim ll}, for some constant $C>0$,
    \begin{equation}
    \max_{1\leq k\leq p}\sup_{t\in[0,1]}\left\|\sum_{i=1}^n \frac{1}{n h_{r,k}} K_r^*\left(\frac{i / n-t}{h_{r,k}},t\right) e_{i,k}\right\|_\lambda\leq C\frac{1}{\sqrt{nh_r}}\lambda^{2}\Theta^{1/\lambda}(p)\sum_{m=0}^{\infty}\chi^{m/\lambda}\leq C\frac{\lambda^{3}\Theta^{1/\lambda}(p)}{\sqrt{nh_r}}.\label{eq:high-dim delta}
    \end{equation}
Using Proposition B.1. of \cite{dette2019change}, we have
\begin{equation}
    \max_{1\leq k\leq p}\left\|\sup_{t\in[0,1]}\sum_{i=1}^n \frac{1}{n h_{r,k}} K_r^*\left(\frac{i / n-t}{h_{r,k}},t\right) e_{i,k}\right\|_\lambda=O\left(\frac{h_r^{-1/\lambda}\lambda^{3}\Theta^{1/\lambda}(p)}{\sqrt{nh_r}}\right)\label{eq:nonstationary rate in lem:High-dim ll}
\end{equation}
Based on universal Lipstchitz condition \hyperref[(A1)]{(A1)}, \eqref{eq:stochastic in lem:High-dim ll}, \eqref{eq:nonstationary rate in lem:High-dim ll} and \eqref{eq:stochastic in lem:High-dim ll}, we have
\begin{equation*}
    \label{eq:rate0 in lem:High-dim ll}
        \max_{1\leq k\leq p}\left\|\sup_{t\in[0,1]}\left|\hat{m}_k(t)-m_k(t)\right|\right\|_\lambda=O\left(h_r^2+\frac{h_r^{-1/\lambda}\lambda^{3}\Theta^{1/\lambda}(p)}{\sqrt{nh_r}}\right).
    \end{equation*}
Then we have
\begin{align}
    \left\| \max_{1\leq k\leq p}\sup_{t\in[0,1]}\left|\hat{m}_k(t)-m_k(t)\right| \right\|_\lambda &\leq \left\{\sum_{k=1}^p\mathbb{E}\sup_{t\in[0,1]}\left|\hat{m}_k(t)-m_k(t)\right|^\lambda \right\}^{1/\lambda},\notag\\
    &\leq p^{1/\lambda}\max_{1\leq k\leq p}\left\| \sup_{t\in[0,1]}\left|\hat{m}_k(t)-m_k(t)\right| \right\|_\lambda\notag\\
    &=O\left(p^{1/\lambda}(h_r^2+\frac{h_r^{-1/\lambda}\lambda^{3}\Theta^{1/\lambda}(p)}{\sqrt{nh_r}})\right).\label{eq:final deduction in High-dim ll}
\end{align}
Setting $\lambda=\log n$, we have $p^{1/\lambda}, h_r^{-1/\lambda},\Theta^{1/\lambda}(p)=O(1)$ by Assumption \hyperref[(B1)]{(B1)} and \hyperref[(C2)]{(C2)}, thus \eqref{eq:rate1 in lem:High-dim ll} holds.
\end{proof}

\begin{lem}
    \label{lem:multi-dim ll}
    In model \eqref{Time-varying_n}, if $\mf m(t)$  follows \hyperref[(A1)]{(A1)} and Assumptions \hyperref[(B1')]{(B1')}-\hyperref[(B4')]{(B4')} hold, then for local linear estimator $\hat{\mf m}(\cdot)$ defined in \eqref{local linear}, we have
    \begin{equation}
    \label{eq:rate of ll in multi-dim}
        \sup_{t\in[0,1]}\left|\hat{\mf m}(t)-\mf m(t)\right|=O_p(h_r^2+\frac{\log n}{\sqrt{nh_r}}).
    \end{equation}
\end{lem}

\begin{proof}[Proof of Lemma \ref{lem:multi-dim ll}.]
    By Assumption \hyperref[(A1)]{(A1)} and Taylor's expansion, if $|t_j-t|\leq h_r$, ${\mf m}\left(t_j\right)={\mf m}(t)+\mf{m}^{\prime}(t)\left(t-t_j\right)+\left\{\mf{m}^{\prime \prime}(t) / 2+O\left(h_r\right)\right\}(t-t_j)^2$. Since $K_r$ has support $[-1,1]$, by \eqref{local linear},
$$
\mf S_n(t)\{\hat{\boldsymbol{\eta}}(t)-\boldsymbol{\eta}(t)\}=\left(\begin{array}{l}
h_r^2 \mf S_{n, 2}(t)\left\{\boldsymbol{m}^{\prime \prime}(t)+O\left(h_r\right)\right\} / 2 \\
h_r^2 \mf S_{n, 3}(t)\left\{\boldsymbol{m}^{\prime \prime}(t)+O\left(h_r\right)\right\} / 2
\end{array}\right)+\left(\begin{array}{l}
\mf T_{n, 0}(t) \\
\mf T_{n, 1}(t)
\end{array}\right),
$$
where $\hat{\eta}(t)=(\hat{\mf m}^\top (t), h_r\hat{\boldsymbol{m}}^{\prime \top}(t))^\top$, $\boldsymbol{\eta}(t)=\left(\boldsymbol{m}^\top(t), h_r \boldsymbol{m}^{\prime \mathrm{T}}(t)\right)^\top, \mf T_n(t)=\left(\mf T_{n, 0}^\top(t), \mf T_{n, 1}^\top(t)\right)^\top$ and
$$
\mf S_{n, l}(t)=\left(n h_r\right)^{-1} \sum_{i=1}^n \mathbf{x}_i \mathbf{x}_i^{\top}\left\{\left(t_i-t\right) / h_r\right\}^l K_{r}\left(\frac{t_i-t}{h_r}\right),
$$
$$
\mf T_{n, l}(t)=\left(n h_r\right)^{-1} \sum_{i=1}^n \mathbf{x}_i e_i\left\{\left(t_i-t\right) / h_r\right\}^l K_{r}\left(\frac{t_i-t}{h_r}\right),
$$
$$
\mf S_n(t)=\left(\begin{array}{ll}
\mathbf{S}_{n, 0}(t) & \mathbf{S}_{n, 1}^{\mathrm{T}}(t) \\
\mathbf{S}_{n, 1}(t) & \mathbf{S}_{n, 2}(t)
\end{array}\right).
$$
Let $\tilde{\mf M}(t)=\operatorname{diag}\left\{\mf M(t), \kappa_2\mf M(t)\right\}, \chi_n=n^{-1 / 2}h_r^{-1}+h_r$. By Lemma 6 in \cite{zhou2010simultaneous} and Lipschitz continuity of $\mf M(t)$, we have $\sup_{t\in[0,1]}|\mf{S}_n(t)-\tilde{\mf M}(t)|=O_{p}\left(\chi_n\right)$. Then we have
\begin{equation}
    \mf S_n(t)\{\hat{\boldsymbol{\eta}}(t)-\boldsymbol{\eta}(t)\}=\left(\begin{array}{c}
\kappa_2 \mf M(t) \boldsymbol{m}^{\prime \prime}(t) h_r^2 / 2+O_{p}\left(h_r^2 \chi_n\right) \\
O_{p}\left(h_r^2 \chi_n\right)
\end{array}\right)+\mf T_n(t)\label{eq:multi-dim stochastic expansion of ll}
\end{equation}
with uniformly $t\in[0,1]$. As for process $\{\mf{x}_ie_i\}_{i=1}^n$, by Assumption \hyperref[(B4')]{(B4')}-\hyperref[(B5')]{(B5')} and Corollary 1 in \cite{Wu2011gaussian}, there exists independent Gaussian random vectors $\mf V_j\sim N_p(0,\mf \Lambda(t_j))$ s.t.
\begin{equation}
\label{eq:partial GA in multi-dim}
    \max_{1\leq j\leq n}\left|\sum_{i=1}^j \mf{x}_ie_i-\sum_{i=1}^j\mf V_i \right|=O_p(n^{1/4}\log^2n).
\end{equation}
Using summation by parts, it can yield
\begin{equation}
    \sup_{t }\left|\frac{1}{n h_r} \sum_{i=1}^n K_r\left(\frac{i / n-t}{h_r}\right)\left(\mf{x}_ie_i-\mf V_i\right)\right|=O_p\left(\frac{n^{1 / 4} \log ^2 n}{n h_r}\right).\label{eq:multi-dim ll GA}
\end{equation}
For Gaussian random vector $\mf V_j$, $\mathbb{E}|\mf V_j|^q\leq (q-1)!!\leq q^{q/2}$. Similar with \eqref{eq:delta_r(m)}, again apply Burkholder's inequality on martingale difference sequence  $\{\mf V_i\}_{i\in \mathbb{Z}}$ for any positive constant $q$, one can show
$$
\sup_{t}\left\| \sum_{i=1}^n\frac{1}{nh_r}K_r\left(\frac{i/n-t}{h_r}\right)\mf V_i  \right\|_q=O(\frac{q}{\sqrt{nh_r}}),\quad
\sup_{t}\left\| \sum_{i=1}^n\frac{1}{nh_r}\frac{\partial}{\partial t}K_r\left(\frac{i/n-t}{h_r}\right)\mf V_i  \right\|_q=O(\frac{qh_r^{-1}}{\sqrt{nh_r}}).
$$
By Proposition B.1 of \cite{dette2019change},
\begin{equation}
    \label{eq:multi-dim rate of Tn}
    \left\|\sup_{t} \sum_{i=1}^n\frac{1}{nh_r}K_r\left(\frac{i/n-t}{h_r}\right)\mf V_i  \right\|_q=O(\frac{qh_r^{-1/q}}{\sqrt{nh_r}}).
\end{equation}
Combining \eqref{eq:multi-dim stochastic expansion of ll}, \eqref{eq:multi-dim ll GA}, \eqref{eq:multi-dim rate of Tn} and $\sup_{t\in[0,1]}|\mf S_n(t)-\tilde{\mf M}(t)|=O_p(\chi_n)$, then \eqref{eq:rate of ll in multi-dim} holds with setting $q=\log h_r^{-1}$. 
\end{proof}

\subsection{Some auxiliary lemmas}
\begin{lem}
\label{nazarov}
(Nazarov's inequality) Let $Y=\left(Y_1, \ldots, Y_p\right)^{\prime}$ be a centered Gaussian random vector in $\mathbb{R}^p$ such that $\mathrm{E}\left[Y_j^2\right] \geq \underline{\sigma}^2$ for all $j=1, \ldots, p$ and some constant $\underline{\sigma}>0$. Then for every $y \in \mathbb{R}^p$ and $a>0$,
\begin{equation}
    \mathbb{P}(Y_j \leq y_j+a,\forall j=1,\dots,p)-\mathbb{P}(Y_j \leq y_j,\forall j=1,\dots,p) \leq \frac{a}{\underline{\sigma}}(\sqrt{2\log p}+2),
    \label{eq:nazarov}
\end{equation}
specially, for $y>0$ 
\begin{equation}
    \mathbb{P}(|Y_j| \leq y+a,\forall j=1,\dots,p)-\mathbb{P}(|Y_j| \leq y,\forall j=1,\dots,p) \leq \frac{a}{\underline{\sigma}}(\sqrt{2\log 2p}+2).
    \label{eq:nazarov-abs}
\end{equation}
\end{lem}
\begin{proof}[Proof of Lemma \ref{nazarov}.]
Proof of \eqref{eq:nazarov} can be seen in \cite{nazarov2003maximal} and \eqref{eq:nazarov-abs} can be immediately obtained by introducing $Y_{j+p}=-Y_j,j=1,\dots,p$. 
\end{proof}

\begin{lem}
    \label{lem:prob inequality} For non-negative random variables $X,Y$, $\forall t\in\mathbb{R},\delta>0$,
    \begin{equation}
        |\mathbb{P}(X\leq t)-\mathbb{P}(Y\leq t)|\leq \mathbb{P}(|X-Y|>\delta)+\mathbb{P}(t-\delta\leq  Y\leq t+\delta).\label{eq:prob ineq}
    \end{equation}
\end{lem}
\begin{proof}[Proof of Lemma \ref{lem:prob inequality}.]
    Note that
    \begin{align}
        \mathbb{P}(X\leq t)&= \mathbb{P}(X\leq t,|X-Y|>\delta)+\mathbb{P}(X\leq t,|X-Y|\leq \delta),\notag\\ 
        &\leq \mathbb{P}(|X-Y|>\delta)+\mathbb{P}(Y\leq t+\delta).\label{eq:prob decomp}
    \end{align}
    Thus $\mathbb{P}(Y\leq t-\delta)\leq \mathbb{P}(|X-Y|>\delta)+\mathbb{P}(X\leq t)$ and
    \begin{align}
                \mathbb{P}(X\leq t)-\mathbb{P}(Y\leq t)&\leq \mathbb{P}(|X-Y|>\delta)+\mathbb{P}(t<Y\leq t+\delta).\label{eq:prob ineq 1}\\
        \mathbb{P}(Y\leq t)-\mathbb{P}(X\leq t)&=\mathbb{P}(Y\leq t-\delta)-\mathbb{P}(X\leq t)+\mathbb{P}(t-\delta<Y\leq t)\notag\\
        &\leq \mathbb{P}(|X-Y|>\delta)+\mathbb{P}(t-\delta<Y\leq t)\label{eq:prob ineq 2}.
    \end{align}
    Combining \eqref{eq:prob ineq 1} and \eqref{eq:prob ineq 2}, \eqref{eq:prob ineq} can hold by
    \begin{align}
        |\mathbb{P}(Y\leq t)-\mathbb{P}(X\leq t)|&= \left(\mathbb{P}(X\leq t)-\mathbb{P}(Y\leq t)\right)\bigvee \left(\mathbb{P}(X\leq t)-\mathbb{P}(Y\leq t)\right).
    \end{align}
\end{proof}

\section{Appendix}
\label{sec:Appendix}

\subsection{Specific examples of $\mf G(u,\Upsilon_i)$}
\label{sec:specific examples}
\TPR{In the following we present examples for the high-dimensional error process satisfying the dependence structure in Section \ref{High dimensional SCB}.
\begin{eg}[MA, AR, GARCH models and time-varying generalization]
\label{eg: stationary models}
Suppose $p$-dimensional random vectors $\boldsymbol{\epsilon}_i$, $i\in\mathbb{Z}$ are i.i.d. with $\mathbb{E}\boldsymbol{\epsilon}_i=\mf 0$ and $\Upsilon_i=(\dots,\boldsymbol{\epsilon}_{i-1},\boldsymbol{\epsilon}_{i})$.
    \begin{description}
        \item[(i)] MA(1) model: $\mf e_i=\boldsymbol{\epsilon}_{i}+a_1\boldsymbol{\epsilon}_{i-1}$ where $\|\boldsymbol{\epsilon}_i\|_q\leq C\sqrt{p}$.
        \item[(ii)] AR(1) model: $\mf e_i=a_1\mf e_{i-1}+\boldsymbol{\epsilon}_{i}$ where $\|\boldsymbol{\epsilon}_i\|_q\leq C\sqrt{p}$.
        \item[(iii)] GARCH(1,1) model (1-dimension):
        $$
        e_i=\epsilon_i V_i^{1 / 2}, \quad V_i=a_1+a_2 e_{i-1}^2+a_3 V_{i-1},
        $$
        where $\epsilon_i$, $i\in\mathbb{Z}$ are i.i.d. with $\mathbb{E}\epsilon_i=0$, $\|\epsilon_i\|_{2q}<\infty$. The coefficients $a_1>0$, $a_2,a_3\geq 0$.
        \item[(iv)] Time-varying MA(1) model: $\mf G(u,\Upsilon_i)=\boldsymbol{\epsilon}_{i}+a_1(u)\boldsymbol{\epsilon}_{i-1}$ where  $\|\boldsymbol{\epsilon}_i\|_q\leq C\sqrt{p}$.
        \item[(v)] Time-varying AR(1) model: $\mf G(u,\Upsilon_i)=a_1(u)\mf G(u,\Upsilon_{i-1})+\boldsymbol{\epsilon}_{i}$ where $\|\boldsymbol{\epsilon}_i\|_q\leq C\sqrt{p}$.
        \item[(vi)] Time-varying GARCH(1,1) model (1-dimension): 
        $$
        G(u,\Upsilon_i)=\epsilon_i H^{1/2}(u,\Upsilon_i), \quad H(u,\Upsilon_i)=a_1(u)+a_2(u)G^2(u,\Upsilon_{i-1})+a_3(u)H(u,\Upsilon_{i-1}),
        $$
        where $\epsilon_i$, $i\in\mathbb{Z}$ are i.i.d. with $\mathbb{E}\epsilon_i=0$, $\|\epsilon_i\|_{2q}<\infty$. The curves  $a_1(u)>0$, $a_2(u),a_3(u)\geq 0$.
    \end{description}
\end{eg}}

\TPR{We can show the examples above follow our dependence structure $\delta_q(\mf G,k)\leq C\Theta(p) \chi^k$. Notice that models (iv),(v),(vi) can be reduced to model (i),(ii),(iii) if $a_k(u)\equiv a_k,k=1,2,3$. Therefore, all these commonly used time series models can be written as the form $\mf G(u,\Upsilon_i)$ and we only need to check the dependence structure for models (iv)-(vi). }
    
\TPR{For model (iv), $\mf G(u,\Upsilon_i)=\boldsymbol{\epsilon}_{i}+a_1(u)\boldsymbol{\epsilon}_{i-1}$. The dependence measure 
\begin{equation*}
\delta_q(\mf G,k) =
\begin{cases}
    O(\|\mf \epsilon_i\|_q), & k = 0 \\[-15pt]
    O(\sup_{u\in[0,1]}|a_1(u)|\|\mf \epsilon_i\|_q), & k = 1 \\[-15pt]
    0, & \text{else}
\end{cases}
\end{equation*}
    then $\delta_q(\mf G,k)=O(\Theta(p)\chi^k)$ for $\chi\in(0,1)$, $\Theta(p)\asymp \sqrt{p}$,  can be satisfied as long as $\sup_{u\in[0,1]}|a_1(u)|<\infty$.\\
    For model (v),  note that $\mf G(u,\Upsilon_i)=\sum_{j=0}^\infty a_1^j(u)\boldsymbol{\epsilon}_{i-j}$. If $\sup_{u\in[0,1]}|a_1(u)|<1$ then the dependence measure $\delta_q(\mf G,k)=O(\Theta(p)\chi^k)$ for $\chi\in(0,1)$, $\Theta(p)\asymp \sqrt{p}$ where constant $\chi=\sup_{u\in[0,1]}|a_1(u)|$.\\
    For model (vi), denote $\mf W(u,\Upsilon_i)=(G^2(u,\Upsilon_i),H(u,\Upsilon_i))^\top$, then 
    \begin{align*}
        \mf W(u,\Upsilon_i)&=\mf M(u,\Upsilon_i)\mf W(u,\Upsilon_{i-1})+a_1(u)(\epsilon_i^2,1)^\top, \text{ where } \mf M(u,\Upsilon_i)=\left(\begin{array}{cc}
a_2(u) \epsilon_i^2 & a_3(u) \epsilon_i^2 \\
a_2(u) & a_3(u)
\end{array}\right).
    \end{align*} 
    Elementary computation can show the operator norm $\|\mf M(u,\Upsilon_i) \|_{\text{op} }=\sqrt{\epsilon_i^4+1}\sqrt{a_2^2(u)+a_3^2(u)}$, thus
    \begin{align*}
        \left\| \mf W(u,\Upsilon_i)- \mf W(u,\Upsilon_{i,i-k})\right\|_q & =\| \mf M(u,\Upsilon_i)\cdots \mf M(u,\Upsilon_{i-k+1})a_1(u)(\epsilon_{i-k}^2-\epsilon^{\prime 2}_{i-k},0)^\top  \|_q,\\
        &\leq |a_1(u)|(a_2^2(u)+a_3^2(u))^{k/2} \left\| \sqrt{\epsilon_i^4+1}\cdots\sqrt{\epsilon_{i-k+1}^4+1} \cdot|\epsilon_{i-k}^2-\epsilon^{\prime 2}_{i-k}|  \right\|_q.
    \end{align*}
    If $\sup_{u\in[0,1]} |a_1(u)|<\infty$ and $\sup_{u\in[0,1]}\sqrt{a_2^2(u)+a_3^2(u)}<1/\|\sqrt{\epsilon_i^4+1}\|_q$, since $\epsilon_i,\dots,\epsilon_{i-k},\epsilon_{i-k}^\prime$ i.i.d. and $\|\epsilon_i\|_{2q}<\infty$, we have $\sup_{u\in[0,1]}\|\mf W(u,\Upsilon_i)\|_q<\infty$ and
    $$
        \sup_{u\in[0,1]}\left\| \mf W(u,\Upsilon_i)- \mf W(u,\Upsilon_{i,i-k})\right\|_q = O(\chi^k), 
    $$
    where $\chi=\sup_{u\in[0,1]}\sqrt{a_2^2(u)+a_3^2(u)}\|\sqrt{\epsilon_i^4+1}\|_q$.
    Note that $H(u,\Upsilon_i)\geq 0$ and
    \begin{align}
        |H^{1/2}(u,\Upsilon_i)-H^{1/2}(u,\Upsilon_{i,i-k})|^2&\leq |H(u,\Upsilon_i)-H(u,\Upsilon_{i,i-k})|,\notag\\
        &\leq |\mf W(u,\Upsilon_i)- \mf W(u,\Upsilon_{i,i-k})|,
    \end{align}
    we further have
        $$
        \sup_{u\in[0,1]}\left\| H^{1/2}(u,\Upsilon_i)- H^{1/2}(u,\Upsilon_{i,i-k})\right\|_{2q} = O(\chi^{k/2}).
    $$
    Then for $k\geq 1$, by Holder's inequality
    $$
    \sup_{u\in[0,1]}\|G(u,\Upsilon_i)-G(u,\Upsilon_{i,i-k})\|_q\leq \|\epsilon_i\|_{2q} \sup_{u\in[0,1]}\|H^{1/2}(u,\Upsilon_i)-H^{1/2}(u,\Upsilon_{i,i-k})\|_{2q}=O(\chi^{k/2}).
    $$
    Together with the boundedness of $\|\mf W(u,\Upsilon_i)\|_q$, we finally have $\delta_{q}(G,k) = O(\chi^{k/2}).$
    For the general GARCH model, we refer to Example 9 in \cite{wu2011asymptotic} with Proposition 3 in \cite{wu2005linear}.}

\TPR{The process $\mf G(i/n,\Upsilon_i),i=1,\dots,n$ in model (iv)-(vi) is not stationary when the time-varying curve $a_k(i/n),k=1,2,3$ changes with $i$. When $a_k(u)$ is Lipschitz's continuous, then the nonstationarity can be regarded as \textbf{locally stationary}. We also give an example for discontinuous $a_k(u)$ as follows.
\begin{eg}[Piecewise locally stationary model] 
Let $\left\{\eta_i\right\}_{i \in \mathbb{Z}}$ be i.i.d. p-dimensional random vector and $\mf e_{i}=\mf G(t_i,\Upsilon_i),t_i=i/n$ where $\mf G(u,\Upsilon_i)= \mf G_0\left(u, \Upsilon_i\right)$ for $u\in[0,2/5]$, $\mf G(u,\Upsilon_{i})=\mf G_1\left(u, \Upsilon_i\right)$ for $u\in(2/5, 1]$, and for any fixed $u\in[0,1]$,
$$
\begin{gathered}
\mf G_0\left(u, \Upsilon_i\right)=0.5 \sin (\pi u) \mf G_0\left(u, \Upsilon_{i-1}\right)+\eta_i+(0.2-0.5 u) \eta_{i-1} \\
\mf G_1\left(u, \Upsilon_i\right)=(0.5-u) \mf G_1\left(u, \Upsilon_{i-1}\right)+\eta_i+\frac{(u-0.2)^2}{2} \eta_{i-1}
\end{gathered}
$$
\end{eg}}

\TPR{If $\|\eta_i\|_q\leq C\sqrt{p}$, we have that for any fixed $u\in[0,1]$, the dependence measure $\delta_q(\mf G,k)\leq \delta_q(\mf G,k-1)/2$ when $k\geq 2$ and $\delta_q(\mf G,k)\leq C\sqrt{p}$ for $k=0,1$. Then we have $\delta_q(\mf G,k)=O(\sqrt{p}\chi^k)$ where constant $\chi=1/2$.}

\begin{eg}[High-dimensional linear process]\label{ex4.1}
    Consider for $t\in [0,1]$, $\mf B_j(t)$, $j\in \mathbb Z$, are time-varying $p\times p'$ matrices ($p'$ is fixed) and
	\begin{align}
		\mf G(t_i,\Upsilon_h)=:\sum_{j=0}^{\infty} \mf B_{j}(t_i) \boldsymbol{\epsilon}_{h-j}=(G_{1}(t_i,\Upsilon_h),...,G_{p}(t_i,\Upsilon_h))^\top\notag
	\end{align}
    where $\boldsymbol{\epsilon}_i=(\epsilon_{i,1},\dots,\epsilon_{i,p'})^{\top}$, $\left(\epsilon_{i,s}\right)_{i\in\mathbb{Z},1\leq s\leq p'}$ are i.i.d. $\sigma^2$-sub-Gaussian variables. 
	Then for  $1\leq v\leq p$,
	\begin{align}	
    G_v(t_{i},\Upsilon_h)=\sum_{j=0}^\infty \sum_{s=1}^{p'} b_{j,v,s}(t_i)\epsilon_{h-j,s}=\sum_{j=-\infty}^{h}\sum_{s=1}^{p'} b_{h-j,v,s}(t_i) \epsilon_{j,s}\notag
	\end{align}
	where $\mf b_{j,v}(t)=(b_{j,v,1}(t),....,b_{j,v,p'}(t))$ is the $v_{th}$ row vector of the matrix $\mf B_j(t)$.
\end{eg}
Then in proposition \ref{PfEx4.1} we gives mild conditions under which the model in Example \eqref{ex4.1} satisfies conditions \hyperref[(B1)]{(B1)}, \hyperref[(B2)]{(B2)}, \hyperref[(B3)]{(B3)} with $\Theta(p)\asymp p^{1/2}$, and fulfills \hyperref[(B5)]{(B5)}.

\begin{prop}\label{PfEx4.1}
    Consider the Example \ref{ex4.1}. Then if 
    $\max_{1\leq v\leq p}\sup_{t\in[0,1]}\sum_{j=0}^\infty\sum_{s=1}^{p'}b^2_{j,v,s}(t)<\infty,\max_{1\leq v\leq p}\sup_{t\in [0,1]}(\sum_{s=1}^{p'} b^2_{k,v,s}(t))^{1/2}=O(\chi^k),
    $ and
    \begin{equation}
            \max_{1\leq v\leq p}\sum_{i=1}^n\sum_{j=0}^\infty\sum_{s=1}^{p'} \left|  b_{j,v,s}(t_i)-b_{j,v,s}(t_{i-1})\right|<\infty,\label{eq:bdd variation}
    \end{equation}
    then Example \eqref{ex4.1} satisfies conditions \hyperref[(B1)]{(B1)}, \hyperref[(B2)]{(B2)}, \hyperref[(B3)]{(B3)} with $\Theta(p)\asymp p^{1/2}$, and fulfills \hyperref[(B5)]{(B5)}. Moreover, if $\sup_{t\in[0,1]}|b'_{j,v,s}(t)|\leq B_{j,v,s}$ and replacing \eqref{eq:bdd variation} with $\max_{1\leq v\leq p}\sum_{j=0}^\infty\sum_{s=1}^{p'} B_{j,v,s}<\infty$, then Assumptions \hyperref[(B1)]{(B1)}, \hyperref[(B2)]{(B2)}, \hyperref[(B3)]{(B3)} and \hyperref[(B5)]{(B5)} also hold and Example \ref{ex4.1} can be viewed as high-dimensional locally stationary process; see also \cite{wu2020adaptive}.
\end{prop}
\begin{proof}[Proof of Proposition \ref{PfEx4.1}.]
Observe that $(\sum_{s=1}^{p'} b_{i-j,v,s}(t) \epsilon_{j,s})_{j\leq i}$ are mean $0$ random variables and are independent of each other. Therefore by Burkholder's inequality, we have 
	\begin{align}
		\|\mf G(t,\Upsilon_0)\|_q^2&= \|\sum_{v=1}^p G_{v}^2(t,\Upsilon_0)\|_{q/2}
 \leq \sum_{v=1}^p\| G_v(t,\Upsilon_0)\|_{q}^2,\notag\\
 &\lesssim \sum_{v=1}^p \left\|\sqrt{\sum_{j=0}^\infty\sum_{s=1}^{p'}\left(  b_{j,v,s}(t)\epsilon_{0-j,s} \right)^2 }\right\|_q^2\lesssim\sum_{v=1}^p\sum_{j=0}^\infty\sum_{s=1}^{p'} \left\|\left(  b_{j,v,s}(t)\epsilon_{0-j,s} \right)^2 \right\|_{q/2} ,\notag\\
 &\lesssim \sum_{v=1}^p \sum_{j=0}^\infty\sum_{s=1}^{p'} b^2_{j,v,s}(t)\notag
	\end{align}
	As a consequence, Assumption \hyperref[(B1)]{(B1)} will be satisfied with $\Theta(p)\asymp\sqrt{p}$ if for any $1\leq v\leq p$, 
    $$
    \sup_{t\in[0,1]}\sum_{j=0}^\infty\sum_{s=1}^{p'}b^2_{j,v,s}(t)<\infty.
    $$ 
    For Assumption \hyperref[(B1)]{(B1)}, by definition \eqref{phd} and Burkholder's inequality, for $q\geq 2$,
	\begin{align}
		\delta_q^2(\mathbf{G},k)&\leq \sup_{t\in[0,1]}\sum_{v=1}^p\left\| \sqrt{\sum_{s=1}^{p'}b^2_{k,v,s}(t)(\epsilon_{0,s}-\epsilon_{0,s}')^2  } \right\|_{q}^2\notag\\
    &\lesssim \sup_{t\in[0,1]}\sum_{v=1}^p\sum_{s=1}^{p'} b_{k,v,s}^2(t).\notag
	\end{align}
	Therefore, Assumption \hyperref[(B1)]{(B1)} will hold with $\Theta(p)\asymp\sqrt{p}$ if
	$\sup_{t\in [0,1]}(\sum_{s=1}^{p'} b^2_{k,v,s}(t))^{1/2}$ is $O(\chi^k)$ for any $1\leq v\leq p$. For Assumption \hyperref[(B3)]{(B3)}, by Burkholder's inequality, for $q\geq 2$,
    \begin{align}
		\|\mf G(t_i,\Upsilon_0)-\mf G(t_{i-1},\Upsilon_0)\|_q^2&\leq \sum_{v=1}^p\| G_v(t_{i},\Upsilon_0)-G_{v}(t_{i-1},\Upsilon_0)\|_{q}^2,\notag\\
 &\lesssim\sum_{v=1}^p \left\|\sqrt{\sum_{j=0}^\infty\sum_{s=1}^{p'}\left(  b_{j,v,s}(t_i)-b_{j,v,s}(t_{i-1})\right)^2 \epsilon^2_{0-j,s}  }\right\|_q^2\notag\\
 &\lesssim\sum_{v=1}^p \sum_{j=0}^\infty\sum_{s=1}^{p'} \left(  b_{j,v,s}(t_i)-b_{j,v,s}(t_{i-1})\right)^2\notag
	\end{align}
    Thus \hyperref[(B3)]{(B3)} is satisfied with $\Theta(p)\asymp \sqrt{p}$ and 
    $
    \max_{1\leq v\leq p}\sum_{i=1}^n\sum_{j=0}^\infty\sum_{s=1}^{p'} \left|  b_{j,v,s}(t_i)-b_{j,v,s}(t_{i-1})\right|<\infty.$
    For Assumption \hyperref[(B5)]{(B5)}, note that $\boldsymbol{\epsilon}_i=(\epsilon_{i,1},\dots,\epsilon_{i,p'})^{\top}$, $\left(\epsilon_{i,s}\right)_{i\in\mathbb{Z},1\leq s\leq p'}$ are i.i.d. $\sigma^2$-sub-Gaussian variables and 
    $$
    \max_{1\leq v\leq p}\sup_{t\in[0,1]}\sum_{j=0}^\infty\sum_{s=1}^{p'}b^2_{j,v,s}(t)<\infty,
    $$
    then we have $G_v(t_{i},\Upsilon_i)=\sum_{j=0}^\infty \sum_{s=1}^{p'} b_{j,v,s}(t_i)\epsilon_{i-j,s}$ is $\bar{\sigma}^2$-sub-Gaussian variable for any $1\leq v\leq p$ and $1\leq i\leq n$ with universal proxy variance
    $$
    \bar{\sigma}^2=:\max_{1\leq v\leq p}\sup_{t\in[0,1]}\sum_{j=0}^\infty\sum_{s=1}^{p'}b^2_{j,v,s}(t)\sigma^2,
    $$
    which means \hyperref[(B5)]{(B5)} holds.
\end{proof}  

\subsection{Discussion}
\label{sec:discussion supp}
\subsubsection{Discussion on reconciling bounded time interval with $n\rightarrow\infty$}
\label{sec:reconcile discussion}
\TPR{The formulation $\mf e_{i,n}=\mf G(i/n,\Upsilon_i)$ is different from `infill' framework. Our theorem framework provides a different way of reconciling the theoretical asymptotic where $n\rightarrow\infty$ to the scenario when time interval remains bounded. The key points are the smoothness Assumption \hyperref[(A1)]{(A1)} on trend function $m(i/n)$ and Assumption \hyperref[(C1)]{(C1)} on kernel function.By solving the least square \eqref{eq:high-dim ll def} and applying Taylor's expansion (detailed proof can be seen in Lemma \ref{local linear stochastic expansion}), we can obtain the following stochastic expansion
$$
    \max_{1\leq k\leq p}\sup_{t\in\mathfrak{T}_{n,k}}\left|\hat{m}_{k}(t)-m_k(t)-\frac{\kappa_2 m^{\prime\prime}_k(t)}{2} h_{r,k}^2-\frac{1}{n h_{r,k}} \sum_{i=1}^n e_{i,k} K_r\left(\frac{t_i-t}{h_{r,k}}\right)\right|=O\left(h_r^3+\frac{1}{n h_r}\right).
$$
Via the above stochastic expansion, we can approximate all $\hat m_k(t)-m_k(t)$ uniformly over a {\it bounded time interval} by the partial sum of process $\mf e_{i,n},i=1,\dots,n$, i.e.,  $\frac{1}{nh_{r,k}}\sum_{i=1}^n e_{i,k}K_r(\frac{t_i-t}{h_{r,k}})$.
To construct a Gaussian process mimicking the stochastic behavior of $\frac{1}{nh_{r,k}}\sum_{i=1}^n e_{i,k}K_r(\frac{t_i-t}{h_{r,k}})$, we apply the high-dimensional Gaussian approximation theory in \cite{Mies2022seq_high-dim}. In other words, we establish a non-asymptotic Gaussian approximation result instead of asymptotic normality in Proposition \ref{prop:GA in high-dim} utilizing the smoothness assumptions on trend and kernel functions, as well as the assumptions \hyperref[(B1)]{(B1)}-\hyperref[(B5)]{(B5)} on $\mf G(u,\Upsilon_i),u\in[0,1]$.}

\TPR{Based on the non-asymptotic Gaussian approximation in Proposition \ref{prop:GA in high-dim}, we proposed bootstrap algorithms to mimic the finite sample behavior of the maximal deviation of the proposed monotone estimator with diminishing approximation error as $n\rightarrow \infty.$ We show that the yielding SCBs are asymptotically correct with shrinking width in simulation. 
We thus reconcile the theoretical asymptotic where $n\rightarrow \infty$ yet the time interval remains bounded. }

\subsubsection{Discussion on the context of our monotone estimation}
\label{sec:context discussion}
\TPR{The context of our results is significantly different from the monotone estimation in \cite{dette2006simple}. The properties of jackkinfe estimator $\tilde m_k(t)$ are not known from \cite{dette2006simple} since it is not used by \cite{dette2006simple}. We study the bias of $\tilde m_k(t)$ and its stochastic variation in Lemma \ref{local linear stochastic expansion}.}

\TPR{We consider nonstationary high-dimensional time series and time series regression, while \cite{dette2006simple} considers an independent sequence. Moreover, \cite{dette2006simple} focuses on the inference on a fixed point while we study the simultaneous inference. Therefore, we have to propose new bootstrap methods for inference, which are also theoretically justified. Due to the different settings and different statistical approaches for different purposes, our assumptions are significantly different from \cite{dette2006simple}. }
\textcolor{black}{
\begin{description}
        \item[(1)] For assumptions on the monotone trend function, we assume $p$-dimensional strictly monotone function with twice continuously differentiable in \hyperref[(A1)]{(A1)}-\hyperref[(A2)]{(A2)} while \cite{dette2006simple} assumes such trend function in 1 dimension. Furthermore, our theorem framework covers the time-varying linear regression model in which the time-varying coefficient follows the monotone condition while \cite{dette2006simple} considered fitting the mean trend.   
        \item[(2)] For assumptions on dependence, we make Assumptions \hyperref[(B1)]{(B1)}-\hyperref[(B5)]{(B5)} to allow for a nonstationary high-dimensional process or \hyperref[(B1')]{(B1')}-\hyperref[(B5')]{(B5')} to accommodate a regression model with nonstationary covariates and errors. In contrast, \cite{dette2006simple} imposes independence assumptions in a one-dimensional setting.
    \end{description}
    }

\subsubsection{Discussion on minimal volume}
\label{sec:width dicussion}
As mentioned before, now we have two approaches to construct SCB for the monotone estimator $\hat{{m}}_I(t)$. The former conservative approach is to apply the monotone rearrangement on local linear estimator $\hat{ m}(t)$, upper and lower bounds of local linear estimator's SCB. This improving procedure admits a narrower SCB that contains the monotone estimator $\hat{ m}_I(t)$ with a conservative significance level. The new way is proposed in this paper, which directly generates SCB for $\hat{  m}_I(t)- m(t)$ by its Gaussian approximation. Although both simulation and empirical study indicate our new SCB is narrower, deriving SCB with minimal length under the monotone condition is still an open question.  In \cite{zhou2010simultaneous}, a Lagranger's multiplier is applied for deriving the minimal length of SCB, however, the improving procedure such as the simple example $\tilde m_L,\tilde m_U$ in Proposition \ref{id-improving} enlarges the range of candidates for the optimization problem from real value space to random functional space. More precisely, the optimization problem that allows monotone improving is formulated as 
\begin{align*}
    &\min_{m_U\in\mathcal{U},m_L\in\mathcal{L}}\int_0^1 m_U(t)-m_L(t)dt\\
    & P(m_L(t) \leq m\left(t\right) \leq m_U(t), \forall t\in[0,1]) \rightarrow 1-\alpha,
\end{align*}
In most nonparametric problems, $U=:m_U-\hat m$ and $L=:m_L-\hat m$ are expected to be nonrandom functions determined by limiting distributions such that the length of SCB is not random. If one wants to simply apply monotone rearrangement on SCB solved from the above optimization problem, then narrower SCB $[\tilde m_U,\tilde m_L]$ is produced if $m_U$ or $m_L$ obtained from the optimization breaks the monotone constraint. This causes the length is again a random variable, which makes it hard to further illustrate whether the improved SCB attains the minimal width. However, if one only wants a constant SCB, which means $U=:m_U-\hat m$ and $L=:m_L-\hat m$ are expected to be critical constant values in $\mathbb{R}$, then our framework can solve this optimization problem and find the SCB with minimal length while the former improving approach in \cite{chernozhukov2009improving} still maintains the random length problem because the origin band deduced from local linear estimator still may violate monotone condition even though only constant band is required. 

\begin{prop}
\label{id-improving}
Suppose $m_L,m_U$ are continuous functions on [0,1]. If \hyperref[(A1)]{(A1)} and \hyperref[(A2)]{(A2)} hold on $m(\cdot)$, then there exists monotone and continuous functions $\tilde m_U(\cdot),\tilde m_L(\cdot)$ s.t. 
\begin{equation}\label{new11.3}
m_L(t)\leq m(t)\leq m_U(t),\forall t\in[a,b]\end{equation} if and only if 
\begin{equation}\label{11.3}
\tilde m_L(t)\leq m(t)\leq \tilde m_U(t), \forall t\in[a,b],
\end{equation}
where $[a,b]\subset [0,1]$. 
\end{prop}
\begin{proof}[Proof of Proposition \ref{id-improving}.]
Consider $\tilde m_L(t)=:\max_{u\in[a,t]}m_L(u)$ and $\tilde m_U(t)=:\min_{u\in[t,b]}m_L(u)$ which are  monotone and continous functions.\\
\textbf{Proof of \eqref{11.3} $\Rightarrow$ \eqref{new11.3}:} By definition,
    $$ 
    \tilde m_L(t)=\max_{u\in[a,t]}m_L(u)\geq m_L(t),\quad
    \tilde m_U(t)=\min_{u\in[t,b]}m_U(u)\leq m_U(t),
    $$
    thus $[m_L,m_U]$ covers $m$ if $[\tilde m_L,\tilde m_U]$ covers $m$.\\
\textbf{Proof of \eqref{new11.3} $\Rightarrow$ \eqref{11.3}:} (Contrapositive) Suppose \eqref{11.3} is not true. If there exists $t_0\in [a,b]$ s.t. $\tilde m_L(t_0)>m(t_0)$ then by the definition, there exists $t^*\in [a,t_0]$ s.t. $m_L(t^*)=\tilde m_L(t_0)$. Note that $m(\cdot)$ is increasing, thus  
    $m_L(t^*)=\tilde m_L(t_0)>m(t_0)\geq m(t^*),$
    which implies \eqref{new11.3} is also not true. Similar arguments yield that if there exists $t_0\in [a,b]$ s.t. $\tilde m_U(t_0)<m(t_0)$, then $\exists t^*\in[t_0,b]$ s.t. $m_U(t^*)<m(t^*)$, which means \eqref{new11.3} is not true.\end{proof} 

\subsubsection{Discussion on interpolation for missing values}
\label{sec:interpolation}
\TPM{
There are 54 missing data points out of the 14580 total observations (across 27 stations from 1979 to 2023). We believe that the missing values have a minimal effect, as the missing points represent a very small fraction of the overall dataset. We apply several interpolation methods and compare the associated empirical study results. In specific, the first interpolation is the basic linear interpolation, which is the method we used in the main paper (implemented by R function `approx()'). The second is performing cubic spline interpolation which has been discussed in \cite{hall1976spline} (implemented by R function `spline()'). The third method is applying the Kalman filtering to interpolate missing values; see \cite{harvey1984estimating}, \cite{gomez1994estimation}. We use the function `na\_kalman()' in R package `imputeTS' to achieve this.}

\TPM{Table \ref{tab:different interpolations supp} shows the widths of SCB in Figure \ref{fig:joint SCBs} and p-values for \eqref{eq:fundement test} and \eqref{eq:null delta} using different interpolation methods. These results remained robust across different interpolations, demonstrating that the statistical findings are not sensitive to the choice of methods for preprocessing missing values. The results for the time-varying regression model \eqref{tmax-SD} in Section \ref{regressionsection} of the supplement remain unchanged, as neither the Heathrow nor Lerwick station has missing values in their historical data from 1979 to 2023.}
\begin{table}[H]
    \centering
    {\renewcommand{\arraystretch}{0.5}
    \begin{tabular}{cccc}
    \hline
       Interpolations & linear & spline & kalman\\
    \hline
       90\% SCBs in Figure \ref{fig:joint SCBs} & 0.7963 & 0.7977 & 0.8013 \\
       95\% SCBs in Figure \ref{fig:joint SCBs} & 0.9261 & 0.9296 & 0.9255\\
       p-value for \eqref{eq:fundement test}  & 0.005 & 0.006 & 0.006\\
       p-value for \eqref{eq:null delta} & 0.04 & 0.038 & 0.039\\
    \hline
    \end{tabular}}
    \setlength{\belowcaptionskip}{-0.5cm}
  \captionsetup{font={small, stretch=1},skip=0cm} 
    \caption{\TPM{Widths of SCB and $p$-values w.r.t. different interpolation methods. Interpolation method `linear' is our original method implemented by function `approx()' in R; Interpolation `spline' stands for performing cubic spline interpolation implemented by function `spline()' in R; Interpolation `kalman' represents using Kalman filtering for interpolation, which is implemented by R package `imputeTS'.}}
    \label{tab:different interpolations supp}
\end{table}

\subsection{Additional simulation}
\label{sec: additional simulation}

\subsubsection{Widths of SCB in Section \ref{Simulation study}}
\label{sec: widths of SCB in simulation}
\TPR{To justify the validity of our SCB method for monotone regression, Table \ref{tab:SCB width high supp} and Table \ref{tab:SCB width high zhou supp} respectively report the lengths of our joint monotone SCBs in the Section \ref{Simulation study} of main paper and the corresponding joint SCBs of local linear estimator without monotone constraint.
We also note that those conservative monotone SCBs obtained from the so-called ``improving procedure" proposed in \cite{chernozhukov2009improving} have the same width as the unconstrained SCB in high dimension. 
Therefore, we only additionally display the widths of the unconstrained SCB 
in Table \ref{tab:SCB width high zhou supp} for comparison with our monotone SCBs. }

\TPR{The comparison between Table \ref{tab:SCB width high supp} and Table \ref{tab:SCB width high zhou supp} shows that the width of our joint SCBs is narrower than the unconstrained SCBs. This is as expected since the monotonicity under high dimensionality has not been utilized by unconstrained estimator to further narrow the width. }
\begin{table}[htpb]
  \centering
  \setlength{\belowcaptionskip}{-0.5cm}
  \captionsetup{font={small}, skip=-5pt}
  \caption{\TPR{Simulated widths of joint SCBs for our monotone estimator $\hat{\mathbf{m}}_I(t)$}}
    \scalebox{0.725}{
    {\renewcommand{\arraystretch}{0.5}
    \begin{tabular}{ccccccc c cccccc}
    \hline
    $(n=300)$ & \multicolumn{6}{c}{\textit{Model (a)}}& & \multicolumn{6}{c}{\textit{Model (b)}} \\
    \cline{2-7}\cline{9-14}
          & \multicolumn{2}{c}{$p=9$} & \multicolumn{2}{c}{$p=18$} & \multicolumn{2}{c}{$p=27$} & & \multicolumn{2}{c}{$p=9$} & \multicolumn{2}{c}{$p=18$} & \multicolumn{2}{c}{$p=27$} \\
    $h_r$ & 90\%  & 95\%  & 90\%  & 95\%  & 90\%  & 95\% & & 90\%  & 95\%  & 90\%  & 95\%  & 90\%  & 95\% \\
    \hline
    GCV   & 0.3078 & 0.3449 & 0.3168 & 0.3539 & 0.3336 & 0.3708 &       & 0.3513 & 0.3943 & 0.3581 & 0.3993 & 0.3776 & 0.4204 \\
    0.05  & 0.5064 & 0.5675 & 0.5742 & 0.6422 & 0.6213 & 0.6945 &       & 0.5474 & 0.6134 & 0.6367 & 0.7137 & 0.6675 & 0.7450 \\
    0.1   & 0.4032 & 0.4489 & 0.4483 & 0.4969 & 0.4677 & 0.5160 &       & 0.4448 & 0.4958 & 0.4938 & 0.5482 & 0.5208 & 0.5759 \\
    0.15  & 0.3820 & 0.4257 & 0.4234 & 0.4689 & 0.4443 & 0.4902 &       & 0.4294 & 0.4789 & 0.4724 & 0.5230 & 0.4992 & 0.5516 \\
    0.2   & 0.3451 & 0.3846 & 0.3815 & 0.4234 & 0.4011 & 0.4431 &       & 0.3904 & 0.4358 & 0.4294 & 0.4760 & 0.4537 & 0.5023 \\
    0.25  & 0.3122 & 0.3477 & 0.3439 & 0.3815 & 0.3622 & 0.4002 &       & 0.3547 & 0.3958 & 0.3896 & 0.4320 & 0.4111 & 0.4552 \\
    0.3   & 0.3020 & 0.3357 & 0.3309 & 0.3654 & 0.3481 & 0.3833 &       & 0.3445 & 0.3832 & 0.3760 & 0.4159 & 0.3954 & 0.4359 \\
    0.35  & 0.2785 & 0.3095 & 0.3045 & 0.3364 & 0.3206 & 0.3526 &       & 0.3180 & 0.3537 & 0.3470 & 0.3833 & 0.3652 & 0.4021 \\
    \hline
     $(n=500)$ & \multicolumn{6}{c}{\textit{Model (a)}}& & \multicolumn{6}{c}{\textit{Model (b)}} \\
    \cline{2-7}\cline{9-14}
          & \multicolumn{2}{c}{$p=9$} & \multicolumn{2}{c}{$p=18$} & \multicolumn{2}{c}{$p=27$} & & \multicolumn{2}{c}{$p=9$} & \multicolumn{2}{c}{$p=18$} & \multicolumn{2}{c}{$p=27$} \\
    $h_r$ & 90\%  & 95\%  & 90\%  & 95\%  & 90\%  & 95\% & & 90\%  & 95\%  & 90\%  & 95\%  & 90\%  & 95\% \\
    \hline
    GCV   & 0.2541 & 0.2840 & 0.2798 & 0.3110 & 0.2957 & 0.3278 &       & 0.2878 & 0.3216 & 0.3184 & 0.3544 & 0.3377 & 0.3741 \\
    0.05  & 0.2123 & 0.2342 & 0.2336 & 0.2567 & 0.2481 & 0.2723 &       & 0.2340 & 0.2589 & 0.2561 & 0.2816 & 0.2690 & 0.2950 \\
    0.1   & 0.2897 & 0.3208 & 0.3201 & 0.3528 & 0.3375 & 0.3704 &       & 0.3218 & 0.3565 & 0.3574 & 0.3945 & 0.3761 & 0.4136 \\
    0.15  & 0.3002 & 0.3342 & 0.3296 & 0.3653 & 0.3460 & 0.3818 &       & 0.3358 & 0.3735 & 0.3708 & 0.4118 & 0.3912 & 0.4325 \\
    0.2   & 0.2719 & 0.3027 & 0.2946 & 0.3268 & 0.3109 & 0.3436 &       & 0.3058 & 0.3412 & 0.3350 & 0.3718 & 0.3554 & 0.3928 \\
    0.25  & 0.2570 & 0.2861 & 0.2775 & 0.3066 & 0.2924 & 0.3221 &       & 0.2917 & 0.3247 & 0.3161 & 0.3496 & 0.3360 & 0.3704 \\
    0.3   & 0.2414 & 0.2684 & 0.2604 & 0.2872 & 0.2744 & 0.3014 &       & 0.2746 & 0.3051 & 0.2970 & 0.3277 & 0.3157 & 0.3470 \\
    0.35  & 0.2166 & 0.2408 & 0.2344 & 0.2586 & 0.2468 & 0.2712 &       & 0.2465 & 0.2740 & 0.2677 & 0.2954 & 0.2837 & 0.3121 \\
    \hline
    $(n=1000)$ & \multicolumn{6}{c}{\textit{Model (a)}}& & \multicolumn{6}{c}{\textit{Model (b)}} \\
    \cline{2-7}\cline{9-14}
          & \multicolumn{2}{c}{$p=9$} & \multicolumn{2}{c}{$p=18$} & \multicolumn{2}{c}{$p=27$} & & \multicolumn{2}{c}{$p=9$} & \multicolumn{2}{c}{$p=18$} & \multicolumn{2}{c}{$p=27$} \\
    $h_r$ & 90\%  & 95\%  & 90\%  & 95\%  & 90\%  & 95\% & & 90\%  & 95\%  & 90\%  & 95\%  & 90\%  & 95\% \\
    \hline
    GCV   & 0.2159 & 0.2397 & 0.2305 & 0.2550 & 0.2366 & 0.2612 &       & 0.2319 & 0.2582 & 0.2554 & 0.2830 & 0.2714 & 0.2994 \\
    0.05  & 0.2005 & 0.2205 & 0.1992 & 0.2187 & 0.1986 & 0.2172 &       & 0.1896 & 0.2090 & 0.2066 & 0.2265 & 0.2188 & 0.2399 \\
    0.1   & 0.2757 & 0.3049 & 0.2791 & 0.3074 & 0.2725 & 0.2990 &       & 0.2683 & 0.2970 & 0.2900 & 0.3194 & 0.3078 & 0.3380 \\
    0.15  & 0.2399 & 0.2657 & 0.2544 & 0.2808 & 0.2602 & 0.2868 &       & 0.2555 & 0.2841 & 0.2807 & 0.3102 & 0.2975 & 0.3278 \\
    0.2   & 0.2123 & 0.2349 & 0.2275 & 0.2508 & 0.2351 & 0.2586 &       & 0.2332 & 0.2588 & 0.2558 & 0.2825 & 0.2704 & 0.2976 \\
    0.25  & 0.1995 & 0.2198 & 0.2133 & 0.2343 & 0.2224 & 0.2437 &       & 0.2219 & 0.2452 & 0.2415 & 0.2659 & 0.2552 & 0.2799 \\
    0.3   & 0.1756 & 0.1938 & 0.1880 & 0.2067 & 0.1970 & 0.2159 &       & 0.1963 & 0.2173 & 0.2136 & 0.2352 & 0.2253 & 0.2472 \\
    0.35  & 0.1568 & 0.1733 & 0.1685 & 0.1853 & 0.1768 & 0.1940 &       & 0.1760 & 0.1952 & 0.1917 & 0.2112 & 0.2022 & 0.2220 \\
        \hline
    \end{tabular}%
  }}
      \label{tab:SCB width high supp}%
\end{table}%

\begin{table}[htpb]
  \centering
  \setlength{\belowcaptionskip}{-0.5cm}
  \captionsetup{font={small}, skip=-5pt}
  \caption{\TPR{Simulated widths of joint SCBs for unconstrained estimator $\hat{\mf m}(t)$}}
  \scalebox{0.725}{
  {\renewcommand{\arraystretch}{0.5}
    \begin{tabular}{ccccccc c cccccc}
    \hline
    $(n=300)$ & \multicolumn{6}{c}{\textit{Model (a)}}& & \multicolumn{6}{c}{\textit{Model (b)}} \\
    \cline{2-7}\cline{9-14}
          & \multicolumn{2}{c}{$p=9$} & \multicolumn{2}{c}{$p=18$} & \multicolumn{2}{c}{$p=27$} & & \multicolumn{2}{c}{$p=9$} & \multicolumn{2}{c}{$p=18$} & \multicolumn{2}{c}{$p=27$} \\
    $h_r$ & 90\%  & 95\%  & 90\%  & 95\%  & 90\%  & 95\% & & 90\%  & 95\%  & 90\%  & 95\%  & 90\%  & 95\% \\
    \hline
    0.05  & 0.4939 & 0.5342 & 0.5296 & 0.5710 & 0.5501 & 0.5915 &       & 0.5675 & 0.6140 & 0.6058 & 0.6532 & 0.6292 & 0.6772 \\
    0.1   & 0.5376 & 0.5861 & 0.5771 & 0.6268 & 0.6022 & 0.6520 &       & 0.6137 & 0.6688 & 0.6579 & 0.7143 & 0.6882 & 0.7454 \\
    0.15  & 0.5197 & 0.5702 & 0.5565 & 0.6079 & 0.5863 & 0.6389 &       & 0.5938 & 0.6508 & 0.6370 & 0.6962 & 0.6719 & 0.7332 \\
    0.2   & 0.4534 & 0.4991 & 0.4892 & 0.5357 & 0.5158 & 0.5636 &       & 0.5209 & 0.5731 & 0.5604 & 0.6136 & 0.5907 & 0.6456 \\
    0.25  & 0.3970 & 0.4373 & 0.4290 & 0.4706 & 0.4522 & 0.4944 &       & 0.4563 & 0.5023 & 0.4913 & 0.5387 & 0.5167 & 0.5655 \\
    0.3   & 0.3502 & 0.3863 & 0.3788 & 0.4159 & 0.3994 & 0.4370 &       & 0.4029 & 0.4449 & 0.4339 & 0.4763 & 0.4564 & 0.5001 \\
    0.35  & 0.3114 & 0.3446 & 0.3376 & 0.3715 & 0.3559 & 0.3900 &       & 0.3579 & 0.3962 & 0.3869 & 0.4260 & 0.4071 & 0.4463 \\
    \hline
     $(n=500)$ & \multicolumn{6}{c}{\textit{Model (a)}}& & \multicolumn{6}{c}{\textit{Model (b)}} \\
    \cline{2-7}\cline{9-14}
          & \multicolumn{2}{c}{$p=9$} & \multicolumn{2}{c}{$p=18$} & \multicolumn{2}{c}{$p=27$} & & \multicolumn{2}{c}{$p=9$} & \multicolumn{2}{c}{$p=18$} & \multicolumn{2}{c}{$p=27$} \\
    $h_r$ & 90\%  & 95\%  & 90\%  & 95\%  & 90\%  & 95\% & & 90\%  & 95\%  & 90\%  & 95\%  & 90\%  & 95\% \\
    \hline
    0.05  & 0.3829 & 0.4136 & 0.4099 & 0.4414 & 0.4245 & 0.4561 &       & 0.4385 & 0.4732 & 0.4680 & 0.5042 & 0.4869 & 0.5234 \\
    0.1   & 0.4165 & 0.4535 & 0.4467 & 0.4850 & 0.4687 & 0.5074 &       & 0.4756 & 0.5181 & 0.5110 & 0.5546 & 0.5378 & 0.5829 \\
    0.15  & 0.4022 & 0.4408 & 0.4335 & 0.4740 & 0.4553 & 0.4963 &       & 0.4576 & 0.5020 & 0.4947 & 0.5408 & 0.5212 & 0.5686 \\
    0.2   & 0.3535 & 0.3890 & 0.3788 & 0.4148 & 0.3981 & 0.4347 &       & 0.4023 & 0.4421 & 0.4325 & 0.4737 & 0.4571 & 0.4993 \\
    0.25  & 0.3095 & 0.3411 & 0.3316 & 0.3636 & 0.3485 & 0.3807 &       & 0.3529 & 0.3890 & 0.3786 & 0.4150 & 0.4011 & 0.4386 \\
    0.3   & 0.2732 & 0.3017 & 0.2935 & 0.3221 & 0.3087 & 0.3375 &       & 0.3116 & 0.3439 & 0.3351 & 0.3681 & 0.3556 & 0.3890 \\
    0.35  & 0.2427 & 0.2684 & 0.2618 & 0.2876 & 0.2754 & 0.3015 &       & 0.2767 & 0.3059 & 0.2988 & 0.3286 & 0.3169 & 0.3476 \\
    \hline
    $(n=1000)$ & \multicolumn{6}{c}{\textit{Model (a)}}& & \multicolumn{6}{c}{\textit{Model (b)}} \\
    \cline{2-7}\cline{9-14}
          & \multicolumn{2}{c}{$p=9$} & \multicolumn{2}{c}{$p=18$} & \multicolumn{2}{c}{$p=27$} & & \multicolumn{2}{c}{$p=9$} & \multicolumn{2}{c}{$p=18$} & \multicolumn{2}{c}{$p=27$} \\
    $h_r$ & 90\%  & 95\%  & 90\%  & 95\%  & 90\%  & 95\% & & 90\%  & 95\%  & 90\%  & 95\%  & 90\%  & 95\% \\
    \hline
    0.05  & 0.3148 & 0.3389 & 0.2974 & 0.3193 & 0.3008 & 0.3235 &       & 0.3108 & 0.3356 & 0.3288 & 0.3536 & 0.3442 & 0.3698 \\
    0.1   & 0.3685 & 0.4021 & 0.3695 & 0.4016 & 0.3591 & 0.3892 &       & 0.3664 & 0.3995 & 0.3918 & 0.4259 & 0.4134 & 0.4484 \\
    0.15  & 0.3085 & 0.3376 & 0.3249 & 0.3543 & 0.3317 & 0.3613 &       & 0.3347 & 0.3672 & 0.3614 & 0.3950 & 0.3820 & 0.4161 \\
    0.2   & 0.2601 & 0.2848 & 0.2770 & 0.3024 & 0.2860 & 0.3118 &       & 0.2884 & 0.3165 & 0.3128 & 0.3424 & 0.3298 & 0.3596 \\
    0.25  & 0.2249 & 0.2462 & 0.2395 & 0.2618 & 0.2494 & 0.2719 &       & 0.2512 & 0.2758 & 0.2720 & 0.2978 & 0.2863 & 0.3124 \\
    0.3   & 0.1977 & 0.2171 & 0.2107 & 0.2307 & 0.2203 & 0.2407 &       & 0.2216 & 0.2440 & 0.2398 & 0.2629 & 0.2521 & 0.2754 \\
    0.35  & 0.1754 & 0.1931 & 0.1874 & 0.2056 & 0.1963 & 0.2149 &       & 0.1969 & 0.2174 & 0.2134 & 0.2344 & 0.2247 & 0.2457 \\
    \hline
    \end{tabular}%
  }}
      \label{tab:SCB width high zhou supp}%
\end{table}%

\subsubsection{Simulation for skewed innovation}
\label{sec:skew simulation}
\TPR{Our theoretical framework for {\it asymptotically correct} joint SCBs allows asymmetric innovations and error processes. To verify this, we add an extra simulation scenario with asymmetric error process under moderate sample size $n=500$ with high dimension $p=27$. We use exponential and log-normal distribution to replace the normal random vectors in model (a) with the same $b(i/n)$ and $\mf \Sigma_e$ in the Section \ref{Simulation study} of main paper:
\begin{description}
    \item[(Exponential)] $\mf G(i/n.\Upsilon_i)=b(i/n)\mf G(i/n,\Upsilon_{i-1})+\mf \xi_i$, where $\mf \xi_i$ i.i.d. and $\mathbb{E}(\xi_i)=0,\operatorname{cov}(\xi_i)=\mf\Sigma_e$. Each entry of $\mf\xi_i=(\xi_{i,1},\dotsm\xi_{i,p})$ follows standardized exponential distribution $\xi_{i,k}\overset{d}{=} \frac{e-\mathbb{E}(e)}{\sqrt{\operatorname{var}(e)}}$, where $e\sim \text{Exp(1)}$.
    \item[(Log-normal)] $\mf G(i/n,\Upsilon_i)=b(i/n)\mf G(i/n,\Upsilon_{i-1})+\mf \xi_i$, where $\mf \xi_i$ i.i.d. and $\mathbb{E}(\xi_i)=0,\operatorname{cov}(\xi_i)=\mf\Sigma_e$. Each entry of $\mf\xi_i=(\xi_{i,1},\dotsm\xi_{i,p})$ follows standardized log-normal distribution $\xi_{i,k}\overset{d}{=} \frac{e-\mathbb{E}(e)}{\sqrt{\operatorname{var}(e)}}$ where $\log(e)\sim N(0,1)$.
\end{description}
By the 400 simulation times, 2000 bootstrap sample size, and same GCV procedure in the simulation set-ups for model (a) in the Section \ref{Simulation study} of main paper, the simulated coverage probabilities and widths for our joint SCBs in Table \ref{tab:skwe SCB supp} show our joint SCBs are asymptotically correct under asymmetric innovations.}
\begin{table}[H]
    \centering
      \setlength{\belowcaptionskip}{-0.5cm}
  \captionsetup{font={small}, skip=5pt}
    \caption{\TPR{Simulated coverage probabilities and widths in skewed innovation.}}
{\renewcommand{\arraystretch}{0.5}
    \begin{tabular}{ccccc}
     \hline
    $(n=500)$ & \multicolumn{2}{c}{coverage} & \multicolumn{2}{c}{width}\\
    $(p=27)$             & 90\% & 95\% & 90\% & 95\% \\
    \hline
    Exponential  & 0.895 & 0.950 & 0.2999 & 0.3325  \\
    Log-normal   & 0.900  & 0.958 & 0.3136 & 0.3504 \\
    \hline
    \end{tabular}}
    \label{tab:skwe SCB supp}
\end{table}

\TPR{We additionally perform a simulation study checking the empirical performance of Gaussian approximation under skewed set-ups. Namely, we compare Gaussian process $\mathcal{V}(t)=(\mathcal{V}_1(t),\dots,\mathcal{V}_p(t))^\top$ in Proposition \ref{prop:GA in high-dim} with the process $\mathcal{E}(t)=(\mathcal{E}_1(t),\dots,\mathcal{E}_p(t))^\top$ where
$$
    \mathcal{V}_k(t)=\frac{m_k'(t)}{Nh_{d,k}nh_{r,k}}\sum_{j=1}^n\sum_{i=1}^N K_{d}\left(\frac{m_k(i/N)-m_k(t)}{h_{d,k}}\right)\tilde K_r^*\left(\frac{j/n-i/N}{h_{r,k}},\frac{i}{N}\right)V_{j,k},
$$
$$
\mathcal{E}_k(t)=:\frac{m_k^{\prime}(t)}{N h_{d, k} n h_{r, k}} \sum_{j=1}^n \sum_{i=1}^N K_d\left(\frac{m_k(i / N)-m_k(t)}{h_{d, k}}\right) \tilde{K}_r^*\left(\frac{j / n-i / N}{h_{r, k}}, \frac{i}{N}\right) e_{j, k}.
$$
Here the Gaussian process $\mf V_i=(V_{i,1},\dots,V_{i,p})^\top,i=1,\dots,n$ shares same long-run covariance structure as $\mf e_{i,n}=(e_{i,1},\dots,e_{i,p})^\top$. We examine how $\max_{1\leq k\leq p}\sup_{t\in\hat{\mathcal{T}}}|\mathcal{V}_k(t)|$ approximates $\max_{1\leq k\leq p}\sup_{t\in\hat{\mathcal{T}}}|\mathcal{E}_k(t)|$ by investigating simulation scenario under the formal skewed exponential and log-normal set-ups at moderate sample size $n=500$ and dimension $p=27$, which is comparable to our real data analysis scenario. The Q-Q plot below shows two distributions are sufficiently close, 
 concluding the adequacy of the Gaussian approximation at a moderate sample size.}
\begin{figure}[H]
    \centering
    \includegraphics[width=0.6\linewidth]{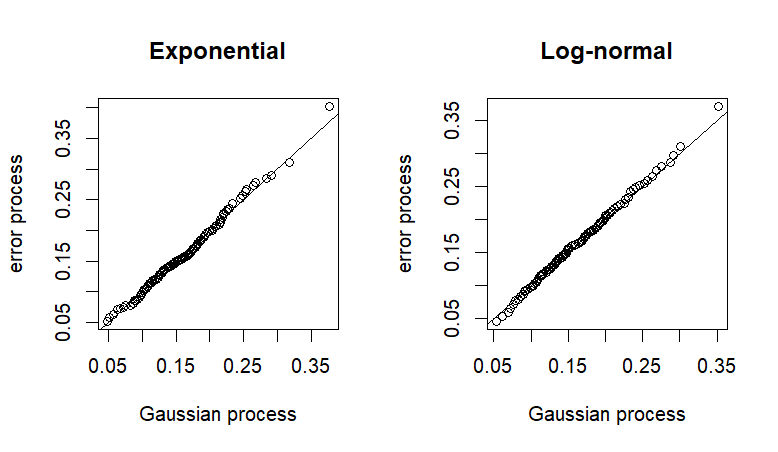}
    \captionsetup{font={small, stretch=1},skip=0cm}    
    \caption{\TPR{Q-Q plot between $\max_{1\leq k\leq p}\sup_{t\in\hat{\mathcal{T}}}|\mathcal{V}_k(t)|$ and $\max_{1\leq k\leq p}\sup_{t\in\hat{\mathcal{T}}}|\mathcal{E}_k(t)|$}}
    \label{fig: QQ plot}
\end{figure}

\subsubsection{GCV selection}
\label{sec:GCV selection}
\TPR{Our simulation results in Table \ref{tab:SCB coverage high main} show that the empirical coverage probabilities of the joint SCBs are close to the nominal levels when bandwidth $h_r$ is larger than $0.1$. Extremely small $h_r = 0.05$ is excluded by the range of the GCV selection criterion in \cite{bai2024difference} that we utilized and the $h_r$ chosen from GCV is larger than $0.1$, as displayed in Figure \ref{fig:GCV hist supp}. 
Together with Table \ref{tab:SCB coverage high main}, our results show that the bandwidth selected by the GCV criterion yields joint SCBs with stable and reasonably accurate coverage probabilities empirically.}
\begin{figure}[htpb]
    \centering
    \includegraphics[width=0.75\linewidth]{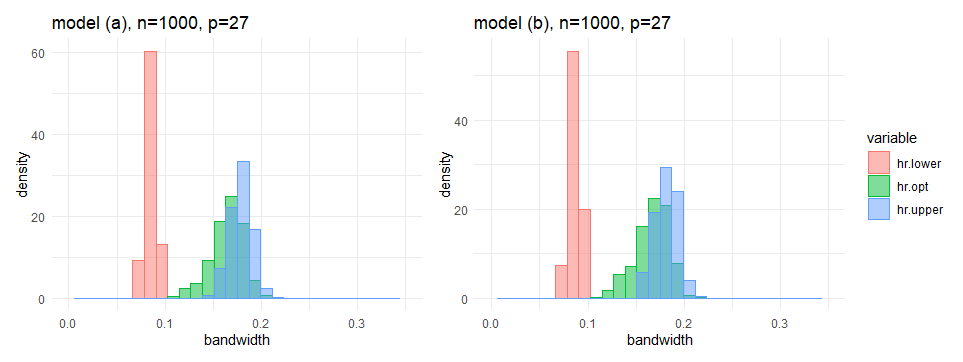}
    \captionsetup{font={small, stretch=1},skip=0cm} 
    \caption{\TPR{$h_r$(green) selected from GCV method and the upper(blue) and lower(red) bounds stand for the candidate range guided by the rule of thumb strategy proposed in \cite{bai2024difference}.}}
    \label{fig:GCV hist supp}
\end{figure}

\subsubsection{Simulation for penalized SCB}
\label{sec:penalization simulation}
\TPM{To verify the penalized SCB, we additionally consider two simulation scenarios by replacing the original strictly monotone functions in Section \ref{Simulation study} with general monotone functions:
\begin{itemize}
    \item $m_1^*(t) = \mathbb{I}_{\{t < \frac{1}{3}\}}(t) \cdot 2e^{1/3}+ \mathbb{I}_{\{t \geq \frac{1}{3}\}}(t) \cdot \left(t^2 - \frac{8t}{3} + \frac{25}{9}\right) \cdot e^{t}$,
    \item $m_2^*(t) = \mathbb{I}_{\{t < \frac{1}{3}\}}(t) \cdot \left(1 + \left(3\left(t - \frac{1}{3}\right)\right)^3\right) + \mathbb{I}_{\{\frac{1}{3} \leq t < \frac{2}{3}\}}(t) \cdot 1 + \mathbb{I}_{\{t \geq \frac{2}{3}\}}(t) \cdot \left(1 + \left(3\left(t - \frac{2}{3}\right)\right)^3\right)$,
    \item $m_3^*(t) = \mathbb{I}_{\{t < \frac{2}{3}\}}(t) \cdot 2 \cdot \cos\left(\frac{3\pi\left(t - \frac{2}{3}\right)}{4}\right) + \mathbb{I}_{\{t \geq \frac{2}{3}\}}(t) \cdot 2$.
\end{itemize}
For simulation, we generate $\mathbf{y}_i=\mathbf{m}^*(i/n)+\mathbf{e}_i$ where $\mathbf{m}^*(t)=\left(\mathbf{A}_1 ^\top m_1^*(t),\mathbf{A}_2^\top m_2^*(t),\mathbf{A}_3^\top m_3^*(t)\right)^\top$. The other settings such as $\mathbf{A}$ and $\mathbf{e}_i$ are the same as models (a) and (b) in Section \ref{Simulation study} of the main paper with bandwidths selected from GCV. Tabel \ref{tab:penal SCB supp} presents the simulated coverage probabilities and widths of our penalized monotone SCBs for no penalization and $C_1=0.1,0.2,0.3$ with $C_2=C_3=C_1^{19/8}$. The penalized SCBs achieve the nominal 90\% and 95\% levels while the original SCB without penalization ($C_1=0$) exhibits poor coverage probabilities when the trend functions contain flat segments. Our penalized SCB can ensure the monotone shape with high probability in theory. In our simulation study, the $\hat{\mathbf{m}}_{correct}$ remains monotone in $100\%$ of 400 simulations for no penalization and $C_1=0.3,0.2,0.1$. According to our proposed criterion for $C_1$, all simulation scenarios would select $C_1=0.3$.}

\begin{table}[H]
    \centering
      \setlength{\belowcaptionskip}{-0.5cm}
  \captionsetup{font={small}, skip=5pt}
    \caption{\TPM{Simulated coverage probabilities and widths of SCB penalized by different $C_1$}}
{\renewcommand{\arraystretch}{0.5}
    \begin{tabular}{cccccccccc}
    \hline
     $(n=500)$&  \multicolumn{4}{c}{Coverage} &  & \multicolumn{4}{c}{Width} \\
     \cline{2-5}\cline{7-10} $(p=27)$ & \multicolumn{2}{c}{Model(a)} & \multicolumn{2}{c}{Model(b)} & & \multicolumn{2}{c}{Model(a)} & \multicolumn{2}{c}{Model(b)} \\
    $C_1$  & 90\% & 95\% & 90\% & 95\% & & 90\% & 95\% & 90\% & 95\%  \\
    \hline
    0.3   & 0.865 & 0.950 & 0.870 & 0.955 &       & 0.3340 & 0.3700 & 0.3780 & 0.4185 \\
    0.2   & 0.838 & 0.933 & 0.845 & 0.955 &       & 0.3249 & 0.3607 & 0.3675 & 0.4074 \\
    0.1   & 0.810 & 0.928 & 0.820 & 0.943 &       & 0.3166 & 0.3522 & 0.3579 & 0.3975 \\
    no penalization     & 0.778 & 0.920 & 0.795 & 0.933 &       & 0.3090 & 0.3441 & 0.3495 & 0.3887 \\
    \hline
    \end{tabular}}
    \label{tab:penal SCB supp}
\end{table}

\subsubsection{Simulation for time-varying linear regression}
\label{sec: simulation for time-varying linear regression}
To examine our approach for time-varying coefficient linear model in Section \ref{Time-varying coefficient regression}, we consider the following model $y_i=m_1(i/n)+m_2(i/n)x_i+e_i,$ where $x_i=\sum_{j=0}^\infty c(t_i)^j\epsilon_{i-j}$ and $c(t)=1/4+t/2$. As for error process $e_i$, we also consider two scenarios:
\begin{description}
[itemsep=0pt,parsep=0pt,topsep=0pt,partopsep=0pt]
    \item[(c)] Locally stationary: $e_i=G(t_i,\Upsilon_i)=\sum_{j=0}^\infty a^j(t_i)\xi_{i-j}$ where $a(t)=0.5-(t-0.5)^2$;
    \item[(d)] Piecewise stationary: $e_i=G^{(0)}(t_i,\Upsilon_i)$ for $i/n\leq 1/3$ and $e_i=G^{(1)}(t_i,\Upsilon_i)$ for $i/n>1/3$ where $G^{(0)}(t_i,\Upsilon_i)=0.5G^{(0)}(t_i,\Upsilon_{i-1})+\xi_i$, $G^{(1)}(t_i,\Upsilon_i)=-0.5G^{(1)}(t_i,\Upsilon_{i-1})+\xi_i,$
\end{description}
where $\xi_l,\epsilon_l,l\in\mathbb{Z}$ are all i.i.d. standard normal variables. Table \ref{tab:SCB coverage regression supp} displays the simulated coverage probabilities for joint SCBs of time-varying coefficients $(m_1(t),m_2(t))^\top$ obtained by Algorithm \ref{alg:multi-dim} using various bandwidths. Similarly to Table \ref{tab:SCB coverage high main}, when using the GCV selected bandwidths, all the coverage probabilities of the induced SCBs are close to the nominal level. The results also show the performance of joint SCBs generated by Algorithm \ref{alg:multi-dim} is reasonably close to the nominal level in most $h_r$ for moderately large sample size $n=500,1000$ and becomes more robust with respect to $h_r$ as the sample size grows.  

\begin{table}[htpb]
  \centering
  \captionsetup{font={small}, skip=5pt}
  \caption{\TPR{Simulated coverage probabilities of SCB for regression coefficients $(m_1(t),m_2(t))^\top$}}
  \scalebox{0.85}{
    {\renewcommand{\arraystretch}{0.5}
    \begin{tabular}{ccccccc c cccccc}
    \hline
     & \multicolumn{6}{c}{\textit{Model (c)}}& & \multicolumn{6}{c}{\textit{Model (d)}} \\
    \cline{2-7}\cline{9-14}
          & \multicolumn{2}{c}{$n=300$} & \multicolumn{2}{c}{$n=500$} & \multicolumn{2}{c}{$n=1000$} & & \multicolumn{2}{c}{$n=300$} & \multicolumn{2}{c}{$n=500$} & \multicolumn{2}{c}{$n=1000$} \\
    $h_r$ & 90\%  & 95\%  & 90\%  & 95\%  & 90\%  & 95\% & & 90\%  & 95\%  & 90\%  & 95\%  & 90\%  & 95\% \\
    \hline
    GCV   & 0.910 & 0.960 & 0.880 & 0.935 & 0.898 & 0.940 &       & 0.905 & 0.958 & 0.863 & 0.938 & 0.900 & 0.953 \\
    0.05  & 0.955 & 0.978 & 0.558 & 0.678 & 0.928 & 0.965 &       & 0.993 & 1.000 & 0.535 & 0.660 & 0.828 & 0.908 \\
    0.1   & 0.938 & 0.960 & 0.930 & 0.958 & 0.943 & 0.978 &       & 0.958 & 0.983 & 0.885 & 0.930 & 0.935 & 0.970 \\
    0.15  & 0.963 & 0.980 & 0.915 & 0.963 & 0.915 & 0.950 &       & 0.925 & 0.965 & 0.883 & 0.935 & 0.910 & 0.958 \\
    0.2   & 0.935 & 0.970 & 0.893 & 0.940 & 0.895 & 0.945 &       & 0.870 & 0.945 & 0.878 & 0.940 & 0.893 & 0.955 \\
    0.25  & 0.910 & 0.960 & 0.880 & 0.930 & 0.878 & 0.933 &       & 0.860 & 0.915 & 0.865 & 0.940 & 0.875 & 0.935 \\
    0.3   & 0.843 & 0.918 & 0.863 & 0.923 & 0.855 & 0.925 &       & 0.848 & 0.920 & 0.860 & 0.935 & 0.880 & 0.925 \\
    0.35  & 0.823 & 0.893 & 0.840 & 0.913 & 0.868 & 0.925 &       & 0.850 & 0.908 & 0.873 & 0.930 & 0.888 & 0.930 \\
    \hline
    \end{tabular}%
      }}
    \label{tab:SCB coverage regression supp}
\end{table}%

\subsection{Temperature-sunshine analysis}\label{regressionsection}
In this section we investigate the relationship between sunshine duration and atmospheric temperature, especially how the former impacts the latter. We consider the following time-varying coefficient linear model for the maximum temperature and sunshine duration:
\begin{equation}
    T_{max,i} = m_0(t_i)+m_1(t_i)SD_i+e_i,\quad i=1,\dots,n,
    \label{tmax-SD}
\end{equation}
where $(T_{max,i})_{i=1}^n$ is the series of monthly maximum temperature and $(SD_i)_{i=1}^n$ represents the series of monthly sunshine duration. In model \eqref{tmax-SD}, we assume $m_1(t)$ is strictly increasing, as illustrated in Section \ref{Time-varying coefficient regression}. 




We investigate \eqref{tmax-SD} using the historical data from Heathrow and Lerwick, separately. By methodology proposed in Section \ref{Time-varying coefficient regression}, we obtain the monotone estimates as well as the corresponding SCB for the coefficient $m_1(t)$, which can be used to test whether the sunshine duration always has a positive effect on temperature. For this purpose, we consider the null hypothesis $H_0:\inf_{t\in \hat{\mathcal{T}}}m_1(t)\leq 0$, and test it using the SCB constructed by Algorithm \ref{alg:multi-dim}.
Alternative SCBs that could be used for testing $H_0$ include  \cite{zhou2010simultaneous} which imposes no monotone constraints, and the conservative monotone SCB obtained by modifying  \cite{zhou2010simultaneous} using order-preserving improving procedure, see for example \cite{chernozhukov2009improving}.

\TPM{In Figure \ref{fig:reg_heathrow}, we apply our Algorithm \ref{alg:multi-dim} for constructing SCBs on $m_1(t)$ in \eqref{tmax-SD} based on the data collected from Heathrow and Lerwick stations. 
In Figure \ref{fig:reg_heathrow}, SCBs in (a) are derived by \cite{zhou2010simultaneous} which does not consider the monotone constraint. SCBs in (b) and (c) are the order-preserving improvings of (a) via \cite{chernozhukov2009improving}, which are conservative. In specific, (b) uses isotonic regression, and (c) uses the monotone spline regression in \cite{tantiyaswasdikul1994isotonic} to improve unconstrained SCBs in (a). SCBs in (d) are generated by Algorithm \ref{alg:multi-dim}. The SCBs obtained by monotone spline in (c) are similar to isotonic regression in (b). The results demonstrate that unconstrained SCBs \cite{zhou2010simultaneous} and conservative SCBs obtained by isotonic regression \cite{brunk1969estimation} and monotone spline \cite{tantiyaswasdikul1994isotonic} fail to reject the null hypothesis $H_0$ at 90\% significance level for Heathrow station. The $p$-values are 0.965, 0.865, 0.864, respectively; see Table \ref{tab:p-value supp}.
Notice that \cite{zhou2010simultaneous} does not guarantee monotone constraints while \cite{chernozhukov2009improving} yields conservative coverage, as a consequence, testing procedures using the two SCBs for $H_0$ might lack power when the underlying time-varying coefficient is monotone as in our scenario. In contrast, our SCB based on the monotone rearranged estimator generated by Algorithm \ref{alg:multi-dim}, suggests a significant positive relationship between sunshine duration and temperature at 95\% confidence level in Heathrow (at $p$-value 0.046). All four considered SCBs fail to reject the null hypothesis in Lerwick.} 


The difference in the testing results between Heathrow and Lerwick in Figure \ref{fig:reg_heathrow} may be due to geographic reasons. This difference shows the temperature mechanism can vary on regional scale, even though on a global scale the sun is the fundamental source of heat on Earth. Specifically, Heathrow is situated adjacent to the city of London, where human activities and urbanization may have diminished the environment's capacity to adjust. As a result, sunshine duration has a significant impact on the weather patterns in Heathrow. On the other hand, Lerwick is located approximately 123 miles off the north coast of the Scottish mainland. Its climate is mainly influenced by maritime factors, which means the effect of sunshine duration on temperature is less pronounced compared with Heathrow. In addition, 
Table \ref{tab:addlabel} shows the detailed information for stations investigated in Section \ref{Empirical study}.
\begin{figure}[htb]
    \centering
    \begin{subfigure}{0.22\textwidth}
        \centering
        \includegraphics[width=\linewidth]{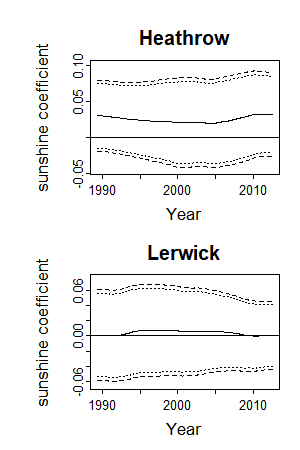}
        \caption{}
        \label{subfig:plot4}
    \end{subfigure}%
    \begin{subfigure}{0.22\textwidth}
        \centering
        \includegraphics[width=\linewidth]{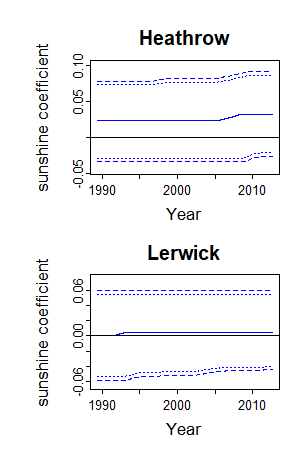}
        \caption{}
        \label{subfig:plot5}
    \end{subfigure}%
    \begin{subfigure}{0.22\textwidth}
        \centering
        \includegraphics[width=\linewidth]{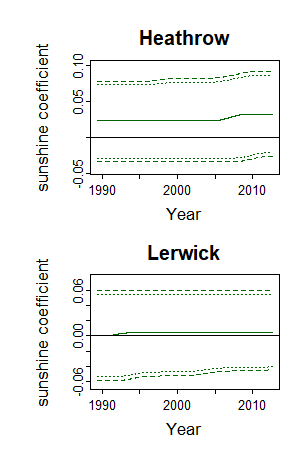}
        \caption{}
        \label{subfig:plot6}
    \end{subfigure}
    \begin{subfigure}{0.22\textwidth}
        \centering
        \includegraphics[width=\linewidth]{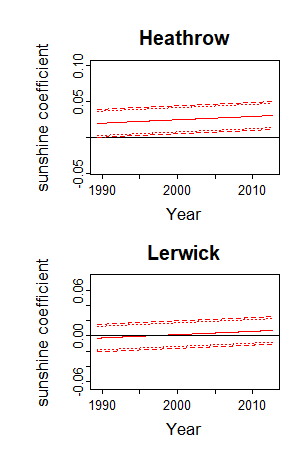}
        \caption{}
        \label{subfig:plot7}
    \end{subfigure}
    \captionsetup{font={small, stretch=1},skip=0cm}
     \caption{\TPM{90\% SCBs (dotted), 95\% SCBs (dashed) and estimations (solid) on $m_1(t)$ in model \eqref{tmax-SD} for Heathrow and Lerwick.}}
         \label{fig:reg_heathrow}
\end{figure}

\begin{table}[htbp]
  \centering
  \captionsetup{font={small, stretch=1},skip=0cm}
    \caption{\TPM{$p$-values for test associated with Figure \ref{fig:reg_heathrow}. Note: The `$\leq$' in (b) is due to its conservativeness.}} 
    \label{tab:p-value supp}
  {\renewcommand{\arraystretch}{0.5}
    \begin{tabular}{lcccc}
    \hline
    Model \eqref{tmax-SD} & \multicolumn{1}{c}{(a)} & \multicolumn{1}{c}{(b)} & \multicolumn{1}{c}{(c)} & \multicolumn{1}{c}{(d)} \\
    \hline
    Heathrow & 0.965 & $\leq$0.865 & $\leq$0.864 & 0.046 \\
    Lerwick & 1     & 1     & 1 & 1 \\
    \hline
    \end{tabular}}%
\end{table}%

\TPM{
The time-varying linear regression model \eqref{Time-varying_n} can include non-instantaneous covariate, which does not violate Assumptions \hyperref[(B1')]{(B1')}-\hyperref[(B5')]{(B5')}. Theorem \ref{thm:multi-SCB} is still valid if the covariate process $\mf x_i$ contains lagged terms. In specific, we consider the non-instantaneous model
\begin{equation}
    T_{max,i} = m_0(t_i)+m_1(t_i)SD_i+m_2(t_i)SD_{i-1}+e_i,\label{eq:tmax-SD lag supp}
\end{equation}
where $(T_{max,i})_{i=1}^n$ is the series of monthly maximum temperature and $(SD_i)_{i=1}^n$ represents the series of monthly sunshine duration. By the similar arguments in \eqref{tmax-SD}, we can obtain four kinds of SCBs to test $H_0: m_1(t)\leq 0$, and we present the associated p-values for (a) unconstrained SCBs \cite{zhou2010simultaneous}, (b) improved SCBs with isotonic regression, (c) improved SCBs with monotone spline, and (d) our monotone SCBs in Table \ref{tab:p-values lag supp}.
\begin{table}[H]
  \centering
  \captionsetup{font={small, stretch=1},skip=0cm}
    \caption{\TPM{$p$-values for testing $\inf_{t\in\hat{\mathcal{T}}}m_1(t)\leq 0$ in Model \eqref{eq:tmax-SD lag supp}. Note: The `$\leq$' in (b) and (c) is due to its conservativeness.}}
        \label{tab:p-values lag supp}
    {\renewcommand{\arraystretch}{0.5}
    \begin{tabular}{lcccc}
    \hline
    Model \eqref{eq:tmax-SD lag supp} & \multicolumn{1}{c}{(a)} & \multicolumn{1}{c}{(b)} & \multicolumn{1}{c}{(c)}  & \multicolumn{1}{c}{(d)} \\
    \hline
    Heathrow & 0.833& $\leq$0.555 & $\leq$0.551 & 0.005 \\
    Lerwick & 1     & 1     & 1 & 1\\
    \hline
    \end{tabular}}%
\end{table}%
Our monotone SCBs generated by Algorithm \ref{alg:multi-dim} suggest a significant positive relationship between sunshine duration and temperature in Heathrow for both Model \eqref{tmax-SD} ($p$-value 0.046) and Model \eqref{eq:tmax-SD lag supp} ($p$-value 0.005).
Interestingly, the $p$-values based on all four kinds of SCB become smaller after imposing the lagged term $SD_{i-1}$ in Heathrow. }

\TPM{To assess which model is better for analyzing the time-varying relationship between temperature and sunshine duration, we introduce the VIC information criterion proposed by \cite{zhang2012inference} which can consistently identify the true set of relevant predictors. Suppose that $D^*=\{1,\dots,p\}$ is the whole set of potential predictors. For a candidate subset $D\subset D^*$, we write $\mf x_{D,i}$ as the selected predictor and $\hat{\mf m}_D(t)$ is the estimated time-varying coefficient obtained by
\begin{equation}
    \left(\hat{\mathbf{m}}_{D}(t), \hat{\mathbf{m}}^{\prime}_{D}(t)\right)=:\underset{\eta_0, \eta_1 \in \mathbb{R}^{|D|}}{\arg \min }\left[\sum_{i=1}^n\left\{y_i-\mathbf{x}_{D,i}^{\top} \eta_0-\mathbf{x}_{D,i}^{\top} \eta_1\left(t_i-t\right)\right\}^2 K_{r}\left(\frac{t_i-t}{h_r}\right)\right],
    \label{local linear subset supp}
\end{equation} 
where $|D|$ is the cardinality of $D$. We utilize the VIC criterion in \cite{zhang2012inference} to select the model that minimizes
\begin{equation}
\label{eq:VIC supp}
    VIC(D)=\log\left\{\sum_{i=1}^n (y_i-\mf x_{D,i}^\top \hat{\mf m}_D(t))^2\right\}+\chi_n|D|,
\end{equation}
where $\chi_n$ is a tuning parameter and is recommended to be $\chi_n=n^{-2/5}$ in \cite{zhang2012inference}. The computation turns out that Model \eqref{tmax-SD} ($VIC=7.042$) is better than Model \eqref{eq:tmax-SD lag supp} ($VIC=7.113$) in Heathrow. In Lerwick, Model \eqref{tmax-SD} ($VIC=6.239$) is still better than Model \eqref{eq:tmax-SD lag supp} ($VIC=6.316$).}

\begin{table}[H]
  \centering
\setlength{\belowcaptionskip}{-0.5cm}
  \captionsetup{font={small}, skip=5pt}
  \caption{Detail information for stations investigated in Section \ref{Empirical study}}
  \scalebox{0.7}{
    {\renewcommand{\arraystretch}{0.5}
    \begin{tabular}{llllllll}
    \hline
    \textbf{Name}  & \textbf{Location} &       & \textbf{Name}  & \textbf{Location} &       & \textbf{Name}  & \textbf{Location} \\
    \hline    
        Aberporth  &  -4.56999, 52.13914  &       &     Durham  &  -1.58455, 54.76786  &       &     Sheffield  &  -1.48986, 53.38101  \\
        Armagh  &  -6.64866, 54.35234  &       &     Eskdalemuir  &  -3.206, 55.311  &       &     Stornoway Airport  &  -6.31772, 58.21382  \\
        Ballypatrick Forest  &  -6.15336, 55.18062  &       &     Heathrow  &  -0.44904, 51.47872  &       &     Sutton Bonington  &  -1.25, 52.8331  \\
        Bradford  &  -1.77234, 53.81341  &       &     Hurn   &  -1.83483, 50.7789  &       &     Tiree  &  -6.8796, 56.49999  \\
        Braemar  &  -3.39635, 57.00612  &       &     Lerwick  &  -1.18299, 60.13946  &       &     Valley  &  -4.53524, 53.25238  \\
        Camborne  &  -5.32656, 50.21782  &       &     Leuchars  &  -2.86051, 56.37745  &       &     Waddington  &  -0.52173, 53.17509  \\
        Cambridge NIAB  &  0.10196, 52.24501  &       &     Newton Rigg  &  -2.78644, 54.6699  &       &     Whitby  &  -0.62411, 54.48073  \\
        Cardiff Bute Park  &  -3.18728, 51.48783  &       &     Oxford  &  -1.2625, 51.76073  &       &     Wick Airport  &  -3.0884, 58.45406  \\
        Dunstaffnage  &  -5.43859, 56.45054  &       &     Shawbury  &  -2.66329, 52.79433  &       &     Yeovilton  &  -2.64148, 51.00586  \\
        \hline
    \end{tabular}}}%
  \label{tab:addlabel}%
\end{table}%

\bibliographystyle{agsm}
\begin{footnotesize}
\spacingset{0.5}
\bibliography{main}
\end{footnotesize}

\end{document}